\documentclass[oneside,a4paper,11pt,reqno,titlepage]{amsbook}

\usepackage[applemac]{inputenc}
\usepackage[T1]{fontenc}
\usepackage{amssymb}
\usepackage{pb-diagram}
\usepackage{microtype}
\usepackage{graphics,url,calc}
\usepackage[shortcuts]{extdash}
\usepackage{mathtools}

\usepackage[figurewithin=none]{caption}
\usepackage{subfig}

\usepackage{tikz}
\usepgflibrary{arrows}
 
\usepackage{loc_def-2}
\usepackage{mypaper}

\title{On Bredon (Co-)Homolo\-gical Dimensions of Groups}

\author{Martin Fluch}
\date{\today}

\newlength{\bskipabstract}
\newlength{\bskipmaintext}
\newenvironment{myabstract}[1]%
{\vspace*{-1cm}\thispagestyle{empty}
\begin{center}\setlength{\baselineskip}{1.4\baselineskip}
#1
\end{center}\vspace{1.3cm}}%
{\cleardoublepage}

\newlength{\theorempreskipamount}
\newlength{\theorempostskipamount}
\newtheoremstyle{plain}%
 {\theorempreskipamount}%
 {\theorempostskipamount}%
 {\itshape}%
 {}%
 {\bfseries}%
 {.}%
 {0.5em}%
 {}%

\newtheoremstyle{definition}%
 {\theorempreskipamount}%
 {\theorempostskipamount}%
 {\normalfont}%
 {}%
 {\bfseries}%
 {.}%
 {0.5em}%
 {}%

\theoremstyle{plain}

\newtheorem*{FJC}{Farrell--Jones Conjecture}
\newtheorem*{EGC}{Eilenberg--Ganea Conjecture}
\newtheorem{theorem}{Theorem}
\newtheorem{proposition}[theorem]{Proposition}
\newtheorem{lemma}[theorem]{Lemma}
\newtheorem{corollary}[theorem]{Corollary}

\newtheorem*{conjecture*}{Conjecture}
\newtheorem{plain-number}[theorem]{}

\theoremstyle{definition}

\newtheorem{DEF}[theorem]{Definition}
\newtheorem{example}[theorem]{Example}
\newtheorem{examples}[theorem]{Examples}
\newtheorem{remark}[theorem]{Remark}
\newtheorem{claim}{Claim}
\newtheorem*{claim*}{Claim}

\newcommand{\OfG}{\OC{G_{1}}{\frakF}}
\newcommand{\OgG}{\OC{G_{2}}{\frakG}}

\newcommand{\ModOfG}{\text{$\Mod$-$\OfG$}}

\newcommand{\ModOgG}{\text{$\Mod$-$\OgG$}}

\newcommand{\cll}{{\scriptstyle\bf L}}

\newcommand{\defeq}{\mathrel{\vcentcolon =}}

\numberwithin{equation}{chapter}
\numberwithin{theorem}{chapter}

\begin{document}

\nocite{schubert-70, brown-82, serre-80, segal-83, dicks-89, 
massey-91, bieri-81, bredon-67, robinson-96, rotman-79,
magnus-76, kawakubo-91}
\nocite{juan-pineda-06, kochloukova-10, martinez-perez-02}
\nocite{luck-00, luck-05, luck-05a, luck-12, luck-09}
\nocite{baumslag-62, farrell-93, flores-05, gildenhuys-79, 
kropholler-09, symonds-05, connolly-06, farley-10, lafont-07,
manion-08, eilenberg-57}


\begin{titlepage}
    \centering
    
    \vspace*{-2cm}
    
    \includegraphics[scale=0.46]{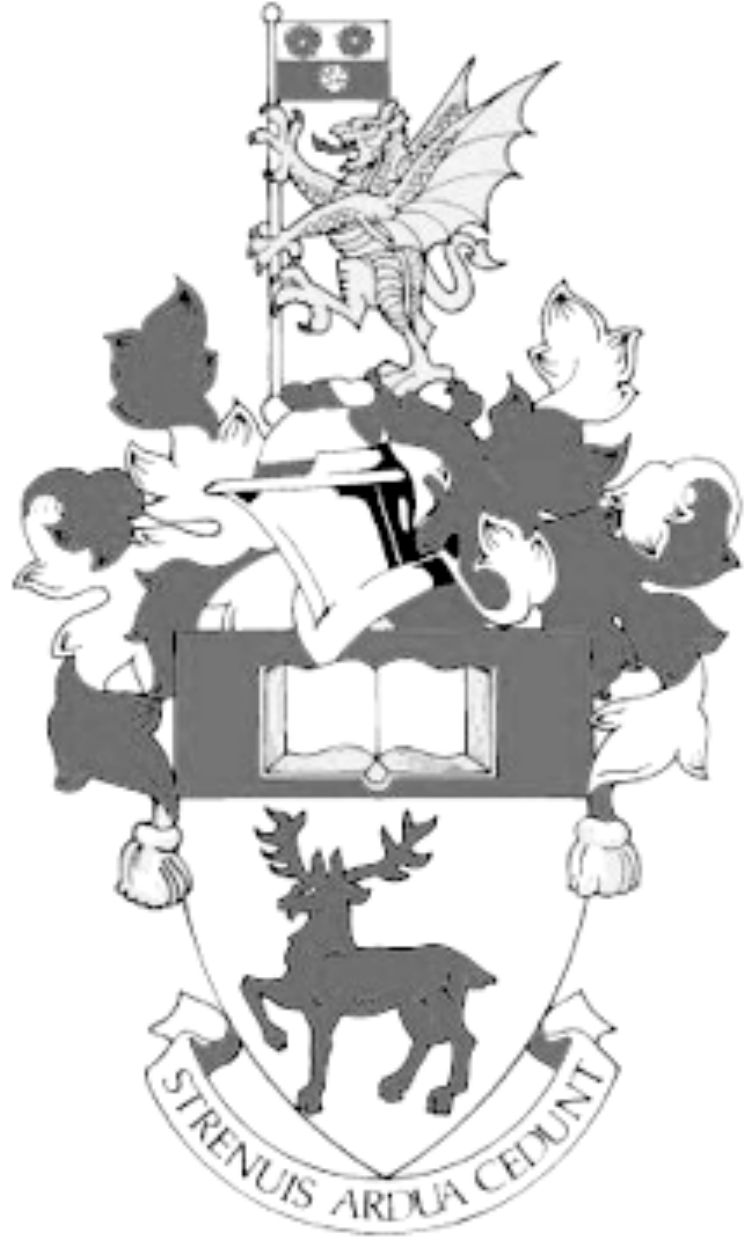}
    
    \vspace{1cm}
    
    UNIVERSITY OF SOUTHAMPTON
    
    \vspace{0.33cm}
    
    FACULTY OF SOCIAL AND HUMAN SCIENCES
    
    \vspace{0.33cm}
    
    School of Mathematics
    
    \vspace{2cm}
    
    {\huge 
    \textsc{\textbf{On Bredon (Co-)Homological\\[1ex] 
    Dimensions of Groups}}}
    
    \vspace{2.2cm}
    
    {\Large Martin Georg Fluch}
    
    \vspace{\fill}
    
    Supervisor: Dr.~Brita Nucinkis
    
    \vspace{4.3cm}
    
    Thesis for the degree of Doctor of Philosophy
    
    \vspace{1.2cm}
    
    January 2011
        
    \vspace*{-1cm}
\end{titlepage}


\setcounter{page}{0}


\cleardoublepage

\addtolength{\textheight}{-0.5cm}
\addtolength{\footskip}{0.75cm}

\pagenumbering{roman}


\begin{myabstract}{%
    UNIVERSITY OF SOUTHAMPTON
    \\
    \underline{ABSTRACT}
    \\
    FACULTY OF SOCIAL AND HUMAN SCIENCES
    \\
    SCHOOL OF MATHEMATICS
    \\[\medskipamount]
    \underline{Doctor of Philosophy}
    \\
    ON BREDON (CO-)HOMOLOGICAL DIMENSIONS OF GROUPS
    \\
    by Martin Georg Fluch}
    
    \setlength{\baselineskip}{\bskipabstract}
    
    The objects of interest in this thesis are classifying spaces
    $E_{\frakF}G$ for discrete groups $G$ with stabilisers in a given
    family $\frakF$ of subgroups of $G$. The main focus of this 
    thesis lies in the family $\Fvc(G)$ of virtually cyclic subgroups 
    of~$G$. A classifying space for this specific family is denoted 
    by $\uu EG$. It has a prominent appearance in the Farrell--Jones 
    Conjecture. Understanding the finiteness properties of $\uu EG$ 
    is important for solving the conjecture.
    
    This thesis aims to contribute to answering the following question
    for a group $G$: what is the minimal dimension a model for $\uu
    EG$ can have?  One way to attack this question is using methods in
    homological algebra.  The natural choice for a cohomology theory
    to study $G$-CW-complexes with stabilisers in a given family
    $\frakF$ is known as Bredon cohomology.  It is the study of
    cohomology in the category of $\OFG$-modules.  This category
    relates to models for $E_{\frakF}G$ in the same way as the
    category of $G$-modules relates to the study of universal covers
    of Eilenberg--Mac~Lane spaces $K(G,1)$.
    
    In this thesis we study Bredon (co-)homological dimensions of
    groups.  A major part of this thesis is devoted to collect
    existing homological machinery needed to study these dimensions
    for arbitrary families $\frakF$.  We contribute to this
    collection.
    
    After this we turn our attention to the specific case of $\frakF =
    \Fvc(G)$.  We derive a geometric method for obtaining a lower
    bound for the Bredon (co\=/)homological dimension of a group
    $G$ for a general family $\frakF$, and subsequently show how to
    exploit this method in various cases for $\frakF=\Fvc(G)$.
    
    Furthermore we construct model for $\uu EG$ in the case that $G$
    belongs to a certain class of infinite cyclic extensions of a
    group $B$ and that a model for~$\uu EB$ is known.  We give bounds
    on the dimensions of these models.  Moreover, we use this
    construction to give a concrete model for $\uu EG$, where $G$ is a
    soluble Baumslag--Solitar group.  Using this model we are able to
    determine the exact Bredon (co-)homological dimensions of these
    groups.
    
    The thesis concludes with the study of groups $G$ of low Bredon 
    dimension for the family $\Fvc(G)$ and we give a 
    classification of countable, torsion-free, soluble groups which 
    admit a tree as a model for $\uu EG$.    
\end{myabstract}

\cleardoublepage

\setlength{\baselineskip}{\bskipmaintext}


\tableofcontents  

\cleardoublepage


\mymakeschapterhead{Declaration of Authorship}

I, Martin Fluch, declare that the thesis entitled
\emph{``\myinserttitle''} and the work presented in the thesis are
both my own, and have been generated by me as the result of my own
original research.  I confirm that:
\begin{itemize}
    \item  this work was done wholly while in candidature for a 
    research degree at this university;
    
    \item  no part of this thesis has previously been submitted for a 
    degree or any other qualification at this university or any other 
    institution;

    \item  where I have consulted the published work of others, this 
    is always clearly attributed;

    \item  where I have quoted from the work of others, the source is 
    always given. With the exception of such quotations, this thesis 
    is entirely my own work;

    \item  I have acknowledged all main sources of help;

    \item  where the thesis is based on work done by myself jointly 
    with others, I have made clear exactly what was done by others 
    and what I have contributed myself;

    \item parts of this work have been published as: ``Classifying
    Spaces with Virtually Cyclic Stabilisers for Certain Infinite
    Cyclic Extensions'', to appear in J.~Pure Appl.~Algebra.
\end{itemize}

\vspace{1.7cm} 
\begin{tabbing}
\hspace*{0.40\textwidth} \= Signed:
\= 
\raisebox{-2.3mm}{\includegraphics[width=4.6cm]{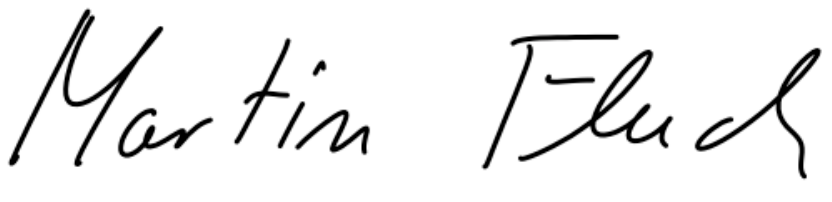}}
\\[0.4cm]
\> Date: \> January 10, 2011
\end{tabbing}

\addtolength{\textheight}{0.5cm}
\cleardoublepage
\addtolength{\footskip}{-0.5cm}


\mymakeschapterhead{Acknowledgements}

First of all, I want to thank my supervisor Brita Nucinkis and my
advisor Bernhard Köck for the ongoing and patient support and guidance
during my three years postgraduate studies.  It has been a pleasure to
work with them and I am grateful for the opportunity this has given to
me.

I would like to thank Armando Martino and Ashot Minasyan for many
inspiring conversations from which I learned so much!  Some important
ideas in this thesis grew out of the seeds planted this way in my
head.  I also want to thank Daniel Juan-Pineda, Ian Leary and Holger
Reich for their helpful comments.  Furthermore, my thanks belong to
Aditi Kar, Daniel Nucinkis, Vesna Perisic, David Singerman, Alexander
Stasinski and Christopher Voll from the School of Mathematics for all
the help and support I have got from them.

I wish to express my gratitude to Martin Dunwoody for reading the
first draft of my thesis thoroughly.  I also thank Tom Norton who
patiently read so many texts of mine and improved their linguistic 
style.

Furthermore, I wish to thank my fellow postgraduate students for
making my three years at the University of Southampton so enjoyable as
they have been.  I remember so many interesting conversations about
mathematics and anything under the sun which have taken place at the
office, the nearby pubs and elsewhere.  In particular, my gratitude
goes to Alex Bailey, Helena Fischbacher-Weitz, Giovanni Gandini, Ana
Khukhro, Ramesh Satkurunath, Richard Slessor, Martin Sell, Robert
Snocken and Glenn Ubly.

I wish to thank the School of Mathematics at the University of
Southampton for the generous financial support I received through the
three years of studies.

Finally I would like to thank my family for their unquestioned support
in good and difficult times.  Without their support I would have never
accomplished what I did!


\addtolength{\textheight}{-0.5cm}

\cleardoublepage

\addtolength{\footskip}{0.5cm}

\pagenumbering{arabic}

%
%

\chapter*{Introduction}

%
%

\section{Classifying Spaces and Bredon (Co-)Homology of Groups}

Classifying spaces and their finiteness conditions form an important 
part of various areas in pure mathematics such as group theory, 
algebraic topology and geometric topology.

Given a group $G$ and a non-empty family~$\frakF$ of subgroups of $G$
which is closed under conjugation and finite intersections, one can
consider the homotopy category of $G$-CW-complexes with stabilisers in
$\frakF$.  This category is known to have terminal objects, see for
example~\cite[p.~275]{luck-05}.  A terminal object in this
category is called a \emph{classifying space} of $G$ for the family
$\frakF$ or alternatively, a \emph{model for $E_{\frakF}G$.}

If $\frakF = \{1\}$ is the trivial family of subgroups, then the
universal cover $EG$ of an Eilenberg--Mac~Lane space $K(G,1)$ is a
model for $E_{\frakF}G$.  If $\frakF = \Ffin(G)$ is the family of
finite subgroups of $G$, then a model for $E_{\frakF}G$ is also known
as the \emph{universal space for proper actions}.  This space is
commonly denoted by $\underline EG$ and it has a prominent appearance
as the geometric object in the Baum--Connes Conjecture.

Recently the classifying space $\uu EG$ of $G$ for the family
$\Fvc(G)$ of virtually cyclic subgroups of $G$ has caught the interest
of the mathematical community (recall that a group is called
\emph{virtually cyclic} if it contains a cyclic subgroup of finite
index).  The reason for this is that the classifying space $\uu EG$
appears on the geometric side of the \emph{Farrell--Jones Conjecture}
for Algebraic $K$- and $L$-Theory.  This conjecture has originally
been stated by Thomas Farrell and Lowell Jones in 1993 in their famous
paper~\cite{farrell-93}.

Let $R$ be a ring with unit and involution.  There exists $G$-homology
theories
\begin{equation*}
    H^{G}_{n}(\?; \mathbf{K}_{R}) \qquad \text{and} \qquad
    H^{G}_{n}(\?; \mathbf{L}^{\langle-\infty\rangle}_{R})
\end{equation*}
in the sense of~\cite[pp.~738f.]{luck-05a} such that, \pagebreak[2]
if evaluated at a singleton space~$\{ * \}$, we recover the algebraic
$K$- and $L$-groups of the group ring $RG$.  That is
\begin{equation*}
    H^{G}_{n}(\{*\};
    \mathbf{K}_{R}) \isom K_{n}(RG)
    \qquad \text{and} \qquad
    H^{G}_{n}(\{*\};
    \mathbf{L}^{\langle-\infty\rangle}_{R}) \isom L^{\langle-\infty
    \rangle}_{n}(RG)
\end{equation*}
for all $n\in \Z$~\cite[p.~735]{luck-05a}.  Now the Farrell--Jones
Conjecture makes the following prediction.

\begin{FJC}
    \cite[p.~736]{luck-05a}
    The \emph{assembly maps}
    \begin{align*}
	A_{\text{vc}}\: & H^{G}_{n}(\uu EG; \mathbf{K}_{R}) \to
	H^{G}_{n}(\{*\}; \mathbf{K}_{R}) 
	\\
	A_{\text{vc}}\: & H^{G}_{n}(\uu EG;
	\mathbf{L}^{\langle-\infty\rangle}_{R}) \to H^{G}_{n}(\{*\};
	\mathbf{L}^{\langle-\infty\rangle}_{R})
    \end{align*}
    induced by the projection $\uu EG\to \{*\}$ are isomorphisms for 
    all $n\in \Z$.
\end{FJC}

The codomains of the assembly maps are the groups which we want to 
compute but whose computation is known to be difficult. On the other 
hand, the domains of the assembly maps are easier to calculate as one 
can apply methods from Algebraic Topology such as spectral sequences 
and Chern characters to it~\cite{luck-05a}.

The Farrell--Jones Conjecture is known to imply numerous other famous
conjectures from different fields of pure mathematics, including the
\emph{Bass Conjecture} in Algebraic $K$-Theory, the \emph{Borel
Conjecture} in Geometric Topology, the \emph{Kaplansky Conjecture} in
Group Theory and the \emph{Novikov Conjecture} in 
Topology~\cite{luck-05a}.

Progress in studying the Farrell--Jones Conjecture relies much on
understanding finiteness conditions of the classifying space $\uu EG$.
Models for~$EG$ and~$\underline EG$ have been studied extensively, see
for example~\cite{luck-05}.  However, there is not much known yet
about the classifying space for the family of virtually cyclic
subgroups.  Classes of groups that are understood are word hyperbolic
groups~\cite{juan-pineda-06}, virtually polycyclic
groups~\cite{luck-12}, relatively hyperbolic groups~\cite{lafont-07}
and
$\operatorname{CAT}(0)$-groups~\cite{luck-09, farley-10}.
Furthermore, there exist general constructions for finite index
extensions~\cite{luck-00a} and direct limits of groups~\cite{luck-12}.
Some more specific constructions can also be found
in~\cite{connolly-06} and~\cite{manion-08}.

The focus in this thesis lies on groups $G$ which admit a finite
dimensional model for $\uu EG$.  This leads to the study of the Bredon
geometric dimension of a group $G$ with respect to the family
$\Fvc(G)$, which by definition is the least integer $n$ (or~$\infty$)
such that there exists an $n$-dimensional model for $\uu EG$.

Homological methods provide suitable tools to study finiteness
conditions of classifying spaces.  The natural choice of a homology
theory for $G$-CW-complexes with stabilisers in a given family
$\frakF$ is the Bredon cohomology of groups.  This homology theory has
been introduced for finite groups by Glen Bredon in~\cite{bredon-67} 
and it has been generalised to arbitrary groups and arbitrary 
families of subgroups by Lück~\cite{luck-89}. Related to the Bredon 
geometric dimension of a group is the concept of the Bredon
homological 
and cohomological dimension of a group which is defined in a purely 
algebraic way.

We aim in this thesis to utilise the algebraic Bredon machinery as
far as possible in order to study the Bredon geometric dimensions of
groups~$G$ with respect to the family $\Fvc(G)$. 


\section{Structure of this Thesis}

The first three chapters in this thesis do not specialise to the 
family of virtually cyclic subgroups but introduce the theory in a 
more general setting.

In Chapter~\ref{ch:basics} the category of Bredon modules over the
orbit category~$\OFG$ is introduced.  Free and projective Bredon
modules are constructed.  It is explained how the categorical tensor
product gives rise to a tensor product over the orbit category $\OFG$
which is the Bredon analogue to the tensor product over the group ring
$\Z G$ in the category of $G$-modules.  This tensor product is used to
define flat Bredon modules.  The chapter is finished with the
definition of the restriction, induction and coinduction functors and
a summary of their basic properties.

In Chapter~\ref{ch:classifying-spaces}, $G$-CW-complexes and
classifying spaces with stabilisers in a given family $\frakF$ are
defined.  It is explained how one derives from the categorical
definition the homotopy characterisation of a classifying space. 
Geometric finiteness conditions are discussed and their relationship 
to algebraic properties in the corresponding category of Bredon 
modules.

Chapter~\ref{ch:bredon-dimensions} introduces the notion of Bredon
(co\=/)homological dimension.  The relationship between the algebraic
and geometric Bredon dimensions is studied as well as how the
algebraic Bredon dimensions depend on the family of subgroups.  We
deduce the algebraic analogue to a result from Lück and
Weiermann~\cite{luck-12} which gives a lower bound for the dimension
when passing to a larger group, see Theorem~\ref{thrm:cdF-G-vs-cdG-G}
and Theorem~\ref{thrm:hdF-G-vs-hdG-G}.  In the same chapter we
construct a standard resolution and derive an algebraic analogue to
a result in~\cite{luck-12} which gives upper bounds on the Bredon
dimensions of direct unions of groups, see
Proposition~\ref{prop:standard-resolution} and
Theorem~\ref{thrm:direct-union-result}; these results are a
generalisation of work by Nucinkis~\cite{nucinkis-04} which she has
carried out for the family of finite subgroups.  In
Section~\ref{sec:tensor-product-resolutions} we study the tensor
product of projective resolutions which gives us the possibility to
derive an upper bound for the Bredon cohomological dimension of direct
products of groups, see Theorem~\ref{thrm:cd-for-direct-products}. 
Finally we derive a Künneth formula for Bredon homology, see 
Theorem~\ref{thrm:kunneth-formula-for-homology}.

In Chapter~\ref{ch:dimensions-for-Fvc} we begin to specialise to the
family of virtually cyclic subgroups.  Using geometric methods we
derive a lower bound for the Bredon (co-)homological dimension of a
group $G$ (still with respect to a general family of subgroups).
Using this tool we use known classifying spaces for the family of
virtually cyclic groups in order to calculate the Bredon
(co\=/)homological dimensions $\uuhd G$ and $\uucd G$ for various
groups.  The results include the dimensions for $\Z^{2}$, free groups
and the fundamental groups of finite graphs of finite groups.  We also
study of the Bredon cohomological dimension $\uucd G$ for nilpotent
groups.  The chapter concludes by investigating under which conditions
an elementary amenable group $G$ admits a finite dimensional model for
$\uu EG$.

In the next chapter we turn our attention to the construction of a
concrete model for $\uu EG$ where $G = B\rtimes \Z$ is an infinite
cyclic extension of a group $B$.  Under certain conditions on the
action of $\Z$ on $B$, we can make a classifying space of $G$ from a
model for $\uu EB$.  The construction relies on a generalisation of a
result by Juan-Pineda and Leary~\cite{juan-pineda-06}, see
Proposition~\ref{prop:JPL}.  The class of groups for which this result
is applicable include certain HNN-extensions with abelian or free base
group and standard wreath products by~$\Z$, see
Section~\ref{sec:examples}.  We calculate the algebraic and geometric
Bredon dimensions of the soluble Baumslag--Solitar groups $BS(1,m)$,
$m\in\Z\setminus \{0\}$, with respect to the family of virtually
cyclic subgroups, see Theorem~\ref{thrm:dim-soluble-BS}. We end 
this chapter by showing that some of the key ideas of this 
chapter can be applied successfully in other settings than infinite 
cyclic extensions. Namely, we use them to calculate the least 
dimension a model for $\uu EG$ can have when $G$ is a free product.

The final chapter of this thesis is an attempt to study and classify
groups with low Bredon dimension with respect to the family of
virtually cyclic subgroups.  Using the result of
Theorem~\ref{thrm:dim-soluble-BS} and a classification result by
Gildenhuys~\cite{gildenhuys-79} we classify countable, torsion-free,
soluble groups $G$ which have Bredon geometric dimension $1$ with
respect to the family of virtually cyclic subgroups, see
Theorem~\ref{thrm:classification}.


\section{Notation, Conventions and Preliminaries}

The set of natural numbers is denoted by $\N$ and $0$ is considered
to be a natural number.  The group of integers is $\Z$, the field of
rational numbers is denoted by $\Q$, the field of real numbers is
denoted by $\R$ and the field of complex numbers is denoted by $\C$.
Rings are always assumed to have a unit. If $G$ is a group and $R$ a 
ring, then $RG$ denotes the \emph{group ring} which consists of all 
formal $R$-linear combinations of elements in $G$.

If $a,b\in \R\cup \{\pm \infty\}$, then $[a,b]$ denotes the closed
interval
\begin{equation*}
    [a,b] \defeq \{ x\in \R : a\leq x\leq b\}.
\end{equation*}

As a topological space $\R$ is considered to have the standard
topology obtained from the Euclidian metric. Similarly $\C$ has the 
topology of the underlying Euclidian space $\R^{2}$.

If $n\in\N$, then the $(n-1)$-sphere $S^{n-1}$ and the $n$-disc 
$D^{n}$ are the subspaces
\begin{align*}
    S^{n-1} & \defeq \{ x\in \R^{n} : |x|=1\},  \\
    D^{n}   & \defeq \{ x\in \R^{n} : |x|\leq 1\}.
\end{align*}
We set $S^{-1}\defeq\emptyset$.  The $1$-sphere $S^{1}$ can be
identified with multiplicative group of complex numbers $z$ with
$|z|=1$ and this multiplicative structure make~$S^{1}$ into a
topological group.

Throughout this thesis we are working in the convenient category of
\emph{compactly generated topological spaces} in the sense
of~\cite{steenrod-67}.  By definition a subset $A\subset X$ of a
compactly generated space $X$ is closed in $X$ if and only if $A\cap
K$ is closed in $X$ for every compact subset $K$ of $X$.  Locally
compact spaces are compactly generated.

We use the following notation for categories: $\Set$ denotes the
category of sets and $\Ab$ denotes the category of abelian groups.  If
$R$ is a ring, then $\ModR$ ($\RMod$) denotes the category of right
(left) $R$-modules.  In the special case that $R$ is the group ring
$\Z G$ we denote this category by $\ModG$ ($\GMod$); the objects in
this category are called $G$-modules.

We use the symbols $\prod$ and $\coprod$ to denote the product and 
coproduct in a category. In particular $\prod$ denotes the cartesian 
product and $\coprod$ denotes the disjoint union in the category of 
sets.

We assume in this thesis that the reader is familiar with the basic
concepts of transformation groups~\cite{kawakubo-91}, category
theory~\cite{mac-lane-98}, and homological algebra~\cite{weibel-94}.

Furthermore we assume that the reader is familiar with the classical
cohomology of groups and classical and cohomological finiteness
conditions of groups~\cite{brown-82, bieri-81}.  In particular, we
denote by $\hd G$ the \emph{homological dimension} of a group $G$, by
$\cd G$ its \emph{cohomological dimension} and by $\gd G$ its
\emph{geometric dimension.} For virtually torsion-free groups $G$ we
have the notion of \emph{virtual cohomological dimension} and this 
dimension is denoted by~$\vcd G$, see~\cite[pp.~225f.]{brown-82}.

\vspace*{3em}

\centerline{\textbf{Note}}

\smallskip

This is a revised version of the authors thesis.  It differs slightly
from the version as it has been accepted by the University of
Southampton in~2011.  In this revised version minor mistakes have been
corrected and small adjustments have been made to increase the
readability.

\hfill Bielefeld, August~31, 2012

%
%

\chapter{The Category of Bredon Modules}
\label{ch:basics}

%
%

\section{Families of Subgroups}
\label{sec:families-of-subgroups}

\begin{DEF}
    Let $G$ be a group.  A set $\frakF$ of subgroups of $G$ is called
    a \emph{family} if it is non-empty and closed under conjugation.
    We say that $\frakF$ is a \emph{semi-full} family if $H\cap K\in
    \frakF$ for any $H,K\in \frakF$.  We say that $\frakF$ is a
    \emph{full} family if $\frakF$ is closed under taking subgroups.
\end{DEF}

\begin{example}
    Commonly used families are the following:
    \begin{enumerate}
        \item
        \label{enum:families-triv}
        the \emph{trivial} family $\{ 1\}$ which consists 
        of the trivial subgroup only;
    
        \item  the family $\Ffin(G)$ of finite subgroups of $G$;
    
        \item  the family $\Fvc(G)$ of virtually cyclic subgroups of 
	$G$;
	
	\item
	\label{enum:families-all}
	the family $\Fall(G)$ of all subgroups of $G$;
	
	\item given a non-empty $G$-set $X$ we have the family
	\begin{equation*}
	   \frakF(X) \defeq \{ G_x : x\in X\}
	\end{equation*}
	 of stabilisers of~$X$.
    \end{enumerate}
    Note that the examples~(\ref{enum:families-triv})
    to~(\ref{enum:families-all}) are full families of subgroups of
    $G$.  However, the family $\frakF(X)$ is in general neither
    subgroup closed or even intersection closed.
\end{example}

There are different common constructions how to obtain a new family 
of subgroups from a given one. In what follows we list those which 
appear in this thesis.

If $\frakF$ is a family of subgroups of $G$ and $K$ a subgroup of $G$ 
then
\begin{equation*}
    \frakF \cap K \defeq \{ H\cap K : H\in \frakF\}
\end{equation*}
is a family of subgroups of $K$ provided. If $\frakF$ is a 
(semi-)full family of subgroups of $G$, then $\frakF\cap K$ is a 
(semi-)full family of subgroups of $K$.

Given two groups $G_{1}$ and $G_{2}$ and families $\frakF_{1}$ and
$\frakF_{2}$ of subgroups of $G_{1}$ and $G_{2}$ respectively we
define their \emph{cartesian product} $\frakF_{1}\times \frakF_{2}$ to
be the set
\begin{equation*}
    \frakF_{1}\times \frakF_{2} \defeq  \{ H_{1}\times H_{2} : H_{1}\in 
    \frakF_{1} \text{ and } H_{2}\in \frakF_{2}\}.
\end{equation*}
This is a family of subgroups of the group $G_{1}\times G_{2}$.  If
$\frakF_{1}$ and $\frakF_{2}$ are semi-full families of subgroups, 
then so is $\frakF_{1}\times \frakF_{2}$. But in general it is not 
true that the cartesian product of two full families is again a full 
family: not every subgroup $K$ of $H_1\times H_2\in \frakF_1	\times
\frakF_2$ is equal to $K_1\times K_2$ for some $K_i\in \frakF_i$.

Given an arbitrary family $\frakF$ of a group $G$ we can always 
\emph{complete} it to a full family of subgroups of $G$. This 
\emph{completion} is denoted by $\bar \frakF$ and is by definition
\begin{equation*}
    \bar{\frakF} \defeq  \{ H\leq G : H\leq K \text{ for some $K\in 
    \frakF$}\}.
\end{equation*}
This is by construction the smallest full family of subgroups of $G$
which contains the family $\frakF$.

\begin{DEF}
    A \emph{pair} $(\frakG, \frakF)$ of families of subgroups of $G$ 
    consists of two families $\frakF$ and $\frakG$ of subgroups of 
    $G$ with $\frakF\subset\frakG$. A pair $(\frakG, \frakF)$ of 
    families of subgroups is called \emph{semi-full} (\emph{full}) if 
    both $\frakF$ and $\frakG$ are semi-full (full).
\end{DEF}

%
%

\section{The Orbit Category}
\label{sec:orbit-category}

\begin{DEF}
    Let $\frakF$ be a family of subgroups of $G$.  Then the
    \emph{orbit category} $\OFG$ is the following small category.  The
    objects of $\OFG$ are homogeneous $G$-spaces $G/H$ with $H\in
    \frakF$ and the morphisms of $\OFG$ are $G$-maps.  In the case
    that $\frakF=\Fall(G)$ we write $\calO G$ for the orbit category.
\end{DEF}

Given two subgroups $H$ and $K$ of $G$ we denote the set of all
$G$-maps from $G/H$ to $G/K$ by $[G/H, G/K]_{G}$.  The set $[G/H,
G/H]_{G}$ is a monoid in general and we denote its identity element 
either by $\id$ or $1$.

A $G$-map $f\: G/H \to G/K$ is characterised by its value on the 
coset $H$. If $f(H) = xK$ for some $x\in G$, then the condition 
that $f$ is a $G$-map implies
\begin{align*}
    xK\in (G/K)^{H} 
    & =
    \{ xK\in G/K : hxK = xK \text{ for all $h\in H$}\}
    \\
    & = \{xK \in G/K : H^{x} \leq K\}.
\end{align*}
Conversely, given any $xK\in (G/K)^{H}$, there exists a unique 
$G$-map $f\: G/H\to G/K$ with $f(H) = xK$. Therefore we have a 
bijective correspondence
\begin{equation}
    \label{eq:G-map-identification}
    [G/H,G/K]_{G} \isom (G/K)^{H}
\end{equation}
given by $f\mapsto f(H)$.

Therefore we can label any $G$-map $f$ between homogeneous $G$-spaces 
as follows: we denote by $f_{x,H,K}$ the unique $G$-map $f\: G/H \to 
G/K$ which maps $H$ to $xK$. With this notation two $G$-maps 
$f_{x,H,K}$ and $f_{x',H',K'}$ are the same if and only if $H=H'$, 
$K=K'$ and $x^{-1}x' \in K$. In particular $f_{x,H,H}$ is the 
identity map on $G/H$ if and only if $x\in H$.

If we are given two $G$-maps $f_{x,H,K}$ and $f'_{y,K,L}$, then the
composite map $f_{y,K,L} \circ f_{x,H,K}$ is a $G$-map $G/H\to G/K$
and we have
\begin{equation*}
    (f_{y,K,L}\circ f_{x,H,K})(H) = f_{y,K,L}(xK) = x f_{y,K,L}(K) =
    xy L.
\end{equation*}
In other words we have the following simple rule to calculate the
composite of two $G$-maps between homogeneous $G$-spaces:
\begin{equation*}
    f_{y,K,L} \circ f_{x,H,K} = f_{xy,H,L}
\end{equation*}

The structure of the orbit category $\OFG$ depends not only on the
group~$G$ but also very much on the family $\frakF$ of subgroups
of~$G$.  We list a few standard facts from the theory of topological
transformation groups which illustrate this situation.
\begin{enumerate}
    \item If $\frakF = \{1\}$, then the orbit category has only one
    object $G/1$.  Clearly every element of $[G/1, G/1]_G$ is
    invertible, that is $[G/1,G/1]_G = \Aut(G/1)$.  We have an
    isomorphism of groups $G \to \Aut(G/1)$ which sends an element $g$
    to the automorphism
    \begin{align*}
       l_g\: & G/1 \mapsto G/1,
       \\
       & x\mapsto gx.
    \end{align*}
    In particular, every morphism in the orbit category $\OFG$ is
    invertible.
   
    \item If $\frakF \subset \Ffin(G)$, then still every endomorphism
    in $\OFG$ is invertible, that is $[G/H,G/H]_G = \Aut(G/H)$ for
    every $H\in\frakF$.  This is because if $f_{g,H,H}$ is an
    endomorphism of $\OFG$, then $H^g\leq H$ and since $H$ is finite
    it follows that $H^g= H$.  Therefore also $H^{g^{-1}}\leq H$
    and~$f_{g^{-1},H,H}$ is a morphism of the orbit category $\OFG$.
    Necessarily $f_{g^{-1},H,H}$ is the inverse to $\varphi$.

    \item In general one has that $\Aut(G/H)$ is isomorphic to the
    Weyl-group~$W_{G}(H) \defeq N_{G}(H)/H$ of $H$ in $G$.  This is,
    because an endomorphism~$f_{g,H,H}$ of $\OFG$ is invertible if and
    only if $g\in N_G(H)$ and two endomorphism $f_{g,H,H}$ and
    $f_{g',H,H}$ are the same if $g'g^{-1}\in H$.  However, if $H$ is
    not finite then there may exists elements in $\mor(G/H,G/H)$ which
    are not invertible and therefore do not belong to the automorphism
    group $\Aut(G/H)$.
\end{enumerate}
Thus broadly speaking, the larger the family $\frakF$ becomes the
more 
the orbit category $\OFG$ loses structure.

%
%

\section{The Category of Bredon Modules}

\begin{DEF}
    Let $\frakF$ be a family of subgroups of a group $G$. A functor
    \begin{equation*}
        M\: \OFG \to \Ab
    \end{equation*}
    from the orbit category $\OFG$ to the category $\Ab$ of abelian
    groups is called a \emph{Bredon module $M$ over the orbit category
    $\OFG$} (or an \emph{$\OFG$-module}). If the functor $M$ is 
    contravariant (covariant) then we call $M$ a \emph{right} 
    (\emph{left})~$\OFG$-module.
 
    Let $M$ and $N$ be two $\OFG$-modules of the same variance. 
    A \emph{morphism} $f\: M\to N$ of $\OFG$-modules is a natural 
    transformation from the functor $M$ to the functor $N$.
\end{DEF}

Let $M$ be a right (left) $\OFG$-module and $\varphi$ a morphism of
the orbit category~$\OFG$.  If there is no danger of confusion, then
we may abbreviate the homomorphism $M(\varphi)$ by $\varphi^{*}$
($\varphi_{*}$~respectively).  In order to avoid complicating the
language we shall understand a statement about Bredon modules without
specified variance to be true for both left and right Bredon modules.

\begin{examples}
    \label{ex:Bredon-modules}
    The following are simple but yet important standard examples of 
    some Bredon modules:
    
    \begin{enumerate}
	\item Let $A$ be an abelian group.  Then $\underline A$
	denotes the \emph{constant} $\OFG$-module given by $\underline
	A (G/H) \defeq A$ and $\underline A(\varphi) \defeq \id$ for
	any object $G/H$ and any morphism $\varphi$ of the orbit
	category $\OFG$.  It is both a left and a right $\OFG$-module.
	If we want to emphasise the dependency on the family $\frakF$
	then we may write $\underline A_{\frakF}$ for the constant
	$\OFG$-module~$\underline A$.
    
	\item A important special case of the above example is the
	\emph{trivial} $\OFG$-module which is the constant
	$\OFG$-module $\underline \Z_{\frakF}$.
	
	\item 
	\label{ex:freeR}
	Let $K$ be a fixed subgroup of $G$. We construct a 
	right $\OFG$-module $\Z[\?,G/K]_{G}$ as follows: Given an 
	object $G/H$ of the orbit category~$\OFG$ we let $\Z[G/H, 
	G/K]_{G}$ be the free abelian group with basis the set $[G/H, 
	G/K]_{G}$. If $\varphi\: G/H\to G/L$ is a morphism in $\OFG$, 
	then $\varphi^{*}\: \Z[G/L, G/K]_{G} \to \Z[G/H,G/K]_{G}$ is 
	the unique homomorphism of abelian groups which maps the basis 
	element $f\in [G/L, G/H]_{G}$ to $f\circ \varphi \in [G/H, 
	G/K]_{G}$.
	
	\item
	\label{ex:freeL}
	In a similar way as above we can construct a left
	$\OFG$-module $\Z[G/K, \?]_{G}$.  Given a morphism $\varphi$
	of the orbit category~$\OFG$ the homomorphism $\varphi_{*}$ is
	defined by pre-composition instead of post-composition.
    \end{enumerate}
\end{examples}

The class of all right $\OFG$-modules together with the morphisms of
$\OFG$-modules form a category which we denote by $\OFGMod$.  Similar
we have the category $\ModOFG$ of all left $\OFG$-modules.  By
construction these categories are the functor categories
$[\OFG^{\op},\Ab]$ and $[\OFG, \Ab]$
respectively~\cite[pp.~63ff.]{mitchell-65}.  It follows from standard
arguments in category theory that the functor categories $\ModOFG$ and
$\OFGMod$ inherit many properties from the abelian category
$\Ab$~\cite{freyd-64, mac-lane-98, weibel-94}.  In what follows we
collect some of those results for $\ModOFG$.
 
The category $\ModOFG$ is abelian, complete and cocomplete (that is
arbitrary limits and colimits exist) since the category $\Ab$ is.
Limits and colimits are calculated componentwise.  This includes:
products, coproducts, direct limits, kernels, images and
intersections.  In particular, filtered limits (which include direct
limits) are exact in $\ModOFG$ as they are exact
in~$\Ab$~\cite[p.~57]{weibel-94}.  Furthermore, since kernels and
images are calculated component wise we have that a sequence
\begin{equation*}
    M'\to M\to M''
\end{equation*}
of right $\OFG$-modules is exact at $M$ if and only if the 
corresponding sequences\pagebreak[1]
\begin{equation*}
    M''(G/H) \to M(G/H) \to M'(G/H)
\end{equation*}
of abelian groups are exact at $M(G/H)$ for every $H\in \frakF$.

Finally, we remark that the category $\ModOFG$ has enough projectives
because the category $\Ab$ is cocomplete and has enough
projectives~\cite[p.~43]{weibel-94}.  Since $\Ab$ is complete and has
enough injectives, it follows by a similar argument that~$\ModOFG$ has
enough injectives, too.  Therefore we can define left and right
derived functors and take advantage of homological methods in the
study of the category of Bredon modules over the orbit
category~$\OFG$.

\begin{DEF}
    Let $G_{1}$ and $G_{2}$ be two groups and $\frakF$ and 
    $\frakG$ families of subgroups of $G_{1}$ and $G_{2}$ 
    respectively. A 
    \emph{$\OfG$-$\OgG$-bimodule} 
    $M$ is a bifunctor
    \begin{equation*}
        M\: \OfG\times \OgG \to \Ab
    \end{equation*}
    that is covariant in the first variable and contravariant in the
    second variable.
\end{DEF}

\begin{example}
    \label{ex:bi-Bredon-module}
    Given a group $G$ and family $\frakF$ of subgroups of $G$, then we
    have a $\OFG$-$\OFG$-bifunctor
    \begin{equation*}
	\Z[\?,\q?]_{G}\: \OFG\times \OFG \to \Ab.
    \end{equation*}
    which is is defined as follows.  Given a pair $G/K$ and $G/H$ of
    objects in~$\OFG$, its value is defined to be the free abelian
    group $\Z[G/H,G/K]_{G}$.  Given any pair $\psi\: G/K\to G/K'$ and
    $\varphi\: G/H'\to G/H$ of morphisms in $\OFG$, the group
    homomorphism
    \begin{equation*}
	\Z[\varphi, \psi]_{G}\: \Z[G/H, G/K]_{G} \to \Z[G/H',G/K']_{G}
    \end{equation*}
    is defined to be the unique group homomorphism which sends a basis
    element $f\in [G/H, G/K]_{G}$ to $\psi\mathop{\circ}
    f\mathop{\circ} \varphi\in [G/H',G/K']_{G}$.  
    \begin{figure}
    \begin{equation*}
    \dgARROWLENGTH=3cm
    \begin{diagram}
    \node{\Z[G/H, G/K]_{G}}
    \arrow{e,t}{\Z[G/H,\varphi]_{G}}
    \arrow{s,l}{\Z[\psi,G/K]_{G}}
    \arrow{se,t}{\Z[\varphi, \psi]_{G}}
    \node{\Z[G/H, G/K']_{G}}
    \arrow{s,r}{\Z[\psi,G/K']_{G}}
    \\
    \node{\Z[G/H', G/K]_{G}}
    \arrow{e,t}{\Z[G/H,\varphi]_G}
    \node{\Z[G/H', G/K']_{G}}
    \end{diagram}
    \end{equation*}
    \caption{}
    \label{fig:free-Bredon-bi-module}
    \end{figure}
    This is precisely
    the necessary definition needed in order to combine the
    constructions~(\ref{ex:freeR}) and~(\ref{ex:freeL}) in
    Example~\ref{ex:Bredon-modules} into a $\OFG$-$\OFG$-bimodule,
    see the diagram in Figure~\ref{fig:free-Bredon-bi-module}
    and~\cite[p.~37]{mac-lane-98}.
\end{example}

%
%

\section{Bredon Modules and $G$-Modules}

Recall that a right $G$-module $M$ is an abelian group $M$ with an
action of $G$ on the right.  The action of $G$ is extended linearly to
a homomorphism from the group ring $\Z G$ into the endomorphism ring
of $M$.  The category of all right $G$-modules is denoted by $\ModG$.

In what follows we consider the special case that $\frakF = \{1\}$ is
the trivial family of subgroups.  In Section~\ref{sec:orbit-category}
we have already noted that $\mor(G/1,G/1) = \Aut(G/1)$ is isomorphic
to the group $G$.  This isomorphism is given by $\varphi_{g,1,1}
\mapsto g^{-1}$.  Its inverse is given by~$g\mapsto
\varphi_{g^{-1},1,1}$.

Now a functor from $\OFG$ to $\Ab$ determines an abelian group $M' = 
M(G/1)$ and a homomorphism $\mor(G/1,G/1) \to \End(M')$. Since all 
endomorphisms of $G/1$ are invertible it follows that this 
homomorphism is actually a homomorphism $G\to \Aut(M')$. It is given 
by $g\mapsto M(\varphi_{g^{-1},1,1})$. If $M$ is a contravariant 
functor, that is a right $\OFG$-module, then we have
\begin{equation*}
    \varphi^{*}_{(gh)^{-1},1,1} = \varphi^{*}_{h^{-1}g^{-1},1,1} = 
    (\varphi_{g^{-1},1,1} \circ \varphi_{h^{-1},1,1})^{*}
    = \varphi^{*}_{h^{-1},1,1} \circ \varphi^{*}_{g^{-1},1,1}
\end{equation*}
for all $g,h\in G$. Therefore $xg \defeq  \varphi^{*}_{g^{-1},1,1}$ 
defines an right action of $G$ on $M'$ and this makes $M'$ into a 
right $G$-module.

In the case that $\frakF=\{1\}$ we can reverse this construction.
Given any right $G$-module $M'$ we can construct a right 
$\OFG$-module in the obvious way as follows. We set $M(G/1) \defeq  M'$ 
and if $\varphi_{g,1,1}$ is a morphism of the orbit category $\OFG$, 
then we let $\varphi^{*}_{g,1,1}$ be the morphism given by $x\mapsto 
xg^{-1}$ for all~$x\in M'$. Then
\begin{align*}
    (\varphi_{g,1,1}\circ \varphi_{h,1,1})^{*} 
    = 
    \varphi^{*}_{hg,1,1} & = x\mapsto x(hg)^{-1}
    \\
    & = 
    x\mapsto (xg^{-1})h^{-1}
    = 
    \varphi^{*}_{h,1,1} \circ \varphi^{*}_{g,1,1}
\end{align*}
which shows that $M$ is indeed a contravariant functor.

Thus in the case that $\frakF=\{1\}$ we have a one-to-one
correspondence between right $\OFG$ modules and right $G$-modules
given by the above construction.  Furthermore a morphism $f\: M\to N$
between two right $\OFG$-modules is given by a single homomorphism
$f'\: M'\to N'$ of abelian groups.  It follows from the fact that $f$
is a natrual transformation that $f'$ is a homomorphism
of~$G$-modules.  It follows that the assignment $M\mapsto M'$ and
$f\mapsto f'$ is functorial.

Therefore one has the known result that the categories $\ModOFG$ and
$\ModG$ are naturally isomorphic if $\frakF = \{1\}$ is the trivial
family of subgroups of~$G$.  Of course one has the dual result that
the category~$\OFGMod$ of left Bredon modules over the orbit category
$\OFG$ and the category $\GMod$ of left~$G$-modules are naturally
isomorphic in the case that $\frakF=\{1\}$.  In other words the theory
of Bredon modules is a generalisation of the theory of modules over
group rings.

%
%

\section{$\frakF$-Sets and Free Bredon Modules}
\label{sec:free-Bredon-modules}

Free objects are usually defined as left adjoint to a suitable
forgetful functor.  In the case of Bredon modules, the target category
of this forgetful functor is not the category $\Set$ of sets but the
category of $\frakF$-sets, which we denote by $\FSet$.  There are
several ways to see and describe this category.

\begin{DEF}
    An \emph{$\frakF$-set} $\Delta=(\Delta,\varphi)$ is a pair
    consisting of a set $\Delta$ and a function $\varphi\: \Delta\to
    \frakF$.  For $H\in \frakF$ we denote by $\Delta_{H}$ the
    pre-image $\varphi^{-1}(\{H\})$ and call it the
    \emph{$H$-component} of the $\frakF$-set $\Delta$.  A \emph{map}
    $f\: (\Delta,\varphi) \to (\Delta',\varphi')$ of~$\frakF$-sets is
    a function $f\: \Delta \to \Delta'$ of sets such that the diagram
    \begin{equation*}
    \dgARROWLENGTH=4em
        \begin{diagram}
            \node{\Delta}
	    \arrow[2]{e,t}{f}
	    \arrow{se,r}{\varphi}
	    \node[2]{\Delta'}
	    \arrow{sw,r}{\varphi'}
	    \\
	    \node[2]{\frakF}
        \end{diagram}
    \end{equation*}
    commutes.  
\end{DEF}

Note that by definition the class of all $\frakF$-sets, together with
maps of~$\frakF$-sets, forms a \emph{comma category over $\frakF$} in
the sense of~\cite[p.~45]{mac-lane-98}.  We denote this category by
$\FSet$.

\begin{lemma}
    Consider the set $\frakF$ as a discrete category.  Then the
    functor category $[\frakF, \Set]$ is isomorphic to $\FSet$.
\end{lemma}

\begin{proof}
    Note, that since $\frakF$ is considered as a discrete category a
    functor~$\frakF\to \Set$ is characterised by its values on the
    objects of $\frakF$.  Given a~$\frakF$-set~$\Delta$, there exists
    precisely one functor $\Delta\: \frakF \to \Set$ that maps~$H$
    to~$\Delta_{H}$ for every~$H\in \frakF$.  This gives a bijection
    between the objects of $\FSet$ and the objects of~$[\frakF,
    \Set]$.  Moreover, any morphism $f\: \Delta\to \Delta'$ in $\FSet$
    induces a collection of functions $f_H\: \Delta_H\to \Delta'_H$
    indexed by the elements $\frakF$.  Since $\frakF$ is a discrete
    category this gives rise to a natural transformation between the
    corresponding functors $\Delta\:\frakF\to \Set$ and
    $\Delta'\:\frakF \to \Set$ and thus a morphism in $[\frakF,\Set]$.
    It follows that we get a bijection between the corresponding
    morphism sets in~$\FSet$ and $[\frakF, \Set]$.  Thus the two
    categories are isomorphic.
\end{proof}

There exists the obvious forgetful functor from the category $\FSet$
to the category $\Set$ which sends a $\frakF$-set $\Delta$ to the
underlying set $\Delta$.  Using this functor we can pull back much of
the terminology for sets to the category of~$\frakF$-sets.  In
particular we speak of a \emph{finite} (\emph{countable}) $\frakF$-set
if the underlying set is finite (countable).  Only with categorical
statements we have to be careful: for example the $\frakF$-set
$\Delta'$ is a \emph{subset} of the $\frakF$-set $\Delta$ if
$\Delta'_{H}\subset \Delta_{H}$ for every $H\in \frakF$.  As in
functor categories limits and colimits are calculated component wise.
In particular this is true for the product (cartesian product) and
coproduct (disjoint union) of $\frakF$-sets.  In detail, if
$\Delta_{i}$ are $\frakF$-sets indexed by some index set~$I$ then
their product and coproduct are given by
\begin{equation*}
    \Bigl(\prod_{i\in I}\Delta_{i}\Bigr)_{H} = \prod_{i\in
    I}\Delta_{i,H}
    \qquad \text{and} \qquad
    \Bigl(\coprod_{i\in I}\Delta_{i}\Bigr)_{H} = \coprod_{i\in
    I}\Delta_{i,H}
\end{equation*}
for every $H\in\frakF$.

Given a $\OFG$-module $M$ we denote the \emph{underlying} 
$\frakF$-set also by $M$, which is given by
\begin{equation*}
    M_{H} \defeq  M(G/H)
\end{equation*}
for all $H\in \frakF$.  A morphism of $\OFG$-modules gives in an
obvious way rise to a map of the underlying $\frakF$-sets.  In this
way we get a forgetful functor
\begin{equation*}
    U\: \ModOFG \to \FSet
\end{equation*}
(and likewise we have a forgetful functor from $\OFGMod$ to $\FSet$).
We say that a $\frakF$-set $X$ is a \emph{subset} of an $\OFG$-module
$M$ if $X$ is a subset of the~$\frakF$-set~$UM$. Any subset of an
$\OFG$-module is implicitly considered as a~$\frakF$-set.

\begin{DEF}
    Let $M$ be an $\OFG$-module and $X$ a subset of $M$.  Then the
    smallest submodule of $M$ containing $X$ is denoted by $\langle
    X\rangle$ and is called the submodule of $M$ \emph{generated} by
    the $\frakF$-set $X$.  If $M=\langle X\rangle$ then we say that
    $M$ is \emph{generated} by $X$ and that $X$ is a
    \emph{$\frakF$-set of generators} of $M$. We say that $M$ is a 
    \emph{finitely generated} $\OFG$-module if there exists a finite 
    $\frakF$-set of generators of $M$.
\end{DEF}

\begin{lemma}
    Let $K\in \frakF$ and consider the right $\OFG$-module $\Z[\?,
    G/K]_{G}$ of Example~\ref{ex:Bredon-modules}.  Then
    the subset $\Delta$ of $\Z[\?, G/K]_{G}$ given by
    \begin{equation}
	\label{eq:singelton-FSet}
        \Delta_{H} \defeq 
	\begin{cases}
	     \{\id\}& \text{if $H=K$,}\\[.5ex]
	     \emptyset & \text{otherwise}
	\end{cases}
    \end{equation}
    is a generating set of $\Z[\?, G/K]_{G}$.
\end{lemma}

\begin{proof}
    Denote by $M$ the submodule of $\Z[\?, G/K]_{G}$ generated by
    $\Delta$.  We know that $M(G/H)$ is a subgroup of
    $\Z[G/H,G/K]_{G}$ for any $H\in \frakF$ and we want to show that
    actually equality holds in every case.  
    
    Therefore let $\varphi\in [G/H,G/K]_{G}$ be a generator of
    $\Z[G/H,G/K]_{G}$.  Since~$\Delta$ generates $M$ we know that
    $\id\in M(G/K)$.  Then $\varphi^{*}(\id) = \id \mathop{\circ}
    \varphi = \varphi\in M(G/H)$.  Since this is true for any
    generator $\varphi$ of the group $\Z[G/H,G/K]_{G}$ we must have
    that $M(G/H)=\Z[G/H,G/K]_{G}$ and the claim follows.
\end{proof}

\begin{proposition}
    \label{prop:free-functor}
    The forgetful functor $U\:\ModOFG\to \FSet$ has a left adjoint
    $F\:\FSet\to \ModOFG$.
\end{proposition}

\begin{proof}
    First we define the functor $F$ for singleton $\frakF$-sets.  Let
    $K\in\frakF$ and consider the singleton $\frakF$-set $\Delta$ with
    $\Delta_{K} \defeq  \{*\}$ and $\Delta_{H}\defeq \emptyset$ for $H\neq K$.
    We set
    \begin{equation*}
        F\Delta \defeq  \Z[\?, G/K]_{G}
    \end{equation*}
    and identify $\Delta$ with the singleton subset of $\Z[\?,
    G/K]_{K}$ as given in~\eqref{eq:singelton-FSet} in the previous
    lemma.  We have to show that for any (right) $\OFG$-module $M$ the
    adjoint relation
    \begin{equation}
	\label{eq:adjont-realation-free}
	\mor_{\frakF}(F\Delta, M) \isom \mor(\Delta, UM)
    \end{equation}
    is satisfied, where the morphism set on the left is in $\ModOFG$
    and the morphism set on the right is in $\FSet$.  But this follows
    from the Yoneda type formula in the next lemma.
    
    A general $\frakF$-set $\Delta$ can always be written as the
    coproduct
    \begin{equation*}
        \Delta = \coprod_{x\in \Delta} \Delta_{x}
    \end{equation*}
    of its singleton subsets $\Delta_{x}$.  The natural way to extend
    the definition of the functor $F$ to arbitrary $\frakF$-sets is to
    set
    \begin{equation*}
	F \Delta \defeq  \coprod_{x\in\Delta} F\Delta_{x}.
    \end{equation*}
    Then we have isomorphisms
    \begin{align*}
	\mor_{\frakF}(F\Delta, M) 
	& \isom
	\prod_{x\in\Delta} \mor_{\frakF}(F\Delta_{x}, M)
	\\
	& \isom
	\prod_{x\in\Delta}\mor(\Delta_{x}, UM)\isom \mor(\Delta, UM)
    \end{align*}
    which are natural, both in $\Delta$ and $M$. Thus $F$ is a left
    adjoint functor to the forgetful functor $U$.
\end{proof}

Note that there is a canonical inclusion of the 
$\frakF$-set $\Delta = (\Delta, \varphi)$ into $F\Delta$ given by
\begin{equation*}
    x\mapsto \id\in (F\Delta_{x})(G/\varphi(x)).
\end{equation*}
Using this inclusion we have a canonical way to identify the
$\frakF$-set $\Delta$ as a subset of the right $\OFG$-module
$F\Delta$. Note that under this identification $\Delta$ becomes a 
generating $\frakF$-set of $F\Delta$.

\begin{lemma}[Yoneda Type Formula]
    \label{lem:yoneda-type-formula}
    Let $K\in\frakF$ and let $M$ be a right $\OFG$-module.  Then there
    exists an isomorphism
    \begin{equation*}
        e_{K}\: \mor_{\frakF}(\Z[\?,G/K]_{G}, M) \isom M(G/K)
    \end{equation*}
    of abelian groups given by the evaluation map 
    $e_{K}(f)\defeq f_{K}(\id)$. This isomorphism is natural in 
    $M$.
\end{lemma}

\begin{proof}
    For the proof of the first part see for
    example~\cite[p.~9]{mislin-03}. The naturality claim states that 
    for any morphism $\eta\: M\to N$ of right~$\OFG$-modules, the 
    diagram
    \begin{equation*}
        \begin{diagram}
            \node{\mor_{\frakF}(\Z[\?,G/K], M)}
	    \arrow{e,t}{e_{K}}
	    \arrow{s,l}{\eta_{*}}
	    \node{M(G/K)}
	    \arrow{s,r}{\eta_{K}}
	    \\
	    \node{\mor_{\frakF}(\Z[\?,G/K], N)}
	    \arrow{e,t}{e_{K}}
	    \node{N(G/K)}
        \end{diagram}
    \end{equation*}
    commutes, where $\eta_{*}$ is the homomorphism which maps any
    morphism $f\in \mor_{\frakF}(\Z[\?,G/K],M)$ to $\eta\circ f\in
    \mor_{\frakF}(\Z[\?,G/K],N)$. But this follows immediately from
    \begin{equation*}
	(e_{K}\circ \eta_{*})(f) = e_{K}(\eta\circ f) = (\eta \circ 
	f)_{K}(\id) = 
	\eta_{K}(f_{K}(\id)) = (\eta_{K} \circ e_{K})(f).\qedhere
    \end{equation*}
\end{proof}

\pagebreak[3]

\begin{DEF}
    Let $M$ be a right $\OFG$-module and let $B$ a subset of $M$.  We
    say that $M$ is \emph{free with basis $B$} if there exists an
    isomorphism
    \begin{equation*}
        FB \isom  M
    \end{equation*}
    that maps $B$ seen as a subset of $FB$ to $B$ as a subset of $M$.

\end{DEF}

In the terms of the adjoint relation~\eqref{eq:adjont-realation-free}
the above definition can be interpreted in the following familiar way:
a right $\OFG$-module $M$ is free with basis $B$ if $B$ is a subset of
$M$ such that for any right $\OFG$-module $N$ and any morphism
$f_{0}\: B\to N$ of $\frakF$-sets there exists a unique extension of
$f_{0}$ to a morphism $f\: M\to N$ of $\OFG$-modules.

Note that from the proof of Proposition~\ref{prop:free-functor} 
follows that the $\OFG$-modules of the form $\Z[\?, G/K]_{G}$, $K\in 
\frakF$, are the building blocks for free right Bredon modules and if 
$\Delta = (\Delta, \varphi)$ is a $\frakF$-set, then
\begin{equation}
    \label{eq:general-free-module}
    F\Delta = \coprod_{\delta\in \Delta} \Z[\?,
G/\varphi(\delta)]_{G}.
\end{equation}

\begin{lemma}
    A $\OFG$-module $M$ is finitely generated if and only if there 
    exists a short exact sequence of $\OFG$-modules
    \begin{equation*}
        0\to K\to F \to M\to 0
    \end{equation*}
    where $F$ is a finitely generated free $\OFG$-module.
\end{lemma}

\begin{proof}
    If $M$ is finitely generated, then there exists a finite
    generating $\frakF$-set $X$ of $M$.  Set $F \defeq  FX$.  Then $F$ is
    finitely generated and surjects onto~$M$.  If one lets $K$ be the
    kernel of this surjection one obtains the above short exact
    sequence.
    
    On the other hand, if $F$ is free with a finite basis $X$, then 
    the image of~$X$ under the surjection $F\to M$ is a finite 
    $\frakF$-subset of $M$ which generates~$M$. Therefore $M$ is 
    finitely generated.
\end{proof}

\begin{DEF}
    A $\OFG$-module $M$ is called \emph{finitely presented} if there 
    exists a short exact sequence
    \begin{equation*}
        0 \to K \to F \to M \to 0
    \end{equation*}
    where $F$ is a finitely generated free $\OFG$-module and $K$ is a
    finitely generated $\OFG$-module.
\end{DEF}

Note that the definitions, results and their proofs in this section
carry word for word over to left $\OFG$-modules with right
$\OFG$-modules of the form $\Z[\?, G/K]_{G}$ replaced by corresponding
left $\OFG$-modules of the form $\Z[G/K, \?]_{G}$.

%
%

\section{From $G$-Sets to Bredon Modules}

There is an alternative way to see the construction of the previous 
section, namely as a functor from the category of $G$-sets to the 
category of right Bredon modules.

Recall that given subgroups $H$ and $K$ of $G$ there exists the 
identification
\begin{equation*}
    \eta\: [G/H, G/K]_{G} \isom (G/K)^{H}
\end{equation*}
which sends a $G$-map $\psi$ to the image $\psi(H)$. Now 
$[\?,\q?]_{G}$ is a bifunctor
\begin{equation*}
    [\?, \q?]_{G}\: \calO G \times \calO G \to \Set
\end{equation*}
contravariant in the first and covariant in the second variable.
Restricting this functor to $\OFG\times \calO G$ and composed with the
functor $\Z[\?]\: \Set \to \Ab$ which sends a set $X$ to the free
abelian group $\Z[X]$ with basis $X$ this gives the functor of
Example~\ref{ex:bi-Bredon-module}.

The functor $[\?,\??]_{G}$ extends to a bi-functor
\begin{equation*}
    [\?, \??]_{G}\: \calO G\times \GSet \to \Set,
\end{equation*}
contravariant in the first and contravariant in the second variable, 
which sends a transitive $G$-set $G/H$ and a $G$-set $X$ to the set 
$[G/H, X]_{G}$ of all~$G$-maps from $G/H$ to $X$. Note that as before 
there is an identification
\begin{equation}
    \eta\: [G/H, X]_{G} \isom X^{H}
    \label{eq:identification-1}
\end{equation}
which sends a $G$-map $\psi\: G/H\to X$ to $\psi(H)$.

If $f\: X\to Y$ is a $G$-map, then this gives a map
\begin{align*}
    f_{*}\:  [G/H, X]_{G} & \to [G/H, Y]_{G},
    \\
    \psi & \mapsto f\circ \psi.
\end{align*}
Under the identification~\eqref{eq:identification-1} is just the
restriction of $f$ to a map $X^{H}\to Y^{H}$.

On the other hand, if $\varphi\: G/H\to G/K$ is a morphism of $\calO 
G$ then this gives a map
\begin{align*}
    \varphi^{*}\: [G/K, X]_{G} & \to [G/H, X]_{G},
    \\
    \psi & \mapsto  \psi \circ \varphi.
\end{align*}
To see what $\varphi_{*}$ becomes under the identification assume
that $\varphi = f_{g,H,K}$.  Then $g$ is uniquely determined up to
right multiplication by an element of $K$ and $H^{g}\leq K$.  If $x\in
X^{K}$ and $h\in H$, then
\begin{equation*}
  hgx = g(g^{-1}hg)x = gx
\end{equation*}
and therefore $gx\in X^{H}$.  If we set $\psi\defeq  \eta^{-1}(x) \in
[G/K,X]_{G}$, then
\begin{equation*}
    \varphi^{*}(\psi)(H) = (\psi\circ \varphi)(H) = \psi(gK) = 
    g\psi(K) = gx.
\end{equation*}
Thus $\varphi^{*}$ becomes under the
identification~\eqref{eq:identification-1} the map $X^{K}\to X^{H}$
which sends $x$ to $gx$.

Let $\Delta = (\Delta, \varphi)$ be a $\frakF$-set and consider the 
$G$-set
\begin{equation*}
    X \defeq  \coprod_{x\in \Delta} G/\varphi(x)
\end{equation*}
which is the disjoint union of homogeneous $G$-spaces $G/\varphi(x)$ 
with $\varphi(x)\in \frakF$. Then
\begin{equation*}
    [\?, X]_{G} = \coprod_{x\in \Delta} [\?, G/\varphi(x)]_{G}
\end{equation*}
and since the functor $\Z[\?]\: \Set \to \Ab$ commutes with 
coproducts it follows that
\begin{equation*}
    \Z[\?, X]_{G} = \coprod_{x\in \Delta} \Z[\?, G/\varphi(x)]_{G}
\end{equation*}
is the free right $\OFG$-module with basis the $\frakF$-set $\Delta$
as introduced in the previous section. On the other hand it is clear 
that if $X$ is a $G$-set with $\frakF(X)\subset \frakF$, then 
$\Z[\?,X]_{G}$ is a free $\OFG$-module. Therefore we obtain the 
following result.

\begin{proposition}
    \label{prop:GSet-to-ModOFG}
    Let $\frakF$ be a family of subgroups of $G$. Then we have a 
    covariant functor
    \begin{equation*}
        \Z[\?,\??]_{G}\: \GSet \to \ModOFG.
    \end{equation*}
    This functor sends disjoint unions of $G$-sets are send to
    coproducts in $\ModOFG$.  The free right $\OFG$-modules are
    precisely all the Bredon modules of the form $\Z[\?,X]_{G}$ where
    $X$ is a $G$-set with $\frakF(X)\subset \frakF$.  \qed
\end{proposition}

Let $X$ be a set and let $X_{\lambda}$, $\lambda\in\Lambda$, be a 
collection of subsets of $X$ indexed by an abstract index set 
$\Lambda$. We say that $X$ is the \emph{directed union of the 
sets}~$X_{\lambda}$ if the following two conditions hold: 
\begin{enumerate}
    \item for every $x\in X$ there exists a $\lambda\in \Lambda$ such 
    that $x\in X_{\lambda}$;
    
    \item for every $\lambda_{1},\lambda_{2}\in \Lambda$ there exists 
    a $\mu\in \Lambda$ such that $X_{\lambda_{1}}\subset 
    X_{\mu}$ and~$X_{\lambda_{2}}\subset X_{\mu}.$
\end{enumerate}
Since a directed union is a special case of a colimit in the category
of sets (or~$G$-sets) we may identify $X = \varinjlim X_{\lambda}$.

\begin{lemma}
    \label{lem:X-H-direct-union}
    Assume that the $G$-set $X$ is the directed union of $G$-invariant
    subsets $X_{\lambda}$, $\lambda\in \Lambda$.  Let $H$ be a
    subgroup $G$.  Then $X^{H}$ is the directed union of the
    subsets $X_{\lambda}^{H}$, $\lambda\in \Lambda$. That is
    \begin{equation*}
        X^{H} = \varinjlim X_{\lambda}^{H}.
    \end{equation*}
\end{lemma}

\begin{proof}
    This follows immediately from the fact that $X_{\lambda}^{H} = 
    X^{H}\cap X_{\lambda}$.
\end{proof}

\pagebreak[3]

\begin{proposition}
    \label{prop:direct-union}
    The homomorphism of right $\OFG$-modules
    \begin{equation}
	\varinjlim \Z[\?, X_{\lambda}]_{G} \to \Z[\?, X]_{G}
	\label{eq:Z-direct-union}
    \end{equation}
    induced by the canonical monomorphisms
    $\Z[\?, X_{\lambda}]_{G} \hookrightarrow \Z[\?, X]_{G}$ is an
    isomorphism.
\end{proposition}

\begin{proof}
    We have
    \begin{align*}
	\Z[\?,X]_{G} 
	& \isom
	\Z[\?, \varinjlim X_{\lambda}]_{G}
	\\[1ex]
	& \isom
	\Z\bigl[ \varinjlim [\?, X_{\lambda}]_{G}\bigr] 
	&&\text{(by Lemma~\ref{lem:X-H-direct-union})}
	\\[1ex]
	& \isom 
	\varinjlim \Z[\?,X_{\lambda}]_{G}
    \end{align*}
    where the last isomorphism is due to the fact that the functor 
    $\Z[\?]\: \Set \to \Ab$ commutes with arbitrary colimits since it 
    is the left adjoint to the forgetful functor $U\: \Ab\to \Set$,
    see~\cite[pp.~118f.]{mac-lane-98}. 
    Now the composite of this sequence of isomorphism is precisely 
    the homomorphism~\eqref{eq:Z-direct-union}.
\end{proof}

\begin{lemma}
    \label{lem:countability-criterion}
    Let $X$ be a $G$-set such that the orbit space $X/G$ is countable.
    Then $\Z[\?, X]_{G}$ is countably generated if one of the
    following conditions holds:
     
    \begin{enumerate}
	\item $\frakF(X)\subset \frakF$;
	
	\item $G$ and $\frakF$ are countable.
    \end{enumerate}
\end{lemma}

\begin{proof}
    Let $R$ be a complete system of representatives for the
    orbits~$X/G$.  Then the $\OFG$-module $\Z[\?, X]_{G}$ is a
    countable coproduct
    \begin{equation*}
        \Z[\?, X]_{G} \isom \coprod_{x\in R} \Z[\?, G/G_{x}]_{G}.
    \end{equation*}
    
    If $\frakF(X)\subset \frakF$ then the right hand side is free and
    thus $\Z[\?, X]_{G}$ is countably generated.  This proves the
    first case.
    
    Thus assume that $G$ and $\frakF$ are countable.  We construct for
    any $x\in R$ a countably generated free $\OFG$ module $F_{x}$ that
    surjects onto $\Z[\?, G/G_{x}]_{G}$.  The construction is dual to
    the construction given in~\cite[p.~43]{weibel-94}.  Since~$G$ is
    countable the set $[G/H, G/G_{x}]_{G}$ is countable for any $H\in
    \frakF$.  Set $F_{\varphi,H} \defeq \Z[\?, G/H]_{G}$ for each
    $\varphi\in [G/H,G/G_{x}]_{G}$ and $H\in\frakF$.  Then
    \begin{equation*}
        F_{x}\defeq  \coprod_{H\in\frakF} \;
	\coprod_{\varphi\in [G/H,G/G_{x}]_{G}} F_{\varphi,H}
    \end{equation*}
    is a countable free $\OFG$-module that surjects onto $\Z[\?,
    G/G_{x}]_{G}$.  This surjection can be constructed as
    follows.  For each $\varphi\in[G/H, G/G_{x}]_{G}$ let $f_{x,H}\:
    F_{\varphi,H}\to \Z[\?, G/G_{x}]_{G}$ be the unique morphism of
    $\OFG$-modules that maps the generator of $F_{\varphi,H}$ to the
    generator $\varphi$ of $\Z[G/H, G/G_{x}]_{G}$.  Then
    \begin{equation*}
	f_{x} \defeq  \coprod_{H\in\frakF} \; \coprod_{\varphi\in
	[G/H,G/G_{x}]_{G}} f_{\varphi,H}
    \end{equation*}
    is a surjection of $F_{x}$ onto $\Z[?,G/G_{x}]$. It follows that
    \begin{equation*}
        F \defeq  \coprod_{x\in R} F_{x}
    \end{equation*}
    is a countably generated free $\OFG$-module that surjects onto 
    $\Z[\?, X]_{G}$.
\end{proof}

%
%

\section{Projective Bredon Modules}

It follows from a categorical argument that free objects
share the following universal property: any morphism $f\: P\to M$ from
a free $\OFG$-module $P$ to an arbitrary $\OFG$-module $M$ factors
through any epimorphism $p\: M'\to M$.  That is we can always find a
morphism $P\to M'$ making the following diagram, with the row exact,
commute:
\begin{equation}
    \label{eq:projective-module-diagram}
    \begin{diagram}
        \node[2]{P}
	\arrow{sw,..}
	\arrow{s,r}{f}\\
	\node{M'}
	\arrow{e,t}{p}
	\node{M}
	\arrow{e}
	\node{0}
    \end{diagram}
\end{equation}

Projective objects are the usual generalisation of free objects.  We
recall the definition of a projective object in the category of Bredon
modules and the following result, which is a standard result for
abelian categories.

\begin{DEF}
    An $\OFG$-module $P$ is called \emph{projective} if for every
    diagram of the form~\eqref{eq:projective-module-diagram} with
    exact row, there exists a morphism $P\to M'$ that makes the
    diagram commute.
\end{DEF}

\pagebreak[3]

\begin{proposition}
    Let $P$ be a Bredon module over the orbit category $\OFG$. Then 
    the following statements for $P$ are equivalent:
    \begin{enumerate}
        \item  $P$ is projective;
    
        \item  every exact sequence $0\to M\to N\to P \to 0$ splits;
    
        \item  $\mor_{\frakF}(P, \?)$ is an exact functor;
    
        \item  $P$ is a direct summand of a free $\OFG$-module.
    \end{enumerate}
\end{proposition}

\begin{proof}
    The result can be found in any homological algebra book, for 
    example~\cite[pp.~33ff.]{weibel-94}.
\end{proof}

Again the above definitions and results are valid in the category of
left~$\OFG$-modules as they are valid in the category of right
$\OFG$-modules.

%
%

\section{Two Tensor Products for Bredon Modules}

There are two possible ways to define a tensor product for Bredon
modules.  The first one generalises the tensor product over the group
ring $\Z G$ for $G$-modules.  The second tensor product is the
generalisation of the tensor product over $\Z$ in the category of
$G$-modules with the diagonal action of~$\Z G$ on the tensor product.

\subsection{The Tensor Product over $\frakF$}

The definition of the tensor product over $\frakF$ involves the
categorical tensor product as described
in~\cite[pp.~45ff.]{schubert-70a}.  Given a small category $\mathcal
B$, the categorical tensor product is a bifunctor
\begin{equation*}
    \?  \mathbin\otimes_{\mathcal B} \q?  \: [\mathcal B^{\op}, \Ab]
    \times [\mathcal B, \Ab] \to \Ab
\end{equation*}
with properties expected from a tensor product.

In the case that $\mathcal B = \OFG$ a concrete model for this tensor
product is given in~\cite[p.~166]{luck-89}: if $M$ is a right
$\OFG$-module and $N$ is a left~$\OFG$-module, then let $P$ be the
abelian group
\begin{equation}
    \label{eq:tensorproduct-def-P}
    P \defeq  \coprod_{H\in \frakF} M(G/H) \otimes N(G/H)
\end{equation}
where the tensor product is taken over $\Z$.  Let $Q$ be the subgroup
of $P$ generated by all elements of the form $\varphi^{*}(m)
\otimes n - m\otimes \varphi_{*}(n)$ with $m\in M(G/H)$, $n\in
N(G/K)$, $\varphi\in [G/K, G/H]_{G}$, $H,K\in \frakF$.  Then the
\emph{tensor product $M\mathop{\otimes_{\frakF}} N$ of $M$ and $N$
over $\frakF$} is defined as the abelian group
\begin{equation*}
    M\mathop{\otimes_{\frakF}} N \defeq  P/Q.
\end{equation*}
If $f\: M\to M'$ and $g\: N\to N'$ are
morphisms of right and left $\OFG$-modules respectively, then
$f\mathop{\otimes_{\frakF}} g\: M\mathop{\otimes_{\frakF}}N \to
M'\mathop{\otimes_{\frakF}}N'$ is defined in the obvious way.

Altogether the tensor product over $\frakF$ becomes an additive
bifunctor
\begin{equation}
    \label{eq:tensor-product-0}
    \? \mathop{\otimes_{\frakF}} \q? \: \ModOFG\times \OFGMod \to \Ab.
\end{equation}

\begin{proposition}
    \label{prop:tensor-product-1}
    Let $M$ be a fixed right $\OFG$-module and let $N$ be a fixed left
    $\OFG$-module.  Then the functors
    \begin{align*}
        M\mathop{\otimes_{\frakF}} \q?\: & \OFGMod\to \Ab\\
	\intertext{and}
	\?\mathop{\otimes_{\frakF}} N\: & \ModOFG \to \Ab
    \end{align*}
    preserve arbitrary colimits.
\end{proposition}

\begin{proof}
    See \cite[pp.~46f.]{schubert-70a}.
\end{proof}

The fact that the functor $\?\mathop{\otimes_{\frakF}} N$ preserves
colimits is not a surprise because functors that have right adjoints
preserve colimits~\cite[pp.~118f.]{mac-lane-98} and the functor
$\?\mathop{\otimes_{\frakF}} N$ has a right adjoint, namely
$\Mor(N,\?)$~\cite[p.~46]{schubert-70a}.  This is the functor that
assigns to each abelian group $A$ the right $\OFG$-module $\Mor(N,A)$,
which is defined on the objects $G/H$ of $\OFG$ by $\Mor(N,A)(G/H)
\defeq \mor(N(G/H), A)$, where the morphism set on the defining side
is in $\Ab$.  Explicitly, for an abelian group $A$ and a right
$\OFG$-module~$M$ the adjoint relation is
\begin{equation}
    \label{eq:tensor-product-adjoint-1}
    \mor(M\mathop{\otimes_{\frakF}} N, A) \isom \mor(M,
    \Mor(N,A)).
\end{equation}

\begin{lemma}
    \label{lem:tensor-product-2}
    Let $M$ be a right $\OFG$-module and $N$ a left $\OFG$-module. 
    Then for every $K\in \frakF$ we have isomorphisms
    \begin{align*}
	M(\?)  \mathop{\otimes_{\frakF}} \Z[G/K, \?]_{G} & \isom
	M(G/K)
	\\
	\intertext{and}
	\Z[\?, G/K]_{G} \mathop{\otimes_{\frakF}} N(\?)  & \isom N(G/K)
    \end{align*}
    which are natural in $M$ and $N$.
\end{lemma}

\begin{proof}
    These are known results, see for example~\cite[p.~166]{luck-89} or
    \cite[p.~14]{mislin-03}.  We carry out the details for the first
    isomorphism in order to exhibit the precise definition of the
    isomorphism.  The second isomorphism is constructed in the same
    way.
    
    Let $P$ and $Q$ be the abelian groups as in the definition of the
    tensor product over $\frakF$.  That is we have that $M(\?)
    \mathop{\otimes_{\frakF}} \Z[G/K, \?]_{G}$ is a quotient of the
    group
    \begin{equation*}
        P = \coprod_{H\in\frakF} M(G/H)\otimes \Z[G/K, G/H]_{G}.
    \end{equation*}
    Observe that each element of the abelian group $M(G/H)\otimes
    \Z[G/K, G/H]_{G}$
    can be written uniquely as a finite sum
    \begin{equation*}
    	m_1\otimes \varphi_1 + \ldots  + m_n\otimes \varphi_n
    \end{equation*}
    with $m_i \in M(G/H)$ and $\varphi_i\in [G/H,G/K]_{G}$.  If $m\in
    M(G/H)$ and $\varphi\in [G/K, G/H]_G$ then $\varphi^{*}(m)\in
    M(G/K)$.  It follows that there exists a unique well defined
    homomorphism
    \begin{equation*}
        \eta_{H}\: M(G/K)\otimes\Z[G/K,G/H]_{G} \to M(G/K)
    \end{equation*}
    of abelian groups for which $\eta_{H}(m\otimes\varphi) =
    \varphi^{*}(m)$ for every $m\in M(G/H)$ and $\varphi\in
    [G/K,G/H]_G$.  The collection $\{ \eta_{H} : H\in \frakF\}$
    defines then a homomorphism
    \begin{equation*}
        \eta\: P\to M(G/K).
    \end{equation*}
    This homomorphism is surjective, since for every $m\in M(G/K)$ we
    have $\eta_{K}(m\otimes \id) = \id^{*}(m) = m$.  Furthermore,
    elements of the form
    \begin{equation*}
	\varphi_{1}^{*}(m) \otimes \varphi_{2} - m\otimes
	(\varphi_{1})_{*}(\varphi_{2}) = \varphi_{1}^{*}(m) \otimes
	\varphi_{2} - m\otimes (\varphi_{1} \comp \varphi_{2})
    \end{equation*}
    are in the kernel of $\eta$, since
    \begin{align*}
	\eta(\varphi_{1}^{*}(m) \otimes \varphi_{2} - m\otimes
	(\varphi_{1} \comp \varphi_{2}))
	& = 
	\eta(\varphi_{1}^{*}(m) \otimes \varphi_{2}) - \eta(m\otimes
	(\varphi_{1} \comp \varphi_{2}))\\
	& = \varphi_{2}^{*}(\varphi_{1}^{*}(m)) - (\varphi_{1}\comp
	\varphi_{2})^{*}(m)
	\\
	& = (\varphi_{1}\comp
	\varphi_{2})^{*}(m) - (\varphi_{1}\comp
	\varphi_{2})^{*}(m)
	\\
	& = 
	0.
    \end{align*}
    It follows that $Q\subset \ker \eta$.
    
    On the other hand, assume that $m\otimes \varphi\in \ker\eta$.
    Then $\varphi^{*}(m) = \eta(m\otimes \varphi) = 0$ and we get
    \begin{equation*}
         m\otimes \varphi = - ( \varphi^{*}(m) \otimes \id - m\otimes
	 \varphi_{*}(\id)) \in Q.
    \end{equation*}
    Hence $\ker \eta = Q$ and $\eta$ induces an isomorphism
    \begin{equation*}
        M(\?) \mathop{\otimes_{\frakF}} \Z[G/K, \?]_{G} = P/Q \isom
	M(G/K).
    \end{equation*}
    The naturality of this isomorphism is evident.
\end{proof}

A priori the tensor product $M\mathop{\otimes_{\frakF}} N$ of two
$\OFG$-modules is only an abelian group.  But if either $M$ or $N$ is
a Bredon bimodule, then the tensor product can be made into a Bredon
module as well.  More precisely, assume we are given two groups
$G_{1}$ and $G_{2}$, a family $\frakF$ of subgroups of
$G_{1}$ and a family $\frakG$ of subgroups of $G_{2}$.  If $M$ is
an $\OgG$-$\OfG$-bimodule and $N$ a left~$\OfG$-module, then
\begin{equation*}
    M(\?,\q?) {\otimes_{\frakF}} N(\q?)
\end{equation*}
is a left $\OgG$-module. Similarly, if $M$ is a right $\OfG$-module 
and $N$ an~$\OfG$-$\OgG$-bimodule, then
\begin{equation*}
    M(\?) {\otimes_{\frakF}} N(\?,\q?)
\end{equation*}
is a right $\OgG$-module.

For a fixed $\OfG$-$\OgG$-bimodule we get an adjoint 
relation for the tensor product similar to 
\eqref{eq:tensor-product-adjoint-1}:

\begin{lemma}
    Let $B$ be a $\OfG$-$\OgG$-bimodule. Then the functor
    \begin{equation*}
        \? {\otimes_{\frakF}} B \: \ModOfG \to \ModOgG
    \end{equation*}
    is left adjoint to the functor
    \begin{equation*}
        \mor_{\frakG}(B,\?) \: \ModOgG \to \ModOfG.
    \end{equation*}
    Explicitly we have for every right $\OfG$-module $M$ and every 
    right $\OgG$-module $N$ the adjoint relation
    \begin{equation}
	\label{eq:tensor-product-adjoint-2}
        \mor_{\frakG}(M {\otimes_{\frakF}} B, N) \isom 
	\mor_{\frakF}(M,\mor_{\frakG}(B,N)).
    \end{equation}
\end{lemma}

\begin{proof}
    See~\cite[p.~169]{luck-89}.
\end{proof}

\subsection{The Tensor Product over $\Z$}

The second tensor product for Bredon modules is the \emph{tensor
product over $\Z$}~\cite[p.~166]{luck-89}. Given two right
$\OFG$-modules $M$ and $N$ define a right $\OFG$-module $M\otimes N$
as follows.

For $H\in \frakF$ let $(M\otimes N)(G/H) \defeq  M(G/H)\otimes N(G/H)$
where the tensor product on the defining side is taken over $\Z$.  If
$\varphi\: G/H \to G/K$ is a morphism in the orbit category $\OFG$,
then define $(M\otimes N)(\varphi) \defeq  \varphi\otimes \varphi$, which
is a homomorphism $(M\otimes N)(G/K) \to (M\otimes N)(G/H)$.

If $f\: M\to M'$ and $g\: N\to N'$ are morphisms in $\ModOFG$, then 
$f\otimes g$ is defined to be the morphism
\begin{equation*}
    f\otimes g\: M\otimes N \to M'\otimes N'
\end{equation*}
which is given by $(f\otimes g)_{H} \defeq  f_{H}\otimes g_{H}$ for every 
$H\in \frakF$.

In this way the tensor product of right $\OFG$-modules over $\Z$ is an
additive bifunctor
\begin{equation*}
    \? \otimes \q? \: \ModOFG\times \ModOFG \to \ModOFG.
\end{equation*}

The tensor product over $\Z$ for left 
$\OFG$-modules is defined in a similar way.

%
%

\section{Flat Bredon Modules}

The tensor product functor over $\frakF$ maps epimorphisms to
epimorphisms and thus this functor is right exact.  But in general the
tensor product over $\frakF$ is not exact.

\begin{DEF}
    A right $\OFG$-module $M$ is called \emph{flat} if the functor
    $M\mathop{\otimes_{\frakF}} \q?$ is exact.  Dually a left
    $\OFG$-module $N$ is called \emph{flat} if the functor
    $\?\mathop{\otimes_{\frakF}} N$ is exact.
\end{DEF}

\begin{proposition}
    Projective $\OFG$-modules are flat.
\end{proposition}

\begin{proof}
    This is true in general in abelian categories.  But the result
    follows also from Proposition~\ref{prop:tensor-product-1} and
    Lemma~\ref{lem:tensor-product-2}.
\end{proof}

Under mild conditions on the family $\frakF$ of subgroups Nucinkis has
given a characterisation of flat
$\OFG$-modules in~\cite[p.~38]{nucinkis-04}, which is the Bredon
equivalent to Lazard's Theorem
in~\cite[p.~84]{lazard-69}.

\begin{proposition}\cite[Theorem~3.2]{nucinkis-04}
    \label{prop:lazard}
    Assume that $\frakF$ is a full family of subgroups of $G$. Then
the 
    following statements are equivalent for a right $\OFG$-module $M$:
    \begin{enumerate}
        \item  $M$ is flat;
	
	\item  any morphism $\varphi\: P \to M$ from a finitely 
	presented $\OFG$-module $P$ to $M$ factors through some 
	finitely generated free $\OFG$-module $F$;
    
        \item  $M$ is the direct limit of finitely generated free 
	$\OFG$-modules.\qed
    \end{enumerate}
\end{proposition}

%
%

\section{Restriction, Induction and Coinduction}
\label{sec:res-ind-coind}

The concept of restriction, induction and coinduction as known for
modules over group rings generalizes to Bredon modules.  Roughly
speaking in the case of group rings these functors are defined using a
ring homomorphism induced by an inclusion $H\hookrightarrow G$ where
$H$ is a subgroup of $G$.  In the case of Bredon cohomology the role
of this ring homomorphism is replaced by a functor between orbit
categories.  The following definition is due to
Lück~\cite[pp.~166f. and p.~350]{luck-89}.

\begin{DEF}
    Let $\frakF$ be a family of subgroups of a group $G_{1}$ and let
    $\frakG$ be a family of subgroups of a group $G_{2}$.
Furthermore let
    $F\: \OfG \to \OgG$ be a functor between the corresponding orbit
    categories.  Associated with the functor $F$ we have the following
    three additive functors:
    \begin{align*}
        \res_{F}\: & \ModOgG \to\ModOfG,
	\\
        & \qquad M \mapsto M(\q?) \mathop{\otimes_{\frakG}} 
        \Z[F(\?), \q?]_{G_{2}}
        &&\text{(\emph{restriction with $F$})},
	\\[\smallskipamount]
        \ind_{F}\: &  \ModOfG \to \ModOgG,  \\
        & \qquad M \mapsto M(\?) \mathop{\otimes_{\frakF}} 
        \Z[\q?, F(\?)]_{G_{2}}
        &&\text{(\emph{induction with $F$})}\\
        \intertext{and}
        \coind_{F}\: &  \ModOfG \to \ModOgG,  \\
        & \qquad M \mapsto \mor_{\frakF}(\Z[F(\?), \q?]_{G_{2}}, 
	M(\?))
        &&\text{(\emph{coinduction with $F$})}.\\
    \end{align*}
\end{DEF}

There are two other ways to interpret the restriction functor.
Namely we have natural equivalences of $\res_F$ with the following
two functors (see~\cite[pp.~116f.]{luck-89}):
\begin{align*}
    \res'_{F}\: &\ModOgG \to\ModOfG,\\
      & \qquad M\mapsto M\circ F\\
\intertext{and}
    \res''_{F}\: &\ModOgG \to\ModOfG,\\
	& \qquad M\mapsto \mor_{\frakG}(\Z[\q?,F(\?)]_{G_{2}},
	M(\q?)).
\end{align*}
Here the first natural equivalence is essentially due to
Lemma~\ref{lem:tensor-product-2}. The second natural equivalence is
due to the Yoneda-style isomorphism
\begin{equation*}
    \mor_{\frakG}(\Z[\q?,F(\?)]_{G_{2}}, M(\q?)) \isom (M \comp F)(\?)
\end{equation*}
which gives a natural equivalence of $\res''_{F}$ with $\res'_{F}$.

As in the ordinary case, restriction, induction and coinduction with
$F$ are closely related functors.  From the adjunction
relation~\eqref{eq:tensor-product-adjoint-2} we get the following
result:

\begin{proposition}
    Induction with $F$ is a left adjoint to restriction with $F$.
    Coinduction with $F$ is a right adjoint to restriction with $F$.
\end{proposition}

\begin{proof}
    Due to~\eqref{eq:tensor-product-adjoint-2} we have the following
    sequences of natural isomorphisms for any right $\OFG$-module 
    $M$ and right $\OGG$-module $N$:
    \allowdisplaybreaks
    \begin{align*}
        \mor_{\frakG}(\ind_{F} M, N) 
	& \isom 
	\mor_{\frakG}(M(\?) \mathop{\otimes_{\frakF}} 
	\Z[\q?, F(\?)]_{G_{2}}, N(\q?))
	\\
	& \isom
	\mor_{\frakF}(M(\?), \mor_{\frakG} (\Z[\q?, F(\?)]_{G_{2}}, 
	N(\q?)))
	\\
	& \isom
	\mor_{\frakF}(M, \res'_{F} N)
	\\
	& \isom
	\mor_{\frakF}(M, \res_{F} N)
	\\
	\intertext{and}
	\mor_{\frakF}(\res_{F} N, M) 
	& \isom
	\mor_{\frakF}(N(\q?) \mathop{\otimes_{\frakG}} 
        \Z[F(\?), \q?]_{G_{2}}, M(\?))
	\\
	& \isom 
	\mor_{\frakG}(N(\q?), \mor_{\frakF}(\Z[F(\?), \q?]_{G_{2}}, 
	M(\?)))
	\\
	& \isom
	\mor_{\frakG}(N, \coind_{F} M).\qedhere
    \end{align*}
\end{proof}

In the following we list some further properties for the above
functors, though not all of them will be needed.  Most of the
following results are direct consequences of the adjunction
result above.

\begin{proposition}
     \label{prop:exactness-properties}
    Restriction with $F$ is an exact functor, induction with~$F$ is a 
    right exact functor and coinduction with $F$ is a left exact
functor.
\end{proposition}

\begin{proof}
    This is a direct application of Theorem 2.6.1 in
    \cite[pp.~51f.]{weibel-94}.  It states that if $L$ and $R$ are
    additive functors and $L$ is a left adjoint to $R$ (and therefore
    $R$ is a right adjoint to $L$), then $L$ is right exact and $R$ is
    left exact.
\end{proof}

\begin{proposition}
    \label{prop:preservation-prop-1}
    Induction and restriction with $F$ preserve arbitrary colimits. 
    Coinduction and restriction with $F$ preserve arbitrary limits.
\end{proposition}

\begin{proof}
    The result follows from the fact that left adjoints preserve all 
    colimits and dually that right adjoints preserve all limits, see 
    \cite[pp.~ 118f.]{mac-lane-98}.
\end{proof}

\begin{proposition}
    \label{prop:preservation-prop-2}
    Induction with $F$ preserves frees and projectives. If $M$ is a 
    finitely generated $\OfG$-module, then so is $\ind_{F} M$.
    If both families $\frakF$ and $\frakG$ are full families of 
    subgroups, then induction with $F$ preserves flats.
\end{proposition}

\begin{proof}
    All the statements except the last one can be found
    in~\cite[p.~169]{luck-89}.  If $M$ is a flat right $\OfG$-module
    and $\frakF$ is a full family of subgroups of~$G_{1}$, then from
    Proposition~\ref{prop:lazard} it follows that $M$ is the direct
    limit of finitely generated free $\OfG$-modules $M_{\lambda}$.
    Then since induction with~$F$ preserves colimits, we get
    \begin{equation*}
        \ind_{F} M  \isom \ind_{F} (\varinjlim M_{\lambda})  \isom 
	\varinjlim (\ind_{F} M_{\lambda}).
    \end{equation*}
    Since the $M_{\lambda}$ are finitely generated free Bredon modules
    so are the $\ind_{F} M_{\lambda}$.  Thus $\ind_{F} M$ is the
    direct limit of finitely generated free $\OgG$-modules and since
    the family~$\frakG$ of subgroups of $G_{2}$ is full we can apply
    again Proposition~\ref{prop:lazard} from which it then follows
    that the $\OgG$-module $\ind_{F} M$ is flat.
\end{proof}

\begin{proposition}
    Coinduction with $F$ preserves injectives.
\end{proposition}

\begin{proof}
    This is a direct application of Theorem 2.3.10 in
    \cite[p.~41]{weibel-94}, since coinduction with $F$ is an additive
    functor that is right adjoint to the exact restriction functor.
\end{proof}

%
%

\chapter{Classifying Spaces}
\label{ch:classifying-spaces}

%
%

\section{$G$-CW-Complexes}

CW-complexes have been introduced by J.~H.~C.~Whitehead 
in~\cite{whitehead-49} and are widely known by now. The concept was 
generalised to the equivariant case in~\cite{matumoto-71}, 
\cite{illman-72} and~\cite{dieck-87}. In this thesis we use  the 
definition described in~\cite[pp.~6f.]{luck-89}. Even though we are 
concerned with the study of classifying spaces for 
discrete group we will state the definition of a $G$-CW-complex and 
of a classifying space first for topological groups before we pass to 
discrete groups in the subsequent studies.

By a \emph{topological group} $G$ we understand a group $G$ which is 
at the same time a Hausdorff space such that the the map 
\begin{align*}
    G\times G & \to G,
    \\
    (g,h) &  \mapsto gh^{-1}
\end{align*}
is continuous.

\begin{DEF}
    \cite[pp.~6f.]{luck-89} Let $G$ be a topological group acting
    continuously on a topological space~$X$.  A \emph{$G$-CW-complex
    structure on $X$} consists of
    \begin{enumerate}
        \item  a filtration $X_{0}\subset X_{1}\subset X_{2} \subset 
	\ldots$  of $X$ which exhausts $X$, and
    
        \item  a collection $\{ e_{i}^{n} : i\in I_{n} \}$ of 
	$G$-subspaces $e^{i}_{n}\subset X_{n}$ for each $n\in \N$,
    \end{enumerate}
    with the properties
    \begin{enumerate}
	\item $X$ has the weak topology with respect to the filtration
	$\{X_{n} : n\in \N\}$ (that is $B\subset X$ is closed in $X$
	if and only if $B\cap X_{n}$ is closed in~$X_{n}$ for every
	$n\in \N$);
	
	\item for each $n\geq 1$ there exists a $G$-pushout as in
	Figure~\ref{fig:def-G-CW-complex-pushout} such that $e_{i}^{n}
	= Q_{i}^{n}(G/H_{i} \times \Int D^{n})$.  Here the $H_{i}$ are
	closed subgroups of~$G$, the $q_i\: G/H_i \times S^{n-1}\to
	X_{n-1}$ are continuous maps and the $Q_i\: G/H_i\times D^n
	\to X_n$ are continuous maps corresponding to the~$q_i$.
    \end{enumerate}

    \begin{figure}[t]
	\centering 
	\begin{tikzpicture}
	    \node (A) at (0,0) 
	    {$\coprod_{i\in I_{n}} G/H_{i} \times S^{n-1}$};
	    \node (C) at (0,-2.8)
	    {$\coprod_{i\in I_{n}} G/H_{i} \times D^{n}$};
	    \node (B) at (4.5,0)
	    {$X_{n-1}$};
	    \node (D) at (4.5,-2.8)
	    {$X_{n}$};
	    \draw[-latex] (A) -- (B) 
	    node[pos=.5, above]{$\coprod q_{i}$};
	    \draw[right hook-latex] (A) -- (C);
	    \draw[-latex] (C) -- (D)  
	    node[pos=.5, above]{$\coprod Q_{i}$};
	    \draw[right hook-latex] (B) -- (D);
	\end{tikzpicture}
	\caption{Attaching equivariant $n$-cells to the
	$(n-1)$-skeleton of $X$, $n\geq 1$.}
	\label{fig:def-G-CW-complex-pushout}
    \end{figure}
    
    The $G$-subspace $X_{n}$ is called the \emph{$n$-skeleton} of $X$.
    The $e_{i}^{n}$ are called the \emph{open equivariant $n$-cells}
    of $X$.  The \emph{(closed) equivariant $n$-cells} are the
    $G$-subspaces $\bar e_{i}^{n} \defeq Q_{i}(G/H_{i}\times D^{n})$.
\end{DEF}

Note that if $G=1$ is the trivial group, one recovers from the above
definition the non-equivariant CW-complex in the sense
of~\cite{whitehead-49}.

If $G$ is a discrete group, then one can express the above definition
also in the following way: a CW-complex $X$ with a $G$-action is a
$G$-CW-complex if the action of $G$ on $X$ is cellular and the cell
stabilisers are the point stabilisers~\cite[p.~8]{luck-89}.  That is,
the action of $G$ on $X$ permutes the cells and any $g\in G$ which
fixes a cell fixes this cell pointwise.
    
There are various finiteness properties for $G$-CW-complexes which
are generalisations of the corresponding finiteness properties of
CW-complexes.  The following is a list of some common finiteness
properties.

\begin{DEF}
    \label{def:G-CW-complex-fincond}
    Let $X$ be a $G$-CW-complex as in the definition before.
    \begin{enumerate}
	\item If there exists a integer $n\geq -1$ such that $X=X_{n}$
	(with the convention that $X_{-1} \defeq \emptyset$), then the
	least such integer is called the dimension of $X$ and we
	denote this fact by $\dim X = n$.  If no such integer exists,
	then we say that $X$ is an \emph{infinite dimensional}
	$G$-CW-complex and we denote this fact by $\dim X = \infty$.
	
	\item We say that $X$ is of \emph{finite type} if it has only
	finitely many equivariant cells in each dimension.
	
	\item We say that $X$ is \emph{finite} if $X$ consists of only
	finitely many equivariant cells.
    \end{enumerate}
\end{DEF}

Note that a $G$-CW-complex is finite if and only if it is of finite 
type and finite dimension. Moreover, a $G$-CW-complex $X$ is finite 
if and only if the quotient space $X/G$ is compact.

%
%

\section{Classifying Spaces}
\label{sec:classifying-space}

In the literature there are several variations of the concept of a
universal $G$-space (also known as a classifying space of $G$) for the
family $\frakF$, see for example the survey article~\cite{luck-05},
which is also the source of the following definition.

\begin{DEF}
    \label{def:classifying-space}
    Let $G$ be a topological group and let $\frakF$ be a semi-full
    family of closed subgroups of $G$.  A $G$-CW-complex $X$ is a
    \emph{classifying space of $G$ for the family $\frakF$} or a
    \emph{model for $E_{\frakF}G$}, if it satisfies the following two
    conditions:
    \begin{enumerate}
        \item  $\frakF(X)\subset \frakF$;
    
        \item  if $Y$ is a $G$-CW-complex with $\frakF(Y)\subset 
	\frakF$, then there exists a $G$-map $f\: Y\to X$ which is 
	unique up to $G$-homotopy.
    \end{enumerate}
\end{DEF}

In other words, a model for $E_{\frakF}G$ is a terminal object in the 
homotopy category of $G$-CW-complexes with isotropy groups in the 
family $\frakF$. In particular, a model for $E_{\frakF}G$ is only 
unique up to $G$-homotopy and the $G$-homotopy class of a classifying 
space of $G$ for the family $\frakF$ can be seen as an invariant of 
the group $G$.

For any given group $G$ and semi-full family $\frakF$ of subgroups $G$
there exists a classifying space of $G$ for the family
$\frakF$~\cite[p.~275]{luck-05}.  Furthermore it has been shown
in~\cite[p.~275]{luck-05}, that a $G$-CW-complex $X$ with
$\frakF(X)\subset \frakF$ is a model for $E_{\frakF}G$ if and only if
the fixed point set $X^{H}$ is weakly contractible for every $H\in
\frakF$. A space $X$ is called \emph{weakly contractible} if the
homotopy groups $\pi_{n}(X,x)$ are trivial for all $n\in \N$ and all
$x\in X$.

A contractible space is always weakly contractible.  However, in
general a weakly contractible space does not need to be contractible.
But if $G$ is discrete and $X$ is a model for $E_{\frakF}G$, then for
every $H\in\frakF$ the fixed point space $X^{H}$ has the homotopy type
of a CW-complex and is therefore
contractible~\cite[pp.~219ff.]{whitehead-78}. Thus we obtain
the following known characterisation result,
see for example~\cite[p.~295]{luck-00}.

\begin{proposition}
    \label{prop:classifying-space-alt}
    Let $G$ be a discrete group and let $\frakF$ be a semi-full family
    of subgroups of $G$.  A $G$-CW-complex $X$ is a model for
    $E_{\frakF}G$ if and only if the following two conditions are 
    satisfied:
    \begin{enumerate}
        \item  $\frakF(X)\subset \frakF$;
    
        \item  $X^{H}$ is contractible for every $H\in \frakF$.\qed
    \end{enumerate}
\end{proposition}

For full families of subgroups the the above result has the following 
corollary, which is often used as the definition of a classifying 
space of discrete groups for full families of subgroups.

\begin{corollary}
    \label{cor:classifying-space-alt}
    Let $G$ be a discrete group and let $\frakF$ be a full family
    of subgroups of $G$.  A $G$-CW-complex $X$ is a model for
    $E_{\frakF}G$ if and only if the following two conditions are 
    satisfied:
    \begin{enumerate}
	\item $X^{H}=\emptyset$ for every subgroup $H$ of $G$ which is
	not in $\frakF$;
    
        \item  $X^{H}$ is contractible for every $H\in \frakF$.
    \end{enumerate}
\end{corollary}

\begin{proof}
    We only need to show that for full families $\frakF$ the 
    assumption~(1) in Proposition~\ref{prop:classifying-space-alt} is 
    equivalent to the assumption~(1) in this corollary.
    
    \smallskip
    
    ``$\Rightarrow$'': Let $H$ be a subgroup of $G$ which is not in
    the family $\frakF$.  Assume towards a contradiction that
    $X^{H}\neq \emptyset$ and let $x\in X^{H}$.  Then $H$ is a
    subgroup of $G_{x}$.  Now $G_{x}\in\frakF(X)\subset \frakF$ and
    since $\frakF$ is a full family of subgroups of $G$ we get $H\in 
    \frakF$ which is a contradiction! Therefore $X^{H}=\emptyset$.
    
    \smallskip 
    
    ``$\Leftarrow$'': \quad Let $H\in \frakF(X)$. Then $X^{H}\neq 
    \emptyset$ and thus $H\in \frakF$.
\end{proof}

\begin{examples}
    \begin{enumerate}
        \item  If $\frakF =\{1\}$ is the trivial family of subgroups 
	then a contractible $G$-CW-complex $X$ is a model for 
	$E_{\frakF}G$ if the action of $G$ of on $X$ is free. In 
	particular the universal cover of an Eilenberg--Mac~Lane 
	space $K(G,1)$ is a model for $E_{\frakF}G$. It is customary 
	to abbreviate $EG \defeq  E_{\frakF}G$.
    
        \item  If $\frakF =\Ffin(G)$ a model for $E_{\frakF}G$ 
	is also known as the universal space of proper actions of 
	$G$. In literature the abbreviation $\underline EG \defeq  
	E_{\frakF}G$ is commonly used.
	
	\item In the case that $\frakF = \Fvc(G)$, which is the family
	of subgroups on which focus of study of this thesis lies, the
	abbreviation $\uu EG \defeq E_{\frakF}G$ is used.
    \end{enumerate}
\end{examples}

%
%

\section{Free Resolutions Obtained from Classifying Spaces}
\label{sec:resolutions-from-classifying-space}

If $\frakF$ is a semi-full family of subgroups of $G$, then a model
for $E_{\frakF}G$ can be used to construct a free resolution of the
trivial $\OFG$-module $\underline\Z_{\frakF}$.  The construction is as
follows (see for example~\cite[pp.~151f.]{luck-89}
or~\cite[pp.~10ff.]{mislin-03}).

Let $X$ be a $G$-CW-complex.  Consider the \emph{cellular chain
complex $C(X) = (C_{*}(X), d_{*})$}.  This chain complex is defined by
\begin{equation*}
    C_{n}(X) \defeq  H^{\Delta}_{n}(X_{n}, X_{n-1})
\end{equation*}
with $H^{\Delta}_{n}$ denoting the singular homology functor. The 
differentials 
\begin{equation}
    d_{n}\: C_{n}(X)\to C_{n-1}(X)
    \label{eq:celular-chain-complex-differentials}
\end{equation}
of the cellular chain complex are the connecting homomorphisms of the
triple $(X_{n}, X_{n-1}, X_{n-2})$, see for
example~\cite[pp.~40ff.]{geoghegan-08}.

Let $\Delta_{n}$ be the set of all $n$-cells of the $G$-CW-complex 
$X$. Since $G$ acts on $X$ by permuting the cells of $X$ the set 
$\Delta_{n}$ is in a natural way a $G$-set. Note that 
$\Delta_{n}^{H}$ is the set of all $n$-cells of the CW-complex 
$X^{H}$ for any group subgroup $H$ of $G$.

We define the right $\OFG$-module $\underline{C}_{n}(X)$ to be
\begin{equation*}
    \underline{C}_{n}(X) \defeq  \Z[\?, \Delta_{n}]_{G}.
\end{equation*}
For each $n\geq 1$ we define homomorphisms 
\begin{equation*}
    d_{n}\: \underline{C}_{n}(X) \to \underline{C}_{n-1}(X)
\end{equation*}
of right $\OFG$-modules as follows.  First note that for every $H\in
\frakF$ and $n\geq 1$ we have $\underline C_{n}(X)(G/H) =
C_{n}(X^{H})$.  Let
\begin{equation*}
    d_{n,H}\: C_{n}(X)(G/H) \to C_{n-1}(X)(G/H)
\end{equation*}
be the differential $d_{n}\: C_{n}(X^{H}) \to C_{n-1}(X^{H})$ of the
cellular chain complex~$C(X^{H})$.  If $\varphi\in [G/H, G/K]_{G}$ is
a morphism of the orbit category $\OFG$, say $\varphi = f_{g,H,K}$,
then $\varphi^{*}\:C_{n}(X)(G/K) \to C_{n}(X)(G/H)$ is for each $n\in
\N$ the homomorphism induced by the map $X^{K}\to X^{H}$ which sends
$x$ to $gx$. Since this map defines a chain map $C(X^{K})\to 
C(X^{H})$ this implies that the diagram
\begin{equation*}
    \begin{diagram}
        \node{C_{n}(X^{K})}
	\arrow{e,t}{d_{n,K}}
	\arrow{s,l}{\varphi^{*}}
	\node{C_{n-1}(X^{K})}
	\arrow{s,r}{\varphi^{*}}
	\\
	\node{C_{n}(X^{H})}
	\arrow{e,t}{d_{n,H}}
	\node{C_{n-1}(X^{H})}
    \end{diagram}
\end{equation*}
commutes for every $n\geq 1$.  In particular this implies that the
homomorphism~$d_{n,H}$ define a homomorphism $d_{n}\:
\underline{C}_{n}(X) \to \underline{C}_{n-1}(X)$ of right
$\OFG$-modules.

Furthermore, for every $H\in \frakF$ there exists an augmentation 
homomorphism $\varepsilon_{H}\: C_{0}(X^{H})\to \Z$ which sends every 
$0$-cell of the CW-complex $X^{H}$ to $1$. It follows that these 
homomorphism define an \emph{augmentation homomorphism}
\begin{equation*}
    \varepsilon\: \underline{C}_{0}(X)\to \underline\Z_{\frakF}.
\end{equation*}

\begin{lemma}
    The sequence
    \begin{equation}
	\label{eq:resolution-2}
	\ldots\longto \underline{C}_{2}(X) \stackrel{d_{2}}{\longto}
	\underline{C}_{1}(X) \stackrel{d_{1}}{\longto}
	\underline{C}_{0}(X) \stackrel{\varepsilon}{\longto}
	\underline\Z_{\frakF} \longto 0
    \end{equation}
    is a chain complex of $\OFG$-modules.
\end{lemma}

\begin{proof}
    Let $H\in\frakF$. Then the sequence~\eqref{eq:resolution-2}
    evaluated at $G/H$ is the augmented cellular chain complex of
    the CW-complex $X^{H}$ and the claim follows.
\end{proof}

\begin{lemma}
    \label{lem:resolution-2-exact}
    Assume that $X^{H}$ is contractible for every $H\in\frakF$. Then
    the sequence~\eqref{eq:resolution-2} is exact.
\end{lemma}

\begin{proof}
    If $X^{H}$ is contractible then the augmented cellular chain
    complex of the CW-complex $X^{H}$ is exact. Thus the claim
    follows by evaluating the sequence~\eqref{eq:resolution-2} at
    $G/H$ for any $H\in\frakF$.
\end{proof}

The results of this section yield the following conclusion.

\begin{proposition}
    \label{prop:free-resolution-from-model}
    Let $\frakF$ be a semi-full family of subgroups and let $X$ be a
model
    for $E_{\frakF}G$.  Then the sequence~\eqref{eq:resolution-2} of
    right $\OFG$-modules is a free resolution of the trivial
    $\OFG$-module $\underline\Z_{\frakF}$.
\end{proposition}

\begin{proof}
    Since $X$ is a model for $E_{\frakF}G$, we have that
    $\frakF(X)\subset \frakF$ and so the $\OFG$-modules
    $\underline{C}_{n}(X)$ are free by
    Proposition~\ref{prop:GSet-to-ModOFG}.  The fixed point sets
    $X^{H}$ are contractible for any $H\in \frakF$ and therefore the
    sequence~\eqref{eq:resolution-2} is exact by
    Lemma~\ref{lem:resolution-2-exact}.
\end{proof}

%
%

\section{Geometric Finiteness Conditions in Terms of Algebraic 
Properties}
\label{sec:geometric-finiteness-conditions}

If follows from the construction of the previous section that the 
finiteness conditions of 
Definition~\ref{def:G-CW-complex-fincond} 
on a model $X$ for $E_{\frakF}G$ imply the following 
statements:
\begin{enumerate}
    \item if $\dim X = n$ then there exists a free
    resolution of the trivial $\OFG$-module $\underline
    \Z_{\frakF}$ of length $n$ in $\ModOFG$;

    \item if $X$ is of finite type then there exists a resolution
    of the trivial $\OFG$-module $\underline \Z_{\frakF}$ by finitely
    generated free Bredon modules in $\ModOFG$;

    \item if $X$ is finite then there exists a finite length
    resolution of the trivial $\OFG$-module $\underline \Z_{\frakF}$
    by finitely generated free Bredon modules in $\ModOFG$.
\end{enumerate}

In~\cite{luck-00} it has been shown that the above statements are 
nearly reversible. The relevant part of the main result 
in this article is the following

\begin{proposition}
    \label{prop:luck-meintrup}
    \cite[Theorem~0.1]{luck-00} Let $G$ be a discrete group, let
    $\frakF$ be a semi-full family of subgroups of $G$ and let $n\geq
    3$.  Then we have:
    \begin{enumerate}
        \item  there is a $n$-dimensional model for $E_{\frakF}G$ if 
	and only if there exists a projective resolution of the 
	trivial $\OFG$-module $\underline \Z_{\frakF}$ of length $n$ 
	in $\ModOFG$.
    
	\item there exists a finite type model for $E_{\frakF}G$ if
	and only if there exists a model for $E_{\frakF}G$ with finite
	equivariant $2$-skeleton and the trivial $\OFG$-module
	$\underline \Z_{\frakF}$ has a resolution by finitely
	generated projective Bredon modules in $\ModOFG$;
    
	\item there exists a finite model for $E_{\frakF}G$ if and
	only if there exists a model for $E_{\frakF}G$ with finite
	equivariant $2$-skeleton and the trivial $\OFG$-module
	$\underline \Z_{\frakF}$ has a resolution of finite length by
	finitely generated \emph{free} Bredon modules in
	$\ModOFG$.\qed
    \end{enumerate}
\end{proposition}

In this thesis we focus on the question whether for a group $G$ and 
full family $\frakF$ of subgroups of $G$, there exists a finite 
dimensional model for $E_{\frakF}G$. This leads to the following 
definition.

\begin{DEF}
    Let $G$ be a group and $\frakF$ semi-full family of subgroups.
    Assume that there exists a finite dimensional model for
    $E_{\frakF}G$.  Then the least integer $n\geq 0$ for which there
    exists an $n$-dimensional model for $E_{\frakF}G$ is called the
    \emph{Bredon geometric dimension of $G$ for the family $\frakF$}
    and we denote this by~$\gd_{\frakF} G \defeq n$.  If there exist
    no finite dimensional model for $E_{\frakF}G$, then we
    set~$\gd_{\frakF}G \defeq \infty$.
\end{DEF}

Following the notation introduced at the end of
Section~\ref{sec:classifying-space} we abbreviate~$\gd_{\frakF}G$ by
$\gd G$ if $\frakF=\{1\}$, by $\ugd G$ if $\frakF = \Ffin(G)$
and by $\uugd G$ if~$\frakF = \Fvc(G)$.

%
%

\chapter{Bredon (Co-)Homological Dimensions}
\label{ch:bredon-dimensions}

%
%

\section{Bredon (Co-)Homology}

Since the category $\ModOFG$ has enough projectives we can define
derived functors and do homological
algebra~\cite[pp.~30ff.]{weibel-94}.  We are interested in the derived
functors of the morphism functor $\mor_{\frakF}(\?, \??)$ and the
tensor product functor $\?\otimes_{\frakF} \??$.  Therefore,
for each right $\OFG$-module $M$, we choose a projective resolution
$P_{*}(M)$ of $M$.

\begin{DEF}
    Let $N$ be a right $\OFG$-module. Then $\Ext^{n}_{\frakF}(\?, 
    N)$ is the $n$-th right derived functor of $\mor_{\frakF}(\?,N)$, 
    that is
    \begin{equation*}
        \Ext^{n}_{\frakF}(M,N) \defeq  H_{n}(\mor_{\frakF}(P_{*}(M),N).
    \end{equation*}
    for any right $\OFG$-module $M$ and all $n\in \N$.
    
    Likewise, if $N$ is a left $\OFG$-module, then
    $\Tor_{n}^{\frakF}(\?, N)$ is the $n$-th left derived functor of 
    $\?\otimes_{\frakF} N$, that is
    \begin{equation*}
        \Tor_{n}^{\frakF}(M,N) \defeq  H_{n}(P_{*}(M)\otimes_{\frakF}N)
    \end{equation*}
    for all right $\OFG$-modules $M$ and all $n\in \N$.
\end{DEF}

It is a standard fact in homological algebra that this definition is
-- up to natural isomorphism -- independent of the choice of the
projective resolutions. Furthermore it is a standard fact that the 
$\Ext^{*}_{\frakF}$ and $\Tor_{*}^{\frakF}$ functors are also 
functorial in the second variable.

\begin{proposition}
    \label{prop:projectives-Ext-classification}
    The following statements about a right $\OFG$-module $M$ are 
    equivalent:
    \begin{enumerate}
        \item $M$ is projective;
	
	\item $\mor_{\frakF}(M, \?)$ is an exact functor;
	
	\item $\Ext^{n}_{\frakF}(M, N) = 0$ for every right
	$\OFG$-module $N$ and every $n\geq 1$;
	
	\item $\Ext^{1}_{\frakF}(M, N) = 0$ for every right
	$\OFG$-module $N$.
    \end{enumerate}
\end{proposition}

\begin{proposition}
    \label{prop:flats-Tor-classification}
    The following statements for a right $\OFG$-module $M$ are 
    equivalent:
    \begin{enumerate}
        \item $M$ is flat;
	
	\item $M \mathbin{\otimes_{\frakF}} \?$ is an exact functor;
	
	\item $\Tor_{n}^{\frakF}(M, N) = 0$ for every left 
	$\OFG$-module $N$ and every $n\geq 1$;
	
	\item $\Tor_{1}^{\frakF}(M, N) = 0$ for every left 
	$\OFG$-module $N$.
    \end{enumerate}
\end{proposition}

\begin{proof}[Proof of Proposition~\ref{prop:projectives-Ext-classification} 
    and~\ref{prop:flats-Tor-classification}] These are standard results in
    homological algebra, see for example~\cite[p.~50]{weibel-94} and
    ~\cite[p.~69]{weibel-94}.
\end{proof}

Let $N$ be a left $\OFG$-module.  We say that a right $\OFG$-module
$M$ is~\emph{$\?  \mathbin{\otimes_{\frakF}} N$-acyclic} if the groups
$\Tor_{n}^{\frakF}(M, N)$ vanish for every $n\geq 1$.  Thus $M$ is
flat if and only if it is $\?  \mathbin{\otimes_{\frakF}} N$-acyclic
for any left $\OFG$-module $N$.

Note that from the theory of derived functors, it follows that we can
relax the requirement on the resolution of $M$ used to calculate the
$\Tor$ groups.  In fact any $\?  \mathbin{\otimes_{\frakF}} N$-acyclic
resolution of $M$ will be sufficient~\cite[p.~47]{weibel-94}. As this
requirement is satisfied by flat $\OFG$-modules this means we can
calculate the $\Tor$ groups using flat resolutions.

\begin{DEF}
    Let $G$ be a group, $\frakF$ a family of subgroups of
    $G$ and let~$M$ be a right $\OFG$-module.  Then the
    \emph{Bredon cohomology groups $H^{n}_{\frakF}(G; M)$ of $G$ with
    coefficients in $M$} are defined as the $\Ext$ groups of the
    trivial $\OFG$-module~$\underline\Z_{\frakF}$ with coefficients in
    $M$, that is
    \begin{equation*}
        H^{n}_{\frakF}(G; M) \defeq 
	\Ext^{n}_{\frakF}(\underline\Z_{\frakF}, M).
    \end{equation*}
    Similarly, if $N$ is a left $\OFG$-module, then the \emph{Bredon
    homology groups $H_{n}^{\frakF}(G; N)$ of $G$ with coefficients in
    $N$} are defined to be the $\Tor$ groups of the trivial
    $\OFG$-module $\underline\Z_{\frakF}$ with coefficients in $N$,
    that is
    \begin{equation*}
        H_{n}^{\frakF}(G; N) \defeq 
	\Tor_{n}^{\frakF}(\underline\Z_{\frakF}, N).
    \end{equation*}
\end{DEF}

%
%

\section{The Standard Resolution}
\label{sec:standard-resolution}

Concrete examples of projective resolutions of the trivial
$\OFG$-module~$\underline\Z_{\frakF}$ are useful in order to
calculate the Bredon (co-)homology groups of a group~$G$.  In
Section~\ref{sec:resolutions-from-classifying-space} we have already 
seen how to obtain a free resolution of $\underline \Z_{\frakF}$ 
from a model for $E_{\frakF}G$. Another example of a very specific 
free resolution of~$\underline\Z_{\frakF}$ is the standard resolution.

Recall that in classical group (co-)homology there
exists the standard resolution
\begin{equation}
    \label{eq:classic-standard-resolution}
    \ldots \to \Z[G\times G\times G] \to \Z[G\times G] \to \Z G \to \Z
    \to 0
\end{equation}
of the trivial $G$-module $\Z$ by free $G$-modules, see for
example~\cite[pp.~18f.]{brown-82}.  This resolution is also known as
the bar resolution.  From the view point of category theory, standard
resolutions arise from simplicial objects constructed from comonads
(also known as triples), see~\cite[pp.~180ff.]{mac-lane-98}
and~\cite[pp.~278ff.]{weibel-94}.

Nucinkis has shown in~\cite[pp.~41f.]{nucinkis-04} how the 
construction~\eqref{eq:classic-standard-resolution} generalises to 
the Bredon setting for the family $\frakF = \Ffin(G)$ of finite
subgroups of 
$G$. That is, there exists a free resolution
\begin{equation*}
    \ldots \to \Z[\?, \Delta_{2}]_G \to \Z[\?,\Delta_{1}]_G \to \Z[\?,
    \Delta_{0}]_G \to \underline\Z_{\frakF}\to 0
\end{equation*}
of the trivial $\OFG$-module $\underline\Z_{\frakF}$ where 
$\Delta_{n}$ is the $G$-set
\begin{equation*}
    \Delta_{n} \defeq  \{ (g_{0}K_{0}, \ldots, g_{n}K_{n}) : g_{i}\in G 
    \text{ and } K_{i}\in \frakF\}.
\end{equation*}

It turns out that we can construct a resolution of this form of the
trivial~$\OFG$-module $\underline\Z_{\frakF}$ for an arbitrary
non-empty family $\frakF$ of subgroups of~$G$.  The details are as
follows.

For $n\geq 1$ and $0\leq i\leq n$ define $G$-maps $\partial_{i}\:
\Delta_{n}\to \Delta_{n-1}$ by
\begin{equation*}
    \partial_{i}(g_{0}K_{0}, \ldots, g_{n}K_{n}) \defeq
    (g_{0}K_{0},\ldots, \widehat{g_{i}K_{i}}, \ldots, g_{n}K_{n})
\end{equation*}
where $(g_{0}K_{0},\ldots, \widehat{g_{i}K_{i}}, \ldots, g_{n}K_{n})$
denotes the $n$-tuple obtained from the $(n+1)$-tuple $(g_{0}K_{0},
\ldots, g_{n}K_{n})$ by deleting the $i$-th component.  With these
maps the collection $\Delta_{*} \defeq  \{\Delta_{n} : n\in\N \}$ of
$G$-sets becomes a semi-simplicial complex.

Let $\Delta_{-1}$ be the singleton $G$-set $\Delta_{-1} \defeq  \{ * \}$.
We get an augmentation $G$-map $\varepsilon\: \Delta_{0}\to
\Delta_{-1}$ if we set $\varepsilon(g_{0}K_{0}) \defeq  *$ for every
$g_{0}K_{0}\in\Delta_{0}$, that is $\varepsilon \comp \partial_{0} =
\varepsilon \comp \partial_{1}$.

Applying the functor $\Z[\?,  \??]_G$ to the semi-simplicial $G$-set
$\Delta$ gives a semi-simplicial $\OFG$-module
\begin{equation*}
    \Z[\?, \Delta_{*}]_G \defeq  \{ \Z[\?, \Delta_{n}]_G : n\in \N \}
\end{equation*}
with augmentation $\varepsilon\: \Z[\?, \Delta_{0}]_G \to \Z[\?,
\{*\}]_G = \underline\Z_{\frakF}$.
The associated augmented chain complex $C_{*}(\Delta_{*})$
is given by
\begin{equation*}
    C_{n}(\Delta_{*}) \defeq 
    \begin{cases}
	\Z[\?, \Delta_{*}]_{G}  & n\geq -1   \\[0.7ex]
	0 & \text{otherwise}
    \end{cases}    
\end{equation*}
with the differentials given by
\begin{equation*}
    d_{n} \defeq  
    \begin{cases}
        \sum_{i=0}^{n} (-1)^{i} \partial_{i} & n > 0
	\\[0.7ex]
        \varepsilon & n=0
	\\[0.7ex]
        0 & n<0
    \end{cases}
\end{equation*}
It follows that $C_{*}(\Delta_{*})$ is necessarily a complex of
$\OFG$-modules, that is $d_{n-1} \comp d_{n} = 0$ for every
$n\in\Z$~\cite[pp.~259ff.]{weibel-94}.

\begin{proposition}
    \label{prop:standard-resolution}
    The sequence
    \begin{equation}
	\label{eq:standard-resolution}
	\ldots \longto \Z[\?, \Delta_{2}]_{G} \stackrel{d_{2}}{\longto}
	\Z[\?, \Delta_{1}]_{G} \stackrel{d_{1}}{\longto} \Z[\?,
	\Delta_{0}]_{G} \stackrel{\varepsilon}{\longto}
	\underline\Z_{\frakF} \longto 0
    \end{equation}
    is a resolution of the trivial $\OFG$-module 
    $\underline\Z_{\frakF}$. If $\frakF$ is a semi-full family of 
    subgroups then this resolution is free.
\end{proposition}

\begin{proof}
    First observe that $\Z[\?, \Delta_{-1}]_{G} =
    \underline\Z_{\frakF}$.  Thus the
    sequence~\eqref{eq:standard-resolution} is nothing else than the
    associated augmented chain complex $C_{*}(\Delta_{*})$.
    
    We need to show that the sequence~\eqref{eq:standard-resolution}
    evaluated at any object $G/H$ of the orbit category $\OFG$ is an
    exact sequence of abelian groups.  We know already that
    \begin{equation*}
	 \ldots \longto \Z[G/H, \Delta_{2}]_{G}
	 \stackrel{d_{2,H}}{\longto} \Z[G/H, \Delta_{1}]_{G}
	 \stackrel{d_{1,H}}{\longto} \Z[G/H, \Delta_{0}]_{G}
	 \stackrel{\varepsilon_{H}}{\longto} \underline\Z_{\frakF}
	 \longto 0
    \end{equation*}
    is a chain complex of abelian groups.  Thus it remains to show
    that there exists a contracting homotopy $h\: \id \homotop 0$.
    But such a contracting homotopy is known to be given by
    \begin{equation*}
	h_{n}(g_{0}K_{0}, \ldots, g_{n}K_{n}) \defeq  (H, g_{0}K_{0},
	\ldots, g_{n}K_{n})
    \end{equation*}
    for $n\in\N$, $h_{-1}(*) \defeq  (H)$ and $h_{n} \defeq  0$ for $n>-1$.
    
    Given $\sigma \defeq  (g_{0}K_{0}, \ldots,g_{n}K_{n})\in\Delta_{n}$ its
    stabiliser $G_{\sigma}$ is
    \begin{equation*}
        G_{\sigma} = K_{0}^{g_{0}^{-1}} \cap \ldots \cap
	K_{n}^{g_{n}^{-1}}.
    \end{equation*}
    If $\frakF$ is a semi-full family of subgroups of $G$, then
    $G_{\sigma}\in\frakF$ for any~$\sigma\in \Delta_{n}$.  That is
    $\frakF(\Delta_{n})\subset \frakF$.  Hence the $\OFG$-modules
    $\Z[\?, \Delta_{n}]_{G}$ are free by
    Proposition~\ref{prop:GSet-to-ModOFG}.
\end{proof}

\begin{DEF}
    We call the resolution~\eqref{eq:standard-resolution} the
    \emph{standard resolution} of the trivial $\OFG$-module
    $\underline\Z_{\frakF}$.
\end{DEF}

\begin{lemma}
    \label{lem:standard-resolution-lemma1}
    Let $G$ be a group and let $\frakF$ be a semi-full family of
    subgroups of $G$.  If both $G$ and $\frakF$ are countable then the
    standard resolution of the trivial~$\OFG$-module
    $\underline\Z_{\frakF}$ is countably generated.
\end{lemma}

\begin{proof}
    If $G$ and $\frakF$ are countable then $\Delta_{n}$ is countable
    and thus $\Delta_{n}/G$ is countable. Now the claim follows from 
    Lemma~\ref{lem:countability-criterion}.
\end{proof}

%
%

\section{Bredon (Co-)Homological Dimensions}

In Section~\ref{sec:geometric-finiteness-conditions} in the previous
chapter we have introduced the Bredon geometric dimension of a
group~$G$ with respect to a family $\frakF$ of subgroups of~$G$.  It
has two closely related algebraic invariants, the Bredon cohomological
and Bredon homological dimension.  They are the obvious
generalisations of the classical (co-)homological dimensions.

\begin{DEF}
    Let $G$ be a group and let $\frakF$ be a family of subgroups
    of~$G$.  Assume that there exists an integer $n\in \N$ such that
    the trivial $\OFG$\-/module~$\underline\Z_{\frakF}$ has a
    projective resolution
    \begin{equation*}
        0\to P_{n} \to \ldots \to P_{1}\to P_{0}\to
	\underline\Z_{\frakF} \to 0
    \end{equation*}
    in $\ModOFG$ of length $n$ but not one of length $n-1$.  We say
    that $G$ has \emph{Bredon cohomological dimension $n$ with respect
    to the family $\frakF$,} which we denote by $\cd_{\frakF} G \defeq  n$.
    If no finite length projective resolution exists, then we say that
    $G$ has \emph{infinite Bredon cohomological dimension with respect
    to $\frakF$,} which we denote by $\cd_{\frakF} G \defeq  \infty$.
\end{DEF}

We abbreviate $\cd_{\frakF} G$ by $\cd G$, $\ucd G$ or 
$\uucd G$ in the case that $\frakF$ is the trivial family, the family 
of finite or the family of virtually cyclic subgroups respectively.

The definition of the Bredon homological dimension follows the same
idea, except that projective $\OFG$-modules are replaced by flat
$\OFG$-modules:

\begin{DEF}
    Let $G$ be a group and let $\frakF$ be a family of subgroups
    of~$G$.  Assume that there exists an integer $n\in \N$ such that
    the trivial $\OFG$\-/module~$\underline\Z_{\frakF}$ has a flat
    resolution
    \begin{equation*}
        0\to Q_{n} \to \ldots \to Q_{1}\to Q_{0}\to
	\underline\Z_{\frakF} \to 0
    \end{equation*}
    in $\ModOFG$ of length $n$ but not one of length $n-1$.  We say
    that $G$ has \emph{Bredon homological dimension $n$ with respect
    to the family $\frakF$,} which we denote by $\hd_{\frakF} G \defeq
    n$.  If no finite length flat resolution exists then we say that
    $G$ has \emph{infinite Bredon homological dimension with respect
    to $\frakF$,} which we denote by $\hd_{\frakF} G \defeq \infty$.
\end{DEF}

Analogous to before we abbreviate $\hd_{\frakF} G$ by $\hd G$, $\uhd 
G$ or $\uuhd G$ in the case that $\frakF$ is the trivial family, the
family 
of finite or the family of virtually cyclic subgroups respectively.

The Bredon (co-)homological dimension is a special case of the
projective and flat dimension of a right $\OFG$-module $M$.  These
dimensions are defined in a similar spirit as the minimal length of a
projective (or respectively flat) resolution of the $\OFG$-module $M$.
We denote the \emph{projective dimension} of $M$ by $\pd_{\frakF} M$
and the \emph{flat dimension} of $M$ by $\fld_{\frakF} M$.  With this
notation, the (co-)homological dimension of a group $G$ is the
projective and flat dimension of the trivial $\OFG$-module
$\underline\Z_{\frakF}$, that is
\begin{equation*}
    \cd_{\frakF} G \defeq \pd_{\frakF} \underline\Z_{\frakF} \qquad
    \text{and} \qquad \hd_{\frakF} G \defeq \fld_{\frakF}
    \underline\Z_{\frakF}.
\end{equation*}

The following two propositions are standard results in homological
algebra in abelian categories; their proof can be found
in~\cite[pp.~93ff.]{weibel-94}, for example.

\begin{proposition}
    \label{prop:pd-via-ext}
    Let $M$ be a right $\OFG$-module.  Then the following statements
    are equivalent:
    \begin{enumerate}
        \item  $\pd_{\frakF} M\leq n$;
    
        \item  $\Ext^{d}_{\frakF}(M,N) = 0$ for every right
	$\OFG$-module $N$ and every $d>n$;
	
        \item  $\Ext^{n+1}_{\frakF}(M,N) = 0$ for every right
	$\OFG$-module $N$;
    
	\item  given any projective resolution of $M$,
	\begin{equation*}
	    \ldots\to P_{2}\to P_{1}\to P_{0} \to M \to 0,
	\end{equation*}
	the kernel of $P_{n}\to P_{n-1}$ is projective.\qed
    \end{enumerate}
\end{proposition}

There are two immediate applications of this result.  The first is
that if we can show that $\Ext^{d}_{\frakF}(M,N) \neq 0$ for some
right $\OFG$-module $N$, then $\pd_{\frakF} M\geq d$.  The second
application is that, given a projective resolution $P_{*}$ of a right
$\OFG$-module $M$ with $\pd_{\frakF} M \leq n$, we obtain a projective
resolution
\begin{equation*}
    0 \to K \to P_{n-1} \to \ldots \to P_{0} \to M \to 0
\end{equation*}
of length $n$, where $K$ is the kernel of $d_{n}\: P_{n}\to P_{n-1}$
and $K\to P_{n-1}$ is the restriction of $d_{n}$ to $K$. That is, any 
projective resolution of $M$ can be truncated by inserting a suitable 
projective kernel as soon as the length of the resolution exceeds the 
projective dimension of~$M$.

\begin{proposition}
     \label{prop:fld-via-tor}
    Let $M$ be a right $\OFG$-module. Then the following statements
are
    equivalent:
    \begin{enumerate}
        \item  $\fld_{\frakF} M\leq n$;
    
        \item  $\Tor_{d}^{\frakF}(M,N) = 0$ for every left
	$\OFG$-module $N$ and every $d>n$;
	
        \item  $\Tor_{n+1}^{\frakF}(M,N) = 0$ for every left
	$\OFG$-module $N$;
    
	\item  given any flat resolution of $M$,
	\begin{equation*}
	    \ldots\to Q_{2}\to Q_{1}\to Q_{0} \to M \to 0,
	\end{equation*}
	the kernel of $Q_{n}\to Q_{n-1}$ is flat.\qed	
    \end{enumerate}
\end{proposition}

Of course this proposition has two analogous immediate applications,
just as the previous proposition had.  Firstly, a non-trivial
$\Tor_{d}^{\frakF}(M,N)$ gives rise to a lower bound for the flat
dimension of $M$.  Secondly, any flat resolution of~$M$ can be
truncated by inserting a suitable flat kernel as soon as the length
of the
resolution exceeds the flat dimension of $M$.

%
%

\section{Cohomological vs.~Homological vs.~Geometric Dimension}
\label{sec:cd-vs-hd-vs-gd}

In this section, we will compare the three Bredon dimensions we have
introduced in the previous section for a fixed family $\frakF$ of
subgroups of $G$.  The first result is just a direct consequence of
the
fact that projective Bredon modules are flat.

\begin{lemma}
    \label{lem:hd-vs-cd}
    For any family $\frakF$ of subgroups of a group $G$ we have
    \begin{equation*}
        \hd_{\frakF} G \leq \cd_{\frakF} G.\tag*{\qed}
    \end{equation*}
\end{lemma}

The next result has been proven by Nucinkis
in~\cite[p.~42]{nucinkis-04} for the family of finite subgroups of
$G$.  The proof also works without modification for more general
families of subgroups.

\begin{theorem}
    \label{thrm:hd-vs-cd}
    Let $G$ be a countable group and let $\frakF$ be a full family of
    subgroups of $G$.  If $\frakF$ is countable then
    \begin{equation*}
	\cd_\frakF G \leq \hd_\frakF G +1.
    \end{equation*}
\end{theorem}

To prove this theorem, we need a result by Nucinkis that gives an 
upper bound on the cohomological dimension of countably presented 
flat modules.

\begin{proposition}\cite[Proposition~3.5]{nucinkis-04}
    \label{prop:nucinkis-04-prop3.5}
    Let $\frakF$ be a full family of subgroups.  Then every countably
    presented flat right $\OFG$-module $M$ has $\pd_{\frakF} M \leq
    1$.\qed
\end{proposition}

\begin{proof}[Proof of Theorem~\ref{thrm:hd-vs-cd}]
    In order to avoid triviality, we assume that $\hd_{\frakF} G$ is
    finite.  Consider the standard resolution
    \begin{equation*}
	\ldots\to F_{2} \to F_{1} \to F_{0} \to \underline\Z_{\frakF} 
	\to 0
    \end{equation*}
    of the trivial $\OFG$-module $\underline\Z_{\frakF}$ as defined in
    section~\ref{sec:standard-resolution}.  Since $G$ and $\frakF$ are
    countable, this resolution is countably generated by
    Lemma~\ref{lem:standard-resolution-lemma1}.
    
    Let $K$ be the $n$-th kernel of the above resolution, which is
    flat.  It follows that $K$ is countably presented.  Since the
    family $\frakF$ of subgroups is assumed to be full we can apply
    Proposition~\ref{prop:nucinkis-04-prop3.5}, which tells us that
    $\pd_{\frakF} K \leq 1$.  From this it follows that there exists a
    projective resolution of the trivial $\OFG$-module
    $\underline\Z_{\frakF}$ of length $n+1$ and hence $\cd_{\frakF}
    G\leq n+1$.
\end{proof}

\begin{remark}
    \label{rem:F-countable}
    Let $G$ be a countable group. Then $G$ has only countably many
    finitely generated subgroups.
\end{remark}

Hence in the case $\frakF = \Ffin(G)$, Theorem~\ref{thrm:hd-vs-cd} and
Lemma~\ref{lem:hd-vs-cd} combine to recover the statement of Theorem
4.1 in~\cite{nucinkis-04}.  Moreover, since virtual cyclic groups are
finitely generated we get that Theorem 4.1 in~\cite{nucinkis-04} also
holds for the family of virtual cyclic subgroups of $G$, that is to
say we have the following result.

\begin{theorem}
    \label{thrm:hd-vs-cd2}
    Let $G$ be a countable group.  Then
    \begin{equation*}
	\uucd G \leq \uuhd G + 1.\tag*{\qed}
    \end{equation*}
\end{theorem}

Next we compare the Bredon cohomological dimension to the Bredon
geometric dimension of a group in the case that $\frakF$ is a full
family of subgroups of $G$.  In
Section~\ref{sec:resolutions-from-classifying-space} we have seen that
a model for $E_{\frakF}G$ gives rise to a projective resolution of the
trivial $\OFG$-module $\underline\Z_{\frakF}$.  If the model of
$E_{\frakF}G$ has finite dimension $n$, then the projective
resolution of
$\underline\Z_{\frakF}$ has length $n$.  Thus we have the following
result.

\begin{lemma}
    For any semi-full family $\frakF$ of subgroups of $G$ we have
    \begin{equation*}
        \cd_{\frakF} G \leq \gd_{\frakF} G.\tag*{\qed}
    \end{equation*}
\end{lemma}

As a consequence of the first part of
Proposition~\ref{prop:luck-meintrup}
we get the following statement about the geometric and cohomological
Bredon dimension of a group $G$.

\begin{proposition}
    \label{prop:luck-meintrup-alt}
    Let $\frakF$ be a semi-full family of subgroups of $G$.  If
    $\cd_{\frakF} G \geq 3$ or $\gd_{\frakF} G\geq 4$ then
    $\cd_{\frakF} G = \gd_{\frakF} G$.
\end{proposition}

\begin{proof}
    If $\cd_{\frakF}\geq 3$ then there exists a model $E_{\frakF}G$ of
    dimension $\cd_{\frakF} G$ by
    Proposition~\ref{prop:luck-meintrup}, that is $\gd_{\frakF} G\leq
    \cd_{\frakF} G$ and thus equality holds by the previous Lemma.  If
    $\gd_{\frakF} G\geq 4$ then there exists no projective resolution
    of length $\gd_{\frakF} G - 1$ by
    Proposition~\ref{prop:luck-meintrup} and therefore $\cd_{\frakF} G
    \geq \gd_{\frakF} G$.  Again equality holds by the previous lemma.
\end{proof}

In~\cite{eilenberg-57} it has been shown that $\cd G = \gd G$ whenever
$\cd G\geq 3$.  In the same paper it has been asked whether $\cd G =
\gd G$ in general.  The statement that this is true is known as the
Eilenberg--Ganea Conjecture.  Groups with $\cd G\leq 1$ cannot give
counter examples to this conjecture since~$\cd G=0$ implies that $G$
is trivial and $\cd G=1$ implies that $G$ is free by a famous work of
Stallings~\cite{stallings-68a} and Swan~\cite{swan-69}.  The trivial
group has geometric dimension~$0$ and since free groups can act freely
on a tree it follows that $\gd G=1$ for free groups.  Therefore a
possible counter example to the Eilenberg--Ganea Conjecture needs to
be a torsion free group $G$ with~$\cd G=2$ and $\gd G=3$.  Until the
present day, neither such counter example has been found nor has the
conjecture been proven.

The conjecture generalises to the Bredon setting as follows.

\begin{EGC}[for Bredon Cohomology]
    Let $\frakF$ be a semi-full of subgroups of a group $G$.  Then
    $\cd_{\frakF} G = \gd_{\frakF} G$.
\end{EGC}

Let $G$ be a group and consider $\frakF=\Ffin(G)$.  We see in the next
section that the implication $\ugd G=0 \Rightarrow \ucd G=0$ is
actually an equivalence.  Therefore $\ugd G=1$ implies $\ucd G=1$.  On
the other hand it is known that~$\ucd G=1$ implies that the rational
cohomological dimension $\cd_{\Q} G=1$, see for
example~\cite[p.~493]{brady-01}, which in turn implies that $\ugd G=1$
by a result by Dunwoody~\cite{dunwoody-79}.  Altogether, this and
Proposition~\ref{prop:luck-meintrup-alt} imply that $\ucd G = \ugd G$
for all groups $G$ with $\ucd G\neq 2$.  Thus a possible counter
example for the above conjecture for the family of finite subgroups
must satisfy $\ucd G=2$ and $\ugd G=3$.  Brady, Leary and Nucinkis
have shown in~\cite{brady-01} that there exist certain right-angled
Coxeter groups which have precisely this property.  That is, the
Eilenberg--Ganea Conjecture is false for the family of finite
subgroups.  This also implies that the statement of
Proposition~\ref{prop:luck-meintrup-alt} is the best possible if one
does not impose any further restriction on the family~$\frakF$.

A natural question question is to ask whether the Eilenberg--Ganea
Conjecture is true for the family $\frakF= \Fvc(G)$ of virtually
cyclic subgroups.  It is unknown, whether in $\uucd G=1$ is equivalent
to $\uugd G=1$ (we will show in the end of this thesis that this is
true for countable, torsion-free soluble groups, see
Theorem~\ref{thrm:classification}).  Therefore possible counter
examples to the conjecture must fall -- similar
to~\cite{eilenberg-57} 
-- into one of the following three 
cases.
\begin{enumerate}
    \item  $\uucd G=1$ and $\uugd G=2$;

    \item  $\uucd G=1$ and $\uugd G=3$;

    \item  $\uucd G=2$ and $\uugd G=3$;
\end{enumerate}
There is not much known about groups $G$ with $\uugd G=2$ or $\uugd
G=3$.  Juan-Pineda and Leary have shown in~\cite{juan-pineda-06} that
for Gromov-hyperbolic groups~$\ugd G\leq 2$ implies that $\uugd G=2$
provided that $G$ is not virtually cyclic.  Lück and Weiermann have
shown that $\vcd G=2$ implies $\uugd G=3$ for virtually polycyclic
groups.  In the next chapter we will show that the Eilenberg--Ganea
Conjecture for the family $\Fvc$ holds for these groups.  Moreover, we
will show in Chapter~\ref{ch:geometric-interlude} that the soluble
Baumslag--Solitar groups $BS(1,m)$ staisfy the Eilenberg--Ganea
Conjecture for the family $\Fvc$.

%
%

\section{Groups of Bredon Dimension Zero}

\begin{proposition}
    \label{prop:gdG=0}
    Let $G$ be a group and $\frakF$ a semi-full family of subgroups of
    $G$.  Then $\gd_{\frakF} G= 0$ if and only if $G\in \frakF$.
\end{proposition}

\begin{proof}
    Again the proof is a standard argument.  If $G\in \frakF$ then any
    singleton space $\{*\}$ is a model for $E_{\frakF}G$ and thus
    $\gd_{\frakF} G = 0$.  On the other hand, if $\gd_{\frakF} G=0$,
    then $G$ has a singleton space $\{*\}$ as a model for
    $E_{\frakF}G$ as this is the only $0$-dimensional contractible
    $G$-CW-complex which exists.  Clearly $\{*\}^{G} \neq \emptyset$
    and thus $G\in \frakF$.
\end{proof}

\begin{proposition}
    \label{prop:cdG=0}
    Let $G$ be a group and $\frakF$ a family of subgroups of $G$. 
    If $G\in \frakF$, then the trivial $\OFG$-module 
    $\underline\Z_{\frakF}$ is free and in particular $\cd_{\frakF} 
    G= 0$. If the family $\frakF$ is semi-full, then $\cd_{\frakF} 
    G=0$ implies that $G\in \frakF$.
\end{proposition}

In order to prove this statement we need a result from Symonds.  For a
family $\frakF$ of subgroups of $G$, he defines a \emph{component} of
$\frakF$ to be an equivalence class under the equivalence relation
generated by inclusion~\cite[p.~265]{symonds-05}.  Note that if
$\frakF$ is a semi-full family of subgroups of $G$, then $\frakF$ has
only one component.  This is because for any $H_{1},H_{2}\in \frakF$,
we have that $H_{1}\cap H_{2}$ is contained in $\frakF$ and is a
common subgroup of $H_{1}$ and $H_{2}$.

\begin{lemma}
    \cite[Lemma~2.5]{symonds-05}
    \label{lem:symonds-05-lem2.5}
    Let $\frakF$ be a family of subgroups of~$G$.  Then the trivial
    $\OFG$-module $\underline\Z_{\frakF}$ is projective if and only if
    each component of $\frakF$ has a unique maximal element $M$ and
    this $M$ is equal to its normaliser~$N_{G}(M)$ in $G$.\qed
\end{lemma}

\begin{proof}[Proof of Proposition~\ref{prop:cdG=0}]
    Assume first that $G\in \frakF$, then a standard argument shows
    that $\cd_{\frakF} G=0$ as follows.  Since $[G/H, G/G]_{G}$
    contains only one map (namely the trivial map) we have that
    $\Z[G/H,G/G]_{G} = \Z$ for every $H\in \frakF$.  Furthermore,
    every morphism of the orbit category $\OFG$ is mapped to the
    identity map.  Thus it follows that the trivial
    $\OFG$-module~$\underline\Z_{\frakF}$ is equal to the free right
    $\OFG$-module $\Z[\?, G/G]_{G}$.  Hence $\cd_{\frakF} G=0$.
    
    Next, assume that the family $\frakF$ is semi-full and that
    $\cd_{\frakF}G =0$.  Let $H_{1},H_{2}\in \frakF$ be two arbitrary
    subgroups.  Since $\frakF$ is closed under finite intersections it
    has only one component.  Then by Lemma~\ref{lem:symonds-05-lem2.5}
    the family~$\frakF$ has a unique maximal element $M$ with
    $M=N_{G}(M)$.  Assume towards a contradiction that $N_{G}(M)$ (and
    therefore $M$) is a proper subgroup of $G$.  Then there exists a
    $g\in G\setminus M$ such that $M^{g}\neq M$.  Since $\frakF$ is
    closed under conjugation, it follows that~$M^{g}\in \frakF$.  Now
    let $N\in \frakF$ with~$M^{g}\leq N$.  Then~$M\leq N^{g^{-1}}$ and
    so~$M=N^{g^{-1}}$ by the maximality of $M$.  Therefore we
    have~$M^{g}=N$ and since $N$ was an arbitrary element of $\frakF$
    with $M^{g}\geq N$ it follows that~$M^{g}$ is maximal in $\frakF$.
    Thus by the uniqueness of a maximal element in $\frakF$ we have
    that $M^{g}=M$, which is a contradiction.  Therefore~$M=G$ and
    so~$G\in \frakF$.
\end{proof}

For completeness, we include the statement about groups $G$ with
$\hd_{\frakF} G$ equal to zero.  This result is just the statement of
Theorem~\ref{thrm:hd-vs-cd} in the case~$\hd_{\frakF} G=0$.

\begin{proposition}
    \label{prop:hdG=0}
    Let $G$ be a group and let $\frakF$ be a semi-full family of
    subgroups of $G$.  If both $G$ and $\frakF$ are countable then
    $\hd_{\frakF} G=0$ implies $\cd_{\frakF} G\leq 1$.\qed
\end{proposition}

Note that the estimation $\cd_{\frakF} G \leq 1$ is sharp.  For
example, $\Q$ is the direct union of its cyclic subgroups.  It follows
by two results which we will prove later in
Section~\ref{sec:dimensions-for-direct-unions}, namely
Theorem~\ref{thrm:direct-union-result} and 
Corollary~\ref{cor:simple-comp-condition}, that $\uuhd \Q=0$ and
$\uucd \Q \leq 1$.  On the other hand $\Q$ is not virtually cyclic
and so $\uucd \Q\neq 0$ by Proposition~\ref{prop:cdG=0}. Therefore
$\Q$ is group with $\uuhd \Q=0$ and $\uucd\Q = 1$.
\vspace{-2pt}

%
%

\section{Induction with $I_{K}$ and Preservation of Exactness}
\label{sec:exatness-of-IndIK}

In Section~\ref{sec:cd-vs-hd-vs-gd} we have compared different types
of Bredon dimensions for a given group and with respect to a given
family of subgroups.  In order to make comparisons of the same kind of
Bredon dimensions, but for different groups, or if we are interested
how the the choice of the family of subgroups affects a certain Bredon
dimension, then the restriction and induction functors are an
important tool.  The exactness of these functors is an important
property and we already know that the restriction functor is always
exact.  The result we present in this section is due to
Symonds~\cite{symonds-05}.

In this section we consider induction with the following functor.  Let
$K$ be a fixed subgroup of $G$.  If $\frakF\cap K\subset
\frakF$ then we obtain a well defined functor
\vspace{-2pt}
\begin{equation*}
    I_{K}\: \OC{K}{\frakF\cap K} \to \OFG
\end{equation*}
of orbit categories as follows: on objects of $\OC{K}{\frakF\cap K}$
we set $I_{K}(K/H) \defeq  G/H$ and if $f\: K/H \to K/L$ is a $G$-map, then
$I_{K}(f)$ is the obvious extension of~$f$ to a $G$-map $G/H\to G/L$,
which by abuse of notation we will also denote by~$f$.

\begin{lemma}
    \label{lem:isomorphism-1}
    Let $H\in \frakF$ and $L\in \frakF\cap K$.  Let $R$ be a complete
    set of representatives for the left cosets in $(G/K)^{H}$.  Then
    there exists a bijection 
    \begin{equation}
	\eta_{L}\: [G/H, G/L]_{G} \to \coprod_{x\in R} [K/H^{x},
	K/L]_{K}
	\label{eq:isomorphism-1}
    \end{equation}
    \vspace{-2pt}
    of sets.
\end{lemma}

\begin{proof}
    If $x\in R$ then $H^{x}\leq K$ and thus the right hand side
    of~\eqref{eq:isomorphism-1} is well defined.  Let $f_{g,H,L}\in
    [G/H,G/L]_{G}$.  Then $H^{g}\leq L$ and since $L\in \frakF\cap K$ 
    we have that $H^{g}\leq K$. This implies that $gK\in (G/K)^{H}$
    and thus there exists a unique $x\in R$ and $y\in K$
    such that $g=xy$.
    
    Assume that $f_{g',H,L} = f_{g,H,L}$.  Then $g'=gl$ for some $l\in
    L\leq K$.  Furthermore there exists a unique $x'\in R$ and $y'\in
    K$ such that $g'=x'y'$.  Thus $gl = xyl=x'y'$ and from this
    follows that $x^{-1}x' = yl(y')^{-1} \in K$.  This means that $x$
    and $x'$ are in the same left coset of $K$ in $G$.  Therefore
    $x=x'$ and $yl = y'$.  That is $x$ is uniquely determined by
    $f_{g,H,L}$ and $y$ is uniquely determined by $f_{g,H,L}$ up to
    right multiplication by an element of $L$.
    
    Since $(H^{x})^{y} = H^{xy} = H^{g} \leq L$ we get that
    $f_{y,H^{x},L}\in [K/H^{x},K/L]_{K}$ and we define
    $\eta_{L}(f_{g,H,L}) \defeq  f_{y,H^{x},L}$.  Since $x$ is uniquely
    determined by the map~$f_{g,H,L}$, and since $y$ is uniquely
    determined by the map~$f_{g,H,L}$ up to right multiplication by an
    element of $L$, this definition of $\eta_{L}(f_{g,H,L})$ is well
    defined and ensures that $\eta_{L}$ is an injective map.
    
    It remains to show that the map $\eta_{L}$ is surjective.
    Therefore choose an arbitrary $x\in R$ and let $f_{y,H^{x},L} \in
    [K/H^{x}, K/L]_{K}$.  Then
    \begin{equation*}
        H^{xy} = (H^{x})^{y} \leq L
    \end{equation*}
    and thus $f_{xy,H,L}\in [G/H, G/L]_{G}$ and $\eta_{L}(f_{xy,H,L})
    = f_{y,H^{x},L}$.  Therefore, we conclude that $\eta$ is
    surjective.
\end{proof}

\begin{lemma}
    \label{lem:free-1}
     Let $H\in \frakF$.  Then the left $\OC{K}{\frakF\cap K}$-module
     \begin{equation*}
	\Z[K/H^{x}, \?]_K
     \end{equation*}
     is free for every $x\in G$ for which $xK\in (G/K)^{H}$.
\end{lemma}

\begin{proof}
    We only need to show that $H^{x}\in \frakF\cap K$.  Since $\frakF$
    is closed under conjugation we have $H^{x}\in\frakF$.  Since
    $xK\in (G/H)^{H}$ we have $H^{x}\leq K$ and thus $H^{x} =
    H^{x}\cap K \in \frakF \cap K$.
\end{proof}

\begin{lemma}
    \label{lem:free-2}
    Assume that $\frakF\cap K$ is a subset of $\frakF$.  Then the
    collection of isomorphisms $\{\eta_{L} : L\in \frakF\cap K\}$
    defined in Lemma~\ref{lem:isomorphism-1} give an isomorphism
    \begin{equation}
	\eta\: \Z[G/H, I_K(\?)]_G \to \coprod_{x\in R} \Z[K/H^{x},
	\?]_K
	\label{eq:isom-1}
    \end{equation}
    of left $\OC{G}{\frakF\cap K}$-modules.  In particular, $\Z[G/H,
    I_K(\?)]_G$ is a free left $\OC{K}{\frakF\cap K}$-module.
\end{lemma}

\begin{proof}
    The assumption $\frakF\cap K\subset \frakF$ ensures that the
    functor $I_{K}$ is defined.  We have to show that the collection
    of isomorphisms $\eta_{L}$ gives a natural transformation of
    covariant functors.  We do this by chasing generators around the
    necessary diagrams.
    
    Let $L_{1}, L_{2}\in \frakF\cap K$, $\varphi\in [K/L_{1},
    K/L_{2}]_{K}$ and $f\in [G/H, I_{K}(K/L_{1})]_{G} = [G/H,
    G/L_{1}]_{G}$.  Then there exists a $z\in K$ such that $\varphi =
    f_{z,L_{1},L_{2}}$ and a $g\in G$ such that $f = f_{g,H,L_{1}}$.
    Furthermore there exists a unique $x\in R$ such that we can write
    $f_{g,H,L_{1}} = f_{xy,H,L_{1}}$ for some $y\in K$. This $y$ is
    unique up to right multiplication by an element of $L_{1}$.  
    We need to chase $f$ around the following diagram.
    \begin{equation*}
	\begin{diagram}
	    \node{[G/H, G/L_{1}]_{G}}
	    \arrow{e,t}{\eta_{L_{1}}}
	    \arrow{s,l}{\varphi_{*}}
	    \node{[K/H^{x},K/L_{1}]_{K}}
	    \arrow{s,r}{\varphi_{*}}
	    \\
	    \node{[G/H, G/L_{2}]_{G}}
	    \arrow{e,t}{\eta_{L_{2}}}
	    \node{[K/H^{x}, K/L_{2}]_{K}}
	\end{diagram}
    \end{equation*}
    On the one hand we have
    \begin{align*}
	(\varphi_{*}\mathbin{\circ}\eta_{L_{1}})(f_{g,H,L_{1}})
	& = \varphi_{*}(f_{y,H^{x},L_{1}}) 
	\\
	& = f_{z,L_{1},L_{2}} \mathbin{\circ} f_{y,H^{x},L_{1}}
	\\
	& = f_{yz, H^{x},L_{2}} \in \Z[K/H^{x}, K/L_{2}]_K.
	\\
	\intertext{On the other hand we have}
        (\eta_{L_{2}} \mathbin{\circ} \varphi_{*})(f_{g,H,L_{1}})
	& = \eta_{L_{2}} (f_{z,L_{1},L_{2}} \mathbin{\circ} 
	f_{xy,H,L_{1}})
	\\
	& = \eta_{L_{2}} (f_{xyz,H,L_{2}})
	\\
	& = f_{yz,H^{x},L_{2}}\in \Z[K/H^{x}, K/L_{2}]_K
    \end{align*}
    where the last equality follows from $yz\in K$.  Therefore the
    maps $\{\eta_{L} : L\in \frakF\cap K\}$ define a homomorphism of
    left $\OC{K}{\frakF\cap K}$-modules.  Since each $\eta_{L}$ is an
    isomorphism it follows that $\eta$ is an isomorphism, too.
    
    The Bredon modules in the coproduct on the right hand side of
    \eqref{eq:isom-1} are all free by Lemma~\ref{lem:free-1}.  We can
    conclude that the right hand side of~\eqref{eq:isom-1} is free and
    the claim follows.
\end{proof}

Note that in Symonds' article~\cite[p.~266]{symonds-05} the above
result together with Lemma~\ref{lem:tensor-product-2} is used as the
\emph{definition} of the induction functor
$\ind_{I_{K}}\: \mathop{\text{$\Mod$-$\OC{K}{\frakF\cap K}$}}
\to \ModOFG$.

\begin{proposition}
    \label{prop:exactness-I-K}
    \cite{symonds-05}
    Let $\frakF$ be a family of subgroups of $G$.  Let $K\leq G$ be a
    subgroup such that $\frakF\cap K\subset \frakF$.  Then induction
    with $I_{K}\: \OC{K}{\frakF\cap K} \to \OFG$ is an exact functor.
\end{proposition}

\begin{proof}
    Let $0 \to L \to M \to N \to 0$ be an exact sequence of right
    $\OC{K}{\frakF\cap K}$-modules.  Applying the functor
    $\ind_{I_{K}}$ to this sequence yields the sequence
    \begin{equation}
	0 \to \ind_{I_{K}} L \to \ind_{I_{K}} M \to \ind_{I_{K}} N \to 0
        \label{eq:seq2}
    \end{equation}
    of $\OFG$-modules. We evaluate this sequence at $H\in \frakF$ and 
    obtain
    \begin{multline}
	0 \to L(\?)  \mathop{\otimes_{\frakF \cap K}}
	\Z[G/H,I_{K}(\?)]_{G} \to M(\?)  \mathop{\otimes_{\frakF \cap
	K}} \Z[G/H,I_{K}(\?)]_{G}\\
	\to N(\?)  \mathop{\otimes_{\frakF \cap K}}
	\Z[G/H,I_{K}(\?)]_{G} \to 0
	\label{eq:seq3}
    \end{multline}
    By the previous lemma the left $\OC{K}{\frakF\cap K}$-module 
    $\Z[G/H,I_{K}(\?)]_{G}$ is free and hence flat. Thus the 
    sequence~\eqref{eq:seq3} is exact. Since this holds for every 
    $H\in \frakF$ we have that the sequence~\eqref{eq:seq2} of 
    $\OFG$-modules is exact.
\end{proof}

%
%

\section{Restriction with $I_{K}$ and Preservation of Projectives}

Let $G$ be a group and $\frakF$ a family of subgroups of $G$.  Let $K$
be a subgroup of $G$ such that $\frakF\cap K\subset \frakF$.
Lemma~\ref{lem:free-2} states that restriction with $I_{K}$ preserves
free \emph{left} Bredon modules.  It turns out that restriction with
$I_{K}$ preserves free \emph{right} Bredon modules, too.  However, we
get a different answer to how the restricted free right Bredon modules
look like.  The following statements together with their proofs are
due to Martínez-Pérez~\cite[p.~167]{martinez-perez-02}.

\begin{lemma}
    \label{lem:preserve-right-frees}
    Let $\frakF$ be a family of subgroups of $G$ and let $K$ be a
    subgroup of $G$.  Then for any $H\in \frakF$ and any complete set
    $R$ of representatives for the double cosets $K\backslash G/H$ we
    have an isomorphism
    \begin{equation*}
	\eta\: \Z[I_{K}(\?),G/H]_{G}\to \coprod_{x\in R} \Z[\?, 
	K/(K\cap H^{x^{-1}})]_{K}
    \end{equation*}
    of right $\OC{K}{\frakF\cap K}$-modules.  In particular if
    $\frakF\cap K\subset \frakF$ then $\Z[I_{K}(\?),G/H]_{G}$ is a
    free right $\OC{K}{\frakF\cap K}$-module.
\end{lemma}

\begin{proof}    
    Let $L\in \frakF\cap K$ and let $f_{g,L,H}\in [G/L,K/H]_{G}$.
    Then there exists a unique $x\in R$ such that we can write $g=yxh$
    for some~$y\in K$ and $h\in H$.  Since $g$ is uniquely determined
    by $f$, up to right multiplication by an element of~$H$, it
    follows that $x$ is uniquely determined by $f$.  Assume that we
    have~$y_{1}\in K$ and $h_{1}\in H$ such that $yxh = y_{1}xh_{1}$.
    Then $y_{1}^{-1}y = xh_{1}h^{-1}x^{-1}\in H^{x^{-1}}$ and since
    $y_{1}^{-1}y\in K$ we get even that $y$ and $y_{1}$ lie in the
    same left coset of $K\cap H^{x^{-1}}$ in $K$.  Furthermore from
    $L^{yxh} = L^{g}\leq H$ follows that $L^{y} \leq H^{h^{-1}x^{-1}}
    = H^{x^{-1}}$.  Since $L\leq K$ and $y\in K$ we get $L^{y}\leq
    K^{y} = K$ and thus altogether that $L^{y}\leq K\cap H^{x^{-1}}$.
    Hence $f_{y,L,K\cap H^{x^{-1}}}\in [K/L, K/(K\cap
    H^{x^{-1}})]_{K}$ and we obtain a well defined map
    \begin{equation*}
	\eta_{L}\: \Z[G/L,G/H]_{G}\to \coprod_{x\in R} \Z[K/L, 
	K/(K\cap H^{x^{-1}})]_{K}
    \end{equation*}
    by $\eta_{L}(f_{yxh,L,G}) \defeq  f_{y,L,K\cap H^{x^{-1}}} \in
    [K/L, K/(K\cap H^{x^{-1}})]_{K}$. 
    
    It follows that the above defined map is bijective.  By chasing
    generators around the usual diagrams it follows that the
    collection $\{\eta_{L} : L\in \frakF\cap L\}$ of maps defines an
    isomorphism $\eta$ of right $\OC{K}{\frakF\cap K}$-modules as
    required in the statement of the lemma.
    
    Let $H\in \frakF$.  Since $\frakF$ is closed under conjugation
    $H^{x^{-1}}\in \frakF$ and thus $K\cap H^{x^{-1}}\in \frakF\cap
    H^{x^{-1}}$.  Therefore all the summands of the codomain of $\eta$
    are free right $\OC{K}{\frakF\cap K}$-modules and therefore the
    domain of $\eta$ must be a free $\OC{K}{\frakF\cap K}$-module,
    too.
\end{proof}

\begin{proposition}\cite[Lemma 3.7]{martinez-perez-02}
    \label{prop:res-I-K-preserves-projectives}
    Let $\frakF$ be a family of subgroups of~$G$.  Let $K\leq G$ be
    some subgroup of~$G$ such that $\frakF\cap K\subset\frakF$.  Then
    restriction with~$I_{K}$ preserves free Bredon modules and
    consequently also projective Bredon modules.
\end{proposition}

\begin{proof}
    The first statement is Lemma~\ref{lem:preserve-right-frees} (for
    right Bredon modules) and~\ref{lem:free-2} (for left Bredon
    modules).  The statement about projective right Bredon modules
    follows from the first statement since restriction is an additive
    functor.
\end{proof}

%
%

\section{Shapiro's Lemma in the Bredon Setting}
\label{sec:shapiros-lemma}

\begin{proposition}
    \label{prop:pre-shapiros-lemma}
    Let $\frakF$ be a family of subgroups of $G$ and let $K$ be a 
    subgroup of $G$ such that $\frakF\cap K\subset\frakF$. Then
    \begin{enumerate}
	\item for any right $\OFG$-module $M$ and any right 
	$\OC{K}{\frakF\cap K}$-module $N$ we have
	\begin{equation*}
	    \Ext^{*}_{\frakF\cap K}(\res_{I_{K}} M,N)
	    \isom
	    \Ext^{*}_{\frakF}(M, \coind_{I_{K}}N);
	\end{equation*}

    	\item for any right $\OFG$-module $M$ and any left 
	$\OC{K}{\frakF\cap K}$-module $N$ we have
	\begin{equation*}
	    \Tor_{*}^{\frakF\cap K}(\res_{I_{K}} M,N)
	    \isom
	    \Tor_{*}^{\frakF}(M, \ind_{I_{K}}N).
	\end{equation*}
    \end{enumerate}
    In both cases the isomorphism is natural in both $M$ and $N$.
\end{proposition}

\begin{proof}
    From a projective resolution 
    \begin{equation*}
	\ldots\to P_{2} \to P_{1} \to P_{0} \to M \to 0
    \end{equation*}
    of the $\OFG$-module $M$ we obtain a sequence
    \begin{equation*}
	 \ldots\to \res_{I_{K}} P_{2} \to \res_{I_{K}} P_{1} \to 
	\res_{I_{K}} P_{0} \to \res_{I_{K}} M \to 0
    \end{equation*}
    of $\OC{K}{\frakF\cap K}$-modules.  This sequence is exact (since
    restriction preserves exactness) and each $\res_{I_{K}} P_{i}$ is
    projective by
    Proposition~\ref{prop:res-I-K-preserves-projectives}.
    
    Now consider the first case, that is $M$ is a right $\OFG$-module 
    and $N$ is a right $\OC{K}{\frakF\cap K}$-module. Then
    \begin{multline*}
	\Ext^{n}_{\frakF\cap K}(M,N)
	=
	H_{n}(\mor_{\frakF\cap K}(\res_{I_{K}} P_{*}, N))
	\isom
	\\
	H_{n}(\mor_{\frakF}(P_{*}, \coind_{I_{K}} N))
	= 
	\Ext^{n}_{\frakF}(M, \coind_{I_{K}} N)
    \end{multline*}
    for any $n\in \N$.  Here the middle isomorphism is due to the fact
    that the restriction functor $\res_{I_{K}}$ is left adjoint to
    coinduction functor $\coind_{I_{K}}$.  In particular the
    isomorphism is natural in $M$ and $N$.  This proves the first
    isomorphism of the statement of the proposition.
    
    Consider the second part of the statement. That is, $M$ is a 
    right $\OFG$-module and $N$ is a left $\OC{K}{\frakF\cap 
    K}$-module. Then we have for any
    $n\in \N$ natural isomorphisms
    \begin{align*}
	\res_{I_{K}} P_{n}(\q?)  \mathop{\otimes_{\frakF\cap K}}
	N(\q?) 
	& \isom \bigl( P_{n}(\?)\mathop{\otimes_{\frakF}}
	\Z[I_{K}(\q?),\?]_{G} \bigr) \mathbin{\otimes_{\frakF\cap K}}
	N(\q?)  \\
	& \isom P_{n}(\?)\mathbin{\otimes_{\frakF}}
	\bigl(\Z[I_{K}(\q?),\?]_{G} \mathop{\otimes_{\frakF\cap K}}
	N(\q?)\bigr) \\
	& \isom P_{n}(\?)  \mathop{\otimes_{\frakF}} \ind_{I_{K}}
	N(\?),
    \end{align*}
    where the middle isomorphism is due to the fact that the
    categorical tensor product is
    associative~\cite[p.~163]{martinez-perez-02}.  Using this we
    obtain 
       \begin{multline*}
	\Tor_{n}^{\frakF\cap K}(M,N)
	=
	H_{n}(\res_{I_{K}} P_{*} \mathop{\otimes_{\frakF\cap K}} N)
	\isom
	\\
	H_{n}(P_{*} \mathop{\otimes_{\frakF}} \ind_{I_{K}} N)
	= 
	\Tor_{n}^{\frakF}(M, \ind_{I_{K}} N)
    \end{multline*}
    which again is natural in both $M$ and $N$.
\end{proof}

\begin{corollary}
    \label{cor:restriction-and-flats}
    Let $\frakF$ be a family of subgroups of $G$.  Furthermore, let
    $K$ be a subgroup of $G$ such that $\frakF\cap
    K\subset \frakF$. Then restriction with $I_{K}$ preserves flat 
    right Bredon modules.
\end{corollary}

\begin{proof}
    Let $M$ be a flat right $\OFG$-module. Then
    \begin{equation*}
	\Tor_{1}^{\frakF\cap K}(\res_{I_{K}} M, N) \isom 
	\Tor_{1}^{\frakF}(M, \ind_{I_{K}} N) = 0
    \end{equation*}
    for any left $\OC{K}{\frakF\cap K}$-module $N$. Therefore 
    $\res_{I_{K}} M$ is flat.
\end{proof}

\begin{proposition}[Shapiro's Lemma]
    \label{prop:shapiros-lemma}
    Let $\frakF$ be a family of subgroups of~$G$ and let $K$ be a
    subgroup of~$G$ such that $\frakF\cap K\subset\frakF$.  Then for
    any right $\OC{K}{\frakF\cap K}$-module $M$ and any left
    $\OC{K}{\frakF\cap K}$-module $N$, there exist isomorphisms
        \begin{align*}
	H_{\frakF\cap K}^{*}(K;M) & \isom
	H_{\frakF}^{*}(G;\coind_{I_{K}}M) \\
        \intertext{and}
	H^{\frakF\cap K}_{*}(K;N) & \isom
	H^{\frakF}_{*}(G;\ind_{I_{K}} N)
    \end{align*}
    which are natural in $M$ and $N$.
\end{proposition}

\begin{proof}
    Note that $\res_{I_{K}} \underline\Z_{\frakF} \isom
    \underline\Z_{\frakF\cap K}$. Then 
    Proposition~\ref{prop:pre-shapiros-lemma} says that we have 
    isomorphisms, natural in $M$ and $N$, such that
    \begin{equation*}
        H_{\frakF\cap K}^{*}(K; M) = \Ext^{*}_{\frakF\cap K} 
	(\underline\Z_{\frakF\cap K}, M) \isom 
	\Ext^{*}_{\frakF}(\underline\Z_{\frakF}, \coind_{I_{K}} M) = 
	H_{\frakF}^{*}(G;M)
    \end{equation*}
    and likewise
    \begin{equation*}
        H^{\frakF\cap K}_{*}(K; N) = \Tor_{*}^{\frakF\cap K} 
	(\underline\Z_{\frakF\cap K}, N) \isom 
	\Tor_{*}^{\frakF}(\underline\Z_{\frakF}, \ind_{I_{K}} N) = 
	H^{\frakF}_{*}(G;N).\tag*{\qedhere}
    \end{equation*}
\end{proof}

%
%

\section{Bredon Dimensions for Subgroups}

The following results about the dimension of subgroups are
generalisations of the corresponding results in classical
(co-)homology of groups.  The proofs for the classical statements work
without structural modifications.

\begin{proposition}
    \label{prop:dim-subgroups-1}
    Let $G$ be a group and $\frakF$ a family of subgroups of $G$.
    Then for any subgroup $K$ of $G$ such that $\frakF\cap
    K\subset\frakF$ we have inequalities
    \begin{equation*}
        \cd_{\frakF\cap K} K \leq \cd_{\frakF} G
	\qquad
	\text{and}
	\qquad 
	\hd_{\frakF\cap K} K \leq \hd_{\frakF} G.
    \end{equation*}
\end{proposition}

\begin{proof}
    Since restriction is exact in general and since restriction with
    $I_{K}$ preservers projectives
    (Proposition~\ref{prop:res-I-K-preserves-projectives}) we get 
    from a projective resolution
    \begin{equation*}
        0 \to P_{n} \to \ldots \to P_{1} \to P_{0} \to 
	\underline\Z_{\frakF} \to 0,
    \end{equation*}
    where $n= \cd_{\frakF}G$, a projective resolution
    \begin{equation*}
        0 \to \res_{I_{K}} P_{n} \to \ldots \to \res_{I_{K}} P_{1} 
	\to \res_{I_{K}} P_{0} \to 
	\res_{I_{K}} \underline\Z_{\frakF} \to 0.
    \end{equation*}
    Now the first claim follows from $\res_{I_{K}} 
    \underline\Z_{\frakF} = \underline\Z_{\frakF\cap K}$.
    
    Similarly, since restriction with $I_{K}$ preserves flats 
    (Corollary~\ref{cor:restriction-and-flats}), one obtains from a 
    flat resolution
    \begin{equation*}
        0 \to Q_{m} \to \ldots \to Q_{1}\to Q_{0} \to 
	\underline\Z_{\frakF}\to 0
    \end{equation*}
    with $m=\hd_{\frakF}G$ a flat resolution of the trivial 
    $\OC{K}{\frakF\cap K}$-module of length~$m$. Therefore the second 
    statement is true, too.
\end{proof}

Note that if $\frakF$ is a family of subgroups of $G$ and if we are
given a chain of subgroups $H\leq K\leq G$ such that $\frakF\cap
H\subset \frakF\cap K\subset \frakF$, then we get the inequalities
\begin{align*}
    \cd_{\frakF\cap H} H & \leq \cd_{\frakF\cap K} K \leq \cd_{\frakF}
    G \\
    \intertext{and}
    \hd_{\frakF\cap H} H & \leq \hd_{\frakF\cap K} K \leq 
    \hd_{\frakF} G
\end{align*}
as in the case of classical group (co-)homology.  In particular, if
$\frakF$ is a full family of subgroups of $G$ then we get the above
inequalities for any chain of subgroups $H\leq K\leq G$.

For the geometric Bredon dimension we have the analogous result to 
Proposition~\ref{prop:dim-subgroups-1}. 

\begin{proposition}
    \label{prop:dim-subgroups-2}
    Let $G$ be a group and let $\frakF$ be a full family of subgroups 
    of $G$. Then for any subgroup $K$ of $G$ we have that
    \begin{equation*}
        \gd_{\frakF\cap K} K \leq \gd_{\frakF} G.
    \end{equation*}
\end{proposition}

\begin{proof}
    The proof is standard.  First of all observe that $\frakF\cap K$
    is a full family of subgroups of $K$, since $\frakF$ is a full
    family of subgroups of $G$.  Thus the geometric dimension
    $\gd_{\frakF\cap K} K$ is defined.  In order to avoid triviality,
    assume that $n\defeq \gd_{\frakF} G$ is finite.  Then there exists
    an $n$-dimensional model~$X$ for $E_{\frakF}G$ which is also a
    $K$-space when restricting the action of $G$ to $K$.
    Since~$\frakF$ is closed under taking subgroups, it follows that
    $\frakF\cap K\subset \frakF$ and as a consequence~$X$ is an
    $n$-dimensional model for $E_{\frakF\cap K}K$.  Thus
    $\gd_{\frakF\cap K} K\leq n$ and the proposition follows.
\end{proof}

As before, if $\frakF$ satisfies the conditions of
Proposition~\ref{prop:dim-subgroups-2} then so does the family of
subgroups $\frakF\cap K$ for any subgroup $K$ of $G$ and we get for
any subgroup $H$ of $K$ the sequence of inequalities
\begin{equation*}
    \gd_{\frakF\cap H} H  \leq \gd_{\frakF\cap K} K \leq 
    \gd_{\frakF} G
\end{equation*}
as in the case of classical group (co-)homology.

%
%

\section{(Co-)Homological Dimension when Passing to Larger Families
of 
Subgroups}

In this section we consider the following setup: Let $(\frakG,
\frakF)$ be a pair of families of subgroups of $G$ and we denote
by~$I$ the
inclusion functor $I\: \OFG \hookrightarrow \OGG$ of the 
corresponding orbit categories.

\begin{proposition}
    \label{prop:hd-cd-result-1}
    \begin{enumerate}
	\item
	If there exists a $k\in \N$ such that $\pd_{\frakF}
	(\res_{I} P)\leq k$ for every projective $\OGG$-module $P$
	then
	\begin{equation*}
	    \cd_{\frakF} G \leq \cd_{\frakG} G + k.
	\end{equation*}
	
	\item If there exists a $k\in \N$ such that $\fld_{\frakF}
	(\res_{I} Q)\leq k$ for every flat $\OGG$-module $Q$
	then
	\begin{equation*}
	    \hd_{\frakF} G \leq \hd_{\frakG} G + k.
	\end{equation*}
    \end{enumerate}
\end{proposition}

To prove this proposition we need the following standard result from 
homological algebra.

\begin{lemma}
    \label{lem:hd-cd-result-1-aux}
    Assume that we have a resolution
    \begin{equation*}
        0\to X_{n} \to \ldots \to X_{0} \to M \to 0
    \end{equation*}
    of the $\OFG$-module $M$.  If there exists a $k\in \N$ such that
    $\pd_{\frakF} X_{i}\leq k$ for all $0\leq i\leq n$, then
    $\pd_{\frakF} M\leq n+k$.  Similarly, if there exists a $k\in \N$
    such that $\fld_{\frakF} X_{i}\leq k$ for all $0\leq i\leq n$,
    then $\fld_{\frakF} M\leq n+k$.
\end{lemma}

\begin{proof}
    We prove the first claim by induction on $n$. If $n=0$ then 
    $X_{0}\isom M$ and $\pd_{\frakF} M = \pd_{\frakF} X_{0} \leq k$ 
    and we are done.
    
    Thus assume that $n'\geq 1$ and that the statement of the lemma is
    true for~$n\leq n'-1$.  Then we have a short exact sequence
    \begin{equation*}
        0\to X_{n'} \to X_{n'-1} \to \im d_{n'} \to 0
    \end{equation*}
    with $\pd_{\frakF} X_{n'}\leq k$ and $\pd_{\frakF} X_{n'-1} \leq
    k$.  Then a standard argument in homological algebra, see for
    example~\cite[p.~95]{weibel-94}, implies that $\pd_{\frakF} (\im
    d_{n'}) \leq k+1$.  Since $\im d_{n'} = \ker d_{n'-1}$ we get a
    resolution
    \begin{equation*}
	0 \to \ker d_{n'-1} \to X_{n'-2} \to \ldots \to X_{0} \to M
	\to 0
    \end{equation*}
    of $M$ by $\OFG$-modules of projective dimension at most $k+1$.
    This resolution has length $n'-1$ and we can apply the induction
    hypothesis.  It follows that~$\pd_{\frakF} M \leq (n'-1) + (k+1)
    = n'+k$.  Therefore the statement of the lemma is also true in the
    case $n=n'$.
    
    Finally, the remaining statement of the lemma for the flat
    dimension of~$M$ is verified in exactly the same way.
\end{proof}

\begin{proof}[Proof of Proposition~\ref{prop:hd-cd-result-1}]
    We prove the cohomological statement first: In order to avoid
    triviality assume that $n\defeq \cd_{\frakG} G$ is finite.  Then
    there exists a projective resolution
    \begin{equation*}
	0 \to P_{n} \to \ldots \to P_{0} \to \underline\Z_{\frakG} \to
	0
    \end{equation*}
    of length $n$ of the trivial $\OGG$-module
    $\underline\Z_{\frakG}$.  Applying the restriction
    functor~$\res_{I}$ to this sequence yields a sequence
    \begin{equation*}
        0 \to \res_{I} P_{n} \to \ldots \to \res_{I} P_{0} \to 
	\underline\Z_{\frakF} \to 0
    \end{equation*}
    of $\OFG$-modules that is exact, since restriction preserves
    exactness in general.  By assumption $\pd_{\frakF}(\res_{I} P_{i})
    \leq k$ for all $0\leq i\leq n$ and so the first claim of the
    proposition follows now from Lemma~\ref{lem:hd-cd-result-1-aux}.
    
    Now the homological statement is proven in exactly the same way
    and this concludes the proof.
\end{proof}

The following result gives a upper bound for the dimensions
$\pd_{\frakF} (\res_{I} P)$, which appear in first part of 
Proposition~\ref{prop:hd-cd-result-1}, in terms of Bredon
cohomological 
dimensions of the subgroups in $\frakG$.

\begin{proposition}
    \label{prop:hd-cd-result-2}
    Assume that $\frakF\cap K\subset\frakF$ for every~$K\in \frakG$.
    Then the following two statements are true:
    \begin{enumerate}
	\item If there exists a $k\in \N$ such that
	$\cd_{\frakF\cap K} K\leq k$ for every $K\in \frakG$ then
	\begin{equation*}
	    \pd_{\frakF} (\res_{I} P)\leq k
	\end{equation*}
	for every projective $\OGG$-module $P$.
	
	\item Assume further that $\frakF$ is a full family of
	subgroups of $G$.  If there exists a $k\in \N$ such that
	$\hd_{\frakF\cap K} K\leq k$ for every $K\in \frakG$ then
	\begin{equation*}
	    \fld_{\frakF} (\res_{I} P)\leq k
	\end{equation*}
	for every projective $\OGG$-module $P$.
    \end{enumerate}
\end{proposition}

\begin{proof}
    Since restriction is an additive functor it is enough to carry out
    the proof for projective $\OGG$-modules of the form $P =
    \Z[\?,G/K]_{G}$ for~$K\in \frakG$.
    
    \smallskip
    
    We prove the cohomological statement first.  By assumption
    $\cd_{\frakF \cap K} K \leq k$ and therefore there exists a
    projective resolution
    \begin{equation*}
        0\to P_{k} \to \ldots \to P_{0} \to \underline{\Z}_{\frakF 
	\cap K}\to 0
    \end{equation*}
    of the trivial $\mathop{\calO_{\frakF\cap K}K}$-module
    $\underline{\Z}_{\frakF \cap K}$.  Since $\frakF\cap
    K\subset\frakF$, the inclusion functor $I_{K}\: \OC{K}{\frakF\cap
    K} \to \OFG$ is defined.  We have by
    Proposition~\ref{prop:exactness-I-K} that induction with $I_{K}$
    is exact.  Therefore we get an exact sequence
    \begin{equation*}
	0\to \ind_{I_{K}} P_{k} \to \ldots \to \ind_{I_{K}} P_{0} \to
	\ind_{I_{K}} \underline{\Z}_{\frakF \cap K}\to 0.
    \end{equation*}
    Induction preserves projectives and thus we have obtained a
    projective resolution of the $\OFG$-module $\ind_{I_{K}}
    \underline{\Z}_{\frakF \cap K}$ of length $k$.  By Lemma 2.7
    in~\cite[p.~268]{symonds-05} we have that $\Z[\?,G/K]_{G} \isom
    \ind_{I_{K}} \underline{\Z}_{\frakF \cap K}$ as $\OFG$-modules.
    On the other hand we have the equality $\Z[\?,G/K]_{G} = \res_{I}
    P$ of $\OFG$-modules and therefore $\pd_{\frakF} (\res_{I} P)\leq
    k$.
    
    \smallskip
    
    Next consider the assumptions of the homological statement of the
    proposition. Since $\hd_{\frakF\cap K}K \leq k$ we have a 
    flat resolution
    \begin{equation*}
        0\to Q_{k} \to \ldots \to Q_{0} \to \underline{\Z}_{\frakF 
	\cap K}\to 0
    \end{equation*}
    of the trivial $\OFG$-module $\underline\Z_{\frakF}$. As before 
    we can apply the induction functor~$I_{K}$ to obtain an exact 
    sequence
    \begin{equation*}
	0\to \ind_{I_{K}} Q_{k} \to \ldots \to \ind_{I_{K}} Q_{0} \to
	\Z[\?, G/K]_{G}\to 0.
    \end{equation*}
    Since $\frakF$ is assumed to be a full family of subgroups of $G$
    if follows that $\frakF\cap K$ is a full family of subgroups of
    $K$.  It follows from Proposition~\ref{prop:preservation-prop-2}
    that induction preserves flats and hence the above resolution is 
    a flat resolution of $\Z[\?, G/K]_{G}$ of length $k$. Now the
claim 
    follows from the fact $\Z[\?,G/K]_{G} = \res_{I}(P)$.
\end{proof}

The next two theorems are the algebraic counterparts to 
Proposition~5.1~(i) in~\cite{luck-12}.

\begin{theorem}
    \label{thrm:cdF-G-vs-cdG-G}
    Let $(\frakG, \frakF)$ be a pair of families of subgroups of $G$.
    Assume that $\frakF \cap K\subset\frakF$ for every $K\in \frakG$.
    If there exists a $k\in \N$ such that $\cd_{\frakF \cap K} K \leq
    k$ for every $K\in \frakG$, then
    \begin{equation*}
        \cd_{\frakF} G \leq \cd_{\frakG} G + k.
    \end{equation*}
\end{theorem}

\begin{proof}
    This is an immediate consequence of 
    Proposition~\ref{prop:hd-cd-result-1} and
    Proposition~\ref{prop:hd-cd-result-2}.
\end{proof}

\begin{theorem}
    \label{thrm:hdF-G-vs-hdG-G}
    Let $(\frakG, \frakF)$ be a pair of families of subgroups of $G$.
    Assume that $\frakF$ is a full family.  If there exists a $k\in
    \N$ such that $\hd_{\frakF \cap K} K \leq k$ for every $K\in
    \frakG$, then
    \begin{equation*}
        \hd_{\frakF} G \leq \hd_{\frakG} G + k.
    \end{equation*}
\end{theorem}

\begin{proof} 
    Let $Q$ be a flat $\OGG$-module.  By Proposition~\ref{prop:lazard}
    we have that~$Q$ is the filtered colimit of finitely generated
    free $\OGG$-modules $Q_{\lambda}$.  Since $\ModOGG$ is an
    AB5--category with enough projectives, and $\?  \otimes_{\frakF}
    B$ is a left adjoint for any left $\OFG$-module, we have by
    Corollary~2.6.16 in~\cite[p.~58]{weibel-94} that
    \begin{equation*}
	\Tor_{k+1}^{\frakF} (\varinjlim \res_{I} Q_{\lambda}, B) \isom
	\varinjlim \Tor_{k+1}^{\frakF} (\res_{I} Q_{\lambda}, 
	B)\tag{$*$}
    \end{equation*}
    for any left $\OFG$-module $B$.  By
    Proposition~\ref{prop:hd-cd-result-2}, part~(2), it follows
    that~$\fld_{\frakF} (\res_{I} Q_{\lambda}) \leq k$.  Therefore
    $\Tor_{k+1}^{\frakF} (\res_{I} Q_{\lambda}, B)=0$ for every
    $\lambda$ and thus the right hand side of~($*$) is equal to $0$,
    too.  Hence we get that
    \begin{equation*}
	\Tor_{k+1}^{\frakF}(\res_{I} Q, B) \isom \Tor_{k+1}^{\frakF}
	(\varinjlim \res_{I} Q_{\lambda}, B) = 0
    \end{equation*}
    for any left $\OFG$ module $B$ since the restriction functor
    commutes with colimits. But this implies that 
    $\fld_{\frakF}(\res_{I} Q)\leq k$. Since $Q$ was an arbitrary 
    flat $\OGG$-module, we can apply 
    Proposition~\ref{prop:hd-cd-result-1} to get $\hd_{\frakF} G\leq 
    \hd_{\frakG} G + k$.
\end{proof}

The following example is the algebraic equivalent to the first part 
of Corollary~5.4 in~\cite[p.~518]{luck-12}.

\begin{example}
    Consider the pair $(\Fvc(G), \Ffin(G))$ of subgroups of $G$.  The
    family $\Ffin(G)$ is closed under conjugation.  For any
    $K\in\Fvc(G)$ we have that~$\Ffin(G)\cap K\subset\Ffin(G)$ and it
    is known that $\uhd K\leq 1$ and $\ucd K \leq 1$, see for
    example~\cite[p.~137]{juan-pineda-06}.  Hence we have by
    Theorem~\ref{thrm:hdF-G-vs-hdG-G} and
    Theorem~\ref{thrm:cdF-G-vs-cdG-G} the two inequalities
    \begin{equation*}
        \uhd G \leq \uuhd G + 1
	\quad 
	\text{and}
	\quad
	\ucd G \leq \uucd G + 1.
    \end{equation*}
\end{example}

%
%

\section{(Co-)Homological Dimension for Direct Unions}
\label{sec:dimensions-for-direct-unions}

This section is motivated by Theorem~4.2 in
\cite[pp.~42f.]{nucinkis-04}.  In this theorem upper bounds for the
Bredon (co-)homological dimension for direct unions of groups are
given for the family of finite subgroups.  The proof given by Nucinkis
extends to a more general setting.

A direct union of groups is a special case of direct limits of groups,
a constructive description for the latter can for example be found
in~\cite[pp.~22ff.]{robinson-96}.  

\begin{DEF}
    Let $\{G_\lambda : \lambda\in \Lambda\}$ be a family of subgroups
    of a group $G$, indexed by an abstract indexing set $\Lambda$.  We
    say that $G$ is the \emph{direct union of the groups $G_\lambda$}
    if the following two conditions hold:
    \begin{enumerate}
    \item for every $\lambda,\mu\in \Lambda$ there exists a $\nu\in
    \Lambda$ such that $G_\lambda\leq G_\nu$ and $G_\mu\leq G_\nu$; 
    
    \item for every $g\in G$ there exists a $\lambda\in \Lambda$ such
    that $g\in G_\lambda$.
    \end{enumerate}
\end{DEF}

Note that we can recover $G$ from this definition as a direct limit of
the subgroups $G_{\lambda}$ in the sense
of~\cite[pp.~22ff.]{robinson-96} as follows.  We define a relation
``$\leq$'' on $\Lambda$ by
\begin{equation*}
\lambda\leq \mu :\Longleftrightarrow G_\lambda\leq G_\mu.
\end{equation*}
In this way $\Lambda$ becomes a directed set.  Whenever $G_\lambda\leq
G_\mu$ there exists an inclusion map $\varphi_\lambda^\mu\: G_\lambda
\to G_\mu$.  It is clear that
\begin{equation*}
\{ G_\lambda, \varphi_\lambda^\mu : \lambda,\mu \in \Lambda \text{
and }\lambda\leq \mu\}
\end{equation*}
forms a directed system of groups and that the obvious inclusion
homomorphisms $\imath_\lambda\: G_\lambda \hookrightarrow G$ give an
isomorphism
\begin{equation*}
\imath\: \varinjlim G_\lambda \to G.
\end{equation*}
We use this isomorphism to identify the direct limit $\varinjlim
G_\lambda$ with $G$.

\begin{DEF}
    \label{def:compatible-families}
    Let $G$ be the direct union of a family $\{G_{\lambda}\}$ of its
    subgroups indexed by the set $\Lambda$.  Assume that we are given
    a family $\frakF$ of subgroups of~$G$ and for each $\lambda\in
    \Lambda$ a family $\frakF_\lambda$ of subgroups of $G_\lambda$.
    We say that these families of subgroups are \emph{compatible with
    the direct union}\pagebreak[3] if the following four conditions are 
    satisfied:
    \begin{enumerate}
	\item\label{cond-1}
	$\frakF_\lambda \subset \frakF_\mu$ for
	every $\lambda,\mu\in \Lambda$ with $\lambda\leq \mu$,

	\item\label{cond-2} $\frakF_\lambda\subset
	\frakF$ for every $\lambda\in \Lambda$,

	\item\label{cond-3} $\frakF \subset
	\bigcup_{\lambda\in\Lambda}
	\frakF_\lambda$,
	
	\item\label{cond-4} $\frakF_{\lambda} = \frakF \cap 
	G_{\lambda}$ for all $\lambda\in \frakF$.
    \end{enumerate}
\end{DEF}

The main result of this section will be a generalised version of
Theorem~4.2 in~\cite{nucinkis-04}:

\begin{theorem}
\label{thrm:direct-union-result}
Assume that a group $G$ is the direct union of a family $\{G_\lambda :
\lambda \in \Lambda\}$ of its subgroups.  Assume that we are given
full families~$\frakF$ and~$\frakF_{\lambda}$, $\lambda\in \Lambda$,
which are compatible with the direct union.  Then
    \begin{enumerate}
	\item\label{claim1} $\hd_{\frakF} G \leq \sup
	\{\hd_{\frakF_{\lambda}}G_{\lambda}\}$ in general, and

	\item\label{claim2} $\cd_{\frakF} G \leq \sup\{
	\cd_{\frakF_{\lambda}}
	G_{\lambda}\} +1$ if the index set $\Lambda$ is countable.
    \end{enumerate}
\end{theorem}

Note that due to Proposition~\ref{prop:dim-subgroups-1} we always have
the inequalities
\begin{equation*}
    \sup \{\hd_{\frakF_{\lambda}}G_{\lambda}\} \leq \hd_{\frakF} G
    \qquad \text{and} \qquad \sup
    \{\cd_{\frakF_{\lambda}}G_{\lambda}\} \leq \cd_{\frakF} G.
\end{equation*}
In particular the inequality for the homological dimension in
Theorem~\ref{thrm:direct-union-result} is always an equality.

Before we prove the above theorem let us first give a criterion for
the families $\frakF$ and $\frakF_\lambda$, $\lambda\in \Lambda$, to 
be compatible with the direct union.

\begin{proposition}
    \label{prop:simple-comp-condition}
    Let $G$ be the direct union of a family $\{G_\lambda : \lambda
    \in \Lambda\}$ of its subgroups.  Let $\frakF$ be a family of
    subgroups of $G$ and $\frakF_{\lambda}$ families of subgroups of
    $G_{\lambda}$.  Assume that $\frakF$ is closed under forming
    subgroups, every $K\in \frakF$ is finitely generated and
    $\frakF_{\lambda} = \frakF\cap G_{\lambda}$ for every $\lambda\in
    \Lambda$.  Then the families of subgroups $\frakF$ and
    $\frakF_{\lambda}$, $\lambda\in \Lambda$, are compatible with the
    direct union.
\end{proposition}

\begin{proof}
    Condition (\ref{cond-4}) of
    Definition~\ref{def:compatible-families} is satisfied by
    assumption.  Since~$\frakF$ is closed under forming subgroups
    condition (\ref{cond-2}) is satisfied.  If~$\lambda\leq \mu$ then
    $G_{\lambda}\leq G_{\mu}$ and it follows that
    $\frakF_{\lambda}\subset \frakF_{\mu}$.  Hence condition
    (\ref{cond-1}) is satisfied.  Now let $K\in \frakF$ be an
    arbitrary group.  Then there exists a finite set of generators
    $\{k_{1}, \ldots,k_{n}\}$ of $K$.  For each $1\leq i\leq n$ there
    exists a~$\lambda_{i}\in \Lambda$ such that~$k_{i}\in
    G_{\lambda_{i}}$.  So there exists a $\lambda\in\Lambda$ such that
    $\lambda_{i}\leq \lambda$ for every~$1\leq i\leq n$.  It follows
    that $\{k_{1},\ldots,k_{n} \}\subset G_{\lambda}$ and therefore
    $K\subset G_{\lambda}$.  Hence $K\in \frakF\cap G_{\lambda} =
    \frakF_{\lambda}$ and condition (\ref{cond-3}) is satisfied.
\end{proof}

\begin{corollary}
    \label{cor:simple-comp-condition}
    The families $\Ffin(G)$ and $\Ffin(G_\lambda)$ are 
    compatible with the direct union. Likewise the families $\Fvc(G)$ 
    and $\Fvc(G_\lambda)$ are compatible with the direct union.
    \qed
\end{corollary}

Note that from the first case of the
Corollary~\ref{cor:simple-comp-condition} and together with
Theorem~\ref{thrm:direct-union-result} we recover Theorem~4.2
in~\cite{nucinkis-04}.

Before proving Theorem~\ref{thrm:direct-union-result} we require
some auxiliary results.

First of all, if $\frakF\cap K\subset \frakF$ then we
can extend the functor $I_{K}\: \OC{K}{\frakF \cap K}\to \OFG$ (see
Section~\ref{sec:exatness-of-IndIK}) to $K$-sets $X$ with
$\frakF(X)\subset \frakF\cap K$ by applying it to each orbit
seperately.  That is, if $X = \coprod_{x\in R} K/K_{x}$ where $R$ is 
a complete system of representatives for the $K$-orbits of $X$, then
\begin{equation*}
    I_{K}(X) \defeq  \coprod_{x\in R} G/K_{x}.
\end{equation*}
The set $I_{K}(X)$ is then a $G$-set with $\frakF(I_{K}(X)) = 
\frakF(X) \subset \frakF\cap K \subset \frakF$.

\begin{lemma}
    \label{lem:direct-union-result-aux1}
    Let $K\leq G$ be a subgroup such that $\frakF\cap K$ is a 
    non-empty subset of $\frakF$. Then
    \begin{equation*}
        \ind_{I_{K}} \Z[\?, X]_{G} \isom \Z[\?, I_{K}(X)]_{G}
    \end{equation*}
    for any $K$-set $X$ with $\frakF(X)\subset \frakF\cap K$. 
\end{lemma}

\begin{proof}
    \allowdisplaybreaks Let $R$ be a complete set of representatives
    of the orbit space~$X/K$.  Then we have the following sequence of
    isomorphisms of $\OFG$-modules.
    \begin{align*}
        \ind_{I_{K}} \Z[?,  X]_{K}
	& \isom
	\coprod_{x\in R} \ind_{I_{K}} \Z[\?, K/K_{x}]_{K}
	\\
	& \isom
	\coprod_{x\in R} \bigl( \Z[\q?, K/K_{x}]_{K} 
	\mathbin{\otimes_{\frakF\cap K}} \Z[\?, I_{K}(\q?)]_{G}\bigr)
	\\
	& \isom
	\coprod_{x\in R} \Z[\?, I_{K}(K/K_{x})]_{G}
	\\
	& \isom
	\coprod_{x\in R} \Z[\?, G/K_{x}]_{G}
	\\
	& \isom
	\Z[\?, I_{K}(X)]_{G}
	\qedhere
    \end{align*}
\end{proof}

\begin{lemma}
    \label{lem:direct-union-result-aux2}
    Consider the direct limiting system
    \begin{equation*}
	\{ \Z[\?, G/G_{\lambda}]_{G}, \varphi^{\mu}_{\lambda} :
	\lambda, \mu\in \Lambda \text{ and } \lambda \leq \mu\}
    \end{equation*}
    directed by $\Lambda$ where the morhpisms
    $\varphi^{\mu}_{\lambda}\: \Z[\?, G/G_{\lambda}]_{G} \to \Z[\?, 
    G/G_{\mu}]_{G}$ are induced by the projections 
    $G/G_{\lambda} \to G/G_{\mu}$. Then
    \begin{equation*}
        \varinjlim \Z[\?, G/G_{\lambda}]_{G} \isom 
	\underline\Z_{\frakF}.
    \end{equation*}
\end{lemma}

\begin{proof}
    For each $\lambda\in\Lambda$ we have a homomorphism
    \begin{equation*}
	\eta_{\lambda}\: \Z[\?, G/G_{\lambda}]_{G} \to 
	\Z[?,G/G]_{G} = \underline\Z_{\frakF}    
    \end{equation*}
    induced by the projection $p_{\lambda} G/G_{\lambda} \to G/G$.
    Clearly $\eta_{\mu} = \varphi^{\mu}_{\lambda} \circ
    \eta_{\lambda}$ for all~$\lambda\leq \mu$.  Thus the
    $\eta_{\lambda}$ define a homomorphism
    \begin{equation*}
        \eta\:  \varinjlim \Z[\?, G/G_{\lambda}]_{G} \to
	\underline\Z_{\frakF}
    \end{equation*}
    We need to show that 
    \begin{equation*}
        \eta_{H}\: \varinjlim \Z[G/G_{\lambda}^{H}] \to \Z
    \end{equation*}
    is an isomorphism for every $H\in \frakF$. Since the functor 
    $\Z[\?]\: \Set \to \Ab$ commutes with arbitrary colimits it is 
    enough to show that $\varinjlim (G/G_{\lambda})^{H} = (G/G)^{H} = 
    G/G$ where the last equality is due to the trivial action of $H$ 
    on the singleton set $G/G$.
    
    We verify this by showing that $G/G$ satisfies the universal
    property of a colimit.  Therefore assume that we are given a set
    $X$ and a collection of maps $f_{\lambda}\: G/G_{\lambda} \to X$
    such that $f_{\lambda} = f_{\mu} \circ \varphi^{\mu}_{\lambda}$
    for all $\lambda\leq \mu$.
    
    First of all, observe that each $f_{\lambda}$ is a constant
    function.  To see this, observe that if $g\in G$ there exists
    $\mu\in\Lambda$ such that $g\in G_{\mu}$ and $G_{\lambda}\leq
    G_{\mu}$.  Then
    \begin{equation*}
        f_{\lambda}(gG_{\lambda}) = 
	f_{\mu}(\varphi^{\mu}_{\lambda}(gG_{\lambda})) = 
	f_{\mu}(G_{\mu}) = 
	f_{\mu}(\varphi^{\mu}_{\lambda}(G_{\lambda})) = 
	f_{\lambda}(G_{\lambda})
    \end{equation*}
    which is independent of the choice of $g\in G$. Furthermore, 
    the value $f_{\lambda}(G_{\lambda})$ is independent of $\lambda$. 
    This is because given any $\lambda_{1}, \lambda_{2}\in \Lambda$ 
    there exists a $\mu$ such that $G_{\lambda_{1}}\leq G_{\mu}$ and 
    $G_{\lambda_{2}}\leq G_{\mu}$. Then
    \begin{equation*}
        f_{\lambda_{1}}(G_{\lambda_{1}}) = 
	f_{\mu}(\varphi^{\mu}_{\lambda_{1}}(G_{\lambda_{1}})) = 
	f_{\mu}(G_{\mu}) = 
	f_{\mu}(\varphi^{\mu}_{\lambda_{2}}(G_{\lambda_{2}})) = 
	f_{\lambda_{2}}(G_{\lambda_{2}}).
    \end{equation*}
    
    Now it follows that there exists a function $f\: G/G\to X$ such
    that $f_{\lambda} = f\circ \varphi_{\lambda}$ for all $\lambda\in
    \Lambda$ and that this function must be unique.  Thus $G/G$ has
    the universal property of a colimit and this concludes the proof.
\end{proof}

\begin{proof}[Proof of Theorem~\ref{thrm:direct-union-result}]
    The main part of the proof consists of constructing a free
    resolution of the trivial $\OFG$-module $\underline\Z_{\frakF}$
    using the standard resolutions of the trivial
    $\OC{G_{\lambda}}{\frakF_{\lambda}}$-modules
    $\underline\Z_{\frakF_{\lambda}}$ for $\lambda\in\Lambda$.

    For every $\lambda\in \Lambda$ let
    \begin{equation}
	\label{eq:direct-union-result-seq1}
	\ldots \longto F_{\lambda,2} \stackrel{d_{\lambda,2}}{\longto}
	F_{\lambda,1} \stackrel{d_{\lambda,1}}{\longto} F_{\lambda,0}
	\stackrel{\varepsilon_{\lambda}}{\longto} Z_{\lambda} \longto
	0
    \end{equation}
    be the sequence of $\OFG$-modules obtained from the standard
    resolution~\eqref{eq:standard-resolution} of the trivial
    $\OC{G_{\lambda}}{\frakF_{\lambda}}$-module
    $\underline\Z_{\frakF_{\lambda}}$ by applying the functor
    $\ind_{I_{G_{\lambda}}}$. Induction by 
    $I_{G_{\lambda}}$ is exact and thus the 
    sequence~\eqref{eq:direct-union-result-seq1} is exact. From the 
    construction of the standard resolution and from  
    Lemma~\ref{lem:direct-union-result-aux1} we know that
    \begin{equation*}
	F_{\lambda,n} \isom \Z[\?,
	I_{G_{\lambda}}(\Delta_{\lambda,n})]_{G}.
    \end{equation*}
    Furthermore $Z_{\lambda} \isom \Z[\?, G/G_{\lambda}]_{G}$.
    
    Observe that we can identify
    \begin{equation*}
	I_{G_{\lambda}}(\Delta_{\lambda,n}) = \{ (gg_{0}K_{0}, \ldots,
	gg_{n}K_{n}) : g\in G, g_{i}\in G_{\lambda} \text{ and } K_{i}\in
	\frakF_{\lambda}\}.
    \end{equation*}
    In particular for any $\lambda\leq \mu$ we have $G_{\lambda}\leq
    G_{\mu}$ and $\frakF_{\lambda} \subset \frakF_{\mu}$ and
    therefore~$I_{G_{\lambda}}(\Delta_{\lambda,n})\subset
    I_{G_{\mu}}(\Delta_{\mu,n})$.  We denote the this inclusion by
    $\varphi_{\lambda}^{\mu}$.  Clearly this inclusion is
    $G$-equivariant.  Therefore we get a morphism
    \begin{equation*}
	\varphi_{\lambda}^{\mu}\: F_{\lambda,n} \to F_{\mu,n}
    \end{equation*}
    of right $\OFG$-modules. The collection $\{F_{\lambda,n},
    \varphi_{\lambda}^{\mu} : \lambda, \mu\in \Lambda \text{ and }
    \lambda \leq \mu\}$ forms a direct limiting system directed by
    $\Lambda$. For each $n\in \N$ we denote its limit by~$F_{n}$.
    
    Similarly, for any $\lambda\leq \mu$ there exists a unique $G$-map
    $\varphi\: G/G_{\lambda} \to G/G_{\mu}$ which sends $G_{\lambda}$
    to $G_{\mu}$.  These $G$-maps give rise to morphisms
    $\varphi_{\lambda}^{\mu}\: Z_{\lambda} \to Z_{\mu}$.  It follows
    that the collection $\{Z_{\lambda}, \varphi_{\lambda}^{\mu} :
    \lambda, \mu\in \Lambda \text{ and } \lambda \leq \mu\}$ forms a
    direct limiting system directed by $\Lambda$.  We denote its
    limit by $Z$.
    
    It follows that for every $\lambda\leq\mu$ and $n\geq 1$ we have
    that the diagrams
    \begin{equation*}
        \begin{diagram}
            \node{F_{\lambda,n}}
	    \arrow{e,t}{d_{\lambda,n}}
	    \arrow{s,l}{\varphi_{\lambda}^{\mu}}
	    \node{F_{\lambda,n-1}}
	    \arrow{s,r}{\varphi_{\lambda}^{\mu}}
	    \\
	    \node{F_{\mu,n}}
	    \arrow{e,t}{d_{\mu,n}}
	    \node{F_{\mu,n-1}}
        \end{diagram}
	\qquad
	\text{and}
	\qquad
        \begin{diagram}
            \node{F_{\lambda,0}}
	    \arrow{e,t}{\varepsilon_{\lambda}}
	    \arrow{s,l}{\varphi_{\lambda}^{\mu}}
	    \node{Z_{\lambda}}
	    \arrow{s,r}{\varphi_{\lambda}^{\mu}}
	    \\
	    \node{F_{\mu,0}}
	    \arrow{e,t}{\varepsilon_{\mu}}
	    \node{Z_{\mu}}
        \end{diagram}
    \end{equation*}
    commute. Therefore we obtain a sequence
    \begin{equation}
	\label{eq:direct-union-result-seq2}
	\ldots \longto F_{2} \stackrel{d_{2}}{\longto} F_{1}
	\stackrel{d_{1}}{\longto} F_{0}
	\stackrel{\varepsilon}{\longto} Z \longto 0.
    \end{equation}
    This sequence is exact, since direct limits preserve exactness.
    
    It follows that the collection 
    $\{I_{G_{\lambda}}(\Delta_{\lambda,n}),
    \varphi_{\lambda}^{\mu} \text{ and } \lambda\leq \mu\}$ of
    $G$-sets is a direct limiting system directed by $\Lambda$. Its
    limit is
    \begin{equation*}
	\Delta_{n} \defeq  \{ (g_{0}K_{0}, \ldots, g_{n}K_{n}) : g_{i}\in
	G \text{ and } K_{i}\in \frakF\}
    \end{equation*}
    and this $G$-set is actually the direct union of the
    $I_{G_{\lambda}}(\Delta_{\lambda,n})$.  As a consequence~$F_{n}
    \isom \Z[\?, \Delta_{n}]_{G}$ by
    Proposition~\ref{prop:direct-union}.  Furthermore, it follows
    that~$Z\isom \underline\Z_{\frakF}$ by
    Lemma~\ref{lem:direct-union-result-aux2}.  Thus the
    sequence~\eqref{eq:direct-union-result-seq2} is nothing else but
    the free standard resolution of the trivial $\OFG$-module
    $\underline\Z_{\frakF}$.  We are now ready to prove the two claims
    of the theorem:
   
    \medskip
    
    We first show that
    \begin{equation*}
	\hd_{\frakF} G \leq \sup \{\hd_{\frakF_{\lambda}}G_{\lambda}\}.
    \end{equation*}
    In order to avoid triviality we assume that $n\defeq \sup \{
    \hd_{\frakF_{\lambda}} G_{\lambda}\}$ is finite.  From the
    resolution~\eqref{eq:direct-union-result-seq2} of the trivial
    $\OFG$-module we obtain the resolution
    \begin{equation}
	\label{eq:res-3}
	0 \to K \to F_{n-1} \to \ldots \to
	F_{0} \to \underline\Z_{\frakF}
	\to 0
    \end{equation}
    where $K$ is the $(n-1)$-th kernel of the
    resolution~\eqref{eq:direct-union-result-seq2}.  We claim that $K$
    is flat.  Since direct limits and induction with $I_{G_{\lambda}}$
    are exact, it follows that
    \begin{equation*}
	K \isom \varinjlim (\ind_{I_{G_{\lambda}}} K_{\lambda})
    \end{equation*}
    where the $K_{\lambda}$ are the $(n-1)$-th kernels of the free
    standard resolution of the trivial
    $\OC{G_{\lambda}}{\frakF_{\lambda}}$-modules
    $\underline\Z_{\frakF_{\lambda}}$.  But these are flat because
    $\hd_{\frakF_{\lambda}} G_{\lambda} \leq n$.  Since the functor
    $\ind_{I_{G_{\lambda}}}$ preserves flats we have that $K$ is the
    direct limit of flats and hence is flat as well.
    Therefore~\eqref{eq:res-3} is a flat resolution of the
    trivial~$\OFG$-module $\underline\Z_{\frakF}$ of length $n$ and so
    $\hd_{\frakF}G \leq \sup \{\hd_{\frakF_{\lambda}}G_{\lambda}\}$.
    
    \medskip
    
    Next, we assume that the set $\Lambda$ is countable and we want to
    verify the second claim of the theorem, that
    \begin{equation*}
	\cd_{\frakF} G \leq \sup\{ \cd_{\frakF_{\lambda}}
	G_{\lambda}\} +1.
    \end{equation*}
    Again, in order to avoid triviality we assume that $n\defeq \sup \{
    \cd_{\frakF_{\lambda}} G_{\lambda}\}$ is finite.  As before, let
    $K$ be the $(n-1)$-th kernel of the standard
    resolution~\eqref{eq:direct-union-result-seq2}.  Similarly, it
    follows that $K$ is the direct limit
    \begin{equation*}
	K \isom \varinjlim (\ind_{G_{\lambda}} K_{\lambda})
    \end{equation*}
    of projectives.  Since $\Lambda$ is assumed to be countable we can
    apply Lemma~3.4 in \cite[p.~40]{nucinkis-04}.  This lemma states
    that the limit of a countable directed system of projective right
    $\OFG$-modules has projective dimension at most~$1$.  Hence
    $\pd_{\frakF} K = 1$ and there exists a projective resolution
    \begin{equation*}
	0 \to P_{1} \to P_{0} \to K\to0
    \end{equation*}
    of $K$.  We can splice this sequence together
    with~\eqref{eq:res-3} to get a projective resolution
    \begin{equation*}
	0 \to P_{1} \to P_{0} \to F_{n} \to \ldots \to F_{0} \to
	\underline\Z_{\frakF} \to 0
    \end{equation*}
    of the trivial $\OFG$-module $\underline\Z_{\frakF}$.  Hence
    $\cd_{\frakF} G \leq \sup \{ \cd_{\frakF_{\lambda}} G_{\lambda}\}
    +1 $.
\end{proof}

\begin{proposition}
    \label{prop:dim-locally-F}
    Let $G$ be a group and $\frakF$ be a full family of finitely
    generated subgroups of $G$.  If $G$ is locally $\frakF$, that is
    every finitely generated subgroup of $G$ is contained in the
    family $\frakF$, then $\hd_{\frakF} G = 0$.  If in addition $G$ is
    countable, then $\cd_{\frakF} G \leq 1$.
\end{proposition}

\begin{proof}
    Every group $G$ is the direct union of its finitely generated
    subgroups $G_{\lambda}$.  Set $\frakF_{\lambda}\defeq \frakF\cap
    G_{\lambda}$.  By assumption $G_{\lambda}\in \frakF$, and so by
    Proposition~\ref{prop:cdG=0}, $\underline\Z_{\frakF_{\lambda}}$ is
    projective and in particular a flat
    $\OC{G_{\lambda}}{\frakF_{\lambda}}$-module, and
    thus~$\cd_{\frakF}G_{\lambda} = \hd_{\frakF}G_{\lambda}=0$.  Since
    $\frakF$ is closed under forming subgroups, we have that the
    families $\frakF$ and $\frakF_{\lambda}$ are compatible with the
    limit by Proposition~\ref{prop:simple-comp-condition}.  Now the
    first part of Theorem~\ref{thrm:direct-union-result} gives
    $\hd_{\frakF} G = 0$.  If $G$ is countable then $\frakF$ is
    countable by Remark~\ref{rem:F-countable} and thus the second part
    of Theorem~\ref{thrm:direct-union-result} gives the estimation
    $\cd_{\frakF} G \leq 1$.
\end{proof}

Examples of families which satisfy the conditions of
Proposition~\ref{prop:dim-locally-F} are~$\Ffin(G)$ and $\Fvc(G)$ in
case that $G$ is locally finite and $\Fvc(G)$ in the case that $G$ is
locally virtually cyclic.  Thus we obtain the following corollary to
Proposition~\ref{prop:dim-locally-F}.

\begin{corollary}
    \label{cor:dim-locally-F}
    Assume that $G$ is a locally finite group.  Then $\uhd G
    = 0$ in general and if in addition $G$ is countable then
    $\ucd G\leq 1$.  Similarly, if $G$ is locally virtually
    cyclic then $\uuhd G=0$ in general, and if in addition $G$ is
    countable, then $\uucd G\leq 1$.\qed
\end{corollary}

%
%

\section{Tensor Product of Projective Resolutions}
\label{sec:tensor-product-resolutions}

Let $G$ and $H$ be groups and $f\: G\to H$ a group homomorphisms. 
Assume that we are given families $\frakF$ and $\frakG$ of subgroups 
of $G$ and $H$ respectively, such that $f(\frakF)\subset \frakG$.

Then we can construct a functor $f\: \OFG \to \OC{H}{\frakG}$ as
follows.  Given an object $G/K$ in $\OFG$ we set $f(G/K) \defeq  H/f(K)$
which is an object of $\OC{H}{\frakG}$.  If $\varphi\: G/K \to G/L$ is
a morphism in $\OFG$ which maps $K\to gL$, $g\in G$, then~$K^{g}\leq
L$ and therefore
\begin{equation*}
    f(K)^{f(g)} = f(K^{g}) \leq f(L).
\end{equation*}
Hence there exists a unique $H$-map $f(G/K)\to f(G/L)$ which maps the
coset~$f(K)$ to the coset~$f(g)f(L)$.  We denote this $H$-map by
$f(\varphi)$.  This way we get a map $f\: \mor(G/K, G/L) \to
\mor(f(G/K), f(G/L))$ for each pair~$G/K$, $G/L$ of objects in $\OFG$.

\begin{lemma}
    The above construction gives a functor $f\: \OFG\to 
    \OC{H}{\frakG}$.
\end{lemma}

\begin{proof}
    If $\id$ is the identity on $G/K$, that is $\id =
    \varphi_{e,K,K}$, then
    \begin{equation*}
        f(\id) = f(\varphi_{e,K,K}) = \varphi_{f(e),f(K),f(K)} = 
	\varphi_{e,f(K),f(K)} = \id
    \end{equation*}
    is the identity on $f(G/K)$.
    
    Assume we are given two morphisms $\varphi_{g_{1},K_{1}, K_{2}}$ 
    and $\varphi_{g_{2},K_{2}, K_{3}}$ in $\OFG$, then we know that
    \begin{equation*}
	\varphi_{g_{2},K_{2}, K_{3}}\circ \varphi _{g_{1},K_{1}, 
	K_{2}} = \varphi_{g_{2}g_{1},K_{1},K_{3}}.
    \end{equation*}
    Therefore
    \begin{multline*}
	f(\varphi_{g_{2},K_{2}, K_{3}}\circ \varphi _{g_{1},K_{1}, 
	K_{3}})
	= f(\varphi_{g_{2}g_{1},K_{1},K_{3}})
	= \varphi_{f(g_{2}g_{1}),f(K_{1}),f(K_{3})}
	\\
	= \varphi_{f(g_{2})f(g_{1}),f(K_{1}),f(K_{3})}
	= \varphi_{f(g_{1}),f(K_{1}), f(K_{2})} \circ 
	\varphi_{f(g_{2}, f(K_{2}), f(K_{3})}. 
	\qedhere
    \end{multline*}
\end{proof}

We apply the above result to the following setting.  Given two
groups~$G_{1}$ and $G_{2}$ we consider the group $G \defeq G_{1}
\times G_{2}$.  Assume we are given two semi-full families of
subgroups $\frakF_{1}$ and $\frakF_{2}$ of $G_{1}$ and $G_{2}$
respectively.  The natural choice for a family $\frakF$ of subgroups
of $G$ is the cartesian product of the families $\frakF_{1}$ and
$\frakF_{2}$.  Recall that by definition this family is
\begin{equation*}
    \frakF_{1}\times \frakF_{2} = \{ H_{1}\times H_{2} : H_{1} \in 
    \frakF_{1} \text{ and } H_{2}\in \frakF_{2}\},
\end{equation*}
see Section~\ref{sec:families-of-subgroups} in
Chapter~\ref{ch:basics}.  The projection homomorphisms $p_{i}\: G\to
G_{i}$ satisfy $p_{i}(\frakF) = \frakF_{i}$.  Therefore we obtain
projection functors
\begin{equation*}
    p_{i}\: \OFG \to \OC{G_{i}}{\frakF_{i}}
\end{equation*}
Given right $\OC{G_{i}}{\frakF_{i}}$-modules $M_{i}$, $i=1,2$, we 
can now form a tensor product of these modules over $\Z$ as follows.
Applying the restriction functor $\res_{p_{i}}$ to~$M_{i}$ we obtain 
a $\OFG$-module $\res_{p_{i}} M_{i}$ and we can form the 
$\OFG$-module
\begin{equation*}
    \res_{p_{1}} M_{1} \otimes \res_{p_{2}} M_{2}.
\end{equation*}
Essentially this module is the tensor product of $M_{1}$ and $M_{2}$
with the orbit category $\OC{G_{1}}{\frakF_{1}}$ acting on the first
factor and the orbit category $\OC{G_{2}}{\frakF_{2}}$ acting on the
second factor of the tensor product.  This construction is the
generalisation of the classical construction of the tensor product
$M_{1}\otimes M_{2}$ of a $G_{1}$-module $M_{1}$ and a $G_{2}$-module
$M_{2}$ with~$G_{1}$ acting on the first factor and~$G_{2}$ acting on
the second factor.  This in turn makes $M_{1}\otimes M_{2}$ a
$G$-module.

\begin{lemma}
    \label{lem:resZ}
    For $i=1,2$ we have the equality $\res_{p_{i}}
    \underline\Z_{\frakF_{i}} = \underline\Z_{\frakF}$.
\end{lemma}

\begin{proof}
    Given a $H = H_{1}\times H_{2} \in \frakF$ we have 
    \begin{equation*}
	(\res_{p_{i}} \underline\Z_{\frakF_{i}})(G/H) = 
	\underline\Z_{\frakF_{i}}(G_{i}/H_{i})  = \Z = 
	\underline\Z_{\frakF}(G/H),
    \end{equation*}
    and similarly if $\varphi$ is a morphism in $\OFG$, then
    \begin{equation*}
        (\res_{p_{i}} \underline\Z_{\frakF_{i}})(\varphi)  = 
	(\underline\Z_{\frakF_{i}}\circ p_{i})(\varphi) = \id = 
	\underline\Z_{\frakF}(\varphi).\qedhere
    \end{equation*}
\end{proof}

\begin{corollary}
    \label{cor:resZ}
    Let $M_{i}$ be an $\OC{G_{i}}{\frakF_{i}}$-module, $i=1,2$. Then
    \begin{equation*} 
	\res_{p_{1}} M_{1} \otimes \res_{p_{2}}
	\underline\Z_{\frakF_{2}} = \res_{p_{1}} M_{1}
	\quad 
	\text{and}
	\quad
	\res_{p_{1}}\underline\Z_{\frakF_{1}} \otimes \res_{p_{2}}
	M_{2} = \res_{p_{2}} M_{2}.
    \end{equation*}
    In particular we have the equality $\res_{p_{1}}
    \underline\Z_{\frakF_{1}} \otimes \res_{p_{2}}
    \underline\Z_{\frakF_{2}} = \underline\Z_{\frakF}$.\qed
\end{corollary}

\begin{lemma}
    \label{lem:free-free}
    Let $F_{i} = \Z[\?, X_{i}]_{G_i}$ be free right
    $\OC{G_{i}}{\frakF_{i}}$-modules, $i=1,2$.  Then $\res_{p_{1}}
    F_{1} \otimes \res_{p_{2}} F_{2}$ is the free right $\OFG$-module
    $\Z[\?, X]_G$ with the diagonal action of $G= G_{1}\times G_{2}$
    on $X\defeq X_1\times X_2$.
\end{lemma}

\begin{proof}
    \allowdisplaybreaks
    For each $H= H_{1}\times H_{2}\in \frakF$ we have
    \begin{align*}
	(\res_{p_{1}} F_{1} \otimes \res_{p_{2}} F_{2})(G/H) 
	& =
	\Z[G_{1}/H_{1}, X_{1}]_{G_{1}} \otimes \Z[G_{2}/H_{2},
	X_{2}]_{G_{2}} \\
	& =
	\Z \bigl[ [G_{1}/H_{1}, X_{1}]_{G_{1}} \times
	\Z[G_{2}/H_{2}, X_{2}]_{G_{2}}\bigr] \\
	& \isom  
 	\Z[G/H, X]_G
    \end{align*}
    and this isomorphism is given by
    \begin{equation*}
        \psi_{1}\otimes \psi_{2} \mapsto \psi
    \end{equation*}
    where $\varphi\: G/H\to X$ is the $G$-map given by
    \begin{equation*}
	H \mapsto \bigl(\psi_{1}(H_{1}), \psi_{2}(H_{2})\bigr).
    \end{equation*}
    Denote this isomorphism by $\eta_{H}$. We claim that these 
    isomorphisms form isomorphism
    \begin{equation*}
	\eta\: \Z[\?, X_{1}]_{G_{1}} \otimes \Z[\?, X_{2}]_{G_{2}} \to
	\Z[\?, X]_{G}
    \end{equation*}
    of $\OFG$-modules. This can be verified by a straightforward 
    diagram chase as follows.
    
    Let $H= H_{1}\times H_{2}, K = K_{1}\times K_{2}\in \frakF$ and 
    $\varphi = (\varphi_{1}, \varphi_{2}) \in [G/H, G/K]_{G}$. 
    Furthermore, let $\psi_{1}\otimes \psi_{2}$ a basis element of 
    $\Z[G_{1}/K_{1}, X_{1}]_{G_{1}} \otimes \Z[G_{2}/K_{2}, 
    X_{2}]_{G_{2}}$.
    
    Then
    \begin{align*}
        (\eta_{H}\circ \varphi^{*})(\psi_{1}\otimes \psi_{2})
	& = 
	\eta_{H}((\psi_{1} \otimes \psi_{2})\circ \varphi)
	\\
	& = 
	\eta_{H}((\psi_{1} \circ \varphi_{1})\otimes (\psi_{2}\circ 
	\varphi_{2}))
	\\
	& =
	(\psi_{1} \circ \varphi_{1}, \psi_{2}\circ 
	\varphi_{2})
	\\
	& = 
	\varphi^{*}(\psi_{1}, \psi_{2})
	\\
	& = (\varphi^{*}\circ \eta_{K})(\psi_{1}\otimes \psi_{2}).
    \end{align*}
    Therefore the $\eta_{G/H}$ form an homomorphism of right 
    $\OFG$-modules and since each $\eta_{G/H}$ is an isomorphism the 
    homomorphism $\eta$ is an isomorphism, too.
\end{proof}

\begin{corollary}
    \label{cor:proj-proj}
    Let $P_{i}$ be projective right $\OC{G_{i}}{\frakF_{i}}$-modules,
    $i=1,2$.  Then $\res_{p_{1}} P_{1} \otimes \res_{p_{2}} P_{2}$ is 
    a projective $\OFG$-module.
\end{corollary}

\begin{proof}
    Since $P_{i}$ is projective it is a direct summand of a free 
    $\OC{G_{i}}{\frakF_{i}}$-module $F_{i}$, say $F_{i} = P_{i} 
    \oplus Q_{i}$ for some projective $Q_{i}$. Since restriction and 
    the tensor product over $\Z$ are additive functors we get by the 
    previous lemma that
    \begin{multline*}
        \res_{p_{1}} F_{1} \otimes \res_{p_{2}} F_{2}
	= 
	(\res_{p_{1}} P_{1} \otimes \res_{p_{2}} P_{2})
	\oplus
	(\res_{p_{1}} P_{1} \otimes \res_{p_{2}} Q_{2})
	\\
	\oplus
	(\res_{p_{1}} Q_{1} \otimes \res_{p_{2}} P_{2})
	\oplus
	(\res_{p_{1}} Q_{1} \otimes \res_{p_{2}} Q_{2})
    \end{multline*}
    is a free $\OFG$-module.  Therefore $\res_{p_{1}} P_{1} \otimes
    \res_{p_{2}} P_{2}$ is a direct summand of a free $\OFG$-module
    and thus projective.
\end{proof}

\begin{lemma}
    Let $F_{i} = \Z[\?, X_{i}]_{G_i}$ be a free right
    $\OC{G_{i}}{\frakF_{i}}$-modules, $i=1,2$.  Then $\res_{p_{i}}
    F_{i}$ is a free right $\OFG$-module.
\end{lemma}

\begin{proof}
    This follows essentially from the observation that
    \begin{equation*}
	\res_{p_{i}} \Z[\?, X_{i}] _{G_i}= \Z[\?, X_{i}]_{G},
    \end{equation*}
    where the set $X_{i}$ on the right hand side of the 
    equation is seen as a $G$-set by~$gx \defeq  p_{i}(g)x$.\pagebreak[3]
\end{proof}

\begin{corollary}
    Let $P_{i}$ be a projective right 
    $\OC{G_{i}}{\frakF_{i}}$-module. Then $\res_{p_{i}} P_{i}$ is a 
    projective $\OFG$-module.
\end{corollary}

\begin{proof}
    This follows again from the fact that $\res_{p_{i}}$ is an 
    additive functor and hence a direct summand of free 
    $\OC{G_{i}}{\frakF_{i}}$-modules is mapped to direct summand of a 
    $\OFG$-module which is free by the previous lemma.
\end{proof}

Let $P_{*}\to \underline\Z_{\frakF_1}$ be a resolution of right
$\OC{G_{1}}{\frakF_{1}}$-modules and let $Q_{*}\to
\underline\Z_{\frakF_2}$ be a resolution of right
$\OC{G_{2}}{\frakF_{2}}$-modules.  Then we can form the double complex
\begin{equation}
    C_{p,q} \defeq  \res_{p_{1}} P_{p} \otimes \res_{p_{2}} Q_{q}
    \label{eq:double-complex}
\end{equation}
of right $\OFG$-modules where $p,q\in \N$.  Taking the total complex
of this double complex gives the chain complex
\begin{equation*}
    \ldots \to C_3 \to C_2 \to C_1 \to C_0.
\end{equation*}
Denote the epimorphisms $P_0\to \underline\Z_{\frakF_1}$ and $Q_0\to
\underline\Z_{\frakF_2}$ by $\varepsilon_1$ and $\varepsilon_2$.  Then
we obtain a morphism $\varepsilon \defeq  \res_{p_1} \varepsilon_1 \otimes
\res_{p_2} \varepsilon_2$ from $C_0$ onto $\res_{p_1} 
\underline\Z_{\frakF_1} \otimes \res_{p_2} \underline\Z_{\frakF_2}$.

\begin{proposition}
    \label{prop:key-result}
    Let $P_{*}\to \underline\Z_{\frakF_1}$ and $Q_{*}\to
    \underline\Z_{\frakF_2}$ be free (projective) resolutions.  Then
    \begin{equation}
	\ldots\to C_2 \to C_1 \to C_0 \stackrel{\varepsilon}{\to}
	\underline\Z_\frakF \to 0
	\label{eq:resolution}
    \end{equation}
    is a free (projective) resolution of the trivial $\OFG$-module
    $\underline\Z_\frakF$.
\end{proposition}

\begin{proof}
    It follows straight from Lemma~\ref{lem:free-free}
    (Corollary~\ref{cor:proj-proj}) that the $\OFG$-modules $C_i$ are
    free (projective).  The domain of $\varepsilon$ is by 
    construction~$C_0$ and its codomain is $\underline\Z_\frakF$ by
    Corollary~\ref{cor:resZ}.  Thus it remains to show that the
    sequence~\eqref{eq:resolution} is exact.  For this we need to show
    that~\eqref{eq:resolution} evaluated at any object $G/H$ of the
    orbit category $\OFG$ is exact.

    Therefore let $H=H_1\times H_2\in \frakF$.  Set $P'_* \defeq 
    P_*(G_1/H_1) = (\res_{p_1} P_*)(G/H)$ and $Q'_* \defeq  Q_*(G_2/H_2) =
    (\res_{p_2} Q_*)(G/H)$.  Since restriction is an exact functor we
    obtain two exact resolutions $P'_*\stackrel{\varepsilon'_1}{\to}
    \Z$ and $Q'_*\stackrel{\varepsilon'_2}{\to} \Z$ of abelian groups
    with $\varepsilon'_i \defeq  (\res_{p_i} \varepsilon_i)_{H}$.  Since
    $P_k$ is projective it is a direct summand of a free
    $\OC{G_1}{\frakF_1}$-module $F_k$.  By definition $F_k(G_1/H_1)$
    is a free abelian group.  Since colimits are calculated
    componentwise it follows that $P_k(G_1/H_1)$ is a direct summand
    of $F_k(G_1/H_1)$ and hence a free abelian group.  Thus~$P'_* \to
    \Z$ is a free resolution of $\Z$ and likewise is $Q'_*\to \Z$.
    Let $C'_*$ be the total complex of the double complex $P'_*
    \otimes Q'_*$.  It follows that this complex gives a free
    resolution of $\Z$ where the augmentation map $\varepsilon'\:
    C'_0\to \Z$ is given by $\varepsilon' \defeq  \varepsilon'_1 \otimes
    \varepsilon'_2$, see for example~\cite[p.~107]{brown-82}.  We
    claim that this resolution of $\Z$ is identical with the
    resolution obtained from evaluating
    \begin{equation*}
	\ldots \to C_2 \to C_1 \to C_0 \to \underline\Z_\frakF \to 0
    \end{equation*}
    at $G/H$. This claim is verified by straightforward calculation 
    as follows.

    First, for any $n\in \N$ we have
    \begin{align}
	C_n(G/H) & = \coprod_{k=0}^n C_{k,n-k}(G/H)
	\nonumber
	\\
	& = 
	\coprod_{k=0}^n (\res_{p_{1}} P_{k})(G/H) \otimes (\res_{p_{2}}
	Q_{n-k})(G/H)
	\label{eq:identification}
	\\
	& = 
	\coprod_{k=0}^n P'_{k} \otimes Q'_{n-k} = C'_n
	\nonumber
    \end{align}
    and $\underline\Z_\frakF(G/H) = \Z$ by definition.  Hence the
    sequences agree on objects and it remains to show that the
    homomorphisms agree as well.
    
    For the homomorphism $\varepsilon_{G/H}\: C_0(G/H) \to \Z$ we have
    the following sequence of equalities
    \begin{equation*}
	\varepsilon_{H} = (\res_{p_1} \varepsilon_1 \otimes
	\res_{p_2} \varepsilon_2)_{H} 
	=
	(\res_{p_1} \varepsilon_1)_{H} \otimes (\res_{p_2}
	\varepsilon_2)_{H}
	=
	\varepsilon'_1 \otimes \varepsilon'_2 = \varepsilon'.
    \end{equation*}
    A similar kind of argument shows that, for any $n\geq 1$, the
    differentials $d_{n,H}\: C_n (G/H) \to C_{n-1}(G/H)$ agree with
    the differentials $d'_{n}\: C'_n \to C'_{n-1}$ under the
    identification~\eqref{eq:identification}.
\end{proof}

\begin{corollary}
    If $\cd_{\frakF_1} G_1 \leq m$ and $\cd_{\frakF_2} G_2 \leq n$
    then $\cd_{\frakF} G\leq m+n$.\qed
\end{corollary}

In order to simplify the notation we set $\tilde H_i \defeq  p_i(H)$ and
$\tilde H \defeq  \tilde H_1\times \tilde H_2$ for any subgroup~$H$ of~$G$.
That is, $\tilde H$ is the smallest subgroup of $G$ which 
contains~$H$ and which is of the form $H_{1}\times H_{2}$ with 
$H_{1}\leq G_{1}$ and $H_{2}\leq G_{2}$.

\begin{lemma}
    \label{lem:XH}
    Let $X_1$ be a $G_1$-set and $X_2$ a $G_2$-set.  Consider the set
    $X= X_1\times X_2$ with the diagonal $G$-action given by $(g_{1},
    g_{2})(x_{1},x_{2}) \defeq  (g_{1}x_{1},g_{2}x_{2})$.  Then
    \begin{equation*}
	X^H = X^{\tilde H} = X_1^{\tilde H_1} \times X_2^{\tilde H_2}.
    \end{equation*}
\end{lemma}

\begin{proof}
    The second equality is clear by construction of $\tilde H$ and
    $X$.  It remains to show that the first equality is true.  Since
    $H\leq \tilde H$ we have that $X^{\tilde H}\subset X^H$.  Thus it
    remains to show that $X^H\subset X^{\tilde H}$.  Let $x = (x_1,
    x_2)\in X^H$ and let $h= (h_1, h_2) \in \tilde H$.  There exist
    $h'_1\in \tilde H_1$ and $h'_2\in \tilde H_2$ such that~$(h_1,
    h'_2), (h'_1, h_2)\in H$.  Then $(h_1 x_1, h'_2 x_2) =
    (h_1,h'_2)(x_1, x_2) = (x_1, x_2)$ and likewise $(h'_1 x_1, h_2
    x_2) = (h'_1, h_2)(x_1, x_2) = (x_1, x_2)$.  Hence $(h_1, h_2)
    (x_1, x_2) = (x_1, x_2)$ and this implies that $x\in X^{\tilde
    H}$.
\end{proof}

Recall that given a family $\frakF$ of subgroups of a group $G$, we
defined its subgroup completion $\bar\frakF$ to be the smallest
\emph{full} family of subgroups of $G$ which contains $\frakF$ (see
Section~\ref{sec:families-of-subgroups} in Chapter~\ref{ch:basics}).

\begin{lemma}
  \label{lem:extension}
  Assume that $\frakF_1$ and $\frakF_2$ are full families of
  subgroups.  Let $X_1$ and~$Y_1$ be $G_1$-sets with stabilisers in
  $\frakF_1$ and let $X_2$ and $Y_2$ be $G_2$-sets with stabilisers in
  $\frakF_2$.  Consider the sets $X\defeq  X_1\times X_2$ and $Y\defeq  Y_1
  \times Y_2$ with the diagonal action of $G= G_{1}\times G_{2}$.
  Then any morphism
  \begin{equation*}
      f\: \Z[\?, X]_G \to \Z[\?, Y]_G
  \end{equation*}
  of right $\OFG$-modules can be extended to a morphism of right
  $\OC{G}{\bar\frakF}$-modules by $f_{H} \defeq  f_{\tilde H}$.
\end{lemma}

\begin{proof}
  Since $\frakF_1$ and $\frakF_2$ are full families of subgroups of
  $G_1$ and $G_2$ it follows that $\tilde H\in \frakF$ for any $H\in
  \bar\frakF$.  Furthermore it follows from the previous lemma that
  $\Z[G/H, X]_G = \Z[G/\tilde H, X]_G$ and $\Z[G/H, Y]_G =
\Z[G/\tilde H,
  Y]_G$.  Therefore $f_{H}$ is defined for any $H\in \bar\frakF$.  We
  need to show that this extension of $f$ indeed gives a morphism of
  $\OC{G}{\bar\frakF}$-modules.

  Let $H,K\in \bar\frakF$ and let $\varphi \defeq  \varphi_{g, H, K}\: G/H
  \to G/K$ be a morphism in $\OC{G}{\bar\frakF}$.  We need to show
  that the diagram
  \begin{equation}
    \begin{diagram}
      \node{\Z[G/K, X]_G}
      \arrow{e,t}{f_{K}}
      \arrow{s,l}{\varphi^*}
      \node{\Z[G/K, Y]_G}  
      \arrow{s,r}{\varphi^*}
      \\
      \node{\Z[G/H, X]_G}
      \arrow{e,t}{f_{H}}
      \node{\Z[G/H, Y]_G}
    \end{diagram}
    \label{dgrm:commutative-square}
  \end{equation}
  commutes. 
  
  Since $H^g\leq K$ it follows that $\tilde H^g\leq \tilde K$.  Let
  $\tilde\varphi \defeq \varphi_{g, \tilde H, \tilde K}$ which is a
  morphism in $\OFG$ as well as a morphism in $\OC{G}{\bar \frakF}$.
  By the definition of the Bredon modules $\Z[\?, X]_G$ and $\Z[\?,
  Y]_G$ and from Lemma~\ref{lem:XH} it follows that~$\varphi^* \circ
  \id = \id \circ \mathop{\tilde \varphi}^*$.  Now in order to show
  that the diagram~\eqref{dgrm:commutative-square} commutes we imbed
  it into the following larger diagram.
  \begin{equation*}
      \begin{diagram}
	  \node{\Z[G/\tilde K, X]_G}
	  \arrow[3]{e,t}{f_{\tilde K}}
	  \arrow[3]{s,l}{\tilde\varphi^*}
	  \node[3]{\Z[G/\tilde K, Y]_G}
	  \arrow[3]{s,r}{\tilde\varphi^*}
	  \\
	  \node[2]{\Z[G/K, X]_G}
	  \arrow{nw,t}{\id}
	  \arrow{e,t}{f_{K}}
	  \arrow{s,l}{\varphi^*}
	  \node{\Z[G/K, Y]_G}
	  \arrow{ne,t}{\id}
	  \arrow{s,r}{\varphi^*}
	  \\
	  \node[2]{\Z[G/H, X]_G}
	  \arrow{e,t}{f_{H}}
	  \arrow{sw,t}{\id}
	  \node{\Z[G/H, Y]_G}
	  \arrow{se,t}{\id}
	  \\
	  \node{\Z[G/\tilde H, X]_G}
	  \arrow[3]{e,t}{f_{\tilde H}}
	  \node[3]{\Z[G/\tilde H, Y]_G}
      \end{diagram}
  \end{equation*}
  
  The large outer square of this diagram commutes by assumption.  The
  upper and lower trapezoid commute by definition and the above said
  means that the left and right trapezoid commute.  It follows that
  the inner small square commutes, that is, we have shown that the
  diagram~\eqref{dgrm:commutative-square} commutes.
\end{proof}

Under the same assumptions as in the previous lemma we can extend the
morphism $\varepsilon\: \Z[\?, X]_G \to \underline\Z_\frakF$
of $\OFG$-modules to a morphism
\begin{equation*}
    \varepsilon\: \Z[\?, X]_G \to \underline\Z_{\bar\frakF}
\end{equation*}
of $\OC{G}{\bar\frakF}$-modules by setting $\varepsilon_{H} \defeq 
\varepsilon_{\tilde H}$ for every $H\in \bar\frakF$.

Assume that $\cd_{\frakF_1} G_1 \leq m$ and $\cd_{\frakF_2}
G_2\leq n$. Then there exists projective resolutions $P_*\to
\underline\Z_{\frakF_1}$ and $Q_*\to \underline\Z_{\frakF_2}$ of
length $m$ and $n$ respectively. By an Eilenberg swindle we may
assume the resolutions are free. Then by Lemma~\ref{lem:free-free}
the double chain complex~\eqref{eq:double-complex} of $\OFG$-modules
satisfies the assumptions of Lemma~\ref{lem:extension}. Hence we may
extend it to a double chain complex of
$\OC{G}{\bar\frakF}$-modules. Passing to the total complex we
obtain a sequence
\begin{equation*}
    \ldots \to C_3 \to C_2 \to C_1 \to C_0\stackrel{\varepsilon}{\to}
    \Z_{\bar\frakF}\to 0
\end{equation*}
of $\OC{G}{\bar\frakF}$-modules. By construction the exactness of
this sequence follows from the exactness of the
sequence~\eqref{eq:resolution} which was proven in
Proposition~\ref{prop:key-result}. This sequence has length $n +
m$. This proves the following result.

\begin{proposition}
    Let $\frakF_1$ and $\frakF_2$ be full families of subgroups. Then
    \begin{equation*}
	\cd_{\bar\frakF} G\leq \cd_{\frakF_1} G_1 + \cd_{\frakF_2}
	G_2.\tag*{\qed}
    \end{equation*}
\end{proposition}

\begin{theorem}
    \label{thrm:cd-for-direct-products}
    Let $\frakF_1$ and $\frakF_2$ be a full families of subgroups of
    $G_1$ and~$G_2$ respectively.  Let $\frakF \defeq  \frakF_{1}\times
    \frakF_{2}$.  Assume that $\frakG \subset \bar \frakF$ is a full 
    family of subgroups of $G$.
    If there exists $k\in \N$ such that $\cd_{\frakG\cap K} K \leq k$
    for every $K\in \frakF$, then
    \begin{equation*}
	\cd_\frakG G \leq \cd_{\frakF_1} G_1 + \cd_{\frakF_2} G_2
	+ k.
    \end{equation*}
\end{theorem}

\begin{proof}
    Let $K\in \bar\frakF$.  Then $K\leq \tilde K \in \frakF$ and
    $\frakG\cap K$ is a non-empty subset of $\frakG\cap \tilde K$.
    Thus $\cd_{\frakG\cap K} K \leq \cd_{\frakG\cap \tilde K} \tilde
    K$ which by the assumption of the theorem is less or equal to $k$.
    Then $\cd_{\frakG} G \leq \cd_{\bar \frakF} G + k$ by
    Theorem~\ref{thrm:cdF-G-vs-cdG-G}.  Now the statement follows from
    the previous proposition.
\end{proof}

\begin{corollary}
    \label{cor:cd-for-direct-products}
    Let $G \defeq  G_{1}\times G_{2}$. Then
    \begin{equation*}
        \ucd G \leq \ucd G_{1} + \ucd G_{2}
	\qquad
	\text{and} 
	\qquad
	\uucd G \leq \uucd G_{1} + \uucd G_{2} + 3.
    \end{equation*}
\end{corollary}
    
\begin{proof}
    The cartesian product $K_{1}\times K_{2}$ of two finite groups is 
    finite and therefore $\ucd (K_{1}\times K_{2})=0$. Thus $\ucd 
    G\leq \ucd G_{1} + \ucd G_{2}$ by 
    Theorem~\ref{thrm:cd-for-direct-products}.
    
    On the other hand, the cartesian product $K_{1}\times K_{2}$ of
    two virtually cyclic groups is a virtually polycyclic group with
    virtually cohomological dimension $\vcd (K_{1}\times K_{2}) \leq
    2$.  In~\cite{luck-12} it has been shown that this implies
    $\uugd(K_{1}\times K_{2})\leq 3$.  Then $\uucd G \leq \uucd G_{1} 
    + \uucd G_{2} + 3$ follows from
    Theorem~\ref{thrm:cd-for-direct-products}.
\end{proof}

Note that the inequality in
Corollary~\ref{cor:cd-for-direct-products} is sharp.  For example, 
take $G_{1}\defeq \Z$ and $G_{2}\defeq \Z$. Then $\ucd G_{i} = 1$ and $\uucd
G_{i} = 
0$. Moreover $\ucd G=2$ and in Proposition~\ref{prop:dimvcZ2=3} we 
will see that $\uucd G=3$. Therefore we have in this case
\begin{equation*}
    \ucd G = \ucd G_{1} + \ucd G_{2}
    \qquad
    \text{and} 
    \qquad
    \uucd G = \uucd G_{1} + \uucd G_{2} + 3.
\end{equation*}

\section{Crossproduct and Künneth Formula for Bredon Homology}

In this section we consider $G \defeq  G_{1}\times G_{2}$ with the notation
as in the section before.  Moreover, throughout this section $M$ and
$M'$ will be a left $\OC{G_{1}}{\frakF_{1}}$-module and
$\OC{G_{2}}{\frakF_{2}}$-module respectively.

\begin{lemma}
    \label{lem:1}
    The map
    \begin{equation}
	\label{eq:iso-rule}
	\begin{split}
	    [G/H, G/K]_{G} \times M(G_{1}/\tilde H_{1}) & \times
	    M'(G_{2}/\tilde H_{2})
	    \\
	    & \to [G_{1}/H_{1} , G_{1}/\tilde
	    K_{1}]_{G_{1}} \times M(G_{1}/H_{1})
	    \\ 
	    & \qquad \times [G_{2}/H_{2} ,
	    G_{2}/\tilde K_{2}]_{G_{2}} \times M'(G_{2}/H_{2})
	    \\
	    (g K, m, m') 
	    & \mapsto
	    (\tilde g_{1}\tilde K_{1}, m, \tilde 
	    g_{2} \tilde K_{2}, m'),
	\end{split}
	\raisetag{8pt}
    \end{equation}
    where $H, K \in \frakF$, defines for every $K\in \frakF$ an
    isomorphism of groups
    \begin{multline}
	\label{eq:iso-map}
        \theta\: \coprod_{H\in \frakF} 
	\Z[G/H, G/K]_{G} \otimes M(G_{1}/\tilde H_{1}) \otimes 
	M'(G_{2}/\tilde H_{2})
	\\
	\longto \Bigl(\coprod_{H_{1}\in \frakF_{1}} 
	\Z[G_{1}/H_{1}, G_{1}/\tilde K_{1}]_{G1} \otimes 
	M(G_{1}/H_{1})\Bigr) \; \otimes
	\\
	\Bigl(\coprod_{H_{2}\in \frakF_{2}} 
	\Z[G_{2}/H_{2} , G_{2}/\tilde K_{2}]_{G_{2}} \otimes
	M'(G_{2}/H_{2})\Bigr).
    \end{multline}
\end{lemma}

\begin{proof}
    First note that the right hand side of~\eqref{eq:iso-map} can be 
    rewritten to
    \begin{align*}
	\coprod_{H\in \frakF} \Bigl( \Z[G_{1}/\tilde
	H_{1}, G_{1}/\tilde K_{1}]_{G_{1}} & \otimes M(G_{1}/\tilde 
	H_{1}) \otimes \mbox{}
	\\
	& \hspace{-1.5em} 
	\Z[G_{2}/\tilde H_{2} , G_{2}/\tilde K_{2}]_{G_{2}} \otimes
	M'(G_{2}/\tilde H_{2})\Bigr).
    \end{align*}
    Thus it is enough to show that the rule~\eqref{eq:iso-rule} gives
    rise to an isomorphism for each fixed $H\in \frakF$.  In turn, for
    this it is enough to show that the obvious restriction of the
    map~\eqref{eq:iso-rule} gives rise to an isomorphism
    \begin{equation}
	\label{eq:iso-help1}
	\Z[G/H,G/K]_{G} \to \Z[G_{1}/H_{1}, G_{1}/
	\tilde{K}_{1}]_{G_{1}} \otimes \Z[G_{2}/H_{2},
	G_{2}/\tilde{K}_{2}]_{G_{2}}.
    \end{equation}
    Note that any $G$-map $f\:G/H\to G/K$ is in fact a $G$-map $f\:
    G_{1}/\tilde H_{1} \times G_{2}/\tilde H_{2} \to G_{1}/\tilde
    K_{1} \times G_{2}/\tilde K_{2}$ which is uniquely characterised
    by $G_{i}$-maps $\tilde f_{i}\: G_{i}/\tilde H_{i} \to
    G_{i}/\tilde K_{i}$ which map $\tilde H_{i} \mapsto \tilde
    g_{i}\tilde K_{i}$, $i=1,2$.  In other words, the map $f\mapsto
    (\tilde f_{1}, \tilde f_{2})$ gives an isomorphism
    \begin{equation}
	\label{eq:iso-help2}
        [G/H, G/K]_{G} \to [G_{1}/\tilde H_{1}, G_{1}/\tilde 
	K_{1}]_{G_{1}} 
	\times [G_{2}/\tilde H_{2}, G_{2}/\tilde K_{2}]_{G_{2}}
    \end{equation}
    of sets. In turn this gives rise to an isomorphism of groups as 
    in~\eqref{eq:iso-help1} and by construction this isomorphism 
    agrees with the isomorphism obtained by the obvious restriction 
    of the map~\eqref{eq:iso-rule}.
\end{proof}

\begin{lemma}
    \label{lem:2}
    Let $K\in \frakF$. Then the 
    map~\eqref{eq:iso-rule} defines (this definition is made precise 
    in the proof) an isomorphism
    \begin{multline}
	\label{eq:iso-2}
	\eta\: (\res_{p_{1}} \Z[\?,G_{1}/\tilde K_{1}]_{G_{1}} \otimes
	\res_{p_{2}} \Z[\?,G_{2}/\tilde K_{2}]_{G_{2}})
	\mathbin{\otimes_{\frakF}} (\res_{p_{1}} M \otimes
	\res_{p_{2}} M') \\
	\longto
        (\Z[\?, G_{1}/\tilde K_{1}]_{G_{1}} 
	\mathbin{\otimes_{\frakF_{1}}} M)
	\otimes 
        (\Z[\?, G_{2}/\tilde K_{2}]_{G_{2}} 
	\mathbin{\otimes_{\frakF_{2}}} M')
    \end{multline}
    of abelian groups. 
\end{lemma}

\begin{proof}
    For each $H\in \frakF$ we have
    \begin{multline*}
	(\res_{p_{1}} \Z[\?,G_{1}/\tilde K_{1}]_{G_{1}} \otimes \res_{p_{2}} 
	\Z[\?,G_{2}/\tilde K_{2}]_{G_{2}})(G/H)
	\\
	= \Z[G_{1}/\tilde H_{1},
	G_{1}/\tilde K_{1}]_{G_{1}} \otimes \Z[G_{2}/\tilde H_{2},
	G_{2}/\tilde K_{2}]_{G_{2}}
	= \Z[G/H, G/K]_{G}
    \end{multline*}
    using the identification given by~\eqref{eq:iso-help2}.
    
    Now the domain of the isomorphism~\eqref{eq:iso-2} is by 
    definition the abelian group $P/Q$ with
    \begin{multline*}
        P = \coprod_{H\in \frakF} 
	\bigl(
	(\res_{p_{1}} \Z[\?,G_{1}/\tilde K_{1}]_{G_{1}} \otimes \res_{p_{2}} 
	\Z[\?,G_{2}/\tilde K_{2}]_{G_{2}})(G/H)
	\\
	\otimes (\res_{p_{1}} M \otimes \res_{p_{2}} M')(G/H)\bigr)
	\\[1.5ex]
	= 
	\coprod_{H\in \frakF} 
	\Z[G/H, G/K]_{G} \otimes M(G_{1}/\tilde H_{1}) \otimes 
	M'(G_{2}/\tilde H_{2})
    \end{multline*}
    where $Q$ is the subgroup of $P$ generated by all elements of the 
    form
    \begin{equation}
	\label{eq:generators-of-Q}
        \varphi^{*}(f) \otimes (m\otimes m') - f \otimes 
	\varphi_{*}(m\otimes m')
    \end{equation}
    where $f\: G/H\to G/K$ is a $G$-map (that is a generator of the
    abelian group $\Z[G/H, G/K]_{G}$), $m\otimes m'\in M(G_{1}/\tilde
    L_{1}) \otimes M'(G_{2}/\tilde L_{2})$, $\varphi\in [G/L,
    G/H]_{G}$ and $H,L\in \frakF$.
    
    Likewise $\Z[\?, G_{1}/\tilde K_{1}]_{G_{1}}
    \mathbin{\otimes_{\frakF_{1}}} M$ is the abelian group $R/S$ with
    \begin{equation*}
        R\defeq  \coprod_{H_{1}\in \frakF_{1}} \Z[G_{1}/H_{1}, 
	G_{1}/\tilde K_{1}]_{G_{1}}
	\otimes M(G_{1}/H_{1})
    \end{equation*}
    where $S$ is the subgroup of $R$ generated by all elements of the 
    form
    \begin{equation}
        \label{eq:generators-of-S}
	\varphi_{1}^{*}(f_{1}) \otimes m - f_{1}\otimes 
	\varphi_{1*}(m)
    \end{equation}
    where $f_{1}\: G_{1}/H_{1}\to G_{1}/\tilde K_{1}$ is a $G_{1}$-map
    (that is a generator of the abelian group $\Z[G_{1}/H_{1},
    G_{1}/\tilde K_{1}]_{G_{1}}$), $m\in M(G_{1}/L_{1})$,
    $\varphi_{1}\in [G_{1}/L_{1}, G_{1}/H_{1}]_{G_{1}}$ and $H_{1},
    L_{1}\in \frakF_{1}$.  And in the very same fashion we express the
    abelian group $\Z[\?, G_{2}/\tilde K_{2}]_{G_{2}}
    \mathbin{\otimes_{\frakF_{2}}} M'$ as the quotient $R'/S'$.
    
    Since the tensor product is right exact the projection
    homomorphism $\pi_{1}\: R \to R/S$ and $\pi_{2}\: R' \to R'/S'$
    give an epimorphism $\pi_{1}\otimes \pi_{2}\: R \otimes R' \to R/S
    \otimes R'/S'$.  If we precompose this epimorphism with the
    isomorphism $\theta\: P\to R\otimes R'$ from Lemma~\ref{lem:1} we
    get an epimorphism
    \begin{equation*}
        P\to R/S \otimes R'/S'
    \end{equation*}
    The first claim of Lemma~\ref{lem:2} is now that this epimorphism 
    factors through $P/Q$, that is there exists a homomorphism $\eta$ 
    making the diagram
    \begin{equation*}
        \begin{diagram}
            \node{P}
	    \arrow{s,l}{\pi}
	    \arrow{e,t}{\theta}
	    \node{R\otimes R'}
	    \arrow{s,r}{\pi_{1} \otimes \pi_{2}}
	    \\
	    \node{P/Q}
	    \arrow{e,t}{\eta}
	    \node{R/S \otimes R'/S'}
        \end{diagram}
    \end{equation*}
    commute.  Since $\pi$ and $(\pi_{1}\otimes \pi_{2}) \circ \theta$
    are both epimorphisms it follows that this $\eta$ is necessarily
    unique and also an epimorphism.  In order to see that the
    epimorphism~$\eta$ exists, we must show that $Q \subset \ker
    ((\pi_{1}\otimes \pi_{2}) \circ \theta)$.  Since~$\theta$ is an
    isomorphism this is equivalent to $\theta (Q) \subset
    \ker(\pi_{1}\otimes \pi_{2})$.  By Proposition~6, 
    \cite[p.~252]{bourbaki-98} the kernel of $\pi_{1}\otimes\pi_{2}$ 
    has the following simple description
    \begin{equation*}
	\ker(\pi_{1}\otimes\pi_{2}) = \langle s\otimes r' , r\otimes
	s' \mathbin{:} r \in R, r'\in R', s\in S, s'\in S'\rangle.
    \end{equation*}
    Thus let $x$ be a generator of $Q$ as
    in~\eqref{eq:generators-of-Q}, then
    \begin{align*}
       \theta(x) 
       & =
       \theta( \varphi^{*}(f) \otimes (m\otimes m')) - \theta(f
       \otimes \varphi_{*}(m\otimes m'))
       \displaybreak[3]\\
       & =
       \theta( (f \circ \varphi) \otimes (m\otimes m')) - \theta(f
       \otimes \varphi_{*}(m) \otimes \varphi_{*}(m')))
       \displaybreak[3] \\
       & =
       \bigl((\tilde f_{1} \circ \tilde \varphi_{1}) \otimes m \otimes
       (\tilde f_{2} \circ \tilde \varphi_{2}) \otimes m'\bigr)
       - 
       \bigl( \tilde f_{1} \otimes \tilde 
       \varphi_{1*}(m) \otimes \tilde f_{2} \otimes \tilde 
       \varphi_{2*}(m')\bigr)       
       \displaybreak[3]\\
       & =
       \bigl((\tilde f_{1} \circ \tilde \varphi_{1}) \otimes m \otimes
       (\tilde f_{2} \circ \tilde \varphi_{2}) \otimes m'\bigr)
       - 
       \bigl( \tilde f_{1} \otimes \tilde \varphi_{1*}(m) \otimes
       (\tilde f_{2} \circ \tilde \varphi_{2}) \otimes m'\bigr)
       \\
       & 
       \qquad + 
       \bigl( \tilde f_{1} \otimes \tilde \varphi_{1*}(m) \otimes
       (\tilde f_{2} \circ \tilde \varphi_{2}) \otimes m'\bigr)
       -
       \bigl( \tilde f_{1} \otimes \tilde 
       \varphi_{1*}(m) \otimes \tilde f_{2} \otimes \tilde 
       \varphi_{2*}(m')\bigr)
       \displaybreak[3] \\
       & =
       \bigl(((\tilde f_{1} \circ \tilde \varphi_{1}) \otimes m - 
       \tilde f_{1} \otimes \tilde \varphi_{1*}(m)) \otimes
       (\tilde f_{2} \circ \tilde \varphi_{2}) \otimes m'\bigr)
       \\
       & 
       \qquad + 
       \bigl( \tilde f_{1} \otimes \tilde \varphi_{1*}(m) \otimes
       ((\tilde f_{2} \circ \tilde \varphi_{2}) \otimes m' -  \tilde 
       f_{2} \otimes \tilde \varphi_{2*}(m'))\bigr)
       \\
       & =
       \Bigl( (
       \underbrace{\tilde \varphi_{1}^{*}(\tilde f_{1}) \otimes m - 
       \tilde f_{1} \otimes \tilde \varphi_{1*}(m)}_{\in S} ) 
       \;
       \otimes 
       \;
       (\underbrace{
       (\tilde f_{2} \circ \tilde \varphi_{2}) \otimes m'}_{\in
       R'})\Bigr)
       \\
       & 
       \qquad + 
       \Bigl( ( \underbrace{\tilde f_{1} \otimes \tilde 
       \varphi_{1*}(m)}_{\in R}) \; \otimes \; (
       \underbrace{\tilde \varphi_{2}^{*}(\tilde f_{2}) \otimes m' -
       \tilde 
       f_{2} \otimes \tilde \varphi_{2*}(m')}_{\in S'})
       \Bigr)
       \\
       &
       \in \ker(\pi_{1}\otimes \pi_{2}).
    \end{align*}
    This concludes the first part of the claim of the lemma.
    
    It remains to show that the epimorphism $\eta$ is actually an
    isomorphism.  For this we need to show that $\theta$ maps $Q$ 
    epimorphically onto $\ker(\pi_{1}\otimes \pi_{2})$.
    
    Let $s$ be a generator of $S$ as in~\eqref{eq:generators-of-S}.
    Let $f_{2}\: G_{2}/H_{2}\to G_{1}/\tilde K_{2}$ be an arbitrary
    $G_{2}$-map and $m'\in M'(G_{2}/H_{2})$, that is $r \defeq  f_{2}
    \otimes m'$ is a generator of $R'$.  Let $\varphi\: G_{1}/L_{1}
    \times G_{2}/H_{2} \to G_{1}/H_{1} \times G_{2}/H_{2}$ be the
    $G$-map given by $\varphi(L_{1} \times H_{2}) \defeq 
    \varphi_{1}(H_{1}) \times H_{2}$.  Then $\tilde \varphi_{1} =
    \varphi_{1}$ and $\tilde\varphi_{2} = \id$.  Set $L\defeq  L_{1}\times
    H_{2}$ and $H \defeq  H_{1}\times H_{2}$.  Furthermore set $f \defeq 
    f_{1}\times f_{2}$.  Then $\tilde f_{1} = f_{1}$ and $\tilde f_{2}
    = f_{2}$ by construction.  With these definitions it follows that
    $x \defeq  \varphi^{*}(f) \otimes (m\otimes m') - f\otimes
    \varphi_{*}(m \otimes m')$ is a generator of $Q$ and
    \begin{align*}
        \theta(x)
	& =
	\theta(\varphi^{*}(f) \otimes (m\otimes m')) 
	\,
	- 
	\,
	\theta(f\otimes
	\varphi_{*}(m \otimes m'))
	\\
	& =
	\bigl(\tilde \varphi_{1}^{*}(\tilde f_{1}) \otimes m \otimes 
	\tilde \varphi_{2}^{*}(\tilde f_{2}) \otimes m' \bigr)
	\,
	-
	\,
	\bigl(\tilde f_{1} \otimes \tilde \varphi_{1*}(m) \otimes \tilde 
	f_{2}\otimes \tilde \varphi_{2*}(m') \bigr)
	\\
	& = 
	\bigl(
	\varphi_{1}^{*}(f_{1}) \otimes m \otimes f_{2} \otimes m' 
	\bigr)
	\,
	-
	\,
	\bigl(
	f_{1}\otimes \varphi_{1}(m) \otimes f_{2} \otimes m'
	\bigr)
	\\
	& =
	\bigl(
	\varphi_{1}(f_{1}) \otimes m - f_{1} \otimes \varphi_{1}(m) 
	\bigr)
	\,
	\otimes 
	\,
	\bigl(
	f_{2}\otimes m'
	\bigr)
	\\
	& = s \otimes r'.
    \end{align*}
    Similarly to~\eqref{eq:generators-of-S} we can show that any
    element $r\otimes s'$ with $r$ a generator of~$R$ and $s'$ a
    generator of $S'$ is contained in $\theta(Q)$.  Thus
    $\ker(\pi_{1}\otimes \pi_{2}) \subset \theta (Q)$ and equality
    holds.
\end{proof}

\begin{lemma}
    \label{lem:3}
    The isomorphism $\eta$ in Lemma~\ref{lem:2} is natural in each of
    its factors.  That is, given morphisms $f_{1}\: \Z[\?,
    G_{1}/\tilde K_{1}]_{G_{1}} \to \Z[\?,G_{1}/\tilde
    L_{1}]_{G_{1}}$, $f_{2}\: \Z[\?, G_{2}/\tilde K_{2}]_{G_{2}}\to
    \Z[\?,G_{2}/\tilde L_{2}]_{G_{2}}$, $f_{3}\: M\to N$, $f_{4}\: M'
    \to N'$ we have
    \begin{equation*}
	\eta \circ \bigl(f_{1}, f_{2}, f_{3}, f_{4}\bigr) =
	\bigl((f_{1}\otimes f_{2}) \mathbin{\otimes_{\frakF}}
	(f_{3}\otimes f_{4})\bigr) \circ \eta.
    \end{equation*}
\end{lemma}

\begin{proof}
    This follows immediately from the simple form of the 
    map~\eqref{eq:iso-rule}.
\end{proof}

\begin{proposition}
    \label{prop:1a}
    Let $F = \Z[\?, X_{1}]_{G_1}$ and $F' = \Z[\?, X_{2}]_{G_2}$ be a
    free $\OC{G_{1}}{\frakF_{1}}$-module and
    $\OC{G_{2}}{\frakF_{2}}$-module.  Then the isomorphism $\eta$
    defined in Lemma~\ref{lem:2} induces in a canonical way an
    isomorphism
    \begin{multline*}
        \eta \: (\res_{p_{1}} F \otimes \res_{p_{2}} F') 
	\mathbin{\otimes_{\frakF}} (\res_{p_{1}} M \otimes 
	\res_{p_{2}}M')\\
	\longto (F \mathbin{\otimes_{\frakF_{1}}} M) \otimes 
	(F' \mathbin{\otimes_{\frakF_{2}}} M').
    \end{multline*}
\end{proposition}

\begin{proof}
    \allowdisplaybreaks The main work has already been done in
    Lemma~\ref{lem:2} and the remaining claim follows from the
    additivity of the setup.  For completeness, the purely technical
    details are as follows.  Write
    \begin{equation*}
        X_{1} = \coprod_{\alpha} G_{1}/H_{1,\alpha}
	\qquad \text{and} \qquad
        X_{2} = \coprod_{\beta} G_{2}/H_{2,\beta}
    \end{equation*}
    as the disjoint union of transitive $G_{1}$-sets and $G_{2}$-sets
    with stabilisers in~$\frakF_{1}$ and $\frakF_{2}$ respectively. 
    Then $F = \coprod F_{\alpha}$ with $F_{\alpha} \defeq  
    \Z[\?,G_{1}/H_{1,\alpha}]_{G_{1}}$ and similarly $F' = 
    \coprod F'_{\beta}$ with $F'_{\beta} \defeq  
    \Z[\?,G_{2}/H_{2,\beta}]_{G_{2}}$. We get
    \begin{multline*}
	(\res_{p_{1}} F \otimes \res_{p_{2}} F')
	\mathbin{\otimes_{\frakF}} (\res_{p_{1}} M \otimes 
	\res_{p_{2}}M') 
	\\[1ex]
	 = \coprod_{\alpha,\beta} \bigl((\res_{p_{1}} F_{\alpha}
	\otimes \res_{p_{2}} F'_{\beta}) \mathbin{\otimes_{\frakF}}
	(\res_{p_{1}} M \otimes \res_{p_{2}}M') \bigr)
    \end{multline*}
    which is mapped by $\eta \defeq 
    \coprod \eta_{\alpha,\beta}$ summand wise
    isomorphically onto
    \begin{equation*}
	\coprod_{\alpha,\beta} \bigl((F_{\alpha}
	\mathbin{\otimes_{\frakF_{1}}} M) \otimes (F'_{\beta}
	\mathbin{\otimes_{\frakF_{2}}} M')\bigr)
	= (F \mathbin{\otimes_{\frakF_{1}}} M) \otimes 
	(F' \mathbin{\otimes_{\frakF_{2}}} M').\tag*{\qedhere}
    \end{equation*}
\end{proof}

Now let $F_{*}\to \underline\Z_{\frakF_{1}}$ be a free resolution of
the trivial $\OC{G_{1}}{\frakF_{1}}$-module and let $F'_{*}\to
\underline\Z_{\frakF_{2}}$ be a free resolution of the trivial
$\OC{G_{2}}{\frakF_{2}}$-module.  From the previous section we know
that the
total complex of $\res_{p_{1}} F_{*}\otimes \res_{p_{2}} F'_{*}$ gives
a free resolution of the trivial $\OFG$-module
$\underline\Z_{\frakF}$.  Tensoring this total complex over the orbit
category $\OFG$ with $\res_{p_{1}} M\otimes \res_{p_{2}}M'$ gives a 
chain complex whose objects are given by
\begin{equation*}
    C_{n} \defeq  \coprod_{k=0}^{n} 
    \bigl(
    (\res_{p_{1}} F_{k}\otimes \res_{p_{2}} F'_{n-k})
    \,\mathbin{\otimes_{\frakF}}\,
    (\res_{p_{1}} M\otimes \res_{p_{2}}M') \bigr)
\end{equation*}
The isomorphism from Proposition~\ref{prop:1a} maps these groups 
isomorphically onto
\begin{equation*}
    C'_{n} \defeq  \coprod_{k=0}^{n} 
    \bigl( 
    (F_{k} \mathbin{\otimes_{\frakF_{1}}} M)
    \, \otimes \,
    (F'_{n-k} \mathbin{\otimes_{\frakF_{2}}} M')
    \bigr).
\end{equation*}
We denote these isomorphisms by $\eta_{n}\: C_{n}\to C'_{n}$, $n\in 
\N$. Using Lemma~\ref{lem:3} one can conclude that this collection of 
isomorphisms defines a chain map
\begin{equation*}
    \eta\: C_{*} \to C'_{*}.
\end{equation*}
Now $C'_{*}$ is nothing else than the chain complex obtained as a
total complex of $(F_{*} \mathbin{\otimes_{\frakF_{1}}} M) \otimes
(F'_{*} \mathbin{\otimes_{\frakF_{2}}} M')$. 
Following~\cite[pp.~108f.]{brown-82}, if $z$ is a $p$-cycle of 
$F_{*}\otimes_{\frakF_{1}} M$ and $z'$ a $q$-cycle of 
$F'_{*} \mathbin{\otimes_{\frakF_{2}}} M'$, then 
\begin{equation*}
    z\times z' \defeq  \eta^{-1}(z \otimes z')
\end{equation*}
is a $(p+q)$-cycle of $C_*$. Its homology class depends only on the 
homology class of $z$ and $z'$. Thus we obtain a \emph{homology cross 
product}
\begin{equation*}
    \times \,\: H^{\frakF_{1}}_{p}(G_{1}; M) \otimes 
    H^{\frakF_{2}}_{q}(G_{2}; M') \to H_{p+q}^{\frakF_{1}\times 
    \frakF_{2}}(G_{1}\times G_{2}; 
    \res_{p_{1}} M \otimes \res_{p_{2}} M')
\end{equation*}
which maps $[z]\otimes [z'] \mapsto [z] \times [z'] \defeq [z\times
z']$ as in the classical case.

\begin{theorem}[Künneth Formula for Bredon Homology]
    \label{thrm:kunneth-formula-for-homology}
    Assume that there exists free resolutions $F_{*}\to
    \underline\Z_{\frakF_{1}}$ and $F'_{*}\to
    \underline\Z_{\frakF_{2}}$ such that the chain complex $F_{*}
    \mathbin{\otimes_{\frakF_{1}}} M$ or the chain complex $F'_{*}
    \mathbin{\otimes_{\frakF_{2}}} M'$ is a free chain complex.  Then
    for every $n\in \N$ there exists a short exact sequence
    \begin{multline*}
        0 
	\to 
	\coprod_{k=0}^{n} H^{\frakF_{1}}_{k}(G_{1}; M) \otimes 
	H^{\frakF_{2}}_{n-k}(G_{2}; M') 	
	\\
	\stackrel{\alpha}{\longto} H^{\frakF_{1}\times
	\frakF_{2}}_{n}(G_{1}\times G_{2}; \res_{p_{1}} M\otimes
	\res_{p_{2}} M') \\
	\longto 
	\coprod_{k=0}^{n-1}
	\Tor_{1}(H^{\frakF_{1}}_{k}(G_{1}; M), 
	H^{\frakF_{2}}_{n-k-1}(G_{2}; M') 
	\to 0
    \end{multline*}
    of abelian groups.  The homomorphism $\alpha$ is given by the
    homological cross product, that is, $\alpha(z \otimes z') \defeq 
    z\times z'$.  If both chain complexes $F_{*}
    \mathbin{\otimes_{\frakF_{1}}} M$ and~$F'_{*}
    \mathbin{\otimes_{\frakF_{2}}} M'$ are dimension wise free, then
    this sequence is split, but this splitting is not natural.
\end{theorem}

\begin{proof}
    This is essentially the Künneth Formula for chain complexes of
    abelian groups applied to chain complexes $F_{*}
    \mathbin{\otimes_{\frakF_{1}}} M$ and  $F'_{*}
    \mathbin{\otimes_{\frakF_{2}}} M'$ (see for 
    example~\cite[pp.~227ff.]{spanier-66}) together with 
    Proposition~\ref{prop:1a}.
\end{proof}

%
%

\chapter{Bredon Dimensions for the Family $\Fvc$}
\label{ch:dimensions-for-Fvc}

%
%

\section{A Geometric Lower Bound for $\hd_{\frakF}G$ and 
$\cd_{\frakF} G$}

Assume that $\frakF$ is a semi-full family of subgroups of $G$.  Then
the
geometric dimension $\gd_{\frakF} G$ is defined and gives an upper
bound for the dimension $\cd_{\frakF} G$, and therefore also for
$\hd_{\frakF} G$.  In this section, we will prove a result that can be
used to obtain a lower bound for the cohomological and homological
Bredon dimension of $G$ using geometrical methods.

\begin{DEF}
    Let $\frakF$ be a semi-full family of subgroups of $G$.  We say
    that~$Y$ is a model for $B_{\frakF}G$ if there exists a model $X$
    for $E_{\frakF} G$ such that~$Y = X/G$, that is, $Y$ is the orbit
    space of some classifying space for the family $\frakF$.
\end{DEF}

The main result of this section will be the following theorem. It is 
essentially the generalisation of the classical fact that
\begin{equation*}
    H_{n}(G; \Z) \isom H_{n}(Y)
    \qquad \text{and} \qquad 
    H^{n}(G; \Z) \isom H^{n}(Y)
\end{equation*}
where $Y$ is an Eilenberg--Mac~Lane space $K(G,1)$,
see for example~\cite[pp.~36ff.]{brown-82}.

\begin{theorem}
    \label{thrm:geometric-lower-bound}
    Let $G$ be a group and let $\frakF$ be a semi-full family of
subgroups 
    of~$G$.  Then for every $n\in \N$ we have isomorphisms
    \begin{equation*}
	H_{n}^{\frakF} (G; \underline\Z_{\frakF}) \isom
	H_{n}(B_{\frakF} G) \qquad \text{and} \qquad H^{n}_{\frakF}
	(G; \underline\Z_{\frakF}) \isom H^{n}(B_{\frakF} G)
    \end{equation*}
    of abelian groups.
\end{theorem}

Now this result together with Proposition~\ref{prop:pd-via-ext} and 
Proposition~\ref{prop:fld-via-tor} implies the following 
immediate result.

\begin{corollary}
    If $H_{n}(B_{\frakF}G)\neq 0$ then $\hd_{\frakF}G\geq n$.
    Likewise $H^{m}(B_{\frakF}G)\neq 0$ implies $\cd_{\frakF}G\geq 
    m$.\qed
\end{corollary}

Before we can prove this theorem we need the following auxiliary
result.  This result is the main reason for the
Theorem~\ref{thrm:geometric-lower-bound} to be true.

\begin{proposition}
    \label{prop:1b}
    Let $\frakF$ be a semi-full family of subgroups of $G$.  Then we
    have an isomorphism $\underline{C}_{*}(E_{\frakF}G)
    \mathbin{\otimes_{\frakF}} \underline\Z_{\frakF} \isom
    C_{*}(B_{\frakF}G)$ of chain complexes.
\end{proposition}

\begin{proof}
    To simplify the notation, denote by $X$ a model for $E_{\frakF}G$
    and by~$\pi\: X\to X/G$ the canonical projection onto the orbit
    space.  For each $H\in \frakF$ this map restricts to the map
    $\pi\: X^{H} \to X/G$ which in turn induces a chain map
    \begin{equation*}
        \pi_{H}\: C_{*}(X^{H}) \to C_{*}(X/G).
    \end{equation*}
    Then the coproduct of these chain 
    maps is a chain map
    \begin{equation*}
        P_{*} \defeq  \coprod_{H\in\frakF} C_{*}(X^{H}) \to C_{*}(X/G)
    \end{equation*}
    which we will denote  (with abuse of notation) by $\pi$.
    
    Let $\tau$ be an $n$-cell of the CW-complex $X/G$.  Then there
    exists an $n$-cell~$\sigma$ of the $G$-CW-complex $X$ such that
    $\pi(\sigma) = \tau$.  Let $H$ be the isotropy group~$G_\sigma$ of
    $\sigma$.  Then $H\in \frakF$ and $\tau$ lies in the image of
    $\pi_{H}$.  Therefore $\pi$ is surjective.
    
    Denote by $Q_{*}$ the kernel of the map $\pi$. We claim that
    \begin{equation}
	\label{eq:claim-1}
        \underline{C}_{*}(X) \mathbin{\otimes_{\frakF}} 
	\underline\Z_{\frakF} = P_{*}/Q_{*}.
    \end{equation}
    
    Recall that by definition
    \begin{equation*}
	\underline{C}_{*}(X) \mathbin{\otimes_{\frakF}}
	\underline\Z_{\frakF} = P'_{*}/Q'_{*}
    \end{equation*}
    where
    \begin{equation*}
        P'_{*} = \coprod_{H\in \frakF} C_{*}(X^{H}) \otimes \Z = 
        \coprod_{H\in \frakF} C_{*}(X^{H}) = P_{*}
    \end{equation*}
    and $Q'_{*}$ is the subcomplex generated by the elements of the
    form $\varphi^{*}(\sigma) - \sigma$.  Note that in the case of the
    trivial $\OFG$-module $\underline\Z_{\frakF}$ we have $\varphi_{*}
    = \id$. Now~\eqref{eq:claim-1} follows from the following claim.
    
    \setcounter{claim}{0}
    
    \begin{claim}
	\label{claim:prop:1b}
	$Q'_{n} = Q_{n}$ for all $n\in \N$.
    \end{claim}
    
    \noindent ``$Q'_{n} \subset Q_{n}$'':\quad This inclusion follows
    immediately from
    \begin{equation*}
	\pi(\varphi^{*}(\sigma) -\sigma) = \pi(g\sigma) -\pi(\sigma) =
	0.
    \end{equation*}
    
    \medskip
    
    \noindent ``$Q'_{n}\supset Q_{n}$'':\quad Let $x\in Q_{n}$. Then 
    there exists pairwise distinct orbits $A_{1}$, \ldots, $A_{n}$ of 
    orbits of $n$-cells of $X$ such that we can write
    \begin{equation*}
        x = x_{1} + \ldots + x_{r}
    \end{equation*}
    with each $x_{s}$, $1\leq s\leq r$ satisfying the following: there
    exists $k_{s}\geq 1$ and for~$1\leq i\leq k_{s}$ there exist
    unique $H_{s,i}\in \frakF$, $\sigma_{s,i}\in
    X^{H_{s,i}}$ and $a_{s,i}\in \Z\setminus \{0\}$ such that
    $\sigma_{s,i} \in A_{s}$ and
    \begin{equation*}
        x_{s} = \sum_{i=1}^{k_{s}} a_{s,i} \, \sigma_{s,i}.
    \end{equation*}

    Now $\pi(x) = 0$ if and only if $\pi(x_{s}) = 0$ for each $1\leq 
    s\leq r$. In particular each $x_{s}\in Q_{n}$ and it is enough to 
    show that each $x_{s}\in Q'_{n}$.
    
    Thus assume that $r=1$. We obmit the variable $s$ from the 
    notation, that is
    \begin{equation}
	\label{eq:x}
        x = \sum_{i=1}^{k} a_{i} \sigma_{i}
    \end{equation}
    where the $\sigma_{i}$ are $n$-cell of $X^{H_{i}}$ and the 
    $a_{i}$ are all non-zero integers and the $\sigma_{i}$ belong all 
    to the same orbit.
    
    We can find elements $g_{i}\in G$ such that $\sigma_{i+1} =
    g_{i}\sigma_{i}$ for $i=1,\ldots, k-1$.  For each $i$ set $K_{i}
    \defeq H_{i} \cap g_{i}H_{i}g_{i}^{-1}$.  Since $\frakF$ is
    assumed to be closed under conjugation and taking finite
    intersections, it follows that the $K_{i}$ are all elements of
    $\frakF$.  By construction we have that $g_{i}^{-1}K_{i} g_{i} =
    g_{i}^{-1}H_{i}g_{i} \cap H_{i} \leq H_{i}$ and thus there exists
    a $G$-map $\varphi_{i}\: G/K_{i} \to G/H_{i}$ that maps $K_{i}$ to
    $g_{i}H_{i}$.  Then $\varphi_{i}^{*}(\sigma_{i}) = g_{i}\sigma_{i}
    = \sigma_{i+1}$.
    
    Using this equality together with $s_{i} \defeq  a_{1} 
    + \ldots + a_{i}$, we can successively rewrite the
    right hand side of \eqref{eq:x} in the following way:
    \begin{align*}
        x & = a_{1}(\sigma_{1} - \sigma_{2}) + (a_{1} + 
	a_{2})\sigma_{2} + a_{3}\sigma_{3} + \ldots + 
	a_{k}\sigma_{k}\\
	& = s_{1}(\sigma_{1}-\varphi^{*}_{1}(\sigma_{1})) + s_{2} 
	(\sigma_{2} -\sigma_{3}) + (s_{2}+a_{3})\sigma_{3} + 
	a_{4}\sigma_{4} + \ldots + a_{k}\sigma_{k}
	\\
	& = s_{1}(\sigma_{1}-\varphi^{*}_{1}(\sigma_{1})) + s_{2}
	(\sigma_{2} -\varphi^{*}_{2}(\sigma_{2})) + \ldots +\\
	&\hspace{4.5cm} + 
	s_{k-1}(\sigma_{k-1} - \varphi^{*}_{k-1}(\sigma_{k-1})) + s_{k}
	\sigma_{k}.
    \end{align*}
    On the other hand, by assumption we have $\pi(\sigma_{1}) = \ldots
    = \pi(\sigma_{k})$ and thus~$\pi(x) = a_{1} \pi(\sigma_{1}) +
    \ldots + a_{k}\pi(\sigma_{k}) = s_{k} \pi(\sigma_{1})$ and this is
    equal to $0$ if and only if $s_{k}=0$.  Therefore
    \begin{equation*}
        x = \sum_{i=1}^{k-1} s_{i} 
	(\sigma_{i}-\varphi^{*}_{i}(\sigma_{i})) \in Q'_{n}.
    \end{equation*}
    
    \medskip
    
    Alltogether $Q'_{n} = Q_{n}$ for every $n\in \N$ and this 
    proves Claim~\ref{claim:prop:1b}. On the other hand, 
    Claim~\ref{claim:prop:1b} implies the claim of the proposition 
    and this concludes the proof.
\end{proof}

\begin{proof}[Proof of Theorem~\ref{thrm:geometric-lower-bound}]
    First of all, we have
    \begin{align*}
        H_{n}(B_{\frakF} G) & \isom H_{n}(C_{*}(B_{\frakF}G))
	\\
	& \isom H_{n}(\underline{C}_{*}(E_{\frakF}G)
	\mathbin{\otimes_{\frakF}} \underline\Z_{\frakF}) && \text{(by
	Proposition~\ref{prop:1b})}
	\\
	& \isom H_{n}^{\frakF} (G;\underline\Z_{\frakF})
    \end{align*}
    and this proves the first isomorphism.
    
    In order to verify the second isomorphism, we first observe that 
    we have
    \begin{align*}
        \hom(C_{*}(B_{\frakF}G),\Z)& \isom 
	\mor_{\frakF}(C_{*}(B_{\frakF}G),\underline\Z_{\frakF})
	\\
	& \isom \mor_{\frakF}(\underline{C}_{*}(E_{\frakF}G)
	\mathbin{\otimes_{\frakF}} \underline\Z_{\frakF}, 
	\underline\Z_{\frakF}) && \text{(by 
	Proposition~\ref{prop:1b})}
	\\
	& \isom \mor_{\frakF} (\underline{C}_{*}(E_{\frakF}G), 
	\mor_{\frakF}(\underline\Z_{\frakF}, \underline\Z_{\frakF})) 
	&& \text{(adjoint isomorphism)}
	\\
	& \isom \mor_{\frakF} (\underline{C}_{*}(E_{\frakF}G),
	\underline\Z_{\frakF}).
    \end{align*}
    From this it follows 
    \begin{align*}
        H^{n}(B_{\frakF}G) & \isom H_{n} (\hom(C_{*}(B_{\frakF}G),\Z))
	\\
	& = H_{n} (\mor_{\frakF} (\underline{C}_{*}(E_{\frakF}G), 
	\underline\Z_{\frakF})\\
	& \isom H^{n}_{\frakF}(G;\underline\Z_{\frakF})
    \end{align*}
    which is the second isomorphism.
\end{proof}

%
%

\section{Applications of Theorem~\ref{thrm:geometric-lower-bound}}

In order to apply the result of 
Theorem~\ref{thrm:geometric-lower-bound} we need groups $G$ for which 
we know nice enough models $X$ for $\uu EG$ so that we can determine 
the homology or cohomology groups of $X/G$. In this section we 
present examples where this is the case.
Juan-Pineda and Leary have described in~\cite[p.~138]{juan-pineda-06}
a model for~$\uu E\Z^{2}$.  The construction goes back to Farrell and it
is as follows.

Label the maximal infinite cyclic subgroups of $\Z^{2}$ by $H_{i}$,
$i\in \Z$.  We have $\Z^{2}/H_{i}\isom \Z$ and there exists a
$1$-dimensional model $X_{i}$ for $E(\Z^{2}/H_{i})$.  It is a line on
which $\Z^{2}/H_{i}$ acts by translation.  Let $\pi_{i}\: \Z^{2} \to
\Z^{2}/H_{i}$ be the canonical projection.  Then $\Z^{2}$ acts on
$X_{i}$ by $gx \defeq  \pi_{i}(g)x$. 

For each $i\in \Z$ we consider the join $X_{i}*X_{i+1}$. There exists 
canonical embeddings $\varphi_{i}\: X_{i}\hookrightarrow 
X_{i}*X_{i+1}$ and $\psi_{i}\: X_{i}\hookrightarrow X_{i-1}* X_{i}$.
We let $X$ be the~$\Z^{2}$-space
\begin{equation*}
    X \defeq  \Bigl(\coprod_{i\in \Z} X_{i} * X_{i+1}\Bigr) \Big/ \sim
\end{equation*}
where the equivalence relation ``$\sim$'' is given by
\begin{equation*}
    y_{1} \sim y_{2} \,:\!\!\iff \exists x\in X_{i}\: \varphi_{i}(x)
= y_{1} 
    \text{ and } \psi_{i}(x) = y_{2}.
\end{equation*}
Note that the embeddings $\varphi_{i}\: X_{i}\hookrightarrow
X_{i}*X_{i+1}$ induce embeddings $\varphi_{i}\: X_{i}\hookrightarrow
X$ and we use them to identify the $X_{i}$ as subspaces of $X$. See 
Figure~\ref{fig:model-for-ZxZ} for a schematic picture of the space 
$X$.

\begin{figure}[t]
    \centering
    \begin{tikzpicture}[scale=\ratio{\textwidth}{7cm},
	line width=\ratio{7cm}{\textwidth}*1pt,
	my ultra thin/.style={line 
	width=\ratio{7cm}{\textwidth}*0.1pt},
	my thick/.style={line 
	width=\ratio{7cm}{\textwidth}*1.2pt}]

        \clip  (-0.5,-1.5) rectangle (6.5,1.65);
        
        \draw[loosely dashed, lightgray] (-3.5,0) -- (6.5,0);
            
            \foreach \i in {-2,-1,0,1,2} {
            
               \coordinate (A1) at (3*\i,-1,0);
               \coordinate (O1) at (3*\i,0,0);
               \coordinate (B1) at (3*\i,1,0);
               \coordinate (A2) at (3*\i+1.5,0,-1);
               \coordinate (O2) at (3*\i+1.5,0,0);
               \coordinate (B2) at (3*\i+1.5,0,1);
               \coordinate (A3) at (3*\i+3,-1,0);
               \coordinate (B3) at (3*\i+3,1,0);
               
               \draw[my thick]
	       (A1) -- (B1) 
	       (A2) -- (B2);
	       
               \filldraw[gray]
	       (O1) circle(0.7pt) (O2) 
	       circle(\ratio{10cm}{\textwidth}*0.7pt);
               \draw 
	       (O1) circle(0.7pt) (O2) 
	       circle(\ratio{10cm}{\textwidth}*0.7pt);
               
	       \draw[lightgray,densely dotted] 
	       (A1) -- (A2)
	       (A2) -- (A3);
           
	       \draw[my ultra thin]
	       (A1) -- (B2)
               (B1) -- (A2)
               (B1) -- (B2)
	       (B2) -- (A3)
	       (B2) -- (B3)
               (A2) -- (B3); 
            }
	    
	    \draw 
	    (1.35,0,0) node [above]{$X_{i-1}$} 
	    (2.8,0,0) node [above]{$X_{i}$}
	    (4.35,0,0) node [above]{$X_{i+1}$} ;
	\end{tikzpicture}
    \caption{A model for $\uu E \Z^{2}$. The thick lines represents 
    the models $X_{i} \isom \R$ for $E\Z^{2}/H_{i}$.}
    \label{fig:model-for-ZxZ}
\end{figure}

It follows that $X$ is a $3$-dimensional model for $\uu E\Z^{2}$. 
This is because of the following observations:
\begin{enumerate}
    \item $X$ is contractible by construction;

    \item if $H$ is an infinite cyclic subgroup of $\Z^{2}$, then 
    $H\leq H_{i}$ for some unique $i\in \Z$ and therefore $X^{H} = 
    X^{H_{i}} = X_{i}$ which is contractible;

    \item  if $H$ is not cyclic, then $H=K_{1}\times K_{2}$ with 
    $K_{1}$ and $K_{2}$ infinite cyclic subgroups of $\Z^{2}$ and
    \begin{equation*}
	X^{H} \subset X^{K_{1}}\cap X^{K^{2}} = X^{H_{i_{1}}}\cap 
	X^{H_{i_{2}}} = \emptyset
    \end{equation*}
    for some $i_{1}, i_{2}\in \Z$, $i_{1}\neq i_{2}$.
\end{enumerate}

We have that $(X_{i}*X_{i+1})/G \isom S^{3}$ and $X_{i}/G\isom S^{1}$
for every $i\in \Z$ and thus $H_{3}(X/G)$ is free abelian of infinite
rank~\cite[p.~138]{juan-pineda-06}.  In particular~$X$ is a model for
$\uu E\Z^{2}$ of minimal dimension, that is $\uugd \Z^{2}=3$.

Now Theorem~\ref{thrm:geometric-lower-bound} states that $\uuhd
\Z^{2}\geq
3$ and therefore we get the following complete statement about the 
Bredon dimensions with respect to the family of virtually cyclic 
subgroups.

\begin{proposition}
    \label{prop:dimvcZ2=3}
    $\uuhd \Z^{2} = \uucd \Z^{2} = \uugd \Z^{2} = 3$.\qed
\end{proposition}

\pagebreak[3]

From this result we obtain immediately two interesting consequences 
regarding virtually polycyclic groups.

\begin{proposition}
    \label{prop:dimvc-polycyclic}
    Let $G$ be a virtually polycyclic group. Then 
    \begin{equation*}
	\uucd G = \uugd G.
    \end{equation*}
\end{proposition}

\begin{proof}
    To avoid triviality we assume that $G$ is not virtually cyclic.
    By Proposition~2 in~\cite[p.~2]{segal-83}, the group $G$ is
    virtually poly-$\Z$.  Since $G$ is not virtually cyclic we have
    $\vcd G \geq 2$.  In particular $G$ contains a subgroup that is an
    extension of $\Z$ by $\Z$, which in turn contains a subgroup
    isomorphic to $\Z^{2}$.  Thus $\uucd G\geq \uucd \Z^{2}=3$.  Now
    the claim follows from Proposition~\ref{prop:luck-meintrup-alt}.
\end{proof}

Note that in Theorem~5.13 in~\cite[pp.~525f.]{luck-12} a complete
description of $\uugd G$ for virtually polycyclic groups are given. 
The above result states that the same theorem gives a complete 
description of $\uucd G$ for virtually polycyclic groups, too. 

\begin{proposition}
    \label{prop:dimvc-polycyclic-1}
    Let $G$ be a virtually polycyclic group with $\vcd G=2$.  Then
    \begin{equation*}
        \uuhd G = \uucd G = \uugd G = 3.
    \end{equation*}
\end{proposition}

\begin{proof}
    Since $G$ has a subgroup isomorphic to $\Z^{2}$ we have $\uuhd
    G\geq \uuhd\Z^{2} = 3$.  On the other hand Theorem~5.13
    in~\cite[pp.~525f.]{luck-12} gives $\uugd G = 3$ and thus all
    three Bredon dimensions agree and are equal to~$3$.
\end{proof}

The next two applications of Theorem~\ref{thrm:geometric-lower-bound}
rely on Juan-Pineda and Leary's construction of a model for $\uu EG$
for a class of group which includes Gromov-hyperbolic groups, see
Proposition~9 and Corollary~10 in~\cite{juan-pineda-06}.

The class of groups considered in~\cite{juan-pineda-06} is
characterised by the following condition: Every infinite virtually
cyclic subgroup $H$ of $G$ is contained in a unique maximal virtually
cyclic subgroup $H_{\max}$ of $G$ which is equal to its own
normaliser.  This class is known to contain all Gromov-hyperbolic
groups~\cite[Theorem~8.37]{ghys-90}.

Recall that if $H$ is a virtually cyclic subgroup it is known (see for
example~\cite{juan-pineda-06}) that $H$ has a unique maximal normal
finite subgroup $N$ and one of the following three cases is true: $H$
is finite, $H/N$ is infinite cyclic (in this case we call $H$
\emph{orientable}) or $H/N$ is infinite dihedral (in this case we call
$H$ \emph{non-orientable}).

\begin{proposition}
    \label{prop:JPL-orig}
    \cite[Proposition~9 and Corollary~10]{juan-pineda-06} Let $G$ be a
    group satisfying the above condition on the set of infinite
    virtually cyclic subgroups.  Let $\calC$ be a complete system of
    representatives for the conjugacy classes of maximal infinite
    virtually cyclic subgroups of $G$.  Denote by $\calC_{o}$ and the
    set of orientable elements of $\calC$ and denote by $\calC_{n}$
    the set of non-orientable elements of $\calC$.  Then a model for
    $\uu EG$ can be obtained from model for $\underline{E}G$ by
    attaching
    \begin{enumerate}
        \item  orbits of $0$-cells indexed by $\calC$;
    
        \item  orbits of $1$-cells indexed by $\calC_{o} \cup 
	\{1,2\}\times \calC_{n}$;
    
        \item  orbits of $2$-cells indexed by $\calC$.
    \end{enumerate}
    Furthermore, a model for $\uu BG$ can be obtained from a
    model for $\underline{B} G$ by attaching $2$-cells indexed by
    $\calC_{o}$.\qed
\end{proposition}

\begin{proposition}
    \label{prop:dim-gromov-hyperbolic-groups}
    Let $G$ be a Gromov-hyperbolic group which is not virtually 
    cyclic. Then $\uuhd G\geq 2$. Moreover, if $\ugd G\leq 2$, then
    \begin{equation*}
        \uuhd G = \uucd G = \uugd G = 2.
    \end{equation*}
\end{proposition}

\begin{proof}
    It has been shown in~\cite[p.~141]{juan-pineda-06}, that if $G$ is
    Gromov-hyperbolic group which not virtually cyclic, then
    $H_{2}(\uu BG)\neq 0$.  This follows from the following two facts:
    \begin{enumerate}
	\item for large enough integers~$d$ the Rips complex
	$R_{d}(G)$ is a \emph{finite} model for $\underline 
	EG$~\cite{baum-94, meintrup-02};
	
	\item $G$ has infinitely many conjugacy classes of orientable maximal
	infinite  virtually cyclic subgroups~\cite[p.~141,
	Theorem~13]{juan-pineda-06}.
    \end{enumerate}
    Thus it follows from Theorem~\ref{thrm:geometric-lower-bound} 
    that $\uuhd G\geq 2$ if $G$ is a Gromov-hyperbolic group which is 
    not virtually cyclic.
    
    If moreover $\ugd G\leq 2$, then there exists a $2$-dimensional
    model for $\uu EG$ by Proposition~\ref{prop:JPL-orig}, that is
    $\uugd G\leq 2$. Therefore we have altogether
    \begin{equation*}
        2\leq \uuhd G \leq \uucd G \leq \uugd G \leq 2
    \end{equation*}
    and equality holds.
\end{proof}

\begin{corollary}
    \label{cor:dimvcF=2}
    Let $F$ be a free group of finite rank at least~$2$. 
    Then 
    \begin{equation*}
        \uuhd F = \uucd F = \uugd F = 2.
    \end{equation*}
\end{corollary}

\begin{proof}
    Free groups are characterised by the fact that they can act freely
    on a tree, that is $\gd F=1$.  Since $F$ is torsion free it
    follows that $\ugd F = \gd F = 1$.  Free groups of finite rank are
    Gromov-hyperbolic and thus the claim follows from
    Proposition~\ref{prop:dim-gromov-hyperbolic-groups}.
\end{proof}

Another example of Gromov-hyperbolic groups are fundamental groups of
finite graphs of finite groups.  We use the notation as introduced by
Serre in~\cite[pp.~41ff.]{serre-80}.  Given a connected, non-empty,
orientated graph $Y$, a graph $(G,Y)$ of groups consists of
\begin{enumerate}
    \item a collection of groups $G_{P}$ indexed by the vertices $P\in
    \vertx Y$ of $Y$;

    \item a collection of groups $G_{y}$ indexed by the edges $y\in
    \edge Y$ of $Y$ subject to the condition $G_{y} = G_{\bar y}$ 
    where $\bar y$ denotes the inverse edge of $y$;

    \item  for each edge $y\in \edge Y$, a monomorphism
    $G_{y}\hookrightarrow G_{t(y)}$ where $t(y)$ denotes the terminal 
    vertex of the edge $y$.
\end{enumerate}
From this data one can construct the \emph{fundamental group
$\pi_{1}(G,Y)$} of the graph of groups $(G,Y)$, see~\cite{serre-80}. 
By abuse of notation we say that a group is the \emph{fundamental 
group} of the graph of groups $(G,Y)$ if $G\isom \pi_{1}(G,Y)$.

The $G$ is the fundamental group of a graph of groups, then there
exists canonical inclusions of the vertex groups $G_{P}$, $P\in \vertx
Y$, into $G$.  Furthermore, one can construct a tree $T$ with an
action of $G$ with vertex stabilisers being precisely the conjugates
of the groups $G_{P}$ and with the edge stabilisers being precisely
the conjugates of the edge groups $G_{y}$, $y\in\edge Y$, and such
that $T/G = Y$.  This tree is called the \emph{universal cover of
$(G,Y)$} or the \emph{Bass--Serre tree} associated with the
fundamental group of the graph of groups~$(G,Y)$.  Conversely, every
group $G$ which admits an action on a tree~$T$ is the fundamental
group of some graph of groups $(G,Y)$ such that $T$ is the associated
Bass--Serre tree of the graph of groups.

\begin{examples}
    \begin{enumerate}
        \item  Let $(1, Y)$ be a graph of groups whose vertex and 
	edge groups are all trivial. Then $\pi_{1}(1,Y) \isom 
	\pi_{1}(Y)$ is the fundamental group of the graph $Y$.
    
	\item Let $Y$ be the graph with two edges (that is
	non-oriented edge) and two vertices, that is a line segment.
	Let $(G,Y)$ be the graph of groups with vertex groups $A$ and
	$B$ and with edge group $C$, see
	Figure~\ref{fig:edge-and-loop}.  Then $\pi_{1}(G,Y) \isom A
	*_{C} B$ is the free product of~$A$ and~$B$ with the common
	subgroup $C$ amalgamated.
	
	\item Let $Y$ be the graph with two edges and
	one vertex, that is a loop.  Let $(G,Y)$ be the graph of
	groups with vertex group $A$ and with edge group $B$, see
	Figure~\ref{fig:edge-and-loop}.  Then $\pi_{1}(G,Y)$ is an
	HNN-extension of the group~$A$.
    \end{enumerate}
    \begin{figure}[t]
	\centering
	
	\subfloat{\begin{tikzpicture}
	    \clip (-1.9,-1.7) rectangle (1.9,1.7);
	    \coordinate (A) at (-1.1,0);
	    \coordinate (B) at (1.1,0);
	    \draw (A) -- node[below] {$C$} (B);
	    \filldraw (A) circle (1.8pt) node [above=1.8pt] {$A$}
	    (B) circle (2pt) node [above=2pt] {$B$};
	\end{tikzpicture}}
	\qquad
	\subfloat{\begin{tikzpicture}
            \clip (-1.9,-1.7) rectangle (1.9,1.7);
	    \draw (0,0) circle (1.2) 
	    (1.2,0) node [right] {$B$};
	    \filldraw (-1.2,0) circle (1.8pt) node [left] {$A$};
	\end{tikzpicture}}
	\caption{A segment of groups and loop of
	groups~\cite[p.~41]{serre-80}.}
	\label{fig:edge-and-loop}
    \end{figure}
\end{examples}

\begin{proposition}
    \label{prop:dimvc-finite-graphs}
    Let $G$ be the fundamental group of a finite graph of finite
    groups $(G,Y)$, and assume that $G$ is not virtually cyclic.  Then
    \begin{equation*}
	\uuhd G = \uucd G =\uugd G = 2.
    \end{equation*}
\end{proposition}

\begin{proof}
    Fundamental groups of a finite graph of finite groups have a free
    subgroup of finite index~\cite[p.~104]{dicks-89} which has finite
    rank.  Free groups of finite rank are Gromov-hyperbolic and the
    property of being Gromov-hyperbolic is preserved by finite
    extensions.  Therefore the group $G$ is Gromov-hyperbolic.
    
    Fundamental groups of finite graphs of groups are known to admit a
    $1$-dimensional model for $\underline EG$.  To see this, let $T$
    be the Bass--Serre tree associated with the graph of groups
    $(G,Y)$.  Then $\frakF(T)\subset \Ffin(G)$.  On the other hand, it
    is a well known fact that a finite group cannot act
    fixed point free on a tree and therefore $T^{H}\neq \emptyset$ for
    all $H\in \Ffin(G)$.  Thus $T$ is a one-dimensional model for
    $\underline{E} G$ and we have $\ugd G\leq 1$.
    
    Altogether the conditions of
    Proposition~\ref{prop:dim-gromov-hyperbolic-groups} are satisfied
    and the claim follows since $G$ is not virtually cyclic by
    assumption.
\end{proof}

%
%

\section{Dimensions of Extensions of $C_{\infty}$}

The results of this section rely on a result which Martínez-Pérez has
obtained using a spectral sequence that she has constructed for Bredon
(co\=/)homology in~\cite{martinez-perez-02}.

\begin{proposition}[Martínez-Pérez]
    \label{prop:MP}
    Let $G$ be an extension
    \begin{equation*}
        0\to N \to G \to Q \to 0
    \end{equation*}
    of a group $N$ by a group $Q = G/N$.  Let 
    \begin{equation*}
	\frakH \defeq  \{ H\leq G : N\leq H \text{ and } H/N\in \Fvc(Q)\}
    \end{equation*}
    and set
    \begin{equation*}
	m \defeq  \sup \{ \uuhd H :  H\in \frakH \}
	\quad \text{and} \quad 
	n \defeq  \sup \{ \uucd H : H\in \frakH\}.
    \end{equation*}
    Then we have the estimates:
    \begin{equation*}
	m\leq \uuhd G\leq \uuhd Q + m
	\quad \text{and} \quad
	n\leq \uucd G\leq \uucd Q + n.
    \end{equation*}
\end{proposition}

\begin{proof}
    The lower bound is a direct consequence of
    Proposition~\ref{prop:dim-subgroups-1}.  The upper bound is due to
    Martínez-Pérez's result in~\cite[pp.~171f.]{martinez-perez-02}.
\end{proof}

For future reference we state the following corollary to this result.

\begin{corollary}
    \label{cor:MP}
    Let $G$ be an extension of a group $N$ by a group $Q$. 
    Assume that $\uuhd Q<\infty$. Then $\uuhd G<\infty$ if 
    and only if there exists an integer $k$ such that $\uuhd 
    H\leq k$ for every $N\leq H\leq G$ such that $H/N$ is 
    virtually cyclic.
    
    The statement remains true if ``$\uuhd$'' is replaced by
    ``$\uucd$'' or ``$\uugd$''.
\end{corollary}

\begin{proof}
    The statements about $\uuhd G$ and $\uucd G$ are direct
    consequences of Proposition~\ref{prop:MP}.  The last statements
    follows from the second since the Bredon cohomological and Bredon
    geometric dimension agree for values greater than $3$ by
    Proposition~\ref{prop:luck-meintrup-alt}.
\end{proof}

In this section we apply Proposition~\ref{prop:MP} to the special
case where 
$N=C_{\infty}$, that is $G$ is an extension
\begin{equation*}
    0 \to C_{\infty} \to G \to Q\to 0
\end{equation*}
of the infinite cyclic group $C_{\infty}$ by an arbitrary group $Q$.

\pagebreak[3]

\begin{proposition}
    \label{prop:main-result}
    Let $G$ be an extension of the infinite cyclic group $C_{\infty}$
    by an arbitrary group $Q$.  Then precisely one of the following
    cases occurs:
    \begin{enumerate}
        \item  $Q$ has no element of infinite order. Then all $H\in 
	\frakH$ are virtually cyclic and
	\begin{equation*}
	    \uuhd G \leq \uuhd Q \qquad \text{and} \qquad \uucd G 
	    \leq \uucd Q;
	\end{equation*}
    
        \item  $Q$ has elements of infinite order. Then
	\begin{equation*}
	    3 \leq \uuhd G \leq \uuhd Q + 3 \qquad \text{and} \qquad 
	    3 \leq \uucd G \leq \uucd Q + 3.
	\end{equation*}
    \end{enumerate}
\end{proposition}

\begin{proof}
    If $Q$ does not have elements of infinite order then every 
    $C_{\infty}\leq H\leq G$ with $H/C_{\infty}$ virtually cyclic is 
    itself virtually cyclic and therefore $\uuhd G = \uucd G = 0$. 
    Now the first claim follows from Proposition~\ref{prop:MP}.
    
    If $Q$ does have an element of infinite order then there exists
    $C_{\infty} \leq H\leq G$ with $H/C_{\infty}$ infinite virtually
    cyclic.  In this case $\vcd H = 2$ and therefore $\uuhd H = \uugd
    H = 3$ by Proposition~\ref{prop:dimvc-polycyclic-1}.  The second
    claim follows now from~\ref{prop:MP}.
\end{proof}

Note that we can replace $C_{\infty}$ by an infinite virtually cyclic 
group $N$ and still get the same result. 

\begin{proposition}
    For the braid group $B_{3}$ we have the estimate 
    \begin{equation*}
	3\leq \uuhd B_{3} \leq \uucd B_{3} = \uugd B_{3} \leq 5.
    \end{equation*}
\end{proposition}

\begin{proof}
    The braid group $B_{3}$ is an extension of the infinite cyclic
    group~$C_{\infty}$ by the modular group $C_{2}*C_{3}$.  Now
    $C_{2}*C_{3}$ is not virtually cyclic and therefore $\uuhd
    (C_{2}*C_{3}) = \uucd (C_{2}*C_{3}) = 2$ by
    Proposition~\ref{prop:dimvc-finite-graphs}.  Moreover it has an
    element of infinite order and thus $3\leq \uuhd B_{3}\leq \uucd
    B_{3} \leq 5$ by Proposition~\ref{prop:main-result}. Furthermore
    $\uucd B_{3}\geq 3$ and therefore it must be equal to~$\uugd 
    B_{3}$ by Proposition~\ref{prop:luck-meintrup-alt}.
\end{proof}

%
%

\section{Nilpotent Groups}

The \emph{Hirsch length} $\h G$ of a group $G$ is an invariant of
groups which has originally been defined for polycyclic groups.  For
polycyclic groups it is the number of infinite cyclic factors in a
infinite cyclic series of $G$~\cite[p.~152]{robinson-96}.

The notion of Hirsch length can be extended to elementary amenable
groups, which is a class of groups which contains all locally
nilpotent
groups and also all soluble groups (more detail will follow in the
next section). The extension can be done in such a way that 
the following holds.

\begin{proposition}
    \cite[Theorem 1]{hillman-91}
    \label{prop:hillman-91}
    Let $G$ be an elementary amenable group. Then
    \begin{enumerate}
        \item  if $H$ is a subgroup of $G$ then $\h H\leq \h G$;
    
        \item  if $G$ is the direct union of subgroups $G_{\lambda}$, 
	$\lambda\in \Lambda$, then
	\begin{equation*}
	    \h G = \sup\{ \h G_{\lambda}\};
	\end{equation*}
    
        \item  if $H$ is a normal subgroup of $G$, then $\h G = \h H 
	+ \h (G/H)$.\qed
    \end{enumerate}
\end{proposition}

A finitely generated nilpotent group is known to be 
polycyclic~\cite[p.~137]{robinson-96}. It is shown in~\cite{luck-12} 
that
\begin{equation*}
    \vcd G -1 \leq \uugd G \leq \vcd G + 1
\end{equation*}
for virtually polycyclic groups $G$.  For virtually polycyclic groups
the virtual cohomological dimension $\vcd G$ is equal to the Hirsch
length $\h G$.  Therefore Proposition~\ref{prop:dimvc-polycyclic}
states that we have for finitely generated nilpotent groups the
estimate
\begin{equation*}
   \h G - 1 \leq  \uucd G \leq \h G + 1.
\end{equation*}

If $G$ is a countable group, then $G$ is the countable direct union of
its finitely generated subgroups $G_{\lambda}$ and we have the
estimate
\begin{equation*}
    k \leq \uucd G \leq k+1
\end{equation*}
where $k\defeq  \sup \{ \uucd G_{\lambda}\}$. On the other hand, if $G$ is 
locally nilpotent group, then 
\begin{equation*}
    \h G = \sup \{ \h G_{\lambda}\},
\end{equation*}
by Proposition~\ref{prop:hillman-91}. Since $\h G_{\lambda} = \uucd 
G_{\lambda}$ for all $\lambda$ it follows that $k = \h G$. We get the 
following estimate for $\uucd G$ for locally nilpotent groups.

\begin{proposition}
    \label{prop:cd-nilpotent-G}
    Let $G$ be a countable locally nilpotent group with finite Hirsch
    length $\h G$.  Then
    \begin{equation*}
        \h G - 1\leq \uucd G \leq \h G + 2.
    \end{equation*}
    In particular this estimate is true for countable nilpotent
    groups.
\end{proposition}

Note that if $G$ locally virtually cyclic, then $\uucd G\leq 1$.
Otherwise~$G$ contains a subgroup isomorphic to $\Z^{2}$ and it
follows that~$\uucd G\geq 3$ and~$\h G\geq 2$.

The proof of Proposition~\ref{prop:cd-nilpotent-G} relies heavily on
Lück and Weiermann's geometric result for virtually polycyclic
groups~\cite{luck-12}.  In what follows we will give an algebraic
proof which avoids the use of geometric results as far as possible.
The only geometric input we need in the following result is that
$\uugd \Z^{n} \leq n+1$.  In the case of $n=2$ this follows from the
very simple model for $\uu E \Z^{2}$ explained
in~\cite{juan-pineda-06} and for general $n\in \N$ it follows from a
construction in~\cite{connolly-06}.  In both cases the results are
obtained with a much simpler machinery than the general result
in~\cite{luck-12}.  Note that the geometric results enter in the proof
of Theorem~\ref{thrm:cd-nilpotent-G-alt} at the following two places:
we need $\uugd \Z^{n} \leq n+1$ for general $n$ in the proof of
Lemma~\ref{lem:nilpotent2} and $\uugd\Z^{2} = 3$ is implicitly used in
the inequality~\eqref{eq:geometric-result} in the proof of
Proposition~\ref{prop:nilpotent3}.

If $G$ is a nilpotent group, then the set
$\tau(G)$\label{def:tau-G-for-N-groups} consisting of all elements
of~$G$ which have finite order is a fully invariant subgroup of~$G$
and the quotient~$G/\tau(G)$ is
torsionfree~\cite[p.~132]{robinson-96}.  The subgroup $\tau(G)$ is
called the \emph{torsion-subgroup} of $G$.

\begin{lemma}
    \label{lem:nilpotent1}
    Let $G$ be a countable nilpotent group.  Then $\uuhd G \leq \uuhd
    G/\tau(G)$ and $\uucd G \leq \uugd G/\tau(G) + 1$.
\end{lemma}

\begin{proof}
    Let $S$ be a subgroup of $G$ such that $\tau(G)\leq S$ and
    $S/\tau(G)$ is a virtually cyclic subgroup of $G/\tau(G)$.  We
    claim that $\uuhd S=0$ and $\uucd S\leq 1$.
 
    Since $G/\tau(G)$ is torsion-free it follows that $S/\tau(G)$ is
    infinite cyclic.  Thus we have a short exact sequence
    \begin{equation*}
	1 \to \tau(G) \to S \to S/\tau(G) \to 1
    \end{equation*}
    The group $S$ is the countable and hence it is the countable union
    of its finitely generated subgroups $S_{\lambda}$.  Since
    $S_{\lambda}$ is a finitely generated nilpotent group it follows
    that $S_{\lambda}$ is a polycyclic group.  In particular
    $S_{\lambda}$ satisfies the maximal condition on
    subgroups~\cite[p.~152]{robinson-96}.  As a consequence this
    implies that~$\tau(G)\cap S_{\lambda}$ cannot have a infinite
    strictly ascending sequence of finite groups and hence it is
    finite.  Therefore $S_{\lambda}$ is virtually cyclic.  It follows
    that $S$ is locally virtually cyclic.  Therefore $\uuhd S = 0$ and
    $\uucd S \leq 1$ and this proves the claim.
    
    We can apply Martínez-Pérez's spectral sequence, that is the
    result of Proposition~\ref{prop:MP}, and we get the desired
    inequalities
    \begin{equation*}
        \uuhd G \leq \uuhd Q
	\qquad \text{and} \qquad
	\uucd G \leq \uucd Q + 1.\qedhere
    \end{equation*}
\end{proof}

\begin{lemma}
    \label{lem:nilpotent2}
    Let $G$ be a torsion-free abelian group. Then
    \begin{equation*}
        \uucd G\leq \h G + 2.
    \end{equation*}
\end{lemma}

\begin{proof}
    In order to avoid triviality we assume that that the Hirsch length
    of $G$ is finite.  In general, torsion-free soluble groups of
    finite Hirsch length are countable~\cite[p.~100]{bieri-81}.  In
    particular $G$ is countable.  The group $G$ is the direct union of
    its finitely generated subgroups $G_{\lambda}$.  Since $G$ is
    torsion-free and abelian, it follows that $G_{\lambda} \isom
    \Z^{n_{\lambda}}$.  Then $\uucd \Z^{n_{\lambda}} \leq \uugd
    \Z^{n_{\lambda}} \leq n_{\lambda}+1 = \h G_{\lambda} +1$ where the
    last inequality is due to~\cite{connolly-06}.  Therefore
    \begin{equation*}
        \uucd G \leq \sup \{ \uucd G_{\lambda} \} + 1 \leq \sup \{ \h 
	G_{\lambda} \} + 2 = \h G + 2.\qedhere
    \end{equation*}
\end{proof}

\begin{proposition}
    \label{prop:nilpotent3}
    Let $G$ be a torsion-free nilpotent group. Then
    \begin{equation*}
        \uucd G \leq \h G + 5(c-1) + 2
    \end{equation*}
    where $c$ is the nilpotency class of $G$.
\end{proposition}

\begin{proof}
    If $c=1$ then $G$ is abelian and the claim is the statement of
    Lemma~\ref{lem:nilpotent2}.  Therefore assume that $c\geq 2$ and
    that the statement is true for groups with nilpotency class
    strictly less then $c$.
    
    Let $N\defeq  \zeta(G)$ be the centre of $G$ and consider the central
    extension
    \begin{equation*}
        1 \to N \to G \to Q \to 1.
    \end{equation*}
    Since $G$ is torsion-free it follows by a theorem of Mal'cev that
    $Q$ is torsion-free~\cite[p.~137]{robinson-96}.  The nilpotency
    class of $Q$ is $c-1$ and therefore we have by induction the 
    inequality
    \begin{equation*}
	\uucd Q \leq \h Q + 5(c-2) +2.
    \end{equation*}
    
    Let $H$ be a subgroup of $G$ with $N\leq H$ and $H/N$ a virtually 
    cyclic subgroup of $Q$. Since $Q$ is torsion-free it follows that 
    $H/N$ is infinite cyclic. Therefore we have a short exact sequence
    \begin{equation*}
        1\to N\to H\to H/N\to 1
    \end{equation*}
    and this sequence is split since $H/N$ is free.  Every
    element of $H$ commutes with every element of $N = \zeta(G)$ and
    thus $N\leq \zeta(H)$.  Therefore the above extension is central
    and hence $H \isom N\times H/N$.  Then
    \begin{align}
        \uucd H & \leq \uucd N + \underbrace{\uucd H/N}_{=0} + 3
	\label{eq:geometric-result}
	\\
	& \leq  \h N + 5 \nonumber
    \end{align}
    where the first inequality is due to
    Corollary~\ref{cor:cd-for-direct-products} and the second
    inequality is due to Lemma~\ref{lem:nilpotent2}.  We can apply
    Proposition~\ref{prop:MP} and get altogether
    \begin{align*}
        \uucd G & \leq \h N + 5 + \uucd Q
	\\
	& \leq \h N + \h Q + 5(c-1) + 2
	\\
	& \leq \h G + 5(c-1) + 2\qedhere
    \end{align*}
\end{proof}

\begin{theorem}
    \label{thrm:cd-nilpotent-G-alt}
    Let $G$ be a countable nilpotent group. Then
    \begin{equation*}
        \uucd G \leq \h G + 5(c-1) + 3
    \end{equation*}
    where $c$ is the nilpotency class of $G$.
\end{theorem}

\begin{proof}
    Let $Q\defeq  G/\tau(G)$ and consider the short exact sequence
    \begin{equation*}
        1\to \tau(G) \to G \to Q \to 1.
    \end{equation*}
    Since $\tau(G)$ is locally finite we have $\h (\tau(G))=0$ and it
    follows that $\h Q = \h G$.  Furthermore the nilpotency class of
    $Q$ is at most~$c$.  We can apply
    Proposition~\ref{prop:nilpotent3} to $Q$ and we get
    \begin{align*}
	\uucd Q
	& \leq
	\h Q + 5(c-1) + 2
	\\
	& = 
	\h G + 5(c-1) + 2.
    \end{align*}
    The claim follows now from Lemma~\ref{lem:nilpotent1}.
\end{proof}

%
%

\section{Elementary Amenable Groups}
\label{sec:EA-groups}

\nocite{lennox-04}

In the previous section we have already made a vague reference to
elementary amenable groups.  This class of groups was first introduced
by von~Neumann in~\cite{neumann-29}.  It is the smallest class of
groups which contains all finite groups and the infinite cyclic group
and which is closed under forming extensions, increasing unions, see
for example~\cite{hillman-92}.  We denote this class of groups in the
following by $\mathscr{C}$.  This class is closed under forming
subgroups and quotients.  It follows that the class of elementary
amenable groups includes all locally finite groups, all locally
nilpotent groups and all virtually soluble groups.

In order to extend the notation of Hirsch length to elementary
amenable groups we need the following hierarchical description of the
class $\mathscr{C}$ of elementary amenable
groups~\cite[pp.~678f.]{kropholler-88}.  We use the following
notation: Let $\mathscr{X}$ and $\mathscr{Y}$ be two classes of
groups.  Then a group $G$ belongs to the class~$\mathscr{X\!Y}$ if it
is isomorphic to a group extension of a group in $\mathscr{X}$ by a
group in $\mathscr{Y}$.  A group $G$ is in the class $\cll\mathscr{X}$
if every finitely generated subgroup of $G$ belongs to the
class~$\mathscr{X}$.  Let $\mathscr{X}_{1}$ be the class of finitely
generated abelian-by-finite groups.  For every ordinal $\alpha\geq 2$
we define inductively the following classes of groups:
\begin{align*}
    \mathscr{X}_{\alpha }& \defeq 
    (\cll\mathscr{X}_{\alpha-1})\mathscr{X}_{1} && \text{if $\alpha$
    is asuccessor ordinal;} \\
    \mathscr{X}_{\alpha }& \defeq  \bigcup_{\beta<\alpha} 
    \mathscr{X_{\beta}} && \text{if $\alpha$ is a limit ordinal.}
\end{align*}
It has been shown in~\cite[p.~679]{kropholler-88} that
\begin{equation*}
   \mathscr{C} = \bigcup_{\alpha\geq 1} \mathscr{X}_{\alpha},
\end{equation*}
where the union is taken over all ordinals greater or equal to $1$.
The Hirsch lengths of an elementary amenable group can now be defined
inductively as follows.  If $G$ is a group in $\mathscr{X}_{1}$ then
the Hirsch length is the rank of an abelian subgroup of finite index
in $G$.  If $G$ is an elementary group which is not
in~$\mathscr{X}_{1}$, then there exists a least ordinal $\alpha$ such
that $G$ is in $\mathscr{X}_{\alpha}$.  This ordinal~$\alpha$ is
necessarily a successor ordinal greater than~$1$.  Then $G$ has a
normal subgroup~$N$ such that $N$ is in~$\cll\mathscr{X}_{\alpha-1}$
and $G/N$ is in $\mathscr{X}_{1}$.  In particular~$\h N_{\lambda}$ is
defined for every finitely generated subgroup of $N$ and $\h(G/N)$ is
defined as well.  We set
\begin{equation*}
    \h G \defeq  \sup \{ \h N_{\lambda}\} + \h(G/N),
\end{equation*}
where the supremum is taken over the Hirsch length of all finitely
generated subgroups $N_{\lambda}$ of $N$.  It follows by transfinite
induction that this way $\h G$ is well defined for all elementary
amenable groups and that the properties of
Proposition~\ref{prop:hillman-91} are
satisfied~\cite[pp.~163f.]{hillman-91}.

In the previous section we have shown that locally nilpotent groups
$G$ with finite Hirsch length have finite Bredon cohomological
dimension with respect to the family $\Fvc(G)$ and thus admit a finite
dimensional model for the classifying space $\uu EG$.
The natural question is to ask, if the same is true for elementary
amenable groups.

Flores and Nucinkis have shown in~\cite{flores-05} 
that for elementary amenable groups the Bredon homological dimension 
$\uhd G$ is equal to the Hirsch length $\h G$. For countable groups 
this implies due to $\ucd G\leq \uhd G + 1$ and 
Proposition~\ref{prop:luck-meintrup} that
\begin{equation*}
    \ugd G \leq \min(\h G +1, 3).
\end{equation*}
In particular countable elementary amenable groups of finite Hirsch
length admit a finite dimensional model for $\underline EG$.  In the
light of the conjecture stated in~\cite{luck-12} that $\uugd G\leq
\ugd G + 1$ one may hope to proof that a similar result holds for
models for $\uu EG$.  However, the proof of the result 
in~\cite{flores-05} does not generalise to the family of virtually 
cyclic subgroups as a finite index argument in the proof of Lemma~2
fails beyond repair. In 
the following we describe a possible strategy to for a proof that 
countable elementary amenable groups of finite Hirsch length admit a 
finite dimensional model for $\uu EG$. However, certain assumptions 
have to be made. We believe these assumptions to be reasonable.

First of all, note that the requirement that an elementary amenable
group needs 
to have finite Hirsch length in order to admit a finite dimensional 
model for $\uu EG$ is necessary:

\begin{lemma}
    \label{lem:EA-aux2}
    Let $G$ be an elementary amenable group with $\h G=\infty$. 
    Then $\uugd G=\infty$.
\end{lemma}

\begin{proof}
    We have $\uugd G \geq \uucd G \geq \ucd G - 1 \geq \cd_{\Q} G - 1$
    where the last inequality is due to the following standard
    argument~\cite{brady-01}: evaluating any projective resolution
    $P_{*}\to \underline{\Z}$ of the trivial $\OC{G}{\Ffin(G)}$-module
    at $G/1$ and tensoring it with~$\Q$ gives a projective resolution
    of the trivial $\Q G$-module~$\Q$.
    
    However, $\cd_{\Q} G \geq \h G$ is true for any elementary
    amenable group by~\cite[pp.~167f]{hillman-91}.  Therefore, if $\h
    G=\infty$ it follows that $\cd_{\Q} G=\infty$ and thus $\uugd
    G=\infty$, too.
\end{proof}

Alternatively the above result follows from~\cite{flores-05} together 
with the inequality $\uugd G \geq \ugd G-1 \geq \uhd G -1$.

In what will follow in this section, two special subgroups of a group
$G$ will be of importance.  Given an arbitrary group $G$ there exists
a unique maximal normal locally finite subgroup which we denote by
$\Lambda(G)$~\cite[p.~436]{robinson-96}.  There exists also a unique
maximal normal torsion subgroup which we denote by $\tau(G)$, see for
example~\cite[p.~90]{lennox-04}.  The subgroup $\tau(G)$ is called the
\emph{torsion radical} of $G$.  Note that for nilpotent groups
$\tau(G)$ agrees with the definition given on
page~\pageref{def:tau-G-for-N-groups}.  Clearly $\Lambda(G)\leq
\tau(G)$ but equality does not hold in
general~\cite[pp.~422ff.]{robinson-96}.  However, equality holds in
case that~$G$ is soluble~\cite[p.~90]{lennox-04}.

The structure of elementary amenable groups of finite Hirsch length
has been studied by Hillman and Linnell in~\cite{hillman-92} and
Wehrfritz in~\cite{wehrfritz-95}.  The relevant result is the 
following.

\begin{proposition}
    \cite{hillman-92, wehrfritz-95}
    \label{prop:EA-structure}
    Let $G$ be an elementary amenable group with $\h G<\infty$. Then
    $G/\Lambda(G)$ is a finite extension of a torsion-free soluble 
    group.\qed
\end{proposition}

Note that since torsion-free soluble groups of finite Hirsch length
are countable~\cite[p.~100]{bieri-81} it follows that an elementary
amenable group~$G$ of finite Hirsch length is countable modulo
$\Lambda(G)$~\cite{hillman-92}.  Therefore an elementary amenable
group which admits a finite dimensional model for $\uu EG$ are not far
away from being countable torsion-free soluble with finite Hirsch
length.  

Hence we will first restrict ourself to the following question: does a
countable soluble group $G$ with finite Hirsch length admit a finite
dimensional model for~$\uu EG$?  Later on we then consider countable
elementary amenable groups with finite Hirsch length.  However, we
will not consider elementary amenable groups of arbitrary large
cardinality since they differ by Proposition~\ref{prop:EA-structure}
from the countable case only by the cardinality of their locally
finite subgroup $\Lambda(G)$.

We collect two results about soluble groups and linear groups which 
will be needed in what follows.

\begin{proposition}
    \label{prop:W-74}
    \cite[Corollary 1.2]{wehrfritz-74} Let $G$ be a finite
    extension of a torsion-free, soluble group.  Then $\h G<\infty$ 
    implies that $G$ is $\Q$-linear.\qed
\end{proposition}

\begin{proposition}
    \label{prop:gruneberg}
    {\rm (Gruneberg, see for example~\cite[p.~102]{wehrfritz-73})}
    Let $G$ be a linear group. Then its Fitting subgroup $\Fit(G)$ 
    is nilpotent.\qed
\end{proposition}

Moreover, we need the following result by Lück which gives an
upper  bound on $\uugd G$ if the group $G$ contains a
finite index subgroup $H$ which admits a finite dimensional model for
$\uu EH$.

\begin{lemma}\cite[p.~191]{luck-00a}.
    \label{lem:luck-1}
    Let $G$ be a group and let $H$ be a finite index subgroup of $G$.
    Then
    \begin{equation*}
        \uugd G \leq |G:H| \cdot \uugd H.
    \end{equation*}
    In particular $\uugd G<\infty$ if and only if $\uugd 
    H<\infty$.\qed
\end{lemma}

For what follows, we will make the following two assumptions.
We return our attention to these assumptions in a moment.

\begin{plain-number}
    \label{num:EA-assumption-1}
    Let $G$ be a countable torsion-free soluble group with $\h
    G<\infty$.  Let $N\defeq \Fit(G)$ and denote by $\zeta(N)$ the centre
    of $N$.  Then there exists an integer $k\geq 0$ such that $\uugd
    K\leq k$ for every $\zeta(N)\leq K\leq G$ with $K/\zeta(N)$
    virtually cyclic.
\end{plain-number}

\begin{plain-number}
    \label{num:EA-assumption-2}
    Let $G$ be a countable soluble group with $\h G<\infty$. Then 
    \begin{equation*}
        \uugd (G/\Lambda(G))<\infty\Rightarrow \uugd(G)<\infty.
    \end{equation*}
\end{plain-number}

\begin{proposition}
    \label{prop:EA-induction}
    Assume that the assumptions~\ref{num:EA-assumption-1}
    and~\ref{num:EA-assumption-2} are satisfied.  Then any countable
    soluble group $G$ with finite Hirsch length admits a finite
    dimensional model for $\uu EG$.
\end{proposition}

\begin{proof}
    We proof by induction on the Hirsch length $\h(G)$.  If $\h(G)=0$
    then $G$ is locally finite which in turn implies $\uugd G\leq 1$
    (Theorem 4.3 in~\cite[pp.~511f.]{luck-12}).
    
    Thus we may assume that $\h(G)\geq 1$.  Then $G/\Lambda(G)$ is a
    finite extension of a torsion free soluble group $H$ by
    Proposition~\ref{prop:EA-structure}.  Since $G$ is not locally
    finite it follows that $H$ is non-trivial.  Since $H$ is soluble
    it follows that~$N\defeq \Fit(H)$ must be non-trivial as it
    contains the smallest non-trivial term of the derived
    series~\cite[p.~133]{robinson-96}.  Moreover $H$ is $\Q$-linear by
    Proposition~\ref{prop:W-74} and therefore $N$ must be nilpotent by
    Proposition~\ref{prop:gruneberg}.  Therefore the centre~$\zeta(N)$
    of $N$ must be non-trivial and since $H$ is torsion-free it
    follows that $\h(\zeta(N))\geq1$.  Since $\zeta(N)$ is
    characteristic in $N$ and since~$N$ is normal in~$H$ it follows
    that~$\zeta(N)$ is normal in $H$.  Therefore we can form the
    quotient group $H/\zeta(N)$ and $\h(H/\zeta(N))\leq
    \h(G/\Lambda(G))-1 = \h(G)-1$.  It follows by induction
    that~$\uugd (H/\zeta(N))<\infty$.
    
    Since we assume that assumption~\ref{num:EA-assumption-1} is
    satisfied for the group $H$ it follows that there exists an
    integer $k$ such that
    \begin{equation*}
        \uugd K\leq k
    \end{equation*}
    for every $\zeta(N)\leq K\leq H$ for which $K/\zeta(N)$ is
    virtually cyclic.  Therefore it follows by Corollary~\ref{cor:MP}
    that $\uugd (H/\zeta(N))<\infty$ implies $\uugd H<\infty$.  Since
    $H$ has finite index in $G/\Lambda(G)$ this implies that $\uugd
    (G/\Lambda(G))<\infty$ by Lemma~\ref{lem:luck-1}.  Finally,
    assumption~\ref{num:EA-assumption-2} implies that $\uugd
    G<\infty$.
\end{proof}

We return our attention to the assumption~\ref{num:EA-assumption-1}
which is clearly a necessary condition in
Proposition~\ref{prop:EA-induction}.  The centre $\zeta(N)$ of the
Fitting subgroup~$N = \Fit(G)$ of a soluble group $G$ is known to be
the centraliser $C_{G}(N)$ of $N$ in $G$~\cite[p.~149]{robinson-96}.
This is another information in addition to the many constraints we
know from the setup in the assumption~\ref{num:EA-assumption-1}.  By
Proposition~\ref{prop:cd-nilpotent-G} we know that $\uugd(\zeta(N))$
is finite and one may hope that with all the additional information
provided one can conclude that virtually cyclic extensions
\begin{equation*}
    1\to \zeta(N) \to K \to K/\zeta(N) \to 1
\end{equation*}
within the given countable torsion-free soluble $\Q$-linear group $G$
do have a bound on $\uugd K$.

\begin{proposition}
    \label{prop:EA-assumption-alt}
    Let $G$ be a torsion-free soluble group with $\h G<\infty$.
    Let~$N\defeq \Fit(G)$.  Assume there exists an integer $k\geq 0$
    such that $\uugd K\leq k$ for every infinite cyclic extension $K$
    of $\zeta(N)$ within $G$.  Then the
    assumption~\ref{num:EA-assumption-1} is satisfied.
\end{proposition}

In order to prove this proposition we need three auxiliary results.

\begin{lemma}
    \label{lem:EA-assumption-alt-aux1}
    Let $G$ be a torsion-free soluble group with $\h G<\infty$.  Let
    $N\defeq \Fit(G)$.  Then $G/N$ has a bound on the order of its
    finite subgroups.
\end{lemma}

\begin{proof}
    Let $Q\defeq  G/N$.  By \cite[pp.~29f.]{berrick-01} we have that
    $\Lambda(Q)$ is finite and $Q/\Lambda(Q)$ is an Euclidean
    crystallographic group.  Since $Q/\Lambda(Q)$ is crystallographic
    it follows that $Q/\Lambda(Q)$ has a bound on the order of its
    finite subgroup, say $|K|\leq k$ for any finite $K\leq
    Q/\Lambda(Q)$.
    
    Now let $H$ be an arbitrary finite subgroup of $Q$. Then
    \begin{align*}
        |H| 
	& = 
	|H:H\cap \Lambda(Q)| \cdot |H \cap \Lambda(Q)|
	\\
	& = 
	\underbrace{|H\Lambda(Q):\Lambda(Q)|}_{\leq k}
	\cdot
	\underbrace{|H\cap \Lambda(Q)|}_{\leq |\Lambda(Q)|}
	\\
	& \leq
	k \cdot  |\Lambda(Q)| < \infty
    \end{align*}
    which is a bound independent of $H$.
\end{proof}

\begin{lemma}
    \label{lem:EA-assumption-alt-aux2}
    Let $G$ and $N$ as in the previous lemma. Then $G/\zeta(N)$ has a 
    bound on the orders of its finite subgroups.
\end{lemma}

\begin{proof}
    The group $G$ is linear (Proposition~\ref{prop:W-74}) and 
    therefore $N$ is nilpotent (Proposition~\ref{prop:gruneberg}).
    
    Now $\zeta(N)$ is torsion-free and therefore each upper central 
    factor of $N$ is torsion-free by a result of 
    Mal'cev~\cite[p.~137]{robinson-96}. In particular 
    $\zeta(N/\zeta(N))$ is torsion-free. Thus every upper central 
    factor of $N/\zeta(N)$ is torsion-free. Since $N/\zeta(N)$ is 
    nilpotent its upper central series reaches $N/\zeta(N)$. it 
    follows that $N/\zeta(N)$ is torsion-free, too.
    
    Consider the short exact sequence
    \begin{equation*}
        1\to N/\zeta(N) \to G/\zeta(N) \to Q \to 1.
    \end{equation*}
    Since $(G/\zeta(N))/(N/\zeta(N) \isom G/N \isom Q$ there exists a
    upper bound $k\geq 1$ on the order of the finite subgroups of $Q$
    by Lemma~\ref{lem:EA-assumption-alt-aux1}. Let $H$ be a finite 
    subgroup of $G/\zeta(N)$. Since $N/\zeta(N)$ is torsion-free it 
    follows that $|H\cap N/\zeta(N)| = 1$. Then
    \begin{align*}
        |H| 
	& = 
	|H : H\cap N/\zeta(N)| \cdot
	\underbrace{|H \cap N/\zeta(N)|}_{=1}
	\\
	& =
	|H (N/\zeta(N)) : N/\zeta(N)| \leq k
	\tag*{\qedhere}
    \end{align*}
\end{proof}

\begin{lemma}
    \label{lem:EA-assumption-alt-aux3}
    Let $G$ be group and assume that there exists $r\geq 1$ such 
    that~$|H|\leq r$ for every finite subgroup $H$ of $G$. Then every 
    infinite virtually cyclic subgroup $K$ of $G$ has an infinite 
    cyclic subgroup $C$ with $|K:C|\leq 2r$.
\end{lemma}

\begin{proof}
    Let $N$ be the unique maximal normal finite subgroup of $K$ such 
    that $K/N$ is either infinite cyclic or infinite dihedral.
    
    If $K/N$ is infinite cyclic, then $K \isom N\rtimes C$ with $C$
    infinite cyclic and $|K:C| = |N| \leq r$.  On the other hand, if
    $K/N$ is infinite dihedral, then there exists $k\in K$ such that
    $kN$ generates an infinite cyclic subgroup of~$K/N$ of index $2$.
    Then $C\defeq  \langle k\rangle$ is an infinite cyclic subgroup of $K$
    with $|K:C| = 2 |N| \leq 2r$.
\end{proof}

\begin{proof}[Proof of Proposition~\ref{prop:EA-assumption-alt}]
    By Lemma~\ref{lem:EA-assumption-alt-aux2} we know that 
    there exists $r\geq 0$ such that $|H|\leq r$ for every finite 
    subgroup $H$ of $G/\zeta(N)$
    
    Let $K$ be an extension of $\zeta(N)$ within $G$ such that 
    $K/\zeta(N)$ is virtually cyclic.
    
    If $K/\zeta(N)$ is finite, then
    \begin{align*}
        \uugd K & \leq |K:\zeta(N)| \cdot \uugd \zeta(N) && 
	\text{(by Lemma~\ref{lem:luck-1})}
	\\
	& \leq 
	r \cdot (\h (\zeta(N)) + 2)
	&& \text{(by Proposition~\ref{prop:cd-nilpotent-G})}
    \end{align*}
    
    If $K/\zeta(N)$ is infinite, then $K$ has a subgroup $C$
    containing $\zeta(N)$ as a subgroup such that $C/\zeta(N)$
    is infinite cyclic and 
    \begin{equation*}
	|K:C| = |K/\zeta(N) : C/\zeta(N)| \leq 2r
    \end{equation*}
    by Lemma~\ref{lem:EA-assumption-alt-aux3}.
    By assumption $\uugd C\leq k$ and thus 
    \begin{align*}
	\uugd K & \leq |K:C| \cdot \uugd C &&
	\text{(by Lemma~\ref{lem:luck-1})}
	\\
	& \leq 2rk
    \end{align*}
    
    Therefore
    \begin{equation*}
        \uugd K \leq r \cdot \max(\h (\zeta(N)) + 2,  2k)
    \end{equation*}
    for any $\zeta (N) \leq K\leq G$ with $K/\zeta(N)$ virtually
    cyclic.  That is, the assumption~\ref{num:EA-assumption-1} is
    satisfied.
\end{proof}

The consequence of Proposition~\ref{prop:EA-assumption-alt} is 
that the assumption~\ref{num:EA-assumption-1} is equivalent to the 
following assumption.

\begin{plain-number}
    \label{num:EA-assumption-1-alt}
    Let $G$ be a countable torsion-free soluble group with $\h
    G<\infty$.  Let $N\defeq \Fit(G)$ and denote by $\zeta(N)$ the
    centre of $N$.  Then there exists an integer $k\geq 0$ such that
    $\uugd K\leq k$ for every $\zeta(N)\leq K\leq G$ with $K/\zeta(N)$
    \emph{infinite} cyclic.
\end{plain-number}

If $K$ is nilpotent, then $\uugd K\leq \h(\zeta(N))+3$ by
Proposition~\ref{prop:cd-nilpotent-G}.  In particular this bound is
satisfied whenever $K\leq N$.  If $\zeta(N)$ is finitely generated
then $\zeta(N)$ and therefore also any infinite cyclic extension $K$
of~$\zeta(N)$ is polycyclic.  In this case~$\uugd K\leq \h K + 1 = \h
(\zeta(N)) + 2$ by~\cite{luck-12}.  In the next chapter we show that
under certain conditions we can ensure $\uugd K \leq \uugd \zeta(N)
+1$, see Proposition~\ref{prop:dimension}.  However, these estimates
do not cover all possibilities yet.  But one may hope that one has
enough constraints to be able to answer all possibilities.  After all
the possible extension~$K$ of~$\zeta(N)$ within $G$ are
well understood.

The validity of the assumption~\ref{num:EA-assumption-2} 
is essential in the induction step which appears in the
proof of Proposition~\ref{prop:EA-induction}.  Whether or whether not
this assumption holds in this form is open at the moment.  However, one
may relax the assumption in case one wants to restrict the attention
to torsion-free soluble groups. In this case the following assumption 
is enough.

\begin{plain-number}
    \label{num:EA-assumption-3}
    Let $G$ be a torsion-free soluble group with finite Hirsch length
    and let~$H\defeq G/\zeta(\Fit(G))$.  Then $\uugd
    (H/\Lambda(H))<\infty$ implies $\uugd H<\infty$.
\end{plain-number}

In the assumption~\ref{num:EA-assumption-2} the subgroup $\Lambda(G)$
is allowed to be any countable soluble locally finite group.
However, in the above assumption~$\Lambda(H)$ may not be anymore as
arbitrary.

Finally, since an elementary amenable group $G$ with finite Hirsch
length is, modulo $\Lambda(G)$, a finite extension of a countable
torsion-free soluble group (Proposition~\ref{prop:EA-structure}) one
may consider the following variation of the
assumption~\ref{num:EA-assumption-2}.

\begin{plain-number}
    \label{num:EA-assumption-4}
    Let $G$ be a countable elementary amenable group with $\h
    G<\infty$.  Then
    \begin{equation*}
        \uugd (G/\Lambda(G))<\infty\Rightarrow \uugd(G)<\infty.
    \end{equation*}
\end{plain-number}

\begin{theorem}
    Suppose that the assumptions~\ref{num:EA-assumption-1-alt}
    and~\ref{num:EA-assumption-4} are satisfied.  Then any countable
    elementary amenable group $G$ with finite Hirsch length admits a
    finite dimensional model for $\uu EG$.
\end{theorem}

\begin{proof}
    Let $G$ be a countable elementary amenable group.  By
    Proposition~\ref{prop:EA-assumption-alt} the
    assumption~\ref{num:EA-assumption-1-alt} is equivalent to the
    assumption~\ref{num:EA-assumption-1}.
    Assumption~\ref{num:EA-assumption-4} implies the
    assumption~\ref{num:EA-assumption-2} and thus we can apply
    Proposition~\ref{prop:EA-induction} to~$G/\Lambda(G)$.  It follows
    that $\uugd (G/\Lambda(G))<\infty$.  Then
    assumption~\ref{num:EA-assumption-4} implies that $\uugd
    G<\infty$.
\end{proof}

%
%

\chapter{A Geometric Interlude}
\label{ch:geometric-interlude}

%
%

\section{A Class of Infinite Cyclic Extensions}

The spectral sequence developed by
Martínez-Pérez~\cite{martinez-perez-02} suggests that virtually cyclic
extensions are the main obstruction to understand the behaviour of
Bredon dimensions under general extensions for the family of virtually
cyclic subgroups.  A first general answer for finite extensions has
been given in~\cite[p.~191]{luck-00a}, see Lemma~\ref{lem:luck-1} in
the previous chapter.  Yet effectively nothing is known for the
general case.  The main objective in this chapter is to construct a
model for $\uu EG$ in the case that $G$ belongs to a certain class of
infinite cyclic extensions.

Infinite cyclic extensions are always split.  Therefore such kind of
extensions are always semidirect products.  Let $B$ be a group and let
$\varphi\in \Aut(B)$.  Recall that the semidirect product
\begin{equation*}
    G \defeq  B\rtimes \Z,
\end{equation*}
where $\Z$ acts on $B$ via the automorphism $\varphi$, is the set
$B\times \Z$ with the multiplication given by
\begin{equation*}
    (x,r)\cdot (y,s) = (x\varphi^{r}(y),r+s).
\end{equation*}
The identity is $(1,0)$ and the inverse of any element $(x,r)$ is
given by $(x,r)^{-1}= (\varphi^{-r}(x^{-1}),-r)$. The group $B$ 
is embedded via $x\mapsto (x,0)$ as a normal subgroup of~$G$ and we 
consider $\Z$ embedded as a subgroup of $G$ via $r\mapsto (1,r)$.

Up to  and including Section~\ref{sec:BS-groups} of this chapter we
will assume that $G = B\rtimes \Z$ and that that this extension
satisfy the following condition:

\begin{quote}
    \emph{The subgroup $\Z$ acts via conjugation \emph{freely} on the
    set of conjugacy classes of nontrivial elements of $B$.}
\end{quote}

\noindent 
Under this condition we can show that there exists a suitable set of 
unique maximal virtually cyclic subgroups to apply a variation of 
Juan-Pineda and Leary's construction (see 
Proposition~\ref{prop:JPL-orig}) in order to construct a model for
$\uu 
EG$ from a model for $\uu EB$.

%
%

\section{Technical Preparations}

\begin{lemma}
    \label{lem:aux1}
    Assume that $B$ is torsion-free and does not contain a subgroup
    isomorphic to $\Z^{2}$.  Then $\Z$ acts freely by conjugation on
    the set of conjugacy classes of the non-trivial elements of $B$ if
    and only if $G$ does not contain a subgroup isomorphic to
    $\Z^{2}$.
\end{lemma}

\begin{proof}
    ``$\Rightarrow$'':\quad Suppose that $H$ is a subgroup of $G$
    which is isomorphic to $\Z^{2}$.  It must have a non-trivial
    intersection with the kernel of the canonical projection $B\rtimes
    \Z \to \Z$.  Therefore there exists $(y,0)\in B\cap H$ with $y\neq
    1$.  On the other hand, $H$ is not contained in $B$ and thus there
    exists $(x,r)\in H\setminus B$.  Then, as $H$ is abelian, the
    commutator
    \begin{equation*}
        [(x,r),(y,0)] = (x\varphi^{r}(y)x^{-1}y^{-1},0)
    \end{equation*}
    must be trivial which is the case if and only 
    if $x\varphi^{r}(y)x^{-1}y^{-1} = 1$. This implies
    that $\varphi^{r}(y)$ and $y$ belong to the same conjugacy class 
    in $B$. Since~$r\neq 0$ and~$y\neq 1$ this implies that 
    $\Z$ does not act freely on the set of conjugacy classes of 
    non-trivial elements of $B$. 
    
    \smallskip
    
    ``$\Leftarrow$'':\quad Suppose that $\Z$ does not act freely on
    the set of conjugacy classes of non-trivial elements of $B$.  Then
    there exists $1\neq y\in B$ and $0\neq r\in \Z$ such that
    $\varphi^{r}(y) = x^{-1}yx$ for some $x\in B$.  This implies that
    the non-trivial elements $(x,r)$ and $(y,0)$ commute.  In general
    $(x,r)$ has infinite order and since $B$ is assumed to be
    torsion-free it follows that the order of $(y,0)$ is also
    infinite.  Furthermore the subgroups generated by $(x,r)$ and
    $(y,0)$ have clearly trivial intersection.  Therefore $(x,r)$ and
    $(y,0)$ generate a subgroup of~$G$ which is isomorphic to
    $\Z^{2}$.
\end{proof}

\begin{lemma}
    \label{lem:aux1a}
    Let $B$ be a non-trivial virtually cyclic group.  Then $\Z$ cannot
    act freely by conjugation on the set of conjugacy classes of the
    non-trivial elements of $B$.
\end{lemma}

\begin{proof}
    To avoid triviality assume that $B$ is infinite.  Then $B$
    contains a characteristic infinite cyclic subgroup $C$.  Therefore
    the automorphism~$\varphi$ restricts to an automorphism
    of $C$ which has order at most~$2$.  Hence the action of $\Z$ on
    the non-trivial elements of $B$ cannot be free and this implies
    the statement of the lemma.
\end{proof}

\begin{lemma}
    \label{lem:aux3}
    Assume that $\Z$ acts freely via conjugation on the set of 
    conjugacy classes of non-trivial elements of $B$. Then for any 
    $(x,r)\in G\setminus B$ and~$y\in B$ we have
    \begin{equation*}
        (x,r)^{y} = (x,r) \iff y=1.
    \end{equation*}
\end{lemma}

\begin{proof}
    $(x,r)^{y} = (x,r)$ is equivalent to $\varphi(y) = xyx^{-1}$, 
    which is by assumption on the action of $\Z$ on $B$ equivalent to 
    $y=1$.
\end{proof}

The statement of the next lemma is only non-trivial if $B$ has
torsion.
 
\begin{lemma}
    \label{lem:aux2}
    Assume that $\Z$ acts freely via conjugation on the set of
    conjugacy classes of non-trivial elements of $B$.  Let $H$ be a
    virtually cyclic subgroup of $G$ which is not a subgroup of
    $B$. Then $H$ is infinite cyclic.
\end{lemma}

\begin{proof}
    Let $\tau(H) \defeq  H\cap B$.  By assumption there exists $g\in
    H\setminus B$.  This element generates an infinite cyclic subgroup
    of $H$ which has trivial intersection with $B$ and hence also
    trivial intersection with $\tau(H)$.  Since $H$ is virtually
    cyclic this implies that $\tau(H)$ is finite.
    
    Since $\tau(H)$ is a normal subgroup of $H$ it follows that
    conjugation by~$g$ induces an automorphism of $\tau(H)$.  Since
    $\tau(H)$ is finite it follows that its automorphism group is
    finite, too.  Hence there exists an $m\geq 1$ such that~$g^{m}$
    commutes with every element of $\tau(H)$, that is $(g^{m})^{y} =
    g^{m}$ for every $y\in \tau(H)$.  However, it follows from
    Lemma~\ref{lem:aux3} that this can happen only if $\tau(H)$ is
    trivial.  Therefore $H$ is infinite cyclic.
\end{proof}

\begin{lemma}
    \label{lem:aux4}
    Under the assumptions of the previous lemma, if $H$ is an infinite
    cyclic subgroup of $G$ that is not a subgroup of $B$, and $y\in
    B$, then $|H\cap H^{y}| = \infty$ if and only if $y=1$.
\end{lemma}

\begin{proof}
    The ``if'' statement is trivial.  Therefore assume that $y\neq 1$
    and let $(x,r)$ be a generator of $H$.  Then $r\neq 0$ and
    \begin{equation*}
	(z,r) \defeq  (x,r)^{y}\neq(x,r)
    \end{equation*}
    is a generator of $H^{y}$ where the inequality is due to
    Lemma~\ref{lem:aux3}.  Suppose, for a contradiction, that
    $|H\cap H^{y}| = \infty$.  Then there must exist $k,l\in
    \Z\setminus \{0\}$ such that $(x,r)^{k} = (z,r)^{l}$.  In
    particular this implies that $k=l$. But then we get 
    \begin{equation*}
        (z,r)^{l} = (z,r)^{k} =  \bigl((x,r)^{y}\bigr)^{k}  = 
	\bigl((x,r)^{k}\bigr)^{y} \neq (x,r)^{k},
    \end{equation*}
    where the last inequality is again due to Lemma~\ref{lem:aux3},
    and so we achieve our desired contradiction.  Hence we must have
    $|H\cap H^{y}| \neq \infty$.
\end{proof}

As in~\cite[p.~502]{luck-12} we define an equivalence relation
``$\sim$'' on the set $\Fvc(G)\setminus \Ffin(G)$ by
\begin{equation*}
    H\sim K :\!\!\iff |H\cap K| = \infty.
\end{equation*}
We denote by $[H]$ the equivalence class of the group $H$.  If $K$ is
not finite then~$K\leq H$ implies that $K\sim H$.  Furthermore the
equivalence relation satisfies $H\sim K$ if and only $H^{g}\sim
K^{g}$.  Therefore the action of $G$ by conjugation on the set
$\Fvc(G)\setminus \Ffin(G)$ gives an action of $G$ on the set of
equivalence classes.  If $[H]$ is an equivalence class, then we denote
by $G_{[H]}$ the stabiliser of~$[H]$.

Given a subgroup $H$ of $G$, the \emph{commensurator} $\Comm_{G}(H)$
of
$H$ in $G$ is defined as the subgroup
\begin{equation*}
    \Comm_{G}(H) \defeq \{ g\in G :  |H : H\cap H^{g}| \text{ and } 
    |H^{g} : H\cap H^{g}| \text{ are finite} \}.
\end{equation*}
This subgroup is also known as the \emph{virtual normaliser}
$VN_{G}(H)$ of the subgroup $H$ in $G$.  In general it contains the
normaliser $N_{G}(H)$ of $H$ in $G$ as its subgroup.  In the case that
$H$ is a virtually cyclic subgroup of $G$ which is not finite we have
\begin{equation*}
    \Comm_{G}(H) = \{ g\in G : |H\cap H^{g}| = \infty\}.
\end{equation*}
In particular we have that $\Comm_{G}(H) = G_{[H]}$ in this case.

\begin{lemma}
    \label{lem:aux5}
    Assume that $\Z$ acts freely by conjugation on the set of
    non-trivial conjugacy classes of non-trivial elements of $B$.
    Then the commensurator $\Comm_{G}(H)$ is infinite cyclic for any
    virtually cyclic subgroup $H$ of $G$ that is not a subgroup
    of~$B$.
\end{lemma}

\begin{proof}
    Any such virtually cyclic subgroup $H$ of $G$ is infinite cyclic
    by Lemma~\ref{lem:aux2}.  Therefore $G_{[H]} = \Comm_{G}(H)$.
    Suppose that $G_{[H]}$ is not infinite cyclic.  Then the canonical
    projection $\pi\: B\rtimes Z \to \Z$ cannot map $G_{[H]}$
    isomorphically onto its image.  Hence there exists a non-trivial
    $y\in G_{[H]}\cap \ker(\pi) = G_{[H]} \cap B$.  Since $H$ is
    infinite cyclic we get $|H\cap H^{y}| \neq \infty$ by
    Lemma~\ref{lem:aux4} which is equivalent to $[H] \neq [H^{y}]$,
    and this is a contradiction to the assumption that $y\in G_{[H]}$.
    Therefore $G_{[H]} = \Comm_{G}(H)$ must be infinite cyclic.
\end{proof}

\begin{proposition}
    \label{prop:main-aux}
    Let $G$ be an arbitrary group and let $\frakF$ and $\frakG$ be
    families of subgroups of $G$ such that
    \begin{equation*}
        \Ffin(G) \subset \frakF \subset \frakG \subset\Fvc(G).
    \end{equation*}
    Assume that the commensurator $\Comm_{G}(H)\in \frakG$ for any 
    $H\in \frakG\setminus\frakF$, then every $H\in \frakG\setminus 
    \frakF$ is contained in a unique maximal element $H_{\max}\in 
    \frakG$ and~$N_{G}(H_{\max}) = H_{\max}$.
\end{proposition}

\begin{proof}
    Since $H$ is an infinite virtually cyclic subgroup of $G$ it 
    follows that $G_{[H]}= \Comm_{G}(H)$ and thus $G_{[H]}\in\frakG$ 
    by assumption.
    
    Trivially we have that $H\leq G_{[H]}$.  If $K\in \frakG$ with
    $H\leq K$, then $H\sim K$ since $H$ is not finite, and for any
    $k\in K$ we get $[H^{k}] = [K^{k}] = [K] = [H]$.  Therefore any
    $k\in K$ stabilises $[H]$.  This implies $K\leq
    G_{[H]}$ and thus $G_{[H]}$ is maximal and unique in
    $\frakG\setminus \frakF$, that is $H_{\max} = G_{[H]}$.
    
    Finally, $H_{\max}\leq N_{G}(H_{\max}) \leq \Comm_{G}(H_{\max}) = 
    G_{[H_{\max}]} = H_{\max}$ and hence $H_{\max} = 
    N_{G}(H_{\max})$.
\end{proof}

Together with Lemma~\ref{lem:aux2}, we get the following result:

\begin{corollary}
    \label{cor:main-aux}
    Let $G = B\rtimes \Z$ and assume that $\Z$ acts freely by
    conjugation on the set of conjugacy classes of non-trivial
    elements of $B$.  Then every $H\in \Fvc(G)\setminus \Fvc(B)$ is
    contained in a unique maximal element $H_{\max}\in
    \Fvc(G)\setminus \Fvc(B)$ and $N_{G}(H_{\max}) = H_{\max}$. 
    Furthermore $\Fvc(B)\cap H = \{1\}$ for any $H\in 
    \Fvc(G)\setminus\Fvc(B)$.\qed
\end{corollary}

%
%

\section{A Generalisation of Juan-Pineda and Leary's Construction}

Let $G$ be an arbitrary group and assume that $\frakF$ and $\frakG$
are two families of subgroups of $G$ which satisfy the conditions of
Proposition~\ref{prop:main-aux}.  Then we have the following
generalisation of Proposition~\ref{prop:JPL-orig}.

\begin{proposition}
    \label{prop:JPL}
    Let $\frakF$ be a full family and $\frakG$ a semi-full family of
    subgroups of $G$ with $\Ffin(G)\subset \frakF\subset \frakG\subset
    \Fvc(G)$.  Assume that every $H\in \frakG\setminus \frakF$ is
    contained in a unique maximal element $H_{\max}\in \frakG$ and
    $N_{G}(H_{\max}) = H_{\max}$.  Moreover, assume that $\frakF\cap H
    \subset \Ffin(H)$ for every $H\in \frakG\setminus \frakF$.
    Let~$\calC$ be a complete set of representatives of conjugacy
    classes of maximal elements in~$\frakG\setminus \frakF$.  Denote
    by $\calC_{o}$ the set of orientable elements of $\calC$ and
    denote by~$\calC_{n}$ the set of non-orientable elements of
    $\calC$.  Then a model for $E_{\frakG}G$ can be obtained from
    model for $E_{\frakF}G$ by attaching
    \begin{enumerate}
        \item  orbits of $0$-cells indexed by $\calC$;
    
        \item  orbits of $1$-cells indexed by $\calC_{o} \cup 
	\{1,2\}\times \calC_{n}$;
    
        \item  orbits of $2$-cells indexed by $\calC$.
    \end{enumerate}
    Furthermore, a model for $B_{\frakG}G$ can be obtained from a
    model for $B_{\frakF}G$ by attaching $2$-cells indexed by
    $\calC_{o}$.
\end{proposition}

\begin{proof}
    \setcounter{claim}{0}
    We only need to verify that Juan-Pineda and Leary's construction
    works unchanged in the more general setting.  We fix a model
    $E$ for~$E_{\frakF}G$.
    
    Let $H\in \calC$.  By~\cite[p.~137]{juan-pineda-06} we can choose
    a $1$-dimensional model~$E_{H}$ for $\underline EH$ which
    homeomorphic to the real line and such that $E_{H}/H$ is a loop if
    $H\in \calC_{o}$ or a line segment if $H\in \calC_{n}$.  Denote by
    \begin{equation*}
	Z_{H} \defeq  G\times_{H} E_{H}
    \end{equation*}
    the $G$ space which is induced from the $H$-space
    $E_{H}$~\cite[pp.~52ff.]{kawakubo-91}.  If $[g,x]\in Z_{H}$, then
    $G_{[g,x]} = G_{g[1,x]} = (G_{[1,x]})^{g^{-1}}$.  Since $G_{[1,x]}
    = H_{x}$ and $H_{x}$ is a finite subgroup of $H$ it follows that
    $G_{[g,x]}\in \Ffin(G)$.  Therefore it follows that
    $\frakF(Z_{H})\subset \Ffin(G) \subset \frakF$ and there exists a
    $G$-map $f_{H}\: Z_{H}\to E$ by the universal property of $E$.
    
    Furthermore, we set
    \begin{equation*}
        X_{H}\defeq  G/H,
    \end{equation*}
    which is a discrete transitive $G$-set.  There exists a
    $G$-equivariant projection $\pi_{H}\: Z_{H} \to X_{H}$ which maps
    $[g,x]$ to $gH$ and an $H$-equivariant inclusion $i_{H}\: E_{H}\to
    Z_{H}$ given by $i_{H}(x) \defeq [1,x]$.  Denote by $V_{gH} \defeq
    \pi_{H}^{-1}(gH)$.  Clearly~$i_{H}(E_{H}) \subset V_{H}$.  On the
    other hand, if $[h,x]\in V_{H}$, then $hx\in E_{H}$ such
    that~$i_{H}(hx) = [1,hx] = [h,x]$.  That is $i_{H}(E_{H}) \subset
    V_{H}$ and we have the equality $i_{H}(E_{H}) = V_{H}$.  Since $G$
    is discrete, it follows that $i_{H}$ is an open map and in
    particular it maps $E_{H}$ homeomorphically onto $V_{H}$.  Let
    $R_{H}$ be a complete system of representatives of the left cosets
    $G/H$.  Since $G/H$ is discrete we have that $Z_{H}$ is the
    disjoint union
    \begin{equation*}
        Z_{H} = \coprod_{g\in R_{H}} V_{gH} 
	= \coprod_{g\in R_{H}} gV_{H}
    \end{equation*}
    of contractible subspaces $gV_{H}$, $g\in R_{H}$, which are 
    permuted by the action of~$G$. 
    
    \begin{claim}	
	Let $K$ be a finite subgroup of $G$.  Then the projection
	$\pi_{H}$ induces a homotopy equivalence
	\begin{equation*}
	    \pi_{H}\:(Z_{H})^{K} \to (X_{H})^{K}.
	\end{equation*}
    \end{claim}
    
    Let $[g,x]\in (Z_{H})^{K}$.  Without any loss of generality we may
    assume that~$g\in R_{H}$.  Now $[g,x]\in (Z_{H})^{K}$ implies
    {\allowdisplaybreaks
    \begin{align*}
         & \forall k\in K\: k[g,x] = [g,x]
	 \\
	 \Longleftrightarrow \quad 
	 & \forall k\in K\: [g^{-1}kg,x] = [1,x]
	 \\
	 \Longleftrightarrow \quad
	 &  \forall k\in K\: \exists h\in H\: g^{-1}kgh^{-1} = 1 
	 \text{ and } hx=x
	 \\
	 \Longleftrightarrow \quad
	 & \forall k\in K\: \exists h\in H_{x}\: g^{-1}kg = h
	 \\
	 \Longleftrightarrow \quad
	 & K^{g}\leq H_{x}
	 \\
	 \Rightarrow \quad
	 & x\in \{ y\in E_{H} : K^{g}\leq H_{y}\} = (E_{H})^{K^{g}} 
	 \text{ and $K$ is finite}.
    \end{align*}}
    Thus $[1,x]\in (V_{H})^{K^{g}} = \{ [1,x] : x\in
    (E_{H})^{K^{g}})\}$ and therefore $[g,x] = g[1,x] \in
    g\bigl((V_{H})^{K^{g}}\bigr)$.  On the other hand, let $[1,x]\in
    (V_{H})^{K^{g}}$.  Then
    \begin{equation*}
	k[g,x] = g(g^{-1}kg)[1,x] = g [1,x] = [g,x]
    \end{equation*}
    for all $k\in K$.  Therefore $[g,x]\in
    (Z_{H})^{K}$ (and $K$ is finite as before).
    
    Altogether this shows
    \begin{equation*}
        (Z_{H})^{K} = \coprod_{g\in R_{H}} g\bigl( 
	(V_{H})^{K^{g}}\bigr)
    \end{equation*}
    is the disjoint union of the subspaces
    $g\bigl((V_{H})^{K^{g}}\bigr)$.  Since $K$ is assumed to be finite
    we have that
    \begin{enumerate}
	\item $(V_{H})^{K^{g}}$ is contractible if $K^{g}\leq H$;
    
	\item $(V_{H})^{K^{g}}=\emptyset$ otherwise.
    \end{enumerate}
    On the other hand $gH\in (G/H)^{K}$ if and only if $K^{g}\leq H$.
    It follows that $\pi_{H}$ induces map
    \begin{equation*}
        \pi_{H}\: (Z_{H})^{K} \to (X_{H})^{K}
    \end{equation*}
    which maps the contractible components $g\bigl ((V_{H})^{K^{g}})$ 
    of $Z_{H}$ in an one-to-one way onto the discrete space 
    $(X_{H})^{K} = (G/H)^{K}$. It follows that $\pi_{H}$ induces a 
    homotopy equivalence $(Z_{H})^{K} \to (X_{H})^{K}$ and the claim 
    follows.
    
    \medskip
    
    We set
    \begin{equation*}
	Z \defeq  \coprod_{H\in \calC} Z_{H}
    \end{equation*}
    and the $G$-maps $f_{H}$ and $\pi_{H}$ give rise to $G$-maps $f\:
    Z\to E$ and $\pi\: Z\to X$ with $X \defeq  \coprod_{H\in \calC} 
    X_{H}$. Note that $\pi$ induces a homotopy equivalence $Z^{K}\to 
    X^{K}$ for every finite subgroup $K$ of $G$.
       
    \begin{claim}
	\label{clm:JPL2}
	The $G$-map
	\begin{equation*}
	    (f, \pi) \: Z\to E\times X
	\end{equation*}
	given by $[g,x] \mapsto \bigl(f([g,x]), \pi([g,x])\bigr)$ is a
	$G$-homotopy equivalence.
    \end{claim}
     
    Let $K$ be a subgroup of $G$ such that $(E\times X)^{K} =
    E^{K}\times X^{K} \neq \emptyset$.  Then~$E^{K}\neq \emptyset$ and
    $X^{K}\neq \emptyset$.  The condition $E^{K}\neq \emptyset$
    implies that $K\in \frakF$ and~$X^{K}\neq \emptyset$ implies that
    there exists a $H\in \calC$ and a $g\in G$ such that~$K^{g}\leq
    H$.  Thus $K^{g} = K^{g}\cap H\in \frakF\cap H \leq \Ffin(H)$.
    Therefore $K$ is a finite subgroup of $G$.  Since $\Ffin(G)\subset
    \frakF$ it follows that $E^{K}$ is contractible by the universal
    property of $E$.  Moreover, since $K$ is finite, it follows that
    $\pi$ induces a homotopy equivalence $Z^{K}\to X^{K}$.
    
    It follows that $(f,\pi)$ induces a homotopy equivalence $Z^{K}\to
    E^{K}\times X^{K} = (E\times X)^{K}$.  To see this, denote by
    $\theta$ the homotopy inverse of $\pi$ restricted to~$X^{K}$, and
    denote by $p$ the projection $E^{K}\times X^{K}\to X^{K}$.  Let
    \begin{equation*}
        \tilde f\: E^{K}\times X^{K} \to Z^{K}
    \end{equation*}
    be the composite map $\tilde f \defeq \theta\circ p$.  Then
    $\tilde f \circ (f,\pi) = \theta \circ p \circ (f,\pi) =
    \theta\circ \pi \homotop \id$.  On the other hand we have that
    $\pi\circ \theta = \id$ since $X$ is discrete.  Therefore
    \begin{equation*}
        (f,\pi) \circ \tilde f = (f\circ \theta\circ p, \pi\circ 
	\theta\circ p) =  (f\circ \theta\circ p, p)
    \end{equation*}
    maps $(x,gH)\mapsto((f\circ\theta)(gH), gH)$ for all $(x,gH)\in
    E^{K}\times X^{K}$.  Since $E^{K}$ is contractible it follows that
    $(f,\pi)\circ\tilde f\homotop \id$. Altogether this shows that 
    $(f,\pi)$ has a homotopy inverse and therefore it is a homotopy 
    equivalence.
    
    Since $(f,\pi)$ is a $G$-map and $K$ has been an arbitrary
    subgroup of $G$ such that $(E\times X)^{K}\neq \emptyset$ we can
    apply the Equivariant Whitehead Theorem~\cite[p.~36]{luck-89} and
    we get that $(f,\pi)$ is a $G$-homotopy equivalence and this
    proves Claim~\ref{clm:JPL2}.
    
    \medskip
    
    As in~\cite[p.~140]{juan-pineda-06} we can attach $Z\times [0,1]$
    to the disjoint union of $E$ and $X$, identifying $(z,0)$ with
    $f(z)\in E$ and $(z,1)$ with $\pi(z)\in X$.  Denote this space by
    $\tilde E$.  Since $Z\homotop_{G} E\times X$ it follows that
    $\tilde E$ is $G$-homotopy equivalent to the join $E*X$ of $E$ and
    $X$.  Note that the join inherits a natural~$G$-CW-complex
    structure from $E$ and $X$.
   
    \begin{claim}
        The join $E*X$ is a model for $E_{\frakG}G$.
    \end{claim}
    
    If $K\in \frakF(E*X)$, then at least one of the following cases
    does hold:
    \begin{enumerate}
	\item $K\in \frakF(E)\subset \frakF\subset \frakG$;
    
	\item $K\in \frakF(X)\subset \frakG$;
	
	\item $K = K_{1}\cap K_{2}$ with $K_{1}\in \frakF\subset
	\frakG$ and $K_{2}\in \frakG$.
    \end{enumerate}
    Since $\frakG$ is semi-full it follows that also in the last case
    $K\in \frakG$ holds.  Altogether $\frakF(E*X)\subset \frakG$. 
    
    If $K\in \frakF$ then $E^{K}$ is contractible and if $K\in
    \frakG\setminus \frakF$, then $X^{K}$ consists of a single point
    and is therefore contractible.  It follows that $(E*X)^{K}$ is
    contractible in both cases, that is for every $K\in \frakG$.
    
    Altogether $E*X$ is a model for $E_{\frakG}G$ by
    Proposition~\ref{prop:classifying-space-alt} and this proves the
    claim.
    
    \medskip
    
    Since $E*X\homotop_{G} \tilde E$ it follows that $\tilde E$ is a
    model for $E_{\frakG}G$, too.  It follows that $\tilde E$ is 
    obtained from $E$ by attaching orbits of $0$-, $1$- and $2$-cells 
    as described in the proposition.
    
    The remaining claim about the construction of a model for
    $B_{\frakG}G$ from a model for $B_{\frakF}G$ follows from the
    argument which proved Corollary~10
    in~\cite[p.~141]{juan-pineda-06}. This concludes the proof of 
    Proposition~\ref{prop:JPL}.
\end{proof}

Note that in the case $\frakG= \Fvc(G)$ and $\frakF=\Ffin(G)$ we
recover the original statement of Proposition~\ref{prop:JPL-orig}.
However we apply it to the case that $G = B\rtimes\Z$, $\frakF =
\Fvc(B)$ and $\frakG = \Fvc(G)$.  If $\Z$ acts freely by conjugation
on the set of conjugacy classes of non-trivial elements of $B$, then
Corollary~\ref{cor:main-aux} tells us that we can use
Proposition~\ref{prop:JPL} in order to construct a model for $\uu EG$
from a model for $E_{\Fvc(B)}G$.  However, in order to obtain this way
a nice model for~$\uu EG$ we need to have a nice model for
$E_{\Fvc(B)}G$ to start with.  In the next section we will give a
general construction for such a model if a nice model for $\uu EB$ is
given.

%
%

\section{Constructing a Model for $E_{\frakF}G$ from a Model for
$E_{\frakF}B$}
\label{sec:model}

We carry out the construction in a setting that is more general than
in the previous section.  Let $G\defeq  B\rtimes \Z$ be an arbitrary
infinite cyclic extension, where $\Z$ acts on $B$ via an automorphism
$\varphi\in \Aut(B)$.  Let $\frakF$ be a family of subgroups of $B$.
We assume that $\frakF$ is invariant under the automorphism $\varphi$,
that is $\varphi^{k}(H)\in \frakF$ for every $H\in \frakF$ and $k\in
\Z$.  This implies that $H\in \frakF$ if and only if
$\varphi(H)\in\frakF$ for any subgroup $H$ of~$B$.  Furthermore this
implies that $\frakF$ is not just a family of subgroups of~$B$ but
also a family of subgroups of~$G$.

We begin our construction with the assumption that we are given a
model~$X$ for $E_{\frakF}B$.  For each $k\in \Z$ let $X_{k}$ be a
copy of $X$ seen as a set. We define a $B$-action
\begin{equation*}
    \Phi_{k}\: B\times X_{k} \to X_{k}
\end{equation*}
on $X_{k}$ by $\Phi_{k}(g,x) \defeq  \varphi^{-k}(g) x$.  Note that each
$X_{k}$ is a model for $E_{\frakF}B$ since~$\frakF$ is assumed to be
invariant under the automorphism $\varphi$.

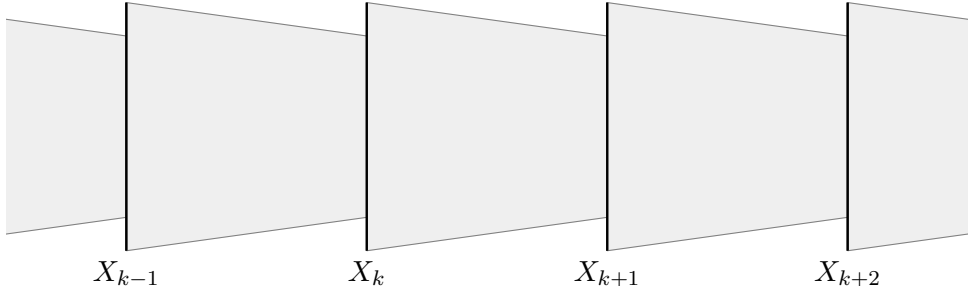
\begin{figure}[tbp]
    \begin{center}
        \begin{tikzpicture}[scale=\ratio{\textwidth}{10cm}]	
	    
        \clip  (-3.75,-1.66) rectangle (6.25,1.65);        
            
	\foreach \i in {-2,-1,0,1,2} {
	    
	    \coordinate (A1) at (2.5*\i    ,-1.3);
	    \coordinate (B1) at (2.5*\i    , 1.3);
	    \coordinate (A2) at (2.5*\i + 2.5,-1.3);
	    \coordinate (B2) at (2.5*\i + 2.5, 1.3);
	    \coordinate (a2) at (2.5*\i + 2.5,-0.95);
	    \coordinate (b2) at (2.5*\i + 2.5, 0.95);
	    
	    \filldraw [lightgray, nearly transparent]
	    (A1) -- (a2) -- (b2) -- (B1) -- cycle;
	
	    \draw [gray] (A1) -- (a2) (B1) -- (b2);
	    
	    \draw [thick] 
	    (B1) -- (A1)
	    (B2) -- (A2);
            }
	    
	    \draw
	    (-2.5,-1.3) node [below]{$X_{k-1}$}
	    (0,-1.3) node [below]{$X_{k}$}
	    (2.5,-1.3) node [below]{$X_{k+1}$}
	    (5,-1.3) node [below]{$X_{k+2}$};
	\end{tikzpicture}
    \end{center}
    \caption{A schematic picture of the $B$-CW-complex $Y$.}
    \label{fig:model}
\end{figure}

Since $X_{0}$ and $X_{1}$ are models for $E_{\frakF}B$ there exists a
$B$-map $f\: X_{0} \to X_{1}$.  In other words $f$ is a continuous map
$f\: X\to X$ which satisfies $f(gx) = \varphi^{-1}(g)f(x)$ for every
$x\in X$ and $g\in B$.  By the equivariant Cellular Approximation
Theorem~\cite[p.~32]{luck-89} we may assume without loss of
generality that
$f$ is cellular.  Denote by $X_\infty$ the disjoint union of
$B$-spaces
\begin{equation*}
    X_{\infty} \defeq  \coprod_{k\in \Z} (X_{k} \times [0,1])
\end{equation*}
and let $Y$ be the quotient space
\begin{equation*}
    Y \defeq  X_{\infty} / \sim
\end{equation*}
under the equivalence relation generated by $(x,1) \sim (f(x), 0)$
whenever $x\in X_{k}$ and $f(x)\in X_{k+1}$ for some $k\in \Z$.  Since
$f$ is a cellular $B$-map it follows that $Y$ is a $B$-CW-complex.
Essentially, it is a mapping telescope which extends to infinity in
both directions, see Figure~\ref{fig:model}.  Note that if $X$ is
an~$n$-dimensional $B$-CW-complex, then $Y$ is $(n+1)$ dimensional
$B$-CW-complex.

\begin{lemma}
    \label{lem:pre-model}
    The $B$-CW-complex $Y$ is a model for $E_{\frakF}B$.
\end{lemma}

\begin{proof}
    Let $H$ be a subgroup of $B$ such that $H\notin \frakF$ and let 
    $x\in X_{k}$ for some $k\in \Z$. Since $\frakF$ is assumed to be 
    invariant under the automorphism $\varphi$ we have 
    $\varphi^{-k}(H) \notin \frakF$. Therefore there exists a $h\in 
    H$ such that $\varphi^{-k}(h)x \neq x$. But then
    \begin{equation*}
        \Phi_{k}(h,x) = \varphi^{-k}(h)x \neq x,
    \end{equation*}
    which implies that $x\notin X_{k}^{H}$.  It follows that
    $X_{k}^{H} = \emptyset$ for all $k\in \Z$.
    Hence $Y^{H}=\emptyset$.
    
    On the other hand, consider the case that $H\in \frakF$.  We want
    to show that $Y^{H}$ is contractible.  Since the subcomplex
    $Y^{H}$ has the structure of a CW-complex it is enough to show
    that $Y^{H}$ is weakly
    contractible~\cite[pp.219ff.]{whitehead-78}.  Since the family
    $\frakF$ is assumed to be invariant under the automorphism
    $\varphi$ it follows that $\varphi^{k}(H)\in \frakF$ for every
    $k\in \Z$.  Then $X_{k}^{H} = X^{\varphi^{k}(H)}$ is contractible
    for every $k\in \Z$.  It follows that $Y^{H}$ is an infinite
    mapping telescope of the collection $X^{H}_{k}$ of contractible
    spaces.  Any image of an~$n$-sphere in $Y^{H}$ will be contained
    in a finite subtelescope.  A finite subtelescope of~$Y^{H}$
    deformation retracts onto its right-hand end space which is 
    contractible. Terefore all homotopy groups of~$Y^{H}$ are 
    trivial, that is $Y^{H}$ is weakly contractible.
\end{proof}

For every $(x,t) \in X_{k}\times [0,1]$ and $(g,r)\in
G$ set 
\begin{equation*}
    \Psi((g,r),(x,t)) \defeq  (\Phi_{k+r}(g,x), t) \in
    X_{k+r} \times [0,1].
\end{equation*}
Straight forward calculation shows that this induces a well defined
action
\begin{equation*}
    \Psi\: G\times Y \to Y
\end{equation*}
of $G$ on $Y$, which extends the already existing $B$-action on
$Y$.  If $(g,r)\in G\setminus B$, then $r\neq 0$ and therefore clearly
$\Psi((g,r),x) \neq x$ for any $x\in Y$.  Then together with
Lemma~\ref{lem:pre-model} this implies that $Y$ is an
$(n+1)$-dimensional model for~$E_{\frakF}G$.  Altogether we have then
shown the following result.

\begin{proposition}
    \label{prop:model}
    Let $G= B\rtimes \Z$ be an arbitrary infinite cyclic extension
    where $\Z$ acts on $B$ via an automorphism $\varphi\in \Aut(B)$.
    Let $\frakF$ be a family of subgroups of $B$ which is invariant
    under the automorphism $\varphi$.  If there exists an
    $n$-dimensional model for $E_{\frakF}B$ then there exists an
    $(n+1)$-dimensional model for~$E_{\frakF}G$.
    \qed
\end{proposition}

%
%

\section{Examples}
\label{sec:examples}

Strictly descending HNN-extensions are a natural source for candidates
for infinite cyclic extensions $G=B\rtimes \Z$ where $\Z$ acts freely
by conjugation on the set of conjugacy classes of the non-trivial
elements of $B$.

The general setup is the following.  Let $B_{0}$ be a group and let
$\varphi\: B_{0}\to B_{0}$ a monomorphism.
Recall that the descending HNN-extension determined by this data is
the group $G$ given by the presentation
\begin{equation*}
    G \defeq  \langle B_{0}, t \mid t^{-1}xt = \varphi(x)
    \text{ for all $x\in B_{0}$}\rangle
\end{equation*}
and this group is usually denoted by $B_{0} *_{\varphi}$ in the
literature.  The group~$B_{0}$ is called the base group of the
HNN-extension.  The HNN-extension is called \emph{strictly descending}
if the monomorphism $\varphi$ is not an isomorphism.  We consider
$B_{0}$  as a subgroup of $G$ in the obvious way.

Conjugation by $t\in G$ defines an automorphism of $G$ which agrees 
on $B_{0}$ with $\varphi$ which we will therefore denote by the same 
symbol. In other words, the monomorphism $\varphi\: B_{0} \to B_{0}$ 
extends to the whole group $G$ if we set
\begin{equation*}
    \varphi\: G\to G, x\mapsto \varphi(x) \defeq  t^{-1}xt.
\end{equation*}
For each $k\in \Z$ we set $B_{k} \defeq  \varphi^{k}(B_{0})$. In this way
we 
obtain a descending sequence
\begin{equation*}
    \ldots \supset B_{-2} \supset B_{-1} \supset B_{0}
    \supset B_{1} \supset B_{2} \supset \ldots
\end{equation*}
of subgroups of $G$.  This sequence of subgroups is strictly
descending if and only if the HNN-extension is strictly descending.
We denote the directed union of all these $B_{k}$ by~$B$.  The
automorphism $\varphi$ restricts to an automorphism of $B$ which is
therefore a normal subgroup of $G$.  It is standard fact that we can
write $G$ as the semidirect product $G = B\rtimes \Z$ where $\Z$ acts
on $B$ via the automorphism~$\varphi$ restricted to~$B$.

\begin{lemma}
    \label{lem:example}
    Assume that $\varphi^{k}(x)\neq x$ for all non-trivial $x\in
    B_{0}$ and all $k\geq 1$.  Given $x\in B_{0}$, denote by $[x]$ the
    set of all elements in $B_{0}$ which are conjugate in $B_{0}$ to
    $x$.  Assume that for each $x\in B$ we are given a finite
    subset~$[x]' \subset [x]$, which only depends on the conjugacy
    class~$[x]$ of~$x$ in~$B_{0}$, such that $\varphi([x]') \subset
    [\varphi(x)]' $ for every $x\in B_{0}$.  Then $\Z$ acts freely on
    the set of conjugacy classes of non-trivial elements of $B$.
\end{lemma}

\begin{proof}
    We suppose that $\Z$ does not act freely on the set of conjugacy
    classes of non-trivial elements of $B$.  Then there exists $x\in
    B$ and $n\geq 1$ such that $\varphi^{n}(x)$ is conjugate in $B$ to
    $x$.  Without any loss of generality we may assume that $x\in
    B_{0}$ (otherwise replace $x$ by $\varphi^{k}(x)$ for a
    suitable $k\in \N$).  Furthermore, without any loss of generality
    we may assume that $x\in [x]'$.  Finally we may assume without any
    loss of generality that $\varphi^{n}(x)$ is actually conjugate in
    $B_{0}$ to $x$ (otherwise, again, replace $x$ by $\varphi^{k}(x)$
    for a suitable~$k\in \N$).
    
    Now $\varphi^{rn}(x)\in [x]'$ for any $r\geq 1$.  Since $[x]'$ is
    finite this implies that $\varphi^{rn}(x) = \varphi^{sn}(x)$ for
    some $s>r$.  Therefore $\varphi^{(s-r)n}(x) = x$ and
    since $(s-r)n>0$ we obtain a contradiction to the assumption of
    the lemma.  Therefore the action of $\Z$ on the conjugacy classes
    of non-trivial elements of~$B$ must be trivial.
\end{proof}  

\begin{example}
    \label{ex:abelian-base-group}
    Let $B_{0}$ be an abelian group and $\varphi\: B_{0}\to B_{0}$ a
    monomorphism such that $\varphi^{k}(x)\neq x$ for every
    non-trivial $x\in B_{0}$ and $k\geq 1$.  Since $B_{0}$ is abelian,
    each conjugacy class $[x]$ of elements in $B_{0}$ contains
    precisely one element and the conditions of
    Lemma~\ref{lem:example} are trivially satisfied.  Thus $\Z$ acts
    freely by conjugation on the set of non-trivial elements of $B$.
    In particular we can use Proposition~\ref{prop:JPL} to obtain a
    model for $\uu EG$ from a model for $E_{\Fvc(B)}G$.
\end{example}

Let $B_{0}$ be a free group.  An element $x\in B_{0}$ is called
\emph{cyclically reduced} if it cannot be written as $x = u^{-1}yu$
for some non-trivial $u, y\in B_{0}$. It follows 
from~\cite[pp.~33ff.]{magnus-76} that every element $x\in B_{0}$ is 
conjugate to a cyclically reduced element $x'$ and that there are 
only finitely many cyclically reduced elements in $B_{0}$ which are 
conjugate to $x$. Therefore
\begin{equation*}
    [x]' \defeq  \{ x' \in [x] : \text{$x'$ is cyclically reduced}\}
\end{equation*}
is a finite subset of $[x]$ for every $x\in B_{0}$.

\begin{example}
    \label{ex:free-base-group}
    Let $X$ be an finite non-empty set and let $B_{0} \defeq  F(X)$ be the
    free group on the basis $X$.  Let $\{\alpha_{x}\}_{x\in X}$ be a
    collection of integers such that $|\alpha_{x}|\geq 2$ for every
    $x\in X$.  Consider the endomorphism $\varphi\: B_{0}\to B_{0}$
    that maps any basis element $x$ to $x^{\alpha_{x}}$.  It follows
    that $\varphi$ is a monomorphism which satisfies the assumptions
    of Lemma~\ref{lem:example}.  Therefore we can use
    Proposition~\ref{prop:JPL} to construct a model for $\uu EG$ from
    a model for $E_{\Fvc(B)}G$.
\end{example}

\begin{example}
    Another example of a strictly descending HNN-extension (in
    disguise) is the restricted wreath product $A\wr \Z$ of an
    arbitrary group $A$ by~$\Z$ which is defined as follows.  Let
    $A_{k}$ be a copy of $A$ for each $k\in \Z$.  Let $B$ be the
    direct product of all these $A_{k}$ and let $\Z$ act on $B$ via
    $\varphi$ which maps~$A_{k}$ identically onto $A_{k+1}$ for all
    $k\in \Z$.  Then
    \begin{equation*}
        A\wr \Z \defeq  B\rtimes \Z.
    \end{equation*}
    Since each $A_{k}$ is normal in $B$ the above definition of
$\varphi$
    forces the action of $\Z$ on the set of conjugacy classes of
    non-trivial elements of $B$ to be free.  Therefore we can apply
    Proposition~\ref{prop:JPL} in this case, too.
\end{example}

%
%

\section{Dimensions}

Given a family $\frakF$ of subgroups of $G$, a model for $E_{\frakF}G$
is only defined uniquely up to $G$-homotopy.  Consider a model for
$E_{\frakF}G$.  One particular invariant of the group $G$ is called
the \emph{geometric dimension} of $G$ with respect to the family
$\frakF$, and this is defined as being the least possible dimension of
a model for $E_{\frakF}G$.  It is denoted by $\gd_{\frakF}G$ and may
be infinite.  In the case that~$\frakF =\{1\}$ we recover the
classical geometric dimension of the group $G$.  In the case that
$\frakF = \Fvc(G)$ we denote the geometric dimension by $\uugd G$.

\begin{proposition}
    \label{prop:dimension}
    Let $G = B\rtimes \Z$ and assume that $\Z$ acts freely via 
    conjugation on the conjugacy classes of non-trivial elements of 
    $B$. Then 
    \begin{equation*}
	\uugd B\leq \uugd G \leq \uugd B + 1.
    \end{equation*}
\end{proposition}

\begin{proof}
    Since (in general) a model for $\uu EG$ is always a model for $\uu
    EB$ via restriction, we have that the second inequality is the
    only non-trivial one.  If $X$ is an $n$-dimensional model for $\uu
    EB$, then the telescope construction in Section~\ref{sec:model}
    gives an $(n+1)$-dimensional model for $E_{\Fvc(B)}G$.
    
    By Lemma~\ref{lem:aux1a} the group $B$ cannot be virtually cyclic.
    Therefore $n+1 \geq 2$ and attaching cells of dimension at most
    $2$ does not increase the dimension of the resulting space.  Hence
    Proposition~\ref{prop:JPL} yields an $(n+1)$-dimensional model for
    $\uu EG$ and this concludes the proof.
\end{proof}

\begin{corollary}
    \label{cor:dimension}
    Let $G = B_{0} *_{\varphi}$ be a descending HNN-extension as in 
    Section~\ref{sec:examples}. If $G = B\rtimes \Z$ satisfies the 
    conditions of the previous proposition then
    \begin{equation*}
	\uugd B_{0} \leq \uugd G \leq  \uugd B_{0} + 2.
    \end{equation*}
\end{corollary}

\begin{proof}
    As exploited previously, since $B_{0}$ is a subgroup of $G$, the
    second inequality is the only non-trivial part of the statement.
    The group $B$ is the countable direct union of the conjugates of
    $B_{0}$ in $G$.  Therefore an $n$-dimensional model for $\uu
    EB_{0}$ gives rise to an $(n+1)$-dimensional model for~$\uu EB$ by
    a construction of Lück and Weiermann~\cite[pp.~510ff.]{luck-12}.
    Now the claim follows from the previous proposition.
\end{proof}

\begin{example}
    Let $G = B_{0} *_{\varphi}$ be a descending HNN-extension with
    $B_{0}$ a free group of finite rank.  If $B_{0}$ has rank $1$,
    then $G$ is a soluble Baumslag--Solitar group and this case is
    treated below in Theorem~\ref{thrm:dim-soluble-BS}.  Thus we may
    assume that $B_{0}$ has rank at least $2$.  Free groups are
    torsion-free and act freely on a tree which is therefore a
    $1$-dimensional model for $\underline{E} B_{0}$.  Free groups of
    finite rank are Gromov-hyperbolic and therefore Proposition~9
    in~\cite{juan-pineda-06} states the existence of a $2$-dimensional
    model for $\uu E B_{0}$.  On the other hand by Remark~16
    in~\cite{juan-pineda-06} there cannot exist a model for $\uu
    EB_{0}$ less than $2$.  Therefore $\uugd B_{0} =2$.  Now the
    direct union $B$ of all conjugates of $B_{0}$ in $G$ is locally
    free and therefore does not contain a subgroup of isomorphic
    to~$\Z^{2}$.  Then Lemma~\ref{lem:aux1} states that we can apply
    Corollary~\ref{cor:dimension} if and only if $G$ does not contain
    a subgroup isomorphic to~$\Z^{2}$.  Therefore we get in this case
    the estimation $2 \leq \uugd G \leq 4$.
\end{example}

\begin{example}
    \label{ex:wreath-product}
    Consider the restricted wreath product $G = A\wr \Z$ where $A$ is
    a countable locally finite group.  Then
    \begin{equation*}
	B \defeq  \coprod_{k\in \Z} A
    \end{equation*}
    is also a countable locally finite group.  Since $B$ is not finite
    it follows that~$\uugd B=1$ by Lemma~\ref{lem:luck-1}.  We have
    seen that $G$ does satisfy the requirements of
    Proposition~\ref{prop:dimension}.  Therefore we get the estimate
    $1 \leq \uugd G \leq 2$.  We will see in the next chapter with
    Corollary~\ref{cor:case1-not-fg-1}, that $\uugd G=1$ implies
    that $G$ is locally virtually cyclic.  However $G$ is not locally
    virtually cyclic and therefore we $\uugd G\neq 1$.  Thus we have
    altogether
    \begin{equation*}
        \uugd G = 2.
    \end{equation*}
    Note that the smallest concrete example of a group of this type is
    the Lamplighter group $L = \Z_{2} \wr \Z$ where $\Z_{2}$ is the
    cyclic group of the integers modulo~$2$.
\end{example}

%
%

\section{Soluble Baumslag--Solitar Groups}
\label{sec:BS-groups}

We conclude this chapter with a complete answer to the geometric
dimension of the soluble Baumslag--Solitar groups with respect to the
family of virtually cyclic subgroups.  These groups belong to a class
of two-generator and one-relator groups introduced by Baumslag and
Solitar in~\cite{baumslag-62}. Their class contains all the groups 
\begin{equation*}
    BS(m,n) = \langle x,t \mid t^{-1}x^{m}t  = x^{n}\rangle.
\end{equation*}
where $m$ and $n$ are non-zero integers.  The soluble
Baumslag--Solitar groups are the groups of the form $BS(1,m)$, $m\neq
0$ and these groups can also be written as
\begin{equation*}
    BS(1,m) = \Z[1/m]\rtimes \Z,
\end{equation*}
where $\Z[1/m]$ is the subgroup of the rational numbers $\Q$ generated
by all powers of $1/m$ and where $\Z$ acts on $\Z[1/m]$ by
multiplication with $m$.  The group $BS(1,1)$ is $\Z^{2}$ and
$BS(1,-1)$ is the Klein bottle group $\Z\rtimes \Z$.  If~$|m|\geq 2$,
then $BS(1,m)$ belongs to the case described in
Example~\ref{ex:abelian-base-group}, as well as to the case described
in Example~\ref{ex:free-base-group}.

\begin{theorem}
    \label{thrm:dim-soluble-BS}
    Let $G=\Z[1/m]\rtimes \Z$ be a soluble Baumslag--Solitar group.
    Then
    \begin{equation*}
        \uuhd G = \uucd G = \uugd G =
	\begin{cases}
	    3 & \text{if $|m| = 1$,}  \\[1ex]
	    2 & \text{if $|m| \geq 2$.}
	\end{cases}
    \end{equation*}
\end{theorem}

\begin{proof}
    The case $|m|=1$ has been answered in the previous chapter.  Thus
    we assume that $|m|\geq 2$.  In this case $G$ is the fundamental
    group of a graph $(G,Y)$ of groups in the sense of~\cite{serre-80}
    where $Y$ is a loop and where the vertex groups are all infinite
    cyclic.  Let $X$ be the Bass--Serre tree associated with this
    graph of groups.  Then $T$ is not only a model for~$\uu E \Z[1/m]$
    but also a model for~$E_{\Fvc(\Z[1/m])} G$.  We can apply
    Proposition~\ref{prop:JPL} and obtain a model $X$ for $\uu EG$ by
    attaching cells of dimension less or equal to~$2$ to~$T$.
    Therefore we get $\uugd G\leq 2$.
    
    In order to see that $\uuhd G\geq 2$ we calculate $H_{2}(\uu BG)$.
    Note that $Y = X/G$ is a model for $B_{\Fvc(B)}G$ and $Y$ consists
    of one $0$-cell and one $1$-cell.  The second part of
    Proposition~\ref{prop:JPL} states that we can obtain a model for
    $\uu BG$ by attaching $2$-cells to $Y$ indexed by the conjugacy
    classes of maximal virtually cyclic subgroups of $G$ that are not
    contained in $\Z[1/m]$.  But there are infinitely many of them.
    Therefore $H_{2}(\uu BG) \neq 0$ which implies that $\uuhd G\geq 
    2$.
    
    Altogether we get $2\leq \uuhd G \leq \uucd G \leq \uugd G\leq 2$ 
    and thus equality holds.
\end{proof}

%
%

\section{Relatively Hyperbolic Groups and Free Products}

In the literature a common strategy to construct a model for $\uu EG$ 
is to begin with a known model for $\underline EG$ and attach cells 
in order to obtain a model for $\uu EG$. One key idea in the 
construction of models for $\uu EG$ in this chapter has been to begin 
with a model for $E_{\frakF}G$ with $\Ffin(G)\subset \frakF\subset 
\Fvc(G)$ where~$\frakF$ is a family of subgroups of $G$ which in 
general is larger than $\Ffin(G)$. In what follows we give another 
example for a fruitful application of this idea.

Let $G$ be a group and $H_{\lambda}$, $\lambda\in \Lambda$, a
collection of subgroups of $G$.  Assume that $G$ is relatively
hyperbolic with respect to the subgroups $H_{\lambda}$ in the sense
of~\cite{osin-06a}.  The subgroups $H_{\lambda}$ are called the
peripheral subgroups of $G$.  Consider the set
\begin{equation}
    \frakF \defeq  \{ H^{g} : H\in \Fvc(H_{\lambda}), \lambda\in \Lambda,
    g\in G \} \cup \Ffin(G),
    \label{eq:F-rel-hyp}
\end{equation}
that is, $\frakF$ consists of all virtually cyclic subgroups which 
are subconjugate to one of the peripheral subgroups of $G$ together 
with all finite subgroups of $G$. This is clearly a full family of 
subgroups of $G$.

Lafont and Ortiz have shown in~\cite[p.~532f.]{lafont-07} using
results of Osin that if $G$ is relatively hyperbolic in the sense
of Bowditch~\cite{bowditch-99} that the following is 
true:
\begin{enumerate}
    \item  $\frakF\cap H\subset \Ffin(H)$ for every $H\in 
    \Fvc(G)\setminus \frakF$;

    \item  every $H\in \Fvc(G)\setminus \frakF$ is contained in a 
    unique maximal $H_{\max}\in \Fvc(G)$;

    \item  $N_{G}(H_{\max}) = H_{\max}$ for every $H\in 
    \Fvc(G)\setminus \frakF$.
\end{enumerate}
The definition of relative hyperbolicity in~\cite{osin-06a} extends
the definition of relative hyperbolicity in~\cite{bowditch-99}. 
Furthermore, the proof in~\cite{lafont-07} of the above result is 
also correct for relatively hyperbolic groups in the sense 
of~\cite{osin-06a}. Therefore we can apply Proposition~\ref{prop:JPL} 
in the current setting. That is, one can obtain a model for $\uu EG$ 
by attaching orbits of at most $2$\=/dimensional cells to any model
for 
$E_{\frakF}G$.

Lafont and Ortiz have constructed in~\cite{lafont-07} a model for $\uu
EG$ for relatively hyperbolic groups in the sense~\cite{bowditch-99}
by forming the join $X * Y$ where~$X$ is a model for $\underline EG$
and~$Y$ is the disjoint union of models for $\uu EH_{\lambda}$,
$\lambda\in \Lambda$, and a set of discrete points.  Their
construction is also valid for relatively hyperbolic groups in the
sense of~\cite{osin-06a} and from the join construction one obtains
\begin{align*}
    \dim (X*Y) & = \dim(X) + \dim(Y) + 1
    \\
    & = 
    \ugd G + \sup\{ \uugd H_{\lambda} : \lambda\in \Lambda\} + 1
\end{align*}
This is the lowest dimension one can achieve with Lafont and Ortiz's 
construction. There is no example known where $\uugd G>\ugd G+1$ and 
this has raised the question whether the bound $\uugd G\leq \ugd G+1$ 
for every group~$G$, see~\cite[p.~500]{luck-12}. Thus, if the 
peripheral subgroups contain groups which are not virtually cyclic, 
then the dimension of $X*Y$ is strictly larger than~$\ugd G+1$ and 
suggests that in this case $X*Y$ is not a model of minimal dimension.

However, if a nice model is known for $E_{\frakF}G$, where $\frakF$ is
as in~\eqref{eq:F-rel-hyp}, then Proposition~\ref{prop:JPL} can give a
model of minimal dimension for $\uu EG$.  We conclude with an example
where we can construct a nice model for $E_{\frakF}G$ such that we
obtain a model for $\uu EG$ of minimal dimension.

Let $G$ be a free product
\begin{equation*}
    G \defeq  H_{1} * \cdots * H_{n}
\end{equation*}
of finitely many groups $H_{i}$.  It follows straight from the
definition in~\cite{osin-06a} that $G$ is relatively hyperbolic with
respect to the factors $H_{i}$, $i=1,\ldots,n$.  For simplicity we
assume in the following that $n=2$ and in order to avoid triviality we
assume that $G$ is  not virtually cyclic.

Since $G$ is not virtually cyclic it follows that $G$ has free
subgroup of rank~$2$.  Thus $\uugd G\geq 2$ and since also $G$
contains $H_{1}$ and $H_{2}$ as subgroups we get altogether
\begin{equation*}
    \uugd G \geq \max(\uugd H_{1}, \uugd H_{2}, 2).
\end{equation*}

Similarly to Example~4.10 in~\cite[p.~290]{luck-05} we construct a
$G$-CW-complex $X$ which is obtained from the Bass--Serre tree $T$ 
associated with the free product $H_{1}*H_{2}$ by replacing the 
vertices $v$ of $T$ equivariantly by models for $\uu EH_{1}$ and $\uu 
EH_{2}$. More precisely, for $i=1,2$ choose once and for all 
$x_{i}\in X_{i}$ where $X_{i}$ is a model for $\uu EH_{i}$ and define 
$G$-equivariant maps
\begin{align*}
    F_{i}\: & G\to G \times_{H_{i}} X_{i},
    \\
    & g\mapsto [g,x_{i}],
\end{align*}
where $G\times_{H_{i}} X_{i}$ denotes the $G$-space induced from the
$H_{i}$-space $X_{i}$.  We obtain $X$ as a $G$-equivariant cellular
pushout
\begin{equation*}
    \dgARROWLENGTH=1.8cm
    \begin{diagram}
        \node{G\times \{0,1\}}
	\arrow{e,t}{F_{1} \coprod F_{2}}
	\arrow{s}
	\node{\bigl( G \times_{H_{1}} X_{1}\bigr) {\textstyle \coprod}
	\bigl(G\times_{H_{2}} X_{2}\bigr)}
	\arrow{s}
	\\
	\node{G\times [0,1]}
	\arrow{e}
	\node{X}
    \end{diagram}
\end{equation*}
It follows that $X$ is a model for $E_{\frakF}G$ where
\begin{equation*}
    \frakF\defeq  \{ H^{g} : H\in \Fvc(H_{1})\cup \Fvc(H_{2}) \text{ and }
    g\in G \},
\end{equation*}
that is $\frakF$ is the family of all virtually cyclic subgroups of
$G$ which are subconjugate to one of the on of the factors $H_{1}$ or
$H_{2}$.  Since any finite subgroup of $G$ is conjugate to one of the
factors $H_{1}$ or $H_{2}$ we have that $\frakF$ includes all finite
subgroups of $G$~\cite[p.~36]{serre-80}.  Thus this family agrees with
the family defined in~\eqref{eq:F-rel-hyp}.  By construction we have
\begin{equation*}
    \dim X = \max(\uugd H_{1}, \uugd H_{2}, 1).
\end{equation*}
Now we can apply Proposition~\ref{prop:JPL} to obtain a model $Z$ for
$\uu EG$ by attaching to $X$ orbits of cells in dimension~$2$ and
less.
Thus
\begin{equation*}
    \uugd G \leq \dim Z = \max(\uugd H_{1}, \uugd H_{2}, 2).
\end{equation*}

\begin{theorem}
    Let $G \defeq  H_{1} * \cdots * H_{n}$ be a free product of finitely
    many groups.  If $G$ is not virtually cyclic then
    \begin{equation*}
        \uugd G = \max(\uugd H_{1}, \ldots, \uugd H_{n}, 2)
    \end{equation*}
\end{theorem}

\begin{proof}
    This statement follows either by adapting the above construction
    to general values of~$n$.  Alternatively one can prove it by
    induction on $n$ and using the fact $ H_{1} * \cdots * H_{n}
    \isom (H_{1} * \cdots * H_{n-1}) * H_{n}.  $
\end{proof}

%
%

\chapter{Groups with Low Bredon Dimension
for~the~Family~$\Fvc$}
\label{ch:low-dimensions}

%
%

\section{Groups $G$ with $\protect\uu{\gd} G=0$}

The classification of groups with $\uugd G=0$ is a straight forward
consequence of Proposition~\ref{prop:gdG=0} and
Proposition~\ref{prop:cdG=0} applied to the family~$\frakF = \Fvc(G)$.

\begin{proposition}
    \label{prop:case0}
    Let $G$ be a group. Then $\uugd G=0$ if and only if $\uucd G = 
    0$ if and only if $G$ is virtually cyclic.\qed
\end{proposition}

%
%

\section{Groups $G$ with $\protect\uu{\gd} G=1$}

\begin{proposition}
    \label{prop:case1-not-fg}
    Let $G$ be a group with $\uugd G=1$. Then $G$ is not finitely 
    generated.
\end{proposition}

\begin{proof}
    By assumption $G$ has a tree $T$ as a model for $\uu E G$.  Assume
    towards a contradiction that $G$ is finitely generated.  For every
    cyclic subgroup~$\langle g\rangle$ of $G$ the fixed point set
    $T^{\langle g\rangle} \neq \emptyset$ since $T$ is a model for
    $\uu E G$.  Hence every element of $G$ has fixed points and
    Corollary~3 to Proposition~25 in~\cite[pp.~64f.]{serre-80} implies
    that $T^{G}\neq \emptyset$.  This can only happen if $G$ is
    virtually cyclic.  Then Proposition~\ref{prop:case0} implies that
    $\uugd G=0$, which is a contradiction to the assumption that
    $\uugd G=1$.  Therefore $G$ cannot be finitely generated.
\end{proof}

\begin{corollary}
    \label{cor:case1-not-fg-1}
    A group $G$ with $\uugd G=1$ is locally virtually cyclic.
\end{corollary}

\begin{proof}
    If $H$ is a finitely generated subgroup of $G$ then $\uugd H\leq
    \uugd G = 1$.  Then Proposition~\ref{prop:case1-not-fg} implies
    $\uugd H\neq 1$ and therefore we must have $\uugd H=0$.  Hence 
    $H$ is virtually cyclic by Proposition~\ref{prop:case0}.
\end{proof}

\begin{corollary}
    \label{cor:case1-not-fg-2}
    If $G$ is a group with $\uugd G=1$, then $\uucd G = 1$ and 
    $\uuhd G =0$.
\end{corollary}

\begin{proof}
    This is true for every locally virtually cyclic group by
    Corollary~\ref{cor:dim-locally-F}.
\end{proof}

Using a result of Lück and Weiermann, we can now prove the following
classification of countable groups $G$ with $\uugd G=1$.

\begin{proposition}
    \label{prop:countable-uugdG1}
    Let $G$ be a countable group. Then $\uugd G =1$ if and only 
    if~$G$ is locally virtually cyclic but not virtually cyclic.
\end{proposition}

\begin{proof}
    The ``only if'' part is covered by
    Corollary~\ref{cor:case1-not-fg-1}.
    
    Conversely, assume that $G$ is locally virtually cyclic.  Since
    $G$ is countable, it has only countably many virtually cyclic
    subgroups and the claim follows from Lemma~4.2 and Theorem~4.3
    in~\cite[pp.~511f.]{luck-12} together with
    Proposition~\ref{prop:case0}.
\end{proof}

A natural question which arises is the following: does $\uucd G=1$
imply $\uugd G = 1$?  If not, under which conditions on the group
$G$ does this implication hold?

\begin{theorem}
    \label{thrm:classification}
    Let $G$ be a countable, torsion-free, soluble group. Then
    \begin{equation*}
        \uucd G = 1 \iff \uugd G = 1.
    \end{equation*}
\end{theorem}

\begin{proof}
    ``$\Leftarrow$'': \quad This is
Corollary~\ref{cor:case1-not-fg-2}.
    
    \smallskip
    
    ``$\Rightarrow$'': Theorem~\ref{thrm:cdF-G-vs-cdG-G} implies that
    $\ucd G\leq \uucd G + 1 = 2$.  Since $G$ is assumed to be torsion
    free, we have that $\cd G = \ucd G$ and thus $\cd G\leq 2$.
    
    Now $\cd G=0$ if and only if $G$ is trivial, and in this case
    $\uucd G=0$ which is a contradiction.  Furthermore $\cd G=1$ if
    and only if $G$ is a free group.  Since free groups of rank
    greater or equal to two are not soluble, $G$ must necessarily be
    cyclic and in this case we obtain the contradiction $\uucd G=0$.
    Thus we must have that $\cd G=2$.
        
    By the classification of soluble groups of cohomological dimension
    $2$ due to Gildenhuys~\cite{gildenhuys-79} lists the following
    possibilities for $G$:
    \begin{enumerate}
	\item $G\isom BS(1,m)$ for some integer $m\neq 0$;
	
	\item  $G$ is isomorphic to a non-cyclic subgroup of $\Q$.
    \end{enumerate}
    
    In the first case we have $\uucd G\geq 2 $ by
    Theorem~\ref{thrm:dim-soluble-BS}.  However, this contradict the
    assumption $\uucd G=1$.  Thus $G$ must be isomorphic to a
    non-cyclic subgroup of $\Q$.  In this case $G$ is locally
    virtually cyclic but not virtually cyclic.  Thus $\uugd G=1$ by
    Proposition~\ref{prop:countable-uugdG1}.
\end{proof}

%
%

\section{Groups $G$ with $\protect{\uugd} G=2$ or $\protect{\uugd} 
G=3$}

There is not much known about which groups $G$ have $\uugd G=2$, and
even less about which groups $G$ have $\uugd G=3$.  We conclude with a
summary of the results obtained in this thesis for groups which 
belong to this class of groups:

\begin{enumerate}
    \item Let $G$ be a Gromov-hyperbolic group with $\ugd G\leq 2$.
    If $G$ is not virtually cyclic, then 
    \begin{equation*}
	\uuhd G = \uucd G = \uugd G = 2
    \end{equation*}
    by Proposition~\ref{prop:dim-gromov-hyperbolic-groups}.  In
    particular this includes the cases where $G$ is a free group of
    rank at least~$2$ (Corollary~\ref{cor:dimvcF=2}) and where $G$ is
    the fundamental group of a finite graph of finite groups
    (Proposition~\ref{prop:dimvc-finite-graphs}).
    
    \item If $G \isom \Z[1/m]\rtimes \Z$ is a soluble
    Baumslag--Solitar group, $|m|\neq 1$, then we have by
    Theorem~\ref{thrm:dim-soluble-BS}
    \begin{equation*}
        \uuhd G = \uucd G = \uugd G = 2.
    \end{equation*}

    \item For any virtually polycyclic group $G$ with $\vcd G = 2$ we
    have
    \begin{equation*}
        \uuhd G = \uucd G = \uugd G = 3
    \end{equation*}
    by Proposition~\ref{prop:dimvc-polycyclic-1}.  In particular this
    holds for $\Z^{2}$ and $\Z\rtimes \Z$.
\end{enumerate}

In particular the above cases are not counter examples for the 
Eilenberg--Ganea Conjecture for Bredon cohomology with respect to the 
family of virtually cyclic subgroups 
(cf.~Section~\ref{sec:cd-vs-hd-vs-gd} in 
Chapter~\ref{ch:bredon-dimensions}).
    
Furthermore, if $G$ is the restricted wreath product $A\wr \Z$ where
$A$ is a non-trivial, countable, locally finite group, then we have
seen in Example~\ref{ex:wreath-product} in the previous chapter that
\begin{equation*}
    \uugd G = 2.
\end{equation*}
In particular this is true for the Lamplighter group $L = \Z_{2}\wr 
\Z$.


\cleardoublepage

\bibliographystyle{alpha}
\bibliography{math}

\end{document}